\documentclass[11pt,reqno]{amsart}
\usepackage{graphicx}
\usepackage{color}
\usepackage{framed}
\usepackage{extarrows}%%%等号不等号上面加东西
\usepackage{amsmath,amssymb,amsthm,amsfonts,mathrsfs,stmaryrd}
\usepackage{cases,indentfirst}
\usepackage{multicol}
\usepackage[colorlinks=true, linkcolor=black,filecolor=black,
    urlcolor=black,citecolor=black]{hyperref}
\makeatletter

\newcommand{\Rmnum}[1]
{\expandafter\@slowromancap\romannumeral #1@}
\makeatother
\newtheorem{theorem}{Theorem}[section]

\newtheorem{lemma}{Lemma}[section]
\newtheorem{corollary}{Corollary}[section]
\newtheorem{proposition}{Proposition}[section]
\newtheorem{remark}{Remark}[section]
\usepackage[numbers,sort&compress]{natbib}
\newcommand{\cala}{\Lambda}
\theoremstyle{definition}
\numberwithin{equation}{section}

\allowdisplaybreaks[4]%best4
\begin{document}

%\author[J.A. Carrillo]{Jos\'{e} A. Carrillo}\address{Jos\'{e} A. Carrillo\newline\indent Mathematical Institute\newline\indent University of Oxford, Oxford OX2 6GG, UK }
%\email{carrillo@maths.ox.ac.uk}

\author[G.-Y. Hong]{Guangyi Hong}\address{Guangyi Hong\newline\indent School of Mathematics\newline\indent South China University of Technology\newline\indent Guangzhou 510641, P. R. China}
\email{magyhong@scut.edu.cn}

\author[Z.-A. Wang]{Zhi-an Wang\textsuperscript{$ \ast $}}\thanks{$^{\ast}$Corresponding author.}\address{Zhi-an Wang\newline\indent Department of Applied Mathematics\newline\indent Hong Kong Polytechnic University\newline\indent Hung Hom,
Kowloon, Hong Kong,  P. R. China}\email{mawza@polyu.edu.hk}
%\author{Changjiang Zhu}\address{School of Mathematics\newline\indent South China University of
%Technology\newline\indent Guangzhou, 510641, P. R. China}
%\email[C.-J. Zhu]{machjzhu@scut.edu.cn}

\title[Convergence of boundary layers]{\bf Convergence of boundary layers of chemotaxis models with physical boundary conditions~II: Non-degenerate\\
initial data}
%\title[Steady states and stability of a hyperbolic-parabolic system]{\bf Steady states and stability of a hyperbolic-parabolic system modeling vascular network formation}
\begin{abstract}
This paper establishes the convergence of boundary-layer solutions of the consumption type Keller-Segel model with non-degenerate initial data subject to physical boundary conditions, which is a sequel of \cite{Corrillo-Hong-Wang-vanishing} on the case of degenerate initial data.
Specifically, we justify that the solution with positive chemical diffusion rate $\varepsilon>0 $ converges to the solution with zero diffusion $\varepsilon=0 $ (outer-layer solution) plus the boundary-layer profiles (inner-layer solution) for any time $t>0$ as $ \varepsilon \rightarrow 0 $.  Compared to \cite{Corrillo-Hong-Wang-vanishing}, the main difficulty in the analysis is the lack of regularity of the outer- and boundary-layer profiles since only the zero-order compatibility conditions for the leading-order boundary-layer profiles can be fulfilled with non-degenerate initial data. Our new strategy is to regularize the boundary-layer profiles with carefully designed corner-corrector functions and approximate the low-regularity leading-order boundary-layer profiles by higher-regularity profiles with regularized boundary conditions. By using delicate weight functions involving boundary-layer profiles to cancel the multi-scaled linear terms in the perturbed equations, we manage to obtain the requisite uniform-in-$ \varepsilon $ estimates for the convergence analysis. This cancellation technique enables us to prove the convergence to boundary-layer solutions for any time $ t >0 $, which is different from the convergence result in \cite{Corrillo-Hong-Wang-vanishing} which  holds true only for some finite time depending on the Dirichlet boundary value.
%Our result seems to be the first one concerning the effect of initial layers on the convergence of boundary layers for the Keller-Segel model without laborious analysis on initial layers. We hope that it may shed some light on the understanding of the phenomenon observed in the experiment by Tuval et al. \cite{tuval2005bacterial}.
%The main idea in the proof is to balance the source terms by the technique of taking anti-derivatives and capture the dissipation of the model by deriving some weighted energy estimates for the $ L ^{2} $-level norms of the solution and some delicate estimates on the stationary profile.

%We shall also emphasize that no smallness assumption upon the variation of $ \bar{\phi} $ is needed for the proof of the existence and uniqueness of the steady solution.

\vspace*{4mm}
\noindent{\sc 2020 Mathematics Subject Classification. }35B40, 35K61, 35Q92, 76D10  

\vspace*{1mm}
 \noindent{\sc Keywords. }Boundary layers; chemotaxis; linear sensitivity; corner-corrector functions.
\end{abstract}
\maketitle

 \section{Introduction} % (fold)
 \label{sec:introduction}
 This paper is concerned with the following Keller-Segel chemotaxis system
\begin{gather}\label{multi-D-model}
\displaystyle \begin{cases}
\displaystyle	u _{t}=\Delta u- \nabla \cdot (u \nabla\phi(v))& \mbox{in }\Omega,\\
	\displaystyle v _{t}=\varepsilon \Delta v-u v& \mbox{in }\Omega,
\end{cases}
\end{gather}
where $ u(x,t) $ and $ v(x,t) $ denote the cell density and  chemical (signal) concentration, respectively. $ \varepsilon>0 $  represent the chemical diffusion. The function $ \phi(v) $ is referred to as the chemotactic sensitivity function accounting for the signal transduction mechanism with two prevailing prototypes: $ \phi(v)=v $ (linear sensitivity) and $ \phi(v)=\log v $ (logarithmic sensitivity).
%The former was originally introduced in \cite{keller1970initiation-produce} to model the self-aggregation of Dictyostelium discoideum in response to cyclic adenosine monophosphate (cAMP), while the latter was first derived in \cite{keller1971traveling} to model the wave
%propagation of bacterial chemotaxis.
The chemotaxis model \eqref{multi-D-model} with logarithmic sensitivity was first proposed \cite{keller1971traveling} to explain the propagation of traveling bands driven by the bacterial chemotaxis observed in the celebrated experiment of Adler \cite{Adler66}. Since then, a large number of mathematical works have been developed based on this model, such as existence and stability of traveling wave solutions \cite{KellerOdell75ms, Schw1, wang2013mathematics,LW091,LLW,peng2018nonlinear, Kyudong-vasseur-2020-M3as, ChoiK1,ChoiK2}, global well-posedness \cite{Li-pan-zhao-2012-Siam-applied, wang2021cauchy, martinez2018asymptotic}, boundary-layer solutions \cite{Carrillo-li-wang-PLms, zhaokun-2015-JDE, Hou-wang-2016-JDE,hou-liu-wang-wang-2018-SIAM, hou-wang-2019-JMPA}, just to mention a few. In contrast, the results on \eqref{multi-D-model} with linear sensitivity are much less. The early results are mainly restricted to global well-posedness of solutions (existence and stability) in bounded domains with Neumann boundary conditions (cf. \cite{fan2017global,tao2011boundedness,tao2012eventual}).

Among other things, this paper is concerned with \eqref{multi-D-model} subject to the following physical boundary conditions
 \begin{gather} \label{bd-conditions}
\displaystyle  \partial_\nu u-u \partial_\nu v =0,\ \ v =v _{\ast}\ \ \mbox{on}\ \ \partial \Omega,
\end{gather}
where $\nu$ denotes the unit outward normal vector of the boundary $\partial \Omega$, and $v_*>0$ is a constant. That is, the zero-flux boundary condition is imposed to the bacterial density while Dirichlet boundary condition to the chemical concentration. The boundary conditions \eqref{bd-conditions} first appeared in \cite{tuval2005bacterial} for a chemotaxis-fluid system describing the boundary accumulation layer formed on the drop edge (air-water interface) in a sessile drop mixed with aerobic bacterial {\it Bacillus ubtilis} undertaking chemotactic movement.
When $\phi(v)=v$, the existence of unique non-constant stationary boundary-layer solution to \eqref{multi-D-model}-\eqref{bd-conditions} was proved in \cite{lee2019boundary} in all dimensions, followed by the stability results in one dimension \cite{Hong-Wang-large-time} and in multi-dimensions \cite{li2023stability}. The existence of classical solutions in two-dimensional radially symmetric domain and weak solutions or small-data classical solutions in higher dimensions were proved in \cite{lankeit2021radial, wang2024smooth}. We also mention another result in \cite{Marcel-Lankeit} on the existence of stationary solutions for \eqref{multi-D-model} with linear sensitivity and a Robin type condition for $v$  based on Henry's law modeling the dissolution of gas in water. When $\phi(v)=\log v$, the existence and stability of the non-constant stationary solution \eqref{multi-D-model}-\eqref{bd-conditions} was obtained in \cite{Carrillo-li-wang-PLms} followed by a work \cite{song2023convergence} further giving the stabilization rate. Recently the existence of unique stationary solution of \eqref{multi-D-model}-\eqref{bd-conditions} in all dimensions for any $\varepsilon>0$ and asymptotic profile of boundary-layer solutions as $\varepsilon \to 0$ were established in \cite{carrillo2024boundary}.

The unique stationary solution of \eqref{multi-D-model}-\eqref{bd-conditions} with $\phi(v)=v$ was shown to be a boundary-layer profile as $ \varepsilon \rightarrow 0 $ with thickness of order $\sqrt{\varepsilon}$ in \cite{lee2019boundary}, which seemed to be the first rigorous result justifying the existence of boundary-layer solutions for chemotaxis models  and supported the experimental findings reported in \cite{tuval2005bacterial} in the absence of fluid dynamics apart from the numerical simulations in \cite{tuval2005bacterial,chertock2012sinking,lee2015numerical}.
However, except the results in \cite{lee2019boundary,Hong-Wang-large-time, li2023stability} on the stationary problem, the analytical results on the boundary-layer solutions of the time-dependent problem \eqref{multi-D-model}-\eqref{bd-conditions} are largely open. In particular, whether the solution of \eqref{multi-D-model}-\eqref{bd-conditions} converges as $\varepsilon \to 0$ remains poorly understood.   Recently,  the model \eqref{multi-D-model}-\eqref{bd-conditions} with linear sensitivity in a bounded interval $\mathcal{I}:=(0,1)$ was considered in \cite{Corrillo-Hong-Wang-vanishing}:
 \begin{align}\label{eq-orignal}
\displaystyle \begin{cases}
\displaystyle  u _{t}=u _{xx}- \left( u v _{x} \right)_{x}, &\ x\in \mathcal{I},\ t>0,\\
  \displaystyle v _{t}=\varepsilon v _{xx}- uv, &\ x\in \mathcal{I},\ t>0,  \\
  (u,v)(x,0)=(u _{0}, v _{0})(x), & \ x\in \overline{\mathcal{I}},
\end{cases}
\end{align}
with boundary conditions
\begin{align}\label{intial-bdary}
\begin{cases}
(u _{x}-u v _{x})\vert _{\partial \mathcal{I}}=0,\ \ v \vert_{ \partial \mathcal{I}}=v _{\ast}, & \text{if}\ \varepsilon>0,\\
(u _{x}-u v _{x})\vert _{\partial \mathcal{I}}=0, & \text{if}\ \varepsilon=0,
\end{cases}
\end{align}
where $\overline{\mathcal{I}}=[0,1]$ and $\partial \mathcal{I}=\{0,1\}$. One may intuitively suspect that the boundary-layer profile may arise from the problem \eqref{eq-orignal}-\eqref{intial-bdary} as $ \varepsilon \rightarrow 0 $. Indeed, when $ \varepsilon =0 $, the boundary value of $v$  is solely determined by directly solving the ODE $v _{t}=- uv$ of the second equation of \eqref{eq-orignal}:
\begin{equation}\label{bv}
v \vert_{\partial \mathcal{I}}=v_0 \vert_{\partial \mathcal{I}} \ {\mathop{\mathrm{e}}}^{-\int_0^t u \vert_{\partial \mathcal{I}}\mathrm{d}\tau}.
 \end{equation}
 This boundary value of $v$ may mismatch the one prescribed in \eqref{intial-bdary} for $\varepsilon>0$. If this occurs, boundary layers will usually develop and consequently $(u^\varepsilon,v^\varepsilon) \nrightarrow (u^0,v^0)$ in $L^\infty(\mathcal{I})$ as $\varepsilon\to 0$, where $(u^\varepsilon, v^\varepsilon)$ and $(u^0,v^0)$ denote the solution of \eqref{eq-orignal}-\eqref{intial-bdary} with $\varepsilon>0$ and $\varepsilon=0$, respectively. Then an immediate question is under what conditions the boundary layer will be present and  what is the limit of $(u^\varepsilon,v^\varepsilon)$ as $\varepsilon\to 0$ then.
%One purpose is to find under conditions, the boundary layer will arise from \eqref{eq-orignal} as $\varepsilon \to 0$.
In \cite{Corrillo-Hong-Wang-vanishing}, we find proper compatibility conditions warranting the development of boundary-layer solutions as $\varepsilon \to 0$ and further show that the solution of \eqref{eq-orignal}-\eqref{intial-bdary} with $\varepsilon >0$ converges to $(u^0,v^0)$ plus the boundary-layer solutions as $ \varepsilon \rightarrow 0$. However, the results of \cite{Corrillo-Hong-Wang-vanishing} have two essential confinements. One is that the analyses in \cite{Corrillo-Hong-Wang-vanishing} heavily rely on the assumption $u _{0}|_{\partial \mathcal{I}}=0 $, namely the initial value $u_0$ is degenerate at the boundary.  The other is the results of \cite{Corrillo-Hong-Wang-vanishing} hold only for $t\in (0, T_0(v_*))$ for a finite time $T_0(v_*)$ depending on $v _{\ast}>0$. This paper is a sequel to \cite{Corrillo-Hong-Wang-vanishing} by studying the case of non-degenerate initial data (i.e. $\inf_{x\in \overline{\mathcal{I}}} u _{0}>0$). In this case we show that the solution of \eqref{eq-orignal}-\eqref{intial-bdary} with $\varepsilon >0$ converges to $(u^0,v^0)$ plus the boundary-layer solutions as $\varepsilon \to 0$ for any $t>0$ irrespective of $v_*>0$.

The main challenges in studying the problem \eqref{eq-orignal}-\eqref{intial-bdary} come from the cross-diffusion (chemotaxis) term $ (uv _{x}) _{x}$ and the complex physical boundary conditions. To deal with the cross-diffusion, the previous works \cite{fan2017global,tao2011boundedness} concerning the global well-posedness of solutions with homogeneous Neumann boundary conditions rely on the construction of a Lyapunov functional. With the physical boundary conditions in \eqref{intial-bdary}, the work \cite{Hong-Wang-large-time} cleverly identified a cancellation structure by eliminating the cross-diffusion with the technique of taking anti-derivatives in one dimension and established the time-asymptotic stability of the stationary boundary-layer profile.  For the convergence of boundary-layer profiles as $\varepsilon \to 0$,  the analyses of  \cite{Corrillo-Hong-Wang-vanishing} need the degeneracy of $u_0$ to meet the compatibility conditions. However, this degeneracy incapacitates cancellation technique of \cite{Hong-Wang-large-time} for the singular linear term. Alternatively, the Hardy inequality was utilized in \cite{Corrillo-Hong-Wang-vanishing} to derive the requisite estimates for convergence analyses as $\varepsilon \to 0$, but this convergence result based on the estimates derived by the Hardy inequality holds up to a finite time depending on $v_*>0$ as mentioned above.  For the non-degenerate initial data considered in this paper, the cancellation technique can be used but the compatibility conditions can only be fulfilled at the leading-order boundary-layer profiles. This confines boundary-layer profiles to have low temporal regularity only which is inadequate for the convergence analysis. To enhance the regularity, we introduce corner-correction functions into the boundary conditions for the leading boundary-layer profile in this paper to meet the full compatibility conditions, see \eqref{cal-A-1} and \eqref{cal-A-2-T}. But introducing these functions raises new challenges to derive the uniform-in-$\varepsilon$ estimates of solutions needed for the convergence analysis, since the time derivative of the corner-correction function is singular in $\varepsilon$.  To overcome this adverse effect caused by the corner-correction functions, we employ various time-weighted functions to derive the requisite uniform-in-$\varepsilon$ estimates by requiring all the higher-order estimates of solutions  be away from the initial time to avoid the possible initial layers (see estimates in Sec. 3).  The advantage of non-degenerate initial data is that we can use cancellation technique instead of the Hardy inequality to deal with the singular linear terms and manage to establish necessary uniform-in-$\varepsilon$ estimates (see the proof of Lemma \ref{lem-L2-pertur}) so that convergence result holds for all time $t>0$, different from the convergence result in \cite{Corrillo-Hong-Wang-vanishing} which holds for $t \in [0, T_0(v_*)]$ for some finite time $T_0(v_*)$ depending on $v_*$. This is a major difference from the degenerate initial data considered in \cite{Corrillo-Hong-Wang-vanishing}. But our current convergence result can not cover the time $t=0$, which is a price paid to the case of non-degenerate initial data.

%Since the corner-correction functions rectify the boundary conditions for the leading-order boundary-layer profile to fulfill the compatibility conditions, it will change the higher-order boundary-layer profiles and affect their regularities. Here  we introduce another correction function depending on $\varepsilon$ (see \eqref{b-vfi-ve} and \eqref{b-v-ve} in Sec. 4) to homogenize the boundary value of approximate solutions to suppress the singularity of the higher-order boundary-layer profiles with respect to $\varepsilon$.

The rest of the paper is organized as follows: In Sec. \ref{sec:main_result}, we first introduce the equations for the regularized outer- and boundary-layer profiles with corner-correction functions, and then state our main results on the convergence of solutions of \eqref{eq-orignal}-\eqref{intial-bdary} as $\varepsilon \to 0$. In Sec. \ref{sec:study_on_the_inner_outer_layers}, we study the regularity and the dependence on $ \varepsilon $ of the regularized outer- and boundary-layer profiles. Finally, in Sec. \ref{sec:stability_of_boundary_layers}, we show the convergence of boundary-layer solutions and prove our main results.

\section{Statement of main results} % (fold)
\label{sec:main_result}
In this section, we first introduce the equations for the regularized outer- and boundary-layer profiles, and then state our main results on the convergence of boundary-layer solutions as $ \varepsilon \rightarrow 0 $. For clarity, we highlight some notations used throughout the paper.\\

\vspace{2mm}
\noindent {\bf Notation.}
\begin{itemize}
	\item Denote $\bar{\mathcal{I}}=[0,1] $, $ \mathbb{R}_{+}:=(0,\infty) $ and $ \mathbb{R}_{-}:=(- \infty,0) $. $ \mathbb{N} $ represents the set of non-negative integers. Let $ L ^{p} $ with $ 1 \leq p \leq \infty $ represent the Lebesgue space $ L ^{p}(\mathcal{I}) $ in which functions are defined with respect to (w.r.t.) the variable $ x \in (0,1) $. $ L _{z}^{p} $ denotes the space $ L ^{p}((0, \infty)) $ for functions defined w.r.t. $z \in (0,\infty)$ and $ L _{\xi}^{p} $ denotes $ L ^{p}((0, \infty)) $ for functions defined w.r.t. $ \xi \in (- \infty,0) $. Similarly, we denote by $  H ^{k} $, $ H _{z}^{k} $ and $ H _{\xi}^{k} $ the standard Sobolev spaces $ W ^{k,2} $ for functions defined w.r.t. $ x \in \mathcal{I} $, $ z \in (0,\infty) $ and $ \xi \in(-\infty,0) $, respectively. We also write $ L _{T}^{p} Y:=L ^{p}(0,T;Y) $~(e.g., $ L _{T}^{\infty}L _{z}^{\infty}:=L ^{\infty}(0,T;L _{z}^{\infty}) $) without confusion caused.

	\item Denote $ \langle z \rangle=\sqrt{1+z ^{2}} $ for $ z \in [0,\infty) $, and $ \langle \xi \rangle=\sqrt{1+\xi ^{2}} $ for $ \xi \in(-\infty,0] $.

	\item Denote by $ C $ a generic positive constant which may depend on the time variable but independent of $ \varepsilon $, and may vary in the context.

  \item We often use $ (\ast)_{i} $ to denote the $ i $-th equation of the system $ (\ast) $ for brevity.
\end{itemize}

\vspace{2mm}
% section main_result (end)
\subsection{Construction of outer-/boundary-layer profiles} % (fold)
\label{sub:boundary_layer_profiel}
In this subsection, we shall first reformulate the problem \eqref{eq-orignal} and then present the equations for the (regularized) outer- and boundary-layer profiles. Notice that the zero-flux boundary condition for $ u $ immediately yields the conservation of mass. That is,
\begin{align*}%\label{h}
\displaystyle \int _{\mathcal{I}}u (x,t)\mathrm{d}x = \int _{\mathcal{I}}u _{0}(x)\mathrm{d}x.
\end{align*}
Denote $ M:=\int _{\mathcal{I}}u _{0}\mathrm{d}x $, and take
\begin{gather}\label{anti-derivatives-transf}
\displaystyle \varphi= \int _{0}^{x}(u-M)\mathrm{d}y \ \ \mbox{with}\ \ \varphi(x,0)= \int _{0}^{x}(u _{0}-M)\mathrm{d}y=:\varphi _{0}(x).
\end{gather}
Then the problem \eqref{eq-orignal} can be reformulated as
\begin{subequations}\label{refor-eq}
\begin{numcases}
   \displaystyle\makebox[-3pt]{~}   \varphi _{t}=\varphi _{xx}-(\varphi _{x}+M)v _{x},\\
   \displaystyle\makebox[-3pt]{~}   v _{t}= \varepsilon v _{xx}-(\varphi _{x}+M)v,\\
         \displaystyle\makebox[-3pt]{~}   \varphi(0,t)=\varphi(1,t)=0, \ \ v(0,t)=v(1,t)=v _{\ast},\label{BD-POSITIVE-VE}\\
          \displaystyle\makebox[-3pt]{~}   (\varphi,v)(x,0)=(\varphi _{0},v _{0}).
 \end{numcases}
\end{subequations}
With the transformation \eqref{anti-derivatives-transf} motivated by the mass conservation of $u$, we define $\varphi$ as the primitive function of $u-M$ and convert the original cross-diffusion problem \eqref{eq-orignal} into the problem \eqref{refor-eq} without cross-diffusion for which the Dirichlet boundary conditions can be fully used. As mentioned before, when $ \varepsilon=0 $,  $v$ is determined by \eqref{bv}, one can not impose boundary conditions for $v$ to avoid overdetermination. Therefore the boundary condition for the case $ \varepsilon=0 $ reads:
 \begin{gather*}%\label{zero-diff-bd}
  \displaystyle \varphi(0,t)=\varphi(1,t)=0.
  \end{gather*}
We proceed to present the equations for the outer- and boundary-layer profiles of the problem \eqref{refor-eq} with small $ \varepsilon>0 $ which are precisely derived in our previous work \cite{Corrillo-Hong-Wang-vanishing} based on the  perturbation method~(cf. \cite{Grenier-1998-JDE,Holems-1995-book,Rousset-2005-JDE}). We assume that the solutions of the transformed problem \eqref{refor-eq} with $ \varepsilon>0 $ have expansions in the form:
\begin{subequations}\label{formal-expan}
\begin{align}
\displaystyle   \varphi ^{\varepsilon}(x,t)&= \sum _{j=0}^{\infty}\varepsilon ^{\frac{j}{2}}\left(  \varphi ^{I,j}(x,t)+ \varphi ^{B,j}\left(z,t \right)+ \varphi ^{b,j}\left(\xi,t \right) \right),\\
\displaystyle    v ^{\varepsilon}(x,t)&= \sum _{j=0}^{\infty}\varepsilon ^{\frac{j}{2}}\left( v ^{I,j}(x,t)+v ^{B,j}\left(z,t \right)+v ^{b,j}\left(\xi,t \right) \right)
\end{align}
\end{subequations}
for $ j \in \mathbb{N} $, where $ z $ and $ \xi $ are called the boundary-layer coordinates defined by
\begin{gather}\label{bd-layer-variable}
\displaystyle z= \frac{x}{\sqrt{\varepsilon}},\ \ \ \xi= \frac{x-1}{\sqrt{\varepsilon}},\ \ x \in [0,1],
\end{gather}
with $ \sqrt{\varepsilon} $ referred to as the boundary-layer thickness which can be formally determined by the balancing argument (cf. \cite{Holems-1995-book,Hou-wang-2016-JDE}). Furthermore, the boundary-layer profiles $ (\varphi ^{B,j}, v ^{B,j}) $ and $ (\varphi ^{b,j}, v ^{b,j}) $ are assumed to be smooth and possess the following property for $ j \geq 0 $:
\begin{gather}\label{profile-decay}
\begin{cases}
   \mbox{$ \varphi ^{B,j} $ and $ v ^{ B,j} $ decay to zero rapidly as $ z \rightarrow \infty $,}\\
  \mbox{while $ \varphi ^{b,j} $ and $ v ^{b,j} $ decay to zero rapidly as $ \xi \rightarrow - \infty$.} \end{cases}
\end{gather}
To match the boundary conditions, we assume that
\begin{gather*}
\displaystyle  \begin{cases}
   \displaystyle \varphi ^{I,j}(0,t)+\varphi ^{B,j}(0,t)=0,\ \  \varphi ^{I,j}(1,t)+\varphi ^{b,j}(0,t)=0,\ \ \ j \geq 0,\\
   \displaystyle v ^{I,0}(0,t)+v ^{B,0}(0,t)=v _{\ast},\ \  v ^{I,0}(1,t)+v ^{b,0}(0,t)=v _{\ast},\\
  \displaystyle v ^{I,j}(0,t)+v ^{B,j}(0,t)=0,\ \  v ^{I,j}(1,t)+v ^{b,j}(0,t)=0,\ \ j \geq 1,
\end{cases}
\end{gather*}
where we have neglected $(\varphi ^{b,j}(-\frac{1}{\varepsilon ^{1/2}},t),v ^{b,j}(- \frac{1}{\varepsilon ^{1/2}},t)) $ at $ x=0 $ and $(\varphi ^{B,j}(\frac{1}{\varepsilon ^{1/2}},t),v ^{B,j}(\frac{1}{\varepsilon ^{1/2}},t)) $  at $ x=1 $ based on the decay properties in \eqref{profile-decay} since $ \varepsilon $ is small. For the initial data, we assume that
\begin{gather*}
\displaystyle \begin{cases}
	\displaystyle  \varphi ^{I,0}(x,0)=\varphi _{0}(x),\ \ \varphi ^{B,0}(z,0)=\varphi ^{b,0}(\xi,0)=0,\\[1mm]
\displaystyle  v ^{I,0}(x,0)=v _{0}(x),\ \ v ^{B,0}(z,0)=v ^{b,0}(\xi,0)=0,
\end{cases}
\end{gather*}
and that
\begin{gather*}
\begin{cases}
	\displaystyle \varphi ^{I,j}(x,0) =\varphi ^{B,j}(z,0)=\varphi ^{b,j}(\xi,0)=0,\ \ j \geq 1,\\
\displaystyle v ^{I,j}(x,0) =v ^{B,j}(z,0)=v ^{b,j}(\xi,0)=0, \ \ j \geq 1.
\end{cases}
\end{gather*}
Substituting the ansatz in \eqref{formal-expan} into the equations in \eqref{refor-eq}, one can get the equations for the outer-/boundary-layer profiles for $ j \geq 0 $. For details, please see \cite{Corrillo-Hong-Wang-vanishing}. We now quote some equations governing the outer-/boundary-layer profiles derived in \cite{Corrillo-Hong-Wang-vanishing} for later use. The leading-order outer-layer profile $ (\varphi ^{I,0}, v ^{I,0}) $ satisfies the problem
\begin{align}\label{eq-outer-0}
\displaystyle \begin{cases}
	\displaystyle\varphi _{t}^{I,0}=\varphi _{xx}^{I,0}-(\varphi _{x}^{I,0}+M)v _{x}^{I,0},& x \in \mathcal{I},\ t >0,\\[1mm]
	\displaystyle v _{t}^{I,0}=-(\varphi _{x}^{I,0}+M)v ^{I,0},& x \in \mathcal{I},\ t>0,\\[1mm]
	\varphi ^{I,0}(0,t)=\varphi ^{I,0}(1,t)=0,\\[1mm]
	\displaystyle (\varphi ^{I,0}, v ^{I,0})(x,0)=(\varphi _{0}, v _{0}),& x \in \bar{\mathcal{I}},
	\end{cases}
\end{align}
which is the zero-diffusion problem of \eqref{refor-eq}. As shown in \cite{Corrillo-Hong-Wang-vanishing}, this problem admits a unique global classical solution if the initial value is compatible with boundary conditions and smooth enough. Now let us turn to the equations describing the leading-order boundary-layer profile. The boundary-layer profile $ \varphi ^{B,0} $ near the left boundary satisfies
\begin{gather*}
\begin{cases}	
\displaystyle  \varphi _{zz}^{B,0}-\varphi _{z}^{B,0}v _{z}^{B,0}=0,\ \ z \in \mathbb{R}_{+}, \\[1mm]
\varphi ^{B,0}(0,t)=0,\ \ \varphi ^{B,0}(+\infty,t)=0,\\[1mm]
\displaystyle \varphi ^{B,0}(x,0)=0,
\end{cases}
\end{gather*}
which along with \eqref{profile-decay} gives $ \varphi ^{B,0}\equiv 0 $. The boundary-layer profile $ v ^{B,0} $ near the left boundary solves
\begin{gather}\label{first-bd-layer-pro}
\displaystyle \begin{cases}
	\displaystyle v _{t}^{B,0}= v ^{B,0}_{zz}- (\varphi _{x} ^{I,0}(0,t)+M)v ^{I,0}(0,t)    ({\mathop{\mathrm{e}}}^{v ^{B,0}}-1)-(\varphi _{x} ^{I,0}(0,t)+M)    {\mathop{\mathrm{e}}}^{v ^{B,0}}v ^{B,0},\ \ z \in \mathbb{R}_{+},\\[1mm]
	\displaystyle v ^{B,0}(0,t)= v _{\ast}-v ^{I,0}(0,t),\ \ v ^{B,0}(+\infty,t)=0,\\[1mm]
		\displaystyle v ^{B,0}(z,0)=0,
\end{cases}
\end{gather}
and the first-order left boundary-layer profile $ \varphi ^{B,1} $ is determined by $ v ^{B,0} $ through
\begin{align}\label{vfi-bd-1ord-lt}
   \displaystyle \varphi ^{B,1}= -\int _{z}^{\infty}(\varphi _{x}^{I,0}(0,t)+M)\left( 		{\mathop{\mathrm{e}}}^{v ^{B,0}(y,t)}-1 \right) \mathrm{d}y .
   \end{align}
The right boundary-layer profile $ v ^{b,0} $ satisfies
\begin{gather}\label{first-bd-pro-rt}
\displaystyle \begin{cases}
	\displaystyle v _{t}^{b,0}= v ^{b,0}_{\xi \xi}- (\varphi _{x} ^{I,0}(1,t)+M)v ^{I,0}(1,t)    ({\mathop{\mathrm{e}}}^{v ^{b,0}}-1)-(\varphi _{x} ^{I,0}(1,t)+M)    {\mathop{\mathrm{e}}}^{v ^{b,0}}v ^{b,0},\ \ \xi \in \mathbb{R}_{-},\\[2mm]
	\displaystyle v ^{b,0}(0,t)= v _{\ast}-v ^{I,0}(1,t),\ \ v ^{b,0}(-\infty,t)=0,\\[2mm]
		\displaystyle v ^{b,0}(\xi,0)=0.
\end{cases}
\end{gather}
Furthermore, we have $ \varphi ^{b,0}\equiv 0 $, and $ \varphi ^{b,1} $ is given by
\begin{gather}\label{firs-bd-1-rt}
\displaystyle  \varphi ^{b,1}=\int _{- \infty}^{\xi}(\varphi _{x}^{I,0}(1,t)+M)\left( 		{\mathop{\mathrm{e}}}^{v ^{b,0}(y,t)}-1 \right)\mathrm{d}y  .
\end{gather}
To obtain the existence and regularity of solutions to  problems \eqref{first-bd-layer-pro} and \eqref{first-bd-pro-rt}, it is natural to assume $ v _{0}|_{\partial \mathcal{I}}=v|_{\partial \mathcal{I}}=v _{\ast} $. Under the scenario $\displaystyle \min _{x \in \bar{\mathcal{I}}}(\varphi _{0x}+M)>0\,(\mbox{i.e.}, \displaystyle\min _{x \in \bar{\mathcal{I}}} u _{0}>0 )$  considered in this paper, one can check that the initial data of \eqref{first-bd-layer-pro} are compatible with boundary conditions at the zero-order only
(i.e. $v ^{B,0}\vert _{t=0}(0)= v ^{B,0}(0,t)\vert _{t=0}=0$), but not at the higher order terms since
$$[\partial _{t}^{k}v ^{B,0}\vert _{t=0}](0)\ne \partial _{t}^{k}v ^{B,0}(0,t)\vert _{t=0} \ \text{for}\ k\geq 1.$$
This applies to the right boundary-layer profile $v ^{b,0}$ of \eqref{first-bd-pro-rt} as well. Consequently  the solutions to the problem \eqref{first-bd-layer-pro} (resp. \eqref{first-bd-pro-rt}) would not be smooth enough at the corner $ (0,0) $ in the $ z $-$ t $ (resp. $ \xi $-$ t $) plane, which may give rise to weak initial layers. To avoid excessive analysis on initial layers and to circumvent the difficulty caused by the lack of regularity of the boundary-layer profiles at the corner, we use a function $ v ^{B,\varepsilon} $ to approximate $ v ^{ B,0} $ by modifying the boundary conditions for $ v ^{B,0} $ with a time-dependent corner-corrector function. Specifically, $ v ^{B,\varepsilon} $ solves
\begin{gather}\label{first-bd-layer-pro-appro}
\displaystyle \begin{cases}
	\displaystyle v _{t}^{B,\varepsilon}= v ^{B,\varepsilon }_{zz}- (\varphi _{x} ^{I,0}(0,t)+M)v ^{I,0}(0,t)    ({\mathop{\mathrm{e}}}^{v ^{B,\varepsilon }}-1)-(\varphi _{x} ^{I,0}(0,t)+M)    {\mathop{\mathrm{e}}}^{v ^{B,\varepsilon }}v ^{B,\varepsilon},\ \ z \in \mathbb{R}_{+},\\[1mm]
	\displaystyle v ^{B,\varepsilon}(0,t)= v _{\ast}-v ^{I,0}(0,t)+\cala_{1}^{\varepsilon}(t),\ \ v ^{B,\varepsilon}(+\infty,t)=0,\\[1mm]
		\displaystyle v ^{B,\varepsilon}(z,0)=0,
\end{cases}
\end{gather}
where $ \cala_{1}^{\varepsilon}(t) $ is a corner-corrector function defined by
\begin{align}\label{cal-A-1}
\displaystyle  \cala _{1}^{\varepsilon}(t)&= - 2\int _{0}^{t} (\varphi _{x}^{I,0}(0,s)+M)v ^{I,0}(0,s)\int _{s}^{\infty}\frac{		{\mathop{\mathrm{e}}}^{- \frac{\sigma ^{2}}{\varepsilon ^{2 \alpha}}}}{\varepsilon ^{\alpha}\sqrt{\pi}} \mathrm{d}\sigma \mathrm{d}s
 \nonumber \\[2mm]
 &~\displaystyle \quad - 2(\varphi _{0x}(0)+M)v _{\ast} \int _{0}^{t}\chi(\frac{\tau}{\varepsilon ^{\alpha}}) \int _{0}^{\tau}\frac{   {\mathop{\mathrm{e}}}^{- \frac{\sigma ^{2}}{\varepsilon ^{2 \alpha}}}}{\varepsilon ^{\alpha}\sqrt{ \pi }} \mathrm{d}\sigma  \mathrm{d}\tau ,
 %\nonumber \\
 %& \displaystyle = - \frac{2}{\varepsilon ^{\alpha}}\int _{0}^{t}\Big[(\varphi _{x}^{I,0}(0,\tau)+M)v ^{I,0}(0,\tau)\int _{\tau}^{\infty}\frac{		{\mathop{\mathrm{e}}}^{- \frac{\sigma ^{2}}{\varepsilon ^{2 \alpha}}}}{\sqrt{\pi}} \mathrm{d}\sigma+(\varphi _{0x}(0)+M)v _{\ast} \chi(\frac{\tau}{\varepsilon}) \int _{0}^{\tau}\frac{   {\mathop{\mathrm{e}}}^{- \frac{\sigma ^{2}}{\varepsilon ^{2 \alpha}}}}{\sqrt{ \pi }} \mathrm{d}\sigma   \Big]\mathrm{d}\tau,
\end{align}
with $ \alpha>0 $ being a constant such that
\begin{align}\label{al-constriant}
\displaystyle 1< \alpha< \frac{5}{4},
\end{align}
and $ \chi(s) $ being a smooth function defined on $ [0,\infty) $ such that $ \chi(0)=1 $ and $ \chi(s)=0 $ for $ s \geq 1 $. For some properties on the corner-corrector $ \cala_{1}^{\varepsilon}(t)$, please see Lemma \ref{LEM-correc-function}. Clearly, if we assume $ v _{0}\vert _{\partial \mathcal{I}}=v _{\ast} $, the initial data for the approximate problem \eqref{first-bd-layer-pro-appro} are compatible with boundary conditions up to order two. That is, $ v ^{B,\varepsilon}(0,t)\vert _{t=0}=0 $, and $ \partial _{t}^{k}v ^{B,\varepsilon}\vert _{t=0}\,(k=1,2) $ defined by the initial data through $ \eqref{first-bd-layer-pro-appro}_{1} $ satisfy  $[\partial _{t}^{k}v ^{B,\varepsilon}\vert _{t=0}](0)= \partial _{t}^{k}v ^{B,\varepsilon}(0,t)\vert _{t=0}$. Furthermore, $ v ^{B,\varepsilon} $ is a good approximation of $ v ^{B,0} $ in the sense that
\begin{align}\label{aprr-esti-v-B}
\displaystyle  \Big\|(v ^{B,\varepsilon}-v ^{B,0})\Big(\frac{x}{\varepsilon ^{1/2}},t\Big)\Big\|_{L _{T}^{\infty}L ^{\infty}} \leq \varepsilon ^{\frac{\alpha}{2}},
\end{align}
for details, please see Appendix \ref{appen-approxi}. With $ v ^{B,\varepsilon} $, we can define the approximation of $ \varphi ^{B,1} $ as follows
\begin{align}\label{vfi-bd-1ord-lt-approxi}
   \displaystyle \varphi ^{B,\varepsilon}= -\int _{z}^{\infty}(\varphi _{x}^{I,0}(0,t)+M)\left( 		{\mathop{\mathrm{e}}}^{v ^{B,\varepsilon}(y,t)}-1 \right) \mathrm{d}y .
   \end{align}
   By \eqref{vfi-bd-1ord-lt}, \eqref{aprr-esti-v-B}, \eqref{vfi-bd-1ord-lt-approxi} and \eqref{l-INFTY-VFI-I-0}, one can deduce that
   \begin{gather}\label{apr-esti-vfi-B}
   \displaystyle  \Big\|\partial _{x}^{k} (\varphi ^{B,\varepsilon}-\varphi  ^{B,1})\Big(\frac{x}{\varepsilon ^{1/2}},t\Big)\Big\|_{L _{T}^{\infty}L ^{\infty}} \leq C\varepsilon ^{\frac{\alpha}{2}},\ \ k=0,1.
   \end{gather}
   Similarly, we approximate $( v ^{b,0},\varphi ^{b,0})  $ by $( v ^{b,\varepsilon},\varphi ^{b,\varepsilon})  $ which satisfies
\begin{gather}\label{first-bd-pro-rt-approxi}
\displaystyle \begin{cases}
	\displaystyle v _{t}^{b,\varepsilon}= v ^{b,\varepsilon}_{\xi \xi}- (\varphi _{x} ^{I,0}(1,t)+M)v ^{I,0}(1,t)    ({\mathop{\mathrm{e}}}^{v ^{b,\varepsilon}}-1)-(\varphi _{x} ^{I,0}(1,t)+M)    {\mathop{\mathrm{e}}}^{v ^{b,\varepsilon}}v ^{b,\varepsilon},\\[2mm]
	\displaystyle\varphi ^{b,\varepsilon}=\int _{- \infty}^{\xi}(\varphi _{x}^{I,0}(1,t)+M)\left( 		{\mathop{\mathrm{e}}}^{ v ^{b,\varepsilon}(y,t)}-1 \right)\mathrm{d}y  ,\\[2mm]
	\displaystyle v ^{b,\varepsilon}(0,t)= v _{\ast}-v ^{I,0}(1,t)+\cala_{2}^{\varepsilon}(t),\ \ v ^{b,\varepsilon}(-\infty,t)=0,\\[2mm]
		\displaystyle v ^{b,\varepsilon}(\xi,0)=0,
\end{cases}
\end{gather}
where $ \cala_{2}^{\varepsilon}(t) $ is given by
\begin{align}\label{cal-A-2-T}
\displaystyle  \cala_{2}^{\varepsilon}(t)&= - 2\int _{0}^{t} (\varphi _{x}^{I,0}(1,s)+M)v ^{I,0}(1,s)\int _{s}^{\infty}\frac{		{\mathop{\mathrm{e}}}^{- \frac{\sigma ^{2}}{\varepsilon ^{2 \alpha}}}}{\varepsilon ^{\alpha}\sqrt{\pi}} \mathrm{d}\sigma \mathrm{d}s
   \nonumber \\[1mm]
& \displaystyle \quad- 2(\varphi _{0x}(1)+M)v _{\ast} \int _{0}^{t}\chi(\frac{\tau}{\varepsilon ^{\alpha}}) \int _{0}^{\tau}\frac{   {\mathop{\mathrm{e}}}^{- \frac{\sigma ^{2}}{\varepsilon ^{2 \alpha}}}}{ \varepsilon ^{\alpha}\sqrt{ \pi}} \mathrm{d}\sigma  \mathrm{d}\tau ,
\end{align}
with $ \chi $ and $ \alpha $ being as in \eqref{cal-A-1}. Similar to \eqref{aprr-esti-v-B} and \eqref{apr-esti-vfi-B}, we have
\begin{gather}\label{appro-v-b-property}
\displaystyle  \Big\|(v ^{b,\varepsilon}-v ^{b,0})\Big(\frac{x}{\varepsilon ^{1/2}},t \Big)\Big\|_{L _{T}^{\infty}L ^{\infty}}+\Big\|\partial _{x}^{k} (\varphi ^{b,\varepsilon}-\varphi  ^{b,1})\Big(\frac{x}{\varepsilon ^{1/2}},t\Big)\Big\|_{L _{T}^{\infty}L ^{\infty}} \leq C\varepsilon ^{\frac{\alpha}{2}}
\end{gather}
for $ k=0,1 $. We refer the reader to Appendix B for some properties of $ \cala_{2}^{\varepsilon}(t) $ and Appendix \ref{appen-approxi} for the proof of \eqref{apr-esti-vfi-B} and \eqref{appro-v-b-property}.

Although we focus only on the convergence result for the leading-order profiles, some regularities and estimates of higher-order outer-/boundary-layer profiles are also needed in our analysis. These higher-order profiles will be constructed based on the approximate leading-order profile $ (v ^{B,\varepsilon}, \varphi ^{B,\varepsilon }) $ and $ (v ^{b,\varepsilon}, \varphi
 ^{b,\varepsilon}) $. Precisely, the first-order outer-layer profile $ (\varphi^{I,1},v ^{I,1}) $ satisfies the following problem:
\begin{align}\label{first-outer-problem}
\displaystyle \begin{cases}
	\displaystyle\varphi _{t}^{I,1}= \varphi _{xx}^{I,1}-(\varphi _{x}^{I,0}+M)v _{x}^{I,1}-\varphi _{x}^{I,1}v _{x}^{I,0},\\[1mm]
	\displaystyle v _{t}^{I,1}=-(\varphi _{x}^{I,0}+M)v ^{I,1}-\varphi _{x}^{I,1}v ^{I,0},\\[1mm]
	\displaystyle \varphi ^{I,1}(0,t) =- \varphi ^{B,\varepsilon}(0,t), \ \ \varphi ^{I,1}(1,t)=-\varphi ^{b,\varepsilon}(0,t),\\[1mm]
	\displaystyle (\varphi ^{I,1}, v ^{I,1})(x,0)=(0,0).
\end{cases}
\end{align}
With $ (\varphi ^{I,0}, v ^{I,0}) $, $ (\varphi ^{I,1}, v ^{I,1}) $ and the leading-order boundary-layer profiles at hand, we continue to construct the boundary-layer profiles $ \varphi ^{B,2} $ and $ v ^{B,1} $ as follows
\begin{align}\label{second-bd-eq}
\displaystyle \begin{cases}
\displaystyle	 -\varphi _{zz}^{B,2}+v _{z}^{B ,\varepsilon}(\varphi _{xx} ^{I,0}(0,t)z+ \varphi _{x}^{I,1}(0,t)+\varphi _{z}^{B,2})
    \\[2mm]
    \displaystyle \quad + v _{z}^{B,1}(\varphi _{x}^{I,0}(0,t)+M+\varphi _{z}^{B,\varepsilon})+\varphi _{z}^{B,\varepsilon }v _{x}^{I,0}(0,t)=0,\\[2mm]
    \displaystyle  v _{t}^{B,1}-v _{zz}^{B,1}+ (\varphi _{x}^{I,0}(0,t)+M)v ^{B,1}
+\varphi _{z}^{B,\varepsilon }(v _{x} ^{I,0}(0,t)z+v ^{I,1}(0,t)+v ^{B,1})
      \\[2mm]
     \qquad+(\varphi _{xx}^{I,0}(0,t) z+\varphi _{x} ^{I,1}(0,t))v ^{B,\varepsilon}+ \varphi _{z}^{B,2}(v ^{I,0}(0,t)+v ^{B,\varepsilon}) =0 ,\\[2mm]
     \displaystyle v ^{B,1}(0,t)=-v ^{I,1}(0,t),\ \  \varphi ^{B,2}(+\infty, t)=v ^{B,1}(+\infty,t)=0,\\[2mm]
     \displaystyle (\varphi ^{B,2},v ^{B,1})(z,0)=(0,0).
\end{cases}
\end{align}
Similarly,  $ (\varphi ^{b,2},v ^{b,1}) $ satisfies
\begin{align}\label{sec-bd-eq-rt}
\displaystyle \begin{cases}
\displaystyle	 -\varphi _{\xi \xi}^{b,2}+v _{\xi}^{ b,\varepsilon}(\varphi _{xx} ^{I,0}(1,t)\xi+ \varphi _{x}^{I,1}(1,t)+\varphi _{\xi}^{b,2})
     \\[2mm]
    \displaystyle \quad + v _{\xi}^{b,1}(\varphi _{x} ^{I,0}(1,t)+M+\varphi _{\xi}^{ b,\varepsilon})+\varphi _{\xi}^{  b,\varepsilon}v _{x} ^{I,0}(1,t)=0,\\[2mm]
    \displaystyle  v _{t}^{b,1}-v _{\xi \xi}^{b,1}+ (\varphi _{x}^{I,0}(1,t)+M)v ^{b,1}
+\varphi _{\xi}^{ b,\varepsilon}(v _{x}^{I,0}(1,t)\xi+v ^{I,1}(1,t)+v ^{b,1})
    \\[2mm]
     \qquad+(\varphi _{xx} ^{I,0}(1,t) \xi+\varphi _{x}^{I,1}(1,t))v ^{b,\varepsilon}+ \varphi _{\xi}^{b,2} (v ^{I,0}(1,t)+v ^{b,\varepsilon}) =0 ,\\[2mm]
     \displaystyle v ^{b,1}(0,t)=-v ^{I,1}(1,t),\ \  \varphi ^{b,2}(-\infty, t)=v ^{b,1}(-\infty,t)=0,\\[2mm]
     \displaystyle (\varphi ^{b,2},v ^{b,1})(\xi,0)=(0,0).
\end{cases}
\end{align}

Finally, we remark that the regularities of higher-order profiles always depend on the regularities of lower-order profiles. For example, $ \varphi ^{B,1} $ is determined by $ v ^{B,0} $ and the leading-order outer-layer profiles.  We can obtain the existence and regularity of solutions to the problem \eqref{first-outer-problem} only if $ \varphi ^{B,\varepsilon}  $ and $ \varphi ^{b,\varepsilon} $ are regular enough and satisfy certain compatibility conditions. We should emphasize that since we have approximated $ v ^{B,0} $ and $ v ^{b,0} $ by the $ \varepsilon $-dependent corner-corrector functions $ \cala_{i}^{\varepsilon}(t)\,(i=1,2) $ and constructed the higher-order profiles based on $ v ^{B,\varepsilon} $ and $ v ^{b,\varepsilon} $ determined by \eqref{first-bd-layer-pro-appro} and \eqref{first-bd-pro-rt-approxi}, respectively, the dependence on $ \varepsilon $ of the norms of these profiles will play a crucial role in proving the convergence of the leading-order boundary-layer profiles. This needs very delicate and careful analysis. We shall show the details of proving the regularities of these profiles in Sec. \ref{sec:study_on_the_inner_outer_layers}.

% subsection boundary_layer_profiel (end)

 \subsection{Statement of main results} % (fold)
 \label{sec:main_results}

 To prove the convergence of the boundary-layer profiles constructed in the last subsection, we require that the initial data satisfy some compatibility conditions at the boundary. Specifically we assume that $ (\varphi _{0},v _{0}) $ satisfies
 \begin{subequations}\label{compatibility-simple}
\begin{align}
%\displaystyle \begin{cases}
&\partial _{t}^{i}\varphi ^{I,0}\vert _{t=0}=0 \ (i=1,2,3) & \text{on}\ \partial \mathcal{I}, \label{compatibility-simple-1}\\
&v _{0}=v _{\ast}  & \text{on}\ \partial \mathcal{I},\label{compatibility-simple-2}
%\end{cases}
 %\end{gather}
 \end{align}
\end{subequations}
 %we postulate that $ (\partial _{t}^{i}\varphi ^{I,0}, \partial _{t}^{i}v ^{I,0})\vert _{t=0} ~(i=1,2,3)$,
where $\partial _{t}^{i}\varphi ^{I,0}\vert _{t=0}$ can be inductively derived from the equations in \eqref{eq-outer-0} as
\begin{gather}\label{compatibility-vfi0}
\begin{cases}
\partial _{t}\varphi ^{I,0}\vert _{t=0}:=\varphi _{0xx}-(\varphi _{0x}+M)v _{0x}, \\[2mm]
\partial _{t}^2\varphi ^{I,0}\vert _{t=0}:=(\partial _{t}\varphi^{I,0}\vert _{t=0})_{xx}+ (\varphi _{0x}+M)((\varphi _{0x}+M)v _{0}) _{x}-(\partial _{t}\varphi^{I,0}\vert _{t=0})_{x}v _{0x}, \\[2mm]
\partial _{t}^3\varphi ^{I,0}\vert _{t=0}:=(\partial _{t}^{2}\varphi ^{I,0}\vert _{t=0})_{xx}-(\partial _{t}^{2}\varphi ^{I,0} \vert _{t=0}) _{x}v _{0x}+2(\partial _{t}\varphi ^{I,0}\vert _{t=0}) _{x}((\varphi _{0x}+M )v _{0})_{x}\\[2mm]
\qquad \qquad \qquad +(\varphi _{0x}+M)((\partial _{t}\varphi ^{I,0}\vert _{t=0})_{x}v _{0}) _{x}-(\varphi _{0x}+M)((\varphi _{0x}+M)^{2}v _{0}) _{x}.
\end{cases}
 \end{gather}
 %with
 %\begin{align*}
 %\displaystyle \partial _{t}v ^{I,0}\vert _{t=0}= -(\varphi _{0x}+M)v_0,\ \
  % \displaystyle \partial _{t}^2v ^{I,0}\vert _{t=0}&=[-\varphi _{0xxx}+((\varphi _{0x}+M)v _{0x}) _{x}+ (\varphi _{0x}+M)^{2}]v_0.
 %\end{align*}
We say that the initial value $\varphi ^{I,0}|_{t=0}$ of the problem \eqref{eq-outer-0} is compatible with boundary conditions up to order three if it fulfills \eqref{compatibility-simple-1}, while the initial values of problem \eqref{first-bd-layer-pro} and \eqref{first-bd-pro-rt} are compatible with boundary conditions at the zero-th order if \eqref{compatibility-simple-2} holds (i.e. the initial values of leading-order boundary-layer profiles and boundary conditions are equal at the corner $(0,0)$).

To proceed, let $ \nu \in (0,\frac{1}{4}) $ be a constant such that
\begin{align}\label{nu-constriant}
\displaystyle  1+\nu>\alpha,
\end{align}
and denote
\begin{align}\label{iota0-defi}
\displaystyle \iota _{0}= \min\left\{ \frac{3}{4}-\nu, \frac{\alpha}{2}, 1+\frac{\nu - \alpha}{2}, 1-\frac{2}{3}\nu, \frac{5}{4}-\frac{\alpha}{2} \right\},
\end{align}
where the constant $ \alpha $ is given in \eqref{al-constriant}. Our main results on the convergence of boundary layer solutions for the reformulated problem \eqref{refor-eq} are given below.
\begin{theorem}\label{thm-stabi-refor}
Assume that $ (\varphi  _{0},v _{0})\in H ^{7} \times  H ^{7} $ and $ (\sqrt{v _{0}})_{x}\in L ^{2} $ with $ \varphi _{0x}+M>0 $ and $ v _{0}\geq 0 $ satisfying \eqref{compatibility-simple}. Then for any $ T>0 $, there exists a constant $ \varepsilon _{T}>0 $ with $ \displaystyle \lim _{T \rightarrow \infty}\varepsilon _{T}=0 $ such that if $ \varepsilon \leq \varepsilon _{T} $, the problem \eqref{refor-eq} admits a unique solution $ (\varphi ^{\varepsilon},v ^{\varepsilon}) \in L  ^{\infty}(0,T;H ^{2}\times H ^{2})$ satisfying:
\begin{align}
&\|\varphi ^{\varepsilon}(\cdot,t)-\varphi ^{I,0}(\cdot,t)\|_{L^\infty} \leq C \varepsilon ^{\frac{3 \iota _{0} }{2}- \frac{3}{8}},\label{vfi-VI-0}\\[2mm]
%t ^{\frac{1}{2}}\varphi ^{\varepsilon}(x,t)&=t ^{\frac{1}{2}}\varphi ^{I,0}(x,t)+ \varepsilon ^{\frac{1}{2}}t ^{\frac{1}{2}}\left[ \varphi ^{I,1}(x,t) +\varphi ^{B,1} \left( z,t\right) +\varphi ^{b,1}\left(\xi ,t\right) \right] +O(\varepsilon ^{\frac{3 \iota _{0} }{2}- \frac{3}{8}}),\\[2mm]
&t ^{\frac{5}{4}}\|\varphi _{x}^{\varepsilon}(\cdot,t)-\varphi _{x}^{I,0}(\cdot,t)-\big(\varphi _{z}^{B,1}(\cdot,t)+\varphi _{\xi}^{b,1}(\cdot,t) \big)\|_{L^\infty}\leq C \varepsilon ^{2 \iota _{0}-1},\label{vfi-VI-01}  \\[2mm]
&\|v ^{\varepsilon}(\cdot,t)-v ^{I,0}(\cdot,t)-\big(v ^{B,0}(\cdot,t) +v ^{ b,0}(\cdot,t)\big)\|_{L^\infty}\leq C \varepsilon ^{\iota _{0}- \frac{1}{4}},
      \label{vfi-VI-02}
%\displaystyle   \displaystyle  \varphi ^{\varepsilon}(x,t)&=\varphi ^{I,0}(x,t)+O(\varepsilon ^{\frac{3 \iota _{0} }{2}- \frac{3}{8}}),\label{vfi-VI-0}\\[2mm]
%t ^{\frac{1}{2}}\varphi ^{\varepsilon}(x,t)&=t ^{\frac{1}{2}}\varphi ^{I,0}(x,t)+ \varepsilon ^{\frac{1}{2}}t ^{\frac{1}{2}}\left[ \varphi ^{I,1}(x,t) +\varphi ^{B,1} \left( z,t\right) +\varphi ^{b,1}\left(\xi ,t\right) \right] +O(\varepsilon ^{\frac{3 \iota _{0} }{2}- \frac{3}{8}}),\\[2mm]
%  \displaystyle t ^{\frac{5}{4}}\varphi _{x}^{\varepsilon}(x,t)&=t ^{\frac{5}{4}}\varphi _{x}^{I,0}(x,t)+ t ^{\frac{5}{4}}\big(\varphi _{z}^{B,1}(z,t)+\varphi _{\xi}^{b,1}(\xi,t) \big)+O(\varepsilon ^{2 \iota _{0}-1}),\nonumber
%   \\[2mm]
%     \displaystyle v ^{\varepsilon}(x,t)&=v ^{I,0}(x,t)+v ^{B,0} \left( z,t \right) +v ^{ b,0}\left( \xi,t \right)+ O(\varepsilon ^{\iota _{0}- \frac{1}{4}}),
%      \nonumber
\end{align}
with $z:=\frac{x}{\varepsilon ^{1/2}},\ \ \xi:= \frac{x-1}{\varepsilon ^{1/2}} $,
where $ (u ^{I,0}, v ^{I,0}) $, $ v ^{B,0} $ and $ v ^{b,0} $ are solutions to the problems \eqref{eq-outer-0}, \eqref{first-bd-layer-pro} and \eqref{first-bd-pro-rt}, respectively; $\varphi ^{I,1}  $ is determined by \eqref{first-outer-problem}; $ \varphi ^{B,1} $ and $ \varphi ^{b,1} $ are given by \eqref{vfi-bd-1ord-lt} and \eqref{firs-bd-1-rt}, respectively; $C>0$ is a constant depending on $T$ but independent of $\varepsilon$.
\end{theorem}
\begin{remark}
From \eqref{al-constriant} and \eqref{iota0-defi}, it follows that
\begin{align}\label{range-iota0}
\displaystyle \frac{1}{2}< \iota _{0} <\frac{7}{12}.
\end{align}
Indeed by $ 0< \nu<1/4 $ and  $1< \alpha<1+\nu $, we can easily get $ 1/2< \iota _{0} $. If $ \iota _{0} \geq \frac{7}{12} $, the fact $ 3/4- \nu \geq \iota _{0} $ gives $ \nu \leq 1/6 $ while the fact  $ \frac{1+\nu}{2}>\frac{\alpha}{2} \geq\iota _{0} $ yields $ \nu>1/6 $. Then a contradiction arises  and thus $\iota _{0} \leq \frac{7}{12} $. In particular, if we take $ \alpha=\frac{11}{10}  $, $ \nu=\frac{1}{5} $, then $ \iota _{0}= \frac{11}{20}$.

\end{remark}
\vspace{2mm}

Based on the transformation \eqref{anti-derivatives-transf}, we can transfer the results of Theorem \ref{thm-stabi-refor} to the original problem \eqref{eq-orignal}. Precisely, let
\begin{gather}
\displaystyle u ^{\varepsilon} = \varphi _{x}^{\varepsilon}+M,\ \ u ^{I,0}=\varphi _{x}^{I, 0}+M,
\end{gather}
with $ \varphi ^{\varepsilon} $ and $ \varphi ^{I,0} $ being the solutions to the problem \eqref{refor-eq} and the problem \eqref{eq-outer-0}, respectively. Then $ (u ^{\varepsilon}, v ^{\varepsilon}) $ and $ (u ^{I, 0}, v ^{I, 0}) $ solve the problem \eqref{eq-orignal} for $ \varepsilon>0 $ and $ \varepsilon=0 $, respectively. Moreover, with \eqref{vfi-bd-1ord-lt} and \eqref{firs-bd-1-rt}, we have
\begin{equation}\label{bb}
\begin{aligned}
u ^{B,0}\left(z,t \right):=&\varphi _{z}^{B,1}(z,t)=(\varphi _{x}^{I,0}(0,t)+M)(		{\mathop{\mathrm{e}}}^{v ^{B,0}(z,t)}-1) ,\\
u ^{b,0}\left(\xi,t \right):=&\varphi _{\xi}^{b,1}(\xi,t)=
(\varphi _{x}^{I,0}(1,t)+M)({\mathop{\mathrm{e}}}^{v ^{b,0}(\xi,t)}-1).
\end{aligned}
\end{equation}
Then the results on the original problem are stated in the following.
\begin{theorem}\label{thm-original}
Assume that $ (u _{0},v _{0})\in H ^{6} \times  H ^{7} $ and $ (\sqrt{v _{0}})_{x}\in L ^{2} $ satisfying $\min _{x \in \bar{\mathcal{I}}} u _{0}>0 $ and \eqref{compatibility-simple} with $ \varphi _{0}= \int _{0}^{x}(u _{0}-M)\mathrm{d}y$ and $ M= \int _{\mathcal{I}}u _{0}\mathrm{d}x $. Then for any $ T>0 $, there exists a constant $ \varepsilon _{T}>0 $  with $ \displaystyle\lim _{T \rightarrow \infty}\varepsilon _{T}=0 $ such that if $ \varepsilon \leq \varepsilon _{T} $, the problem \eqref{eq-orignal}-\eqref{intial-bdary} admits a unique solution $ (u ^{\varepsilon},v ^{\varepsilon}) \in L  ^{\infty}(0,T;H ^{1}\times H ^{2})$ on $ [0,T] $, which satisfies
\begin{align}
&t ^{\frac{5}{4}}\|u ^{\varepsilon}(\cdot,t)-u ^{I,0}(\cdot,t)-\big(v ^{B,0}(\cdot,t) +u ^{ b,0}(\cdot,t)\big)\|_{L^\infty}\leq C \varepsilon ^{2\iota _{0}- 1}, \label{u-converg} \\[2mm]
&\|v ^{\varepsilon}(\cdot,t)-v ^{I,0}(\cdot,t)-\big(v ^{B,0}(\cdot,t) +v ^{ b,0}(\cdot,t)\big)\|_{L^\infty}\leq C \varepsilon ^{\iota _{0}- \frac{1}{4}}, \nonumber
%& t ^{\frac{5}{4}} u ^{\varepsilon}(x,t)=t ^{\frac{5}{4}} u ^{I,0}(x,t)+  t ^{\frac{5}{4}} u^{B,0}(z,t)+t ^{\frac{5}{4}} u ^{b,0}(\xi,t) +O(\varepsilon ^{2 \iota _{0}-1}),\label{u-converg}
%   \\[2mm]
%& v ^{\varepsilon}(x,t)=v ^{I,0}(x,t)+v ^{B,0}(z,t)  +v ^{ b,0}\left(\xi,t \right)+ O(\varepsilon ^{\iota _{0}- \frac{1}{4}}),
%      \nonumber
\end{align}
where $z:=\frac{x}{\varepsilon ^{1/2}},\ \ \xi:= \frac{x-1}{\varepsilon ^{1/2}} $, and $C>0$ is a constant depending on $T$ but independent of $\varepsilon$. Here $ (u ^{I,0}, v ^{I,0})(x,t) $, $ v ^{B,0}(z,t) $ and $ v ^{b,0}(\xi,t) $ are solutions to the problems \eqref{eq-outer-0}, \eqref{first-bd-layer-pro} and \eqref{first-bd-pro-rt}, respectively; $ u ^{B,0}(z,t) $ and $ u ^{b,0}(\xi,t) $ are given in \eqref{bb}.

\end{theorem}
\begin{remark}
From \eqref{u-converg}, for any $v_*>0$, we get  that
\begin{align*}
\|u ^{\varepsilon}(x,t)-u ^{I,0}(x,t)-   (u^{B,0}(z,t)+ u ^{b,0}(\xi,t))\|_{L^\infty}\leq C \varepsilon ^{2 \iota _{0}-1}, \ \forall \ 0< t \leq T
\end{align*}
where $T>0$ is independent of $v_*$ and can be arbitrarily large. That is, the convergence of the boundary-layer profile $u^\varepsilon$ holds for any $t>0$ while $v^\varepsilon$ for any $t\geq0$. However, the convergence result in \cite{Corrillo-Hong-Wang-vanishing} with degenerate initial data $u_0$ satisfying $u_0|_{\partial \mathcal{I}}=0$ holds only for some finite time $T_0(v_*)$ depending on $v_*$. This is the major difference between degenerate and non-degenerate initial data $u_0$.
\end{remark}

\section{Regularity of the outer-/boundary-layer profiles} % (fold)
\label{sec:study_on_the_inner_outer_layers}

In this section, we shall derive the requisite regularities of solutions to  problems \eqref{eq-outer-0}, \eqref{first-bd-layer-pro-appro}, \eqref{first-bd-pro-rt-approxi}, \eqref{first-outer-problem}, \eqref{second-bd-eq}, \eqref{sec-bd-eq-rt}, respectively. The proofs are based on the energy method along with some basic mathematical tools introduced in Appendix \ref{Appen-Analysis}.
%We start with several preparatory results.
%\subsection{Properties of the corner-corrector functions}
We start by showing the following estimates on the corner-corrector functions $\cala_{i} ^{\varepsilon}(t), i=1,2$.
\begin{lemma}\label{LEM-correc-function}
It holds for $ i=1,2 $ and any $ \lambda \geq 0 $ that
\begin{gather}
\displaystyle  \displaystyle \varepsilon ^{-\alpha}\|\partial _{t}^{j}\cala_{i} ^{\varepsilon}(t)\|_{L ^{2}((0,T))}^{2}+ \|\partial _{t}^{2}\cala_{i} ^{\varepsilon}(t)\|_{L ^{1}((0,T))}+\|(t ^{\frac{2k-3}{2}}+\varepsilon ^{\frac{\alpha(2k-3)}{2}})\partial _{t}^{k}\cala_{i} ^{\varepsilon}(t)\|_{L ^{2}((0,T))}^{2} \leq C, \label{Corre-property}\\
\displaystyle  \|t ^{\frac{2k-3+\lambda}{2}}\partial _{t}^{k}\cala_{i} ^{\varepsilon}(t)\|_{L ^{2}((0,T))}^{2} +\|t ^{\lambda} \partial _{t}\cala _{i}^{\varepsilon}\|_{L ^{\infty}((0,T))}\leq C \varepsilon ^{\alpha \lambda},\label{cal-A-INFTY-T-SMALL}
\end{gather}
where $ j=0,1 $, $ k=2,3 $, $ C>0 $ is a constant independent of $ \varepsilon $.

\end{lemma}
\begin{proof}
We shall prove the estimates in \eqref{Corre-property} and \eqref{cal-A-INFTY-T-SMALL} for $ \cala_{1}^{\varepsilon}(t) $ only, and the estimates for $ \cala_{2}^{\varepsilon}(t) $ can be obtained in the same manner. Recalling \eqref{cal-A-1}, we have
 \begin{align}\label{cal-A-1-appen}
\displaystyle  \cala _{1}^{\varepsilon}(t)&= - 2\int _{0}^{t} (\varphi _{x}^{I,0}(0,s)+M)v ^{I,0}(0,s)\int _{s}^{\infty}\frac{		{\mathop{\mathrm{e}}}^{- \frac{\sigma ^{2}}{\varepsilon ^{2 \alpha}}}}{\varepsilon ^{\alpha}\sqrt{\pi}} \mathrm{d}\sigma \mathrm{d}s
   \nonumber \\
& \displaystyle \quad- 2(\varphi _{0x}(0)+M)v _{\ast} \int _{0}^{t}\chi(\frac{\tau}{\varepsilon ^{\alpha}}) \int _{0}^{\tau}\frac{   {\mathop{\mathrm{e}}}^{- \frac{\sigma ^{2}}{\varepsilon ^{2 \alpha}}}}{\varepsilon ^{ \alpha}\sqrt{ \pi }} \mathrm{d}\sigma  \mathrm{d}\tau .
\end{align}
It follows from direct computation that
\begin{align}\label{integral-esti-1}
&\displaystyle \int _{0}^{t} (\varphi _{x}^{I,0}(0,s)+M)v ^{I,0}(0,s)\int _{s}^{\infty}\frac{		{\mathop{\mathrm{e}}}^{- \frac{\sigma ^{2 }}{\varepsilon ^{2 \alpha}}}}{\varepsilon ^{\alpha}\sqrt{\pi}} \mathrm{d}\sigma \mathrm{d}s
 \nonumber \\
 &~\displaystyle = \int _{0}^{t}  (\varphi _{x}^{I,0}(0,s)+M)v ^{I,0}(0,s) \int _{\frac{s}{\varepsilon ^{\alpha}} } ^{\infty}		\frac{{\mathop{\mathrm{e}}}^{-\tau ^{2}}}{\sqrt{\pi}} \mathrm{d}\tau \mathrm{d}s
 \leq C \varepsilon ^{\alpha}\int _{0}^{\frac{t}{\varepsilon ^{\alpha}}}  \int _{\lambda} ^{\infty}		\frac{{\mathop{\mathrm{e}}}^{-\tau ^{2}}}{\sqrt{\pi}} \mathrm{d}\tau \mathrm{d}\lambda
  \nonumber \\
  &~\displaystyle \leq C \varepsilon ^{\alpha}\int _{0}^{\infty}  \int _{\lambda} ^{\infty}		\frac{{\mathop{\mathrm{e}}}^{-\tau ^{2}}}{\sqrt{\pi}} \mathrm{d}\tau \mathrm{d}\lambda \leq C \varepsilon ^{\alpha},
 \end{align}
 where we have used  Corollary \ref{cor-l-infty-vfi-I0} and \eqref{Sobolev-infty}, and the constant $ C>0 $ is independent of $ \varepsilon $. Similarly, we have
 \begin{align}\label{integral-esti-2}
 \displaystyle  \int _{0}^{t}\chi(\frac{\tau}{\varepsilon ^{\alpha}}) \int _{0}^{\tau}\frac{   {\mathop{\mathrm{e}}}^{- \frac{\sigma ^{2}}{\varepsilon ^{2 \alpha}}}}{\varepsilon ^{\alpha}\sqrt{ \pi }} \mathrm{d}\sigma  \mathrm{d}\tau \leq C \varepsilon ^{\alpha} ,
 \end{align}
 where $ C>0 $ is a constant independent of $ \varepsilon $. Therefore we get $ \|\cala_{1}^{\varepsilon}(t)\|_{L ^{\infty}((0,T))} \leq C \varepsilon ^{\alpha} $ for some constant $ C>0 $ which may depend on $ T $ but independent of $ \varepsilon $. Differentiating $ \cala_{1}^{\varepsilon}(t) $ with respect to $ t $ leads to
 \begin{align}\label{A-1-diff}
 \displaystyle   \partial _{t}\cala _{1}^{\varepsilon}(t)&= - 2(\varphi _{x}^{I,0}(0,t)+M)v ^{I,0}(0,t)\int _{t}^{\infty}\frac{		{\mathop{\mathrm{e}}}^{- \frac{\sigma ^{2}}{\varepsilon ^{2 \alpha}}}}{\varepsilon ^{\alpha}\sqrt{\pi}} \mathrm{d}\sigma
   \nonumber \\
& \displaystyle \quad- 2(\varphi _{0x}(0)+M)v _{\ast} \chi(\frac{t}{\varepsilon ^{\alpha}}) \int _{0}^{t}\frac{   {\mathop{\mathrm{e}}}^{- \frac{\sigma ^{2}}{\varepsilon ^{2 \alpha}}}}{\varepsilon ^{\alpha}\sqrt{ \pi }} \mathrm{d}\sigma  .
 \end{align}
 This along with \eqref{integral-esti-1} and \eqref{integral-esti-2} gives
 \begin{align*}
  \displaystyle \| \partial _{t}\cala _{1}^{\varepsilon}(t)\|_{L ^{2}((0,T))}^{2} \leq C \varepsilon ^{\alpha}.
  \end{align*}
  Differentiating \eqref{A-1-diff} with respect to $ t $, we get, thanks to Corollary \ref{cor-l-infty-vfi-I0},
   \begin{align*}
  \displaystyle  \partial _{t}^{2}\cala_{1}^{\varepsilon}(t) \sim \int _{t}^{\infty}\frac{		{\mathop{\mathrm{e}}}^{- \frac{\sigma ^{2}}{\varepsilon ^{2 \alpha}}}}{\varepsilon ^{\alpha}\sqrt{\pi}} \mathrm{d}\sigma + \frac{   {\mathop{\mathrm{e}}}^{- \frac{ t ^{2}}{ \varepsilon ^{2 \alpha}}}}{\varepsilon ^{\alpha} \sqrt{ \pi }} +
  \frac{1}{\varepsilon ^{\alpha}} \chi'(\frac{t}{\varepsilon ^{\alpha}})\int _{0}^{t}\frac{   {\mathop{\mathrm{e}}}^{- \frac{\sigma ^{2}}{4 \varepsilon ^{2 \alpha}}}}{\varepsilon
   ^{\alpha}\sqrt{ \pi }} \mathrm{d}\sigma
+\chi(\frac{t}{\varepsilon ^{\alpha}}) \frac{   {\mathop{\mathrm{e}}}^{- \frac{ t ^{2}}{4 \varepsilon ^{2 \alpha}}}}{\varepsilon ^{\alpha}\sqrt{ \pi }}.
  \end{align*}
  Then it follows from direct computation that
  \begin{align*}
  \displaystyle  \|\partial _{t}^{2}\cala_{1}^{\varepsilon}(t) \|_{L ^{1}((0,T))}
     &\leq C\int _{0}^{T}\int _{t}^{\infty}\frac{		{\mathop{\mathrm{e}}}^{- \frac{\sigma ^{2}}{\varepsilon ^{2 \alpha}}}}{\varepsilon ^{\alpha}\sqrt{\pi}} \mathrm{d}\sigma \mathrm{d}t + \int _{0}^{T} \frac{   {\mathop{\mathrm{e}}}^{- \frac{ t ^{2}}{ \varepsilon ^{2 \alpha}}}}{\varepsilon ^{\alpha} \sqrt{ \pi }} \mathrm{d}t + \int _{0}^{T}
  \frac{1}{\varepsilon ^{\alpha}} \chi'(\frac{t}{\varepsilon ^{\alpha}})\int _{0}^{t}\frac{   {\mathop{\mathrm{e}}}^{- \frac{\sigma ^{2}}{4 \varepsilon ^{2 \alpha}}}}{\varepsilon
   ^{\alpha}\sqrt{ \pi }} \mathrm{d}\sigma \mathrm{d}t
    \nonumber \\
    &~\displaystyle \quad+\int _{0}^{T}\chi(\frac{t}{\varepsilon ^{\alpha}}) \frac{   {\mathop{\mathrm{e}}}^{- \frac{ t ^{2}}{4 \varepsilon ^{2 \alpha}}}}{\varepsilon ^{\alpha}\sqrt{ \pi }} \mathrm{d}t \leq C
  \end{align*}
  and
  \begin{align*}
  \displaystyle  \|(t ^{\frac{1}{2}}+\varepsilon ^{\frac{\alpha}{2}})[\partial _{t}^{2}\cala_{1}^{\varepsilon}(t) \|_{L ^{2}((0,T))}^{2}&\leq C\int _{0}^{T}(t+\varepsilon ^{\alpha})\int _{t}^{\infty}\frac{		{\mathop{\mathrm{e}}}^{- \frac{\sigma ^{2}}{\varepsilon ^{2 \alpha}}}}{\varepsilon ^{\alpha}\sqrt{\pi}} \mathrm{d}\sigma \mathrm{d}t + C\int _{0}^{T} (t+\varepsilon ^{\alpha})\frac{   {\mathop{\mathrm{e}}}^{- \frac{ t ^{2}}{ \varepsilon ^{2 \alpha}}}}{\varepsilon ^{2\alpha}\pi } \mathrm{d}t
    \nonumber \\
    &\displaystyle \quad+ C\int _{0}^{T}(t+\varepsilon ^{\alpha})
  \frac{1}{\varepsilon ^{2\alpha}} [\chi'(\frac{t}{\varepsilon ^{\alpha}})]^{2}\int _{0}^{t}\frac{   {\mathop{\mathrm{e}}}^{- \frac{\sigma ^{2}}{4 \varepsilon ^{2 \alpha}}}}{\varepsilon
   ^{\alpha}\sqrt{ \pi }} \mathrm{d}\sigma \mathrm{d}t
    \nonumber \\
    &\displaystyle \quad+C\int _{0}^{T}(t+\varepsilon ^{\alpha})\chi(\frac{t}{\varepsilon ^{\alpha}}) \frac{   {\mathop{\mathrm{e}}}^{- \frac{ t ^{2}}{\varepsilon ^{2 \alpha}}}}{\varepsilon ^{2\alpha}\pi} \mathrm{d}t \leq C,
  \end{align*}
  where the constant $ C>0 $ is independent of $ \varepsilon $. Similar arguments further imply for any $ \lambda>0 $ that
  \begin{align*}
  &\displaystyle \|(t ^{\frac{1+\lambda}{2}}+\varepsilon ^{\frac{\alpha+\lambda}{2}})\partial _{t}^{2}\cala_{1}^{\varepsilon}(t) \|_{L ^{2}((0,T))} ^{2} + \|(t ^{\frac{3+\lambda}{2}}+\varepsilon ^{\frac{3\alpha+\lambda}{2}})\partial _{t}^{3}\cala_{1}^{\varepsilon}(t) \|_{L ^{2}((0,T))} ^{2} \leq C \varepsilon ^{\lambda \alpha}.
    \end{align*}
    Finally, by \eqref{A-1-diff}, the basic fact $ t ^{\lambda} \int _{t}^{\infty}		{\mathop{\mathrm{e}}}^{- \tau ^{2}}\mathrm{d}\tau \leq C _{\lambda} $ for some $ C _{\lambda} >0$ depending on $ \lambda $ and the fact that $ \chi $ is bounded and compactly supported, we have
    \begin{align*}
    \displaystyle  \|t ^{\lambda} \partial _{t}\cala_{1}^{\varepsilon}(t)\|_{L ^{\infty}}& \leq C t ^{\lambda}\int _{t}^{\infty}\frac{		{\mathop{\mathrm{e}}}^{- \frac{\sigma ^{2}}{\varepsilon ^{2 \alpha}}}}{\varepsilon ^{\alpha}\sqrt{\pi}} \mathrm{d}\sigma  + C\chi(\frac{t}{\varepsilon ^{\alpha}})  t ^{\lambda}\int _{0}^{t}\frac{   {\mathop{\mathrm{e}}}^{- \frac{\sigma ^{2}}{\varepsilon ^{2 \alpha}}}}{\varepsilon ^{\alpha}\sqrt{ \pi }} \mathrm{d}\sigma   \leq C
     \nonumber \\
     & \displaystyle \leq C \varepsilon ^{\alpha \lambda}\left( \frac{t}{\varepsilon ^{
     \alpha
     }} \right)^{\lambda}\left[ \int _{\frac{t}{\varepsilon ^{\alpha}}}^{\infty}		{\mathop{\mathrm{e}}}^{- \sigma ^{2}}\mathrm{d}\sigma+ \chi(\frac{t}{\varepsilon ^{\alpha}})\int _{0}^{\frac{t}{\varepsilon ^{\alpha}}}	{\mathop{\mathrm{e}}}^{- \sigma ^{2}}\mathrm{d}\sigma    \right] \leq C \varepsilon ^{\alpha \lambda}.
    \end{align*}
    This finishes the proof of Lemma \ref{LEM-correc-function}.
\end{proof}

Next we proceed with the problem on the leading-order outer-layer profile $ (\varphi ^{I,0}, v ^{I,0}) $. In \cite{Corrillo-Hong-Wang-vanishing}, we prove the global existence and uniqueness of classical solutions to the problem \eqref{refor-eq}. We quote the result here for later use.

% section study_on_the_inner_outer_layers (end)

\begin{lemma}\label{lem-regul-outer-layer-0}
Assume that $ (\varphi _{0}, v _{0}) \in H ^{7}\times H ^{7} $ and $ (\sqrt{v _{0}}) _{x} \in L ^{2}	 $ satisfying \eqref{compatibility-simple} and $\min _{x \in \bar{\mathcal{I}}} (\varphi _{0x}+M)>0 $. Then for any $ T>0 $, there exists a unique solution $ (\varphi ^{I,0}, v ^{I,0}) $ to the problem \eqref{eq-outer-0} on $ [0,T] $ satisfying
\begin{subequations}\label{con-vfi-v-I-0-regula}
\begin{align}
\displaystyle \displaystyle & K ^{-1} \leq \varphi _{x}^{I,0}+M \leq K, \ \ \ \partial _{t}^{k}\varphi ^{I,0} \in L _{T}^{2}H ^{8-2k}\ \ \mbox{for } k=0,1,2,3,4, \label{con-vfi-I-0-only}\\
\displaystyle v ^{I,0} \in &L _{T}^{\infty}H ^{7},\ \ \partial _{t}^{k}v ^{I,0} \in L _{T}^{2}H ^{9-2k}\ \ \mbox{for }k=1,2,3,4,  %\label{con-v-I-0}
\end{align}
where $ K>0 $ is a constant.
\end{subequations}

%\begin{align*}
%&\displaystyle \sup _{t \in [0,T]}\left( \|\varphi ^{I,0}\|_{H ^{7}}^{2}+\|\varphi _{t}^{I,0}\|_{H ^{5}}^{2}+\textcolor{blue}{\|\varphi _{tt}^{I,0}\|_{H^{3}}^{2}} +\left\|v ^{I,0}\right\|_{H ^{7}}^{2}+\| v  _{t}^{I,0}\|_{H ^{6}}^{2}+\left\|v _{tt}^{I,0}\right\|_{H ^{4}}^{2} \right)
 %\nonumber \\
 %&\quad+ \int _{0}^{T}\left( \|\varphi ^{I,0}\|_{H ^{8}}^{2}+\|\varphi _{t}^{I,0}\|_{H ^{6}}^{2}+ \textcolor{blue}{\|\varphi _{tt}^{I,0}\|_{ H ^{4} }^{2}} +\|v ^{I,0}\|_{H ^{7}}^{2}+\|v _{t}^{I,0}\|_{H ^{7}}^{2}+\|v _{tt}^{I,0}\|_{H^{5}}^{2} \right)\mathrm{d}\tau \leq C,
%\end{align*}
%where the constant $ C>0 $ may depend on $ T $.
\end{lemma}
Using Proposition \ref{prop-embeding-spacetime} in Appendix A and Lemma \ref{lem-regul-outer-layer-0}, we obtain the following estimates on $ (\varphi ^{I,0},v ^{I,0})$.
\begin{corollary}\label{cor-l-infty-vfi-I0}
Under the conditions of Lemma \ref{lem-regul-outer-layer-0}, it holds that
\begin{align}\label{l-INFTY-VFI-I-0}
\displaystyle \left\|\varphi ^{I,0}\right\|_{L _{T}^{\infty}H ^{7}}+\|v ^{I,0}\|_{L _{T}^{\infty}H ^{7}} +\|\partial _{t}^{k}\varphi ^{I,0}\|_{L _{T}^{\infty}H ^{7-2k}} +\|\partial _{t}^{k}v ^{I,0}\|_{L _{T}^{\infty}H ^{8-2k}}\leq C \ \ \  k=1,2,3.
\end{align}

\end{corollary}
\begin{remark}%\label{rem:-3-1}
As mentioned before, the problem \eqref{eq-outer-0} is exactly the zero-diffusion problem of \eqref{refor-eq}, which is equivalent to the zero-diffusion problem of \eqref{eq-orignal} in the sense of classical solutions. Precisely, let $ (\varphi ^{I,0},v ^{I,0}) $ be the solution obtained in Lemma \ref{lem-regul-outer-layer-0} and denote $ u ^{I,0}=\varphi _{x}^{I,0}+M $. Then $ (u ^{I,0},v ^{I,0}) $ is the unique classical solution to the zero-diffusion problem of \eqref{eq-orignal}.
\end{remark}
%%%%

% subsection gs (end)
% subsection gs (end)
Next lemma presents the regularity of the boundary-layer profiles $ v ^{B,\varepsilon} $ and $ \varphi ^{B,\varepsilon} $.
\begin{lemma}\label{lem-v-B-0}
Let $ (\varphi ^{I,0}, v ^{I,0}) $ be the solution to the problem \eqref{eq-outer-0} obtained in Lemma \ref{lem-regul-outer-layer-0}. Then for any $ T>0 $, there exists a positive constant $ \hat{\varepsilon} _{T} $ such that if $ \varepsilon \leq \hat{\varepsilon} _{T} $, the problem \eqref{first-bd-layer-pro-appro} admits a unique solution $ v ^{B,\varepsilon} $ on $ [0,T] $ such that for any $ l \in \mathbb{N} $, $ \langle z \rangle ^{l}\partial _{t}^{k} v ^{B,\varepsilon} \in L _{T}^{2}H _{z}^{6-2k}\,(k=0,1,2,3) $ and
\begin{subequations}\label{v-B-0-regularity}
\begin{gather}
\displaystyle 0 \leq v ^{B,\varepsilon} \leq v _{\ast},\ \ \ \|\langle z \rangle ^{l}v _{t}^{B,\varepsilon }\|_{L _{T}^{\infty}L _{z}^{2}}+ \|\langle z \rangle ^{l}\partial _{t}^{k}v ^{B,\varepsilon}\|_{L _{T}^{2} H _{z}^{3-2k}} \leq C  \ \ \mbox{for  } k=0,1,\label{v-B-0-regularity-a}
\\[2mm]
\displaystyle  \|\langle z \rangle ^{l}(t ^{\frac{1}{2}}+\varepsilon ^{\frac{\alpha}{2}})\partial _{t}^{i}\partial _{z}^{4-2i}v ^{B,\varepsilon}\|_{L _{T}^{2}L _{z}^{2}}+\|\langle z \rangle ^{l}(t ^{\frac{3}{2}}+\varepsilon ^{\frac{3 \alpha}{2}})\partial _{t}^{i}\partial _{z}^{5-2i}v ^{B,\varepsilon}\|_{L _{T}^{2} L _{z}^{2}} \leq C \ \ \mbox{for  }  i=0,1,2,\\[2mm]
\displaystyle  \|\langle z \rangle ^{l}(t ^{2}+\varepsilon ^{\frac{3 \alpha}{2}})\partial _{t}^{k}\partial _{z}^{6-2k}v ^{B,\varepsilon}\|_{L _{T}^{2} L _{z}^{2}}\leq C \ \ \mbox{for  } k=0,1,2,3.
\label{v-B-0-regularity-c}
%\displaystyle \textcolor{red}{ \|\langle z \rangle ^{l} t^{5/2}\partial _{t}^{4}v ^{B,\varepsilon}\|_{L _{T}^{2}L _{z}^{2}} \leq Cv _{\ast}.}\label{v-B-0-regularity-d}
\end{gather}
\end{subequations}
Furthermore, for $ \varphi ^{B,\varepsilon} $ defined in \eqref{vfi-bd-1ord-lt-approxi}, it holds that $ \langle z \rangle ^{l}\partial _{t} ^{k}\varphi ^{B,\varepsilon} \in L _{T}^{2}H ^{7-2k}\,(k=0,1,2,3) $ with
\begin{subequations}\label{con-vfi-B-1}
\begin{gather}
\displaystyle \|\langle z \rangle ^{l}\varphi _{t}^{ B,\varepsilon}\|_{L _{T}^{\infty}H _{z}^{1}}+ \|\langle z \rangle ^{l}\partial _{t}^{k}\varphi ^{B,\varepsilon}\|_{L _{T}^{2}H _{z}^{4-2k}} \leq C  \ \mbox{for }\, k=0,1, \label{con-vfi-B-1-a}
  \\[2mm]
 \displaystyle\|\langle z \rangle ^{l}(t ^{\frac{1}{2}}+\varepsilon ^{\frac{\alpha}{2}})\partial _{t}^{2}\varphi ^{B,\varepsilon}\|_{L _{T}^{2}L_{z}^{2}}+
 \|\langle z \rangle ^{l}(t ^{\frac{1}{2}}+\varepsilon ^{\frac{\alpha}{2}})\partial _{t}^{i}\partial _{z}^{k}\varphi ^{B,\varepsilon}\|_{L _{T}^{2}L _{z}^{2}} \leq C  \ \mbox{for }\, 2i+k =5,
 \\[2mm]
\displaystyle  \|\langle z \rangle ^{l}(t ^{\frac{3}{2}}+\varepsilon ^{\frac{3 \alpha}{2}})\partial _{t}^{k}\partial _{z}^{6-2k}\varphi ^{B,\varepsilon}\|_{L _{T}^{2}L _{z}^{2}} \leq C\ \ \mbox{for }\,k=0,1,2,
 \\[2mm]
 \displaystyle  \|\langle z \rangle ^{l}(t ^{2}+\varepsilon ^{\frac{3 \alpha}{2}})\partial _{t}^{3}\varphi ^{B,\varepsilon}\|_{L _{T}^{2}L _{z}^{2}}+\|\langle z \rangle ^{l}(t ^{2}+\varepsilon ^{\frac{3 \alpha}{2}})\partial _{t}^{k}\partial _{z}^{7-2k}\varphi ^{B,\varepsilon}\|_{L _{T}^{2}L _{z}^{2}} \leq C\ \ \mbox{for }\, k=0,1,2,3,
\end{gather}
\end{subequations}
where $ \alpha $ is as in \eqref{cal-A-1}, $ C >0$ is a constant depending on $ T $, but independent of $ \varepsilon $.
\end{lemma}

Before proving Lemma \ref{lem-v-B-0}, we extract some important facts resulting from Lemma \ref{lem-v-B-0} and Proposition \ref{prop-embeding-spacetime} in the following corollary.
\begin{corollary}\label{cor-V-B-0-INFTY}
Under the conditions of Lemma \ref{lem-v-B-0}, it holds for any $ l \in \mathbb{N} $ and $ k=0,1 $ that
\begin{align}%\label{v-infty-esti-lem}
&\displaystyle \|\langle z \rangle ^{l}\partial _{t}^{k} v ^{B,\varepsilon}\|_{L _{T}^{\infty}H _{z}^{2-2k}} +\|\langle z \rangle ^{l}(t ^{\frac{1}{2}}+\varepsilon ^{\frac{\alpha}{2} })\partial _{t}^{k}\partial _{z}^{3-2k}v ^{B,\varepsilon}\|_{L _{T}^{\infty}L _{z}^{2}}+\|\langle z \rangle ^{l}(t ^{2}+\varepsilon ^{\frac{3 \alpha}{2}})\partial _{t}^{2}v ^{B,\varepsilon}\|_{L _{T}^{\infty}H _{z}^{1}}
 \nonumber \\
 &\displaystyle  \quad+\|\langle z \rangle ^{l}(t ^{\frac{3}{2}}+\varepsilon ^{\frac{3 \alpha}{2}})\partial _{t}^{k}\partial _{z}^{4-2k}v ^{B,\varepsilon}\|_{L _{T}^{\infty}L _{z}^{2}}+ \|\langle z \rangle ^{l}(t ^{2}+\varepsilon ^{\frac{3 \alpha}{2}})\partial _{t}^{k}\partial _{z}^{5-2k}v ^{B,\varepsilon}\|_{L _{T}^{\infty}L _{z}^{2}} \leq C, \label{L-INFT-V-B-0}
\end{align}
and that
\begin{subequations}\label{L-INFT-Vfi-B-0}
\begin{gather}
\displaystyle \|\langle z \rangle ^{l}\partial _{t}^{k}\varphi ^{B,\varepsilon}\|_{L _{T}^{\infty}H _{z}^{3-2k}}  +\|\langle z \rangle ^{l}(t+\varepsilon ^{\frac{\alpha}{2} })\partial _{t}^{k}\partial _{z}^{4-2k}\varphi ^{B,\varepsilon}\|_{L _{T}^{\infty}L _{z}^{2}} \leq C,
  \\[2mm]
  \|\langle z \rangle ^{l}(t ^{2}+\varepsilon ^{\frac{3 \alpha}{2}})\partial _{t}^{2}\varphi ^{B,\varepsilon}\|_{L _{T}^{\infty}H _{z}^{2}}+ \|\langle z \rangle ^{l}(t ^{\frac{3}{2}}+\varepsilon ^{\frac{3 \alpha}{2}})\partial _{t}^{k}\partial _{z}^{5-2k}\varphi ^{  B,\varepsilon}\|_{L _{T}^{\infty}L _{z}^{2}} \leq C,\\[2mm]
  \displaystyle \|\langle z \rangle ^{l}(t ^{2}+\varepsilon ^{\frac{3 \alpha}{2}})\partial _{t}^{k}\partial _{z}^{6-2k}\varphi ^{  B,\varepsilon}\|_{L _{T}^{\infty}L _{z}^{2}} \leq C,
\end{gather}
\end{subequations}
where $ C >0$ is a constant depending on $ T $, but independent of $ \varepsilon $.
\end{corollary}
\begin{proof}[Proof of Lemma \ref{lem-v-B-0}]
Given any $ \varepsilon>0 $, the global existence and uniqueness of solutions to the problem \eqref{first-bd-layer-pro-appro} with regularity as in \eqref{v-B-0-regularity} can be proved by similar arguments as \cite{Corrillo-Hong-Wang-vanishing}. Here we mainly focus on the \emph{a priori} estimates of the solution on $ [0,T] $ for any $ T>0 $, as well as the dependence of the norms of solutions on $ \varepsilon $ which are of importance in the study on the convergence of  boundary-layer profiles in the sequel. We first prove the bounds of the solution, i.e.,
\begin{gather}\label{bd-v-B-VE}
\displaystyle  0 \leq v ^{B,\varepsilon} \leq v _{\ast}.
\end{gather}
Recalling $  \varphi_{x}^{I,0}+M \geq K ^{-1}>0 $ from \eqref{con-vfi-I-0-only}, the estimate of $ \cala _{1}^{\varepsilon}(t) $ in \eqref{Corre-property} and the fact
\begin{gather}%\label{v-I-0-positive}
\displaystyle v ^{I,0}(0,t)=v _{\ast}   \mathop{\mathrm{exp}}\nolimits \left( - \int _{0}^{t}(\varphi _{x}^{I,0}(0,t)+M)\mathrm{d}\tau \right),
\end{gather}
we conclude that there exists a positive constant $ \varepsilon _{1} $~(may depend on $ T $) such that for any $ \varepsilon \leq \varepsilon _{1} $, it holds that
\begin{align*}
\displaystyle  v _{\ast}\geq v _{\ast}- v ^{I,0}(0,t) +\cala_{1}^{\varepsilon}(t)>0.
\end{align*}
Then testing the equation \eqref{first-bd-layer-pro-appro} against $ v ^{-}:=-\max\{0,-v ^{B,\varepsilon}\} $, we derive that
\begin{align*}
&\displaystyle \frac{1}{2}\frac{\mathrm{d}}{\mathrm{d}t}\int _{\mathbb{R}_{+}}\left\vert v ^{-}\right\vert ^{2} \mathrm{d}z+\int _{\mathbb{R}_{+}}\left\vert \partial _{z}v ^{-}\right\vert ^{2}\mathrm{d}z +\int _{\mathbb{R}_{+}}(\varphi _{x}^{I,0}(0,t)+M)    {\mathop{\mathrm{e}}}^{v ^{B,\varepsilon}} \left\vert v ^{-}\right\vert ^{2} \mathrm{d}z
 \nonumber \\[2mm]
 &~\displaystyle \quad +\int _{\{v ^{B,\varepsilon}<0\}}(\varphi _{x} ^{I,0}(0,t)+M)\, v ^{I,0}(0,t)(    {\mathop{\mathrm{e}}}^{v ^{B,\varepsilon}}-1)v ^{B,\varepsilon} \mathrm{d}z=0,
\end{align*}
This entails that $ \|v ^{-}\|_{L _{T}^{\infty}L _{z}^{2}} \leq 0 $. Then we get $ v ^{B,\varepsilon} \geq 0 $. Similarly, testing the equation \eqref{first-bd-layer-pro} against $ v ^{+}:=\max\{v ^{B,\varepsilon}-v _{\ast},0\} $, we have $ v ^{B,\varepsilon} \leq v _{\ast}  $. Therefore \eqref{bd-v-B-VE} is proved.

To proceed, let $ \eta(z) \in C ^{\infty}([0,\infty)) $ such that
\begin{gather}\label{eta-defi}
\displaystyle \eta (0)=1,\ \ \eta(z)=0\ \ \mbox{for}\ z \geq 1,
\end{gather}
and denote by $ \overline{u ^{I,0}}:=\varphi  _{x}^{I,0}(0,t)+M $, $ \overline{v ^{I,0}}:=v ^{I,0}(0,t) $. Furthermore, take $ \vartheta= v ^{B,\varepsilon}- \eta(z)(v _{\ast} - \overline{v ^{I,0}}+\cala_{1}^{\varepsilon}(t))=:v ^{B,\varepsilon}-\phi(z,t)  $. Then $ \vartheta $ solves
\begin{align}\label{eq-for-esti-v-B-0}
 \displaystyle \begin{cases}
  \displaystyle \vartheta_{t}=\vartheta_{zz} -\overline{u ^{I,0}}    {\mathop{\mathrm{e}}}^{\vartheta+ \phi}(\vartheta+ \phi)- \overline{u ^{I,0}}\,\overline{v ^{I,0}}(    {\mathop{\mathrm{e}}}^{\vartheta+ \phi}-   1) + \varrho,\\
\displaystyle \vartheta(0,t)=0,\ \ \vartheta(+\infty,t)=0,\\
\displaystyle \vartheta(z,0)=0,
 \end{cases}
  \end{align}
where
\begin{gather*}
\displaystyle \varrho =  \eta _{zz}(z)(v _{\ast} - \overline{v ^{I,0}}+\cala_{1}^{\varepsilon}(t)) - \eta(z)(v _{\ast} - \overline{v ^{I,0}}+\cala_{1}^{\varepsilon}(t)) _{t}.
\end{gather*}
In view of \eqref{con-vfi-v-I-0-regula}, \eqref{l-INFTY-VFI-I-0} and \eqref{Corre-property}, it holds for any $ l \in \mathbb{N} $ that
\begin{subequations}\label{fip-vrho-esti}
\begin{gather}
\displaystyle  \|\langle z \rangle ^{l}\partial _{t}^{i}\phi\|_{L _{T}^{\infty} H _{z}^{5}}
 +\|\langle z \rangle ^{l}( t ^{\frac{2k-1}{2}}+ \varepsilon ^{\frac{\alpha(2k-1)}{2}}) \partial _{t}^{1+k}\phi\|_{L _{T}^{2} H _{z}^{6-2k}}\leq C,\ \ i=0,1,\, k=1,2,  \\
\displaystyle \|\langle z \rangle ^{l} \varrho\|_{L _{T} ^{\infty}H _{z}^{4}}+\|\langle z \rangle ^{l}\partial _{t}\varrho \|_{L _{T} ^{1}H _{z}^{2}} + \|\langle z \rangle ^{l} ( t ^{\frac{2k-1}{2}}+ \varepsilon ^{\frac{\alpha(2k-1)}{2}})  \partial _{t}^{k} \varrho\|_{L _{T}^{2}H _{z}^{4-2k}}\leq C , \ \    k=1,2,%\label{vrho-properti-0}
 \end{gather}
\end{subequations}
where $ C>0 $ is a constant independent of $ \varepsilon $. Multiplying the equation $ \eqref{eq-for-esti-v-B-0}_{1} $ by $ \langle z \rangle ^{2l}\vartheta$ followed by an integration over $ \mathbb{R}_{+} $, we get
\begin{align}\label{v-B-0-esti-0}
   \displaystyle & \frac{1}{2}\frac{\mathrm{d}}{\mathrm{d}t}\int _{\mathbb{R}_{+}}\langle z \rangle ^{2l}\vartheta^{2} \mathrm{d}z+\int _{\mathbb{R}_{+}}\langle z \rangle ^{2l}\vartheta_{z}^{2}\mathrm{d}z +\int _{\mathbb{R}_{+}}\langle z \rangle ^{2l}\overline{u ^{I,0}} {\mathop{\mathrm{e}}}^{\vartheta+ \phi}\vartheta^{2}\mathrm{d}z
    \nonumber \\
    & \displaystyle  = \int _{\mathbb{R}_{+}}\langle z \rangle ^{2l}\vartheta  \varrho \mathrm{d}z- 2l \int _{\mathbb{R}_{+}}\langle z \rangle ^{2l-2}z \vartheta _{z} \vartheta\mathrm{d}z-\int _{\mathbb{R}_{+}}\langle z \rangle ^{2l}\overline{u ^{I,0}} {\mathop{\mathrm{e}}}^{\vartheta+ \phi}\vartheta  \phi\mathrm{d}z
     \nonumber \\
     &~\displaystyle \quad -\int _{\mathbb{R}_{+}}\langle z \rangle ^{2l} \overline{u ^{I,0}}\,\overline{v ^{I,0}}\left(     {\mathop{\mathrm{e}}}^{\vartheta+ \phi}-   1 \right) \vartheta\mathrm{d}z=:\cala,
   \end{align}
   where, due to $ \overline{u ^{I,0}}=\varphi _{x} ^{I,0}(0,t)+M  \geq K ^{-1}>0 $ from \eqref{con-vfi-I-0-only} and $  0 \leq v ^{B,\varepsilon} \leq v _{\ast} $, it holds that
   \begin{align}\label{v-b-0-good}
   \displaystyle  \int _{\mathbb{R}_{+}}\langle z \rangle ^{2l}\overline{u ^{I,0}} {\mathop{\mathrm{e}}}^{\vartheta+ \phi}\vartheta^{2}\mathrm{d}z \geq C \int _{\mathbb{R}_{+} }\langle z \rangle ^{2l} \vartheta ^{2} \mathrm{d}z,
   \end{align}
   where $ C $ is a positive constant depending on $T $.
   We now estimate the terms on the right hand side of \eqref{v-B-0-esti-0}. By \eqref{con-vfi-v-I-0-regula}, $  0 \leq v ^{B,\varepsilon} \leq v _{\ast} $ and the Cauchy-Schwarz inequality, we get
   \begin{align}\label{cal-A-esti}
   \cala
    &\displaystyle  \leq \|\langle z \rangle ^{l}\vartheta\|_{L  _{z}^{2}}\|\langle z \rangle ^{l}\varrho\|_{L _{z} ^{2}}+C  \|\langle z \rangle ^{l}\vartheta\|_{L  _{z}^{2}}\|\langle z \rangle ^{l}\vartheta _{z}\|_{L _{z} ^{2}}+ C\|\langle z \rangle ^{l}\vartheta\|_{L  _{z}^{2}}\|\langle z \rangle ^{l}\phi\|_{L _{z} ^{2}}
     \nonumber \\
     & \displaystyle \quad+C  \int _{\mathbb{R}_{+}} \langle z \rangle ^{2l} \overline{u ^{I,0}}\,\overline{v ^{I,0}}\left(\left\vert \vartheta\right\vert+ \left\vert \phi\right\vert  \right)\vartheta \mathrm{d}z
     \nonumber \\
          & \displaystyle \leq  \frac{1}{4}\int _{\mathbb{R}_{+}}\langle z \rangle ^{2l}\vartheta _{z}^{2} \mathrm{d}z+C \int _{\mathbb{R}_{+}}\langle z \rangle ^{2l} \vartheta ^{2} \mathrm{d}z + C \|\langle z \rangle ^{l} \varrho\|_{L _{z}^{2}}^{2}+C \|\langle z \rangle ^{l}\phi\|_{L _{z}^{2}}^{2} .
   \end{align}
   Inserting \eqref{v-b-0-good} and \eqref{cal-A-esti} into \eqref{v-B-0-esti-0}, we get after integrating the resulting inequality over $ (0,t) $ for any $ t \in (0,T] $ that
\begin{align}\label{con-vfi-B-0-0-pre}
   \displaystyle  \int _{\mathbb{R}_{+}}\langle z \rangle ^{2l}\vartheta ^{2} (\cdot,t)\mathrm{d}z+\int _{0}^{t} \int _{\mathbb{R}_{+}}\langle z \rangle ^{2l}\left( \vartheta ^{2}+ \vartheta _{z} ^{2} \right)  \mathrm{d}z \mathrm{d}\tau \leq C + \int _{0}^{t}\int _{\mathbb{R}_{+}}\langle z \rangle ^{2l} \vartheta ^{2} \mathrm{d}z \mathrm{d}\tau.
   \end{align}
  Applying the Gronwall inequality to \eqref{con-vfi-B-0-0-pre}, it follows that
   \begin{align}\label{con-vfi-B-0-0}
   \displaystyle  \int _{\mathbb{R}_{+}}\langle z \rangle ^{2l}\vartheta ^{2} (\cdot,t)\mathrm{d}z+\int _{0}^{t} \int _{\mathbb{R}_{+}}\langle z \rangle ^{2l}\left( \vartheta ^{2}+ \vartheta _{z} ^{2} \right)  \mathrm{d}z \mathrm{d}\tau \leq C.
   \end{align}
  Testing $ \eqref{eq-for-esti-v-B-0}_{1} $ against $ \langle z \rangle ^{2l}\vartheta _{t} $, we get
\begin{align}\label{vte-t-esti}
&\displaystyle \frac{1}{2}\frac{\mathrm{d}}{\mathrm{d}t}\int _{\mathbb{R}_{+}}\langle z \rangle ^{2l}\left( \vartheta_{z} ^{2}+\overline{u ^{I,0}} \vartheta ^{2}     {\mathop{\mathrm{e}}}^{\vartheta+ \phi}\right)  \mathrm{d}z + \int _{\mathbb{R}_{+}}\langle z \rangle ^{2l}\vartheta_{t} ^{2}\mathrm{d}z
 \nonumber \\
 &~\displaystyle =\frac{1}{2}\int _{\mathbb{R}_{+}} \langle z \rangle ^{2l} \partial _{t}\overline{u ^{I,0}} \vartheta^{2}    {\mathop{\mathrm{e}}}^{\vartheta+ \phi}\mathrm{d}z+ \frac{1}{2}\int _{\mathbb{R}_{+}}\langle z \rangle ^{2l} \overline{u ^{I,0}}\vartheta^{2}\left( \vartheta_{t}+ \phi _{t} \right)     {\mathop{\mathrm{e}}}^{\vartheta+ \phi} \mathrm{d}z  -2 l \int _{\mathbb{R}_{+}}\langle z \rangle ^{2l-2}z \vartheta _{t}\vartheta _{z} \mathrm{d}z
  \nonumber \\
  &~\displaystyle \quad- \int _{\mathbb{R}_{+}}\langle z \rangle ^{2l}\overline{u ^{I,0}}    {\mathop{\mathrm{e}}}^{\vartheta+ \phi}  \phi \vartheta_{t} \mathrm{d}z-\overline{u ^{I,0}}\,\overline{v ^{I,0}}\int _{\mathbb{R}_{+}}\langle z \rangle ^{2l}(    {\mathop{\mathrm{e}}}^{\vartheta+ \phi}-    1)\vartheta_{t} \mathrm{d}z- \int _{\mathcal{I} }\langle z \rangle ^{2l} \varrho\vartheta_{t} \mathrm{d}z
   \nonumber \\
   &~\displaystyle \leq  C|\partial _{t} \overline{u ^{I,0}}|  \int _{\mathbb{R}_{+}}\langle z \rangle ^{2l}\vartheta^{2} \mathrm{d}z+C | \overline{u ^{I,0}}| \int _{\mathbb{R}_{+}}\langle z \rangle ^{2l} \vartheta ^{2}(\left\vert \vartheta _{t}\right\vert+\left\vert \phi _{t}\right\vert)  \mathrm{d}z + C \int _{\mathbb{R}_{+}}\langle z \rangle ^{2l-1}\left\vert \vartheta _{t}\right\vert \left\vert \vartheta _{z}\right\vert \mathrm{d}z
    \nonumber \\
    &\displaystyle~ \quad +C \int _{\mathbb{R}_{+}}\langle z \rangle ^{2l}\left\vert \phi\right\vert \left\vert \vartheta _{t}\right\vert\mathrm{d}z+C | \overline{u ^{I,0}}||\overline{v ^{I,0}}|\int _{\mathbb{R}_{+}}\langle z \rangle ^{2l}\left( \left\vert \vartheta\right\vert+\left\vert \phi\right\vert \right) \left\vert \vartheta_{t}\right\vert  \mathrm{d}z - \int _{\mathbb{R}_{+} }\langle z \rangle ^{2l} \varrho\vartheta_{t} \mathrm{d}z
      \nonumber \\
        &~\displaystyle \leq \frac{1}{8}\int _{\mathbb{R}_{+}}\langle z \rangle ^{2l}\vartheta_{t} ^{2}\mathrm{d}z+C\int _{\mathbb{R}_{+}}\langle z \rangle ^{2l}\vartheta _{z}^{2}  \mathrm{d}z+ C \int _{\mathbb{R}_{+}}\langle z \rangle ^{2l} (\phi ^{2}+\phi _{t}^{2})\mathrm{d}z +\int _{\mathbb{R}_{+}}\langle z \rangle ^{2l} \varrho^{2}\mathrm{d}z
         \nonumber \\
         & ~\displaystyle \quad + C(\vert \partial _{t} \overline{u ^{I,0}}\vert+ \vert \overline{u ^{I,0}}\vert ^{2}+ | \overline{u ^{I,0}}|^{2}|\overline{v ^{I,0}}|^{2})\int _{\mathbb{R}_{+}}\langle z \rangle ^{2l}\vartheta^{2} \mathrm{d}z,
\end{align}
where we have used \eqref{con-vfi-v-I-0-regula}, \eqref{con-vfi-B-0-0}, $  0 \leq v ^{B,\varepsilon} \leq v _{\ast} $ and the Cauchy-Schwarz inequality. By \eqref{con-vfi-v-I-0-regula} and the Sobolev inequality \eqref{Sobolev-infty}, we obtain that
  \begin{align*}%\label{vfi-bd-I-0}
 \begin{cases}
 	\displaystyle  \|\partial _{t}^{k}\varphi _{x}^{I,0}(0,t)\|_{L ^{2}(0,T)} \leq C\|\partial _{t}^{k}\varphi _{x}^{I,0}\|_{L _{T}^{2}H ^{1}} \leq C \ \ \mbox{for}\ \ 0 \leq k \leq 3,\\[1mm]
  \displaystyle \|\partial _{t}^{k}v ^{I,0}(0,t)\|_{L ^{2}(0,T)} \leq  C\|\partial _{t}^{k}v ^{I,0}\|_{L _{T}^{2}H ^{1}} \leq C \ \ \mbox{for}\ \ 0 \leq k \leq 4,
 \end{cases}
  \end{align*}
  which gives rise to
\begin{gather}\label{vfi-I-0-x-bd-infty}
\displaystyle   \|\partial _{t}^{k} \varphi _{x}^{I,0}(0,t)\|_{L ^{\infty}(0,T)} \leq C \ \ \mbox{for} \ 0 \leq k \leq 2 \ \ \mbox{and}\ \ \|\partial _{t}^{k} v^{I,0}(0,t)\|_{L ^{\infty}(0,T)} \leq C \ \ \mbox{for} \ \ 0 \leq k \leq 3.
\end{gather}
With \eqref{vfi-I-0-x-bd-infty}, we update \eqref{vte-t-esti} as
\begin{align*}
\displaystyle  &\displaystyle \frac{1}{2}\frac{\mathrm{d}}{\mathrm{d}t}\int _{\mathbb{R}_{+}}\langle z \rangle ^{2l}\left( \vartheta_{z} ^{2}+\overline{u ^{I,0}} \vartheta ^{2}     {\mathop{\mathrm{e}}}^{\vartheta+ \phi}\right)  \mathrm{d}z + \int _{\mathbb{R}_{+}}\langle z \rangle ^{2l}\vartheta_{t} ^{2}\mathrm{d}z
 \nonumber \\
 &~\displaystyle \leq C\int _{\mathbb{R}_{+}}\langle z \rangle ^{2l}\left( \vartheta^{2} +\vartheta _{z}^{2} \right)  \mathrm{d}z+ C \int _{\mathbb{R}_{+}}\langle z \rangle ^{2l} (\phi ^{2}+\phi _{t}^{2})\mathrm{d}z +\int _{\mathbb{R}_{+}}\langle z \rangle ^{2l} \varrho^{2}\mathrm{d}z .
\end{align*}
This along with \eqref{con-vfi-v-I-0-regula}, \eqref{fip-vrho-esti}, \eqref{v-b-0-good} and the Gronwall inequality yields for any $ t \in (0,T] $ that
\begin{align}\label{con-vfi-B-0-1}
\displaystyle  \int _{\mathbb{R}_{+}}\langle z \rangle ^{2l}\left( \vartheta^{2} +\vartheta _{z}^{2} \right) (\cdot,t) \mathrm{d}z+ \int _{0}^{t}\int _{\mathbb{R}_{+}}\langle  z\rangle ^{2l}\vartheta _{\tau}^{2} \mathrm{d}z \mathrm{d}\tau \leq C.
\end{align}
By \eqref{con-vfi-v-I-0-regula}, \eqref{fip-vrho-esti} and \eqref{con-vfi-B-0-1}, we further get from  $ \eqref{eq-for-esti-v-B-0}_{1} $ that
\begin{gather}\label{vte-xx}
\displaystyle \int _{0}^{T}\int _{\mathbb{R}_{+}}\langle z \rangle ^{2l}\vartheta _{zz}^{2} \mathrm{d}z \mathrm{d}t \leq C.
\end{gather}
Denote by $ \tilde{\vartheta}:=\vartheta_{t} $. Then by the compatibility condition \eqref{compatibility-vfi0}, $ \tilde{\vartheta} $ satisfies
\begin{align}\label{tild-vte-eq}
  \displaystyle \begin{cases}
    \displaystyle  \tilde{\vartheta}_{t}= \tilde{\vartheta}_{zz} -\overline{u ^{I,0}}    {\mathop{\mathrm{e}}}^{\vartheta+ \phi}\tilde{\vartheta}-\overline{u ^{I,0}}    {\mathop{\mathrm{e}}}^{\vartheta+ \phi}(\vartheta + \phi) \tilde{\vartheta}- \overline{u ^{I,0}}\,\overline{v ^{I,0}}{\mathop{\mathrm{e}}}^{\vartheta+ \phi} \tilde{\vartheta}+\tilde{\varrho},\\
\displaystyle \tilde{\vartheta}(0,t)=0,\ \ \tilde{\vartheta}(+\infty,t)=0,\\
\displaystyle \tilde{\vartheta}(z,0)=0,
  \end{cases}
  \end{align}
  where $ \tilde{\varrho} $ is given by
  \begin{align*}
  \displaystyle  \tilde{\varrho}&=- \partial _{t}\overline{u ^{I,0}}    {\mathop{\mathrm{e}}}^{\vartheta+ \phi}(\vartheta + \phi)-\overline{u ^{I,0}}    {\mathop{\mathrm{e}}}^{\vartheta+ \phi} \phi _{t}(1+\vartheta+ \phi)-\overline{u ^{I,0}}\,v ^{I,0}(0,t){\mathop{\mathrm{e}}}^{\vartheta+ \phi}\phi _{t}
   \nonumber \\
   & \displaystyle \quad -\partial _{t} \left( \overline{u ^{I,0}}\,\overline{v ^{I,0}} \right) (    {\mathop{\mathrm{e}}}^{\vartheta+ \phi}-   1)+ \partial _{t}\varrho.
  \end{align*}
  By \eqref{con-vfi-v-I-0-regula}, \eqref{fip-vrho-esti}, \eqref{con-vfi-B-0-0},
  \eqref{vfi-I-0-x-bd-infty}--\eqref{vte-xx} and \eqref{Corre-property}, it holds for any $ l \in \mathbb{N} $ that
  \begin{align}\label{tild-varrho-esti-1}
  \displaystyle \|\langle z \rangle ^{l}  \tilde{\varrho}\|_{L _{T}^{1}H _{z}^{2}}+\|\langle z \rangle ^{l} (t ^{\frac{2k+1}{2}}+\varepsilon ^{\frac{\alpha(2k+1)}{2}})\partial _{t}^{k}  \tilde{\varrho}\|_{L _{T}^{2}H _{z}^{2-2k}} \leq C  \ \ \mbox{for }k=0,1.
  \end{align}
Testing the equation in \eqref{tild-vte-eq} against $ \langle z \rangle ^{2l}\tilde{\vartheta} $, using integration by parts, the H\"older inequality and the Cauchy-Schwarz inequality, we have
\begin{align*}
&\displaystyle  \frac{1}{2}\frac{\mathrm{d}}{\mathrm{d}t}\int _{\mathbb{R}_{+}} \langle z \rangle ^{2l}\tilde{\vartheta}^{2}\mathrm{d}z+ \int _{\mathbb{R}_{+}}\langle z \rangle ^{2l} \tilde{\vartheta}^{2} \left(\overline{u ^{I,0}}    {\mathop{\mathrm{e}}}^{\vartheta+ \phi}+\overline{u ^{I,0}}\,\overline{v ^{I,0}}{\mathop{\mathrm{e}}}^{\vartheta+ \phi}  +\overline{u ^{I,0}}    {\mathop{\mathrm{e}}}^{\vartheta+ \phi}(\vartheta + \phi)  \right) \mathrm{d}z  + \int _{\mathbb{R}_{+}}\langle z \rangle ^{2l}\tilde{\vartheta}_{z}^{2}\mathrm{d}z
 \nonumber \\
 &~\displaystyle
 =  \int _{\mathbb{R}_{+}}\langle z \rangle ^{2l} \tilde{\vartheta}\tilde{\varrho}\mathrm{d}z-2 l \int _{\mathbb{R}_{+}}\langle z \rangle ^{2l-2}z \tilde{\vartheta}_{z}\tilde{\vartheta}\mathrm{d}z
  \nonumber \\
  &~\displaystyle\leq \frac{1}{2} \int _{\mathbb{R}_{+}}\langle z \rangle ^{2l}\tilde{\vartheta}_{z}^{2}\mathrm{d}z +C (1+\|\langle z \rangle ^{l}\tilde{\varrho}\|_{L _{z}^{2}}) \int _{\mathbb{R}_{+}} \langle z \rangle ^{2l}\tilde{\vartheta}^{2}\mathrm{d}z+C\|\langle z \rangle ^{l}\tilde{\varrho}\|_{L _{z}^{2}}.
\end{align*}
That is,
\begin{align}\label{v-B-0-t-sig}
&\displaystyle  \frac{\mathrm{d}}{\mathrm{d}t}\int _{\mathbb{R}_{+}} \langle z \rangle ^{2l}\tilde{\vartheta}^{2}\mathrm{d}z+\int _{\mathbb{R}_{+}}\langle z \rangle ^{2l}\tilde{\vartheta}_{z}^{2}\mathrm{d}z \leq C (1+\|\langle z \rangle ^{l}\tilde{\varrho}\|_{L _{z}^{2}}) \int _{\mathbb{R}_{+}} \langle z \rangle ^{2l}\tilde{\vartheta}^{2}\mathrm{d}z+\|\langle z \rangle ^{l}\tilde{\varrho}\|_{L _{z}^{2}}.
\end{align}
Then applying the Gronwall inequality to \eqref{v-B-0-t-sig}, by virtue of \eqref{tild-varrho-esti-1}, one has
\begin{align}\label{esti-con-v-B-0-T}
\displaystyle  \int _{\mathbb{R}_{+}} \langle z \rangle ^{2l}\tilde{\vartheta}^{2}(\cdot,t)\mathrm{d}z+\int _{0}^{t}\int _{\mathbb{R}_{+}}\langle z \rangle ^{2l}\tilde{\vartheta}_{z}^{2}\mathrm{d}z \mathrm{d}\tau \leq C
\end{align}
for any $ t \in (0,T] $, where $ C >0$ is a constant independent of $ \varepsilon $. Furthermore, by \eqref{fip-vrho-esti}, \eqref{con-vfi-B-0-1}, the equation in \eqref{eq-for-esti-v-B-0} and the bounds of $ v ^{B,\varepsilon} $, we have
\begin{align}\label{vte-zzz-multi}
\displaystyle \|\langle z \rangle ^{l} \vartheta _{zz}(\cdot,t)\|_{L _{z}^{2}}^{2}+  \int _{0}^{t}\int _{\mathbb{R}_{+}}\langle z \rangle ^{2l}\vartheta_{zzz}^{2}\mathrm{d}z \mathrm{d}\tau \leq C .
\end{align}
 Gathering \eqref{con-vfi-B-0-0}, \eqref{con-vfi-B-0-1}, \eqref{vte-xx}, \eqref{esti-con-v-B-0-T} and \eqref{vte-zzz-multi},  we get for any $ l \in \mathbb{N} $ and any $ t \in [0,T] $ that
 \begin{align}\label{con-vte-indep}
  \displaystyle  \displaystyle \|\langle z \rangle ^{l} \vartheta(\cdot,t)\|_{H _{z}^{2}}^{2}+ \|\langle z \rangle ^{l} \vartheta _{t}(\cdot,t)\|_{L _{z}^{2}}^{2}+ \int _{0}^{t} \int _{\mathbb{R}_{+}}\langle z \rangle ^{2l}\left( \vert \vartheta _{\tau z}\vert ^{2}+ \vert \vartheta _{zzz}\vert ^{2} \right)\mathrm{d}z \mathrm{d}\tau \leq C ,
  \end{align}
  where $ C>0 $ is a constant independent of $ \varepsilon $. This along with the definition of $ \vartheta $ gives
\begin{align}\label{con-v-B-0-indep}
\displaystyle \|\langle z \rangle ^{l} v ^{B,\varepsilon}(\cdot,t)\|_{H _{z}^{2}}^{2}+ \|\langle z \rangle ^{l} v _{t}^{ B,\varepsilon}(\cdot,t)\|_{L _{z}^{2}}^{2}+ \int _{0}^{t} \int _{\mathbb{R}_{+}}\langle z \rangle ^{2l}\left( \vert v _{\tau z}^{ B,\varepsilon}\vert ^{2}+ \vert v _{zzz}^{ B,\varepsilon}\vert ^{2} \right)\mathrm{d}z \mathrm{d}\tau \leq C .
\end{align}

Now let us turn to some higher-order estimates on $ v ^{B,\varepsilon} $ which, similar to the higher-order estimates on $ \tilde{\varrho} $ in \eqref{tild-varrho-esti-1}, may be time-weighted or $ \varepsilon $-dependent. Precisely, multiplying the equation in \eqref{tild-vte-eq} by $ \langle z \rangle ^{2l}\tilde{\vartheta}_{t} $ followed by an integration over $ \mathbb{R}_{+} $, one deduces that
\begin{align}\label{tild-te-esti-diff}
&\displaystyle \frac{1}{2} \frac{\mathrm{d}}{\mathrm{d}t}\int _{\mathbb{R}_{+}} \langle z \rangle ^{2l}\tilde{\vartheta} _z^{2}\mathrm{d}z+\int _{\mathbb{R}_{+}} \langle z \rangle ^{2l}\tilde{\vartheta} _t^{2}\mathrm{d}z
 \nonumber \\
 &~\displaystyle = -\int _{\mathbb{R}_{+}}\langle z \rangle ^{2l} \tilde{\vartheta} \tilde{\vartheta}_{t} \left(\overline{u ^{I,0}}    {\mathop{\mathrm{e}}}^{\vartheta+ \phi}+\overline{u ^{I,0}}\,\overline{v ^{I,0}}{\mathop{\mathrm{e}}}^{\vartheta+ \phi}  +\overline{u ^{I,0}}    {\mathop{\mathrm{e}}}^{\vartheta+ \phi}(\vartheta + \phi)  \right) \mathrm{d}z
  \nonumber \\
  & \displaystyle~ \quad -2l \int _{\mathbb{R}_{+}}\langle z \rangle ^{2l-2}z \tilde{\vartheta}_{z}\tilde{\vartheta}_{t}\mathrm{d}z+ \int _{\mathbb{R}_{+}}\langle z \rangle ^{2l} \tilde{\vartheta}_{t}\tilde{\varrho}\mathrm{d}z
   \nonumber \\
   &~\displaystyle \leq \frac{1}{2}\int _{\mathbb{R}_{+}} \langle z \rangle ^{2l}\tilde{\vartheta} _t^{2}\mathrm{d}z +C \int _{\mathbb{R}_{+}}\langle z \rangle ^{2l}\left( \tilde{\vartheta}^{2}+\tilde{\vartheta} _{z}^{2} \right)     \mathrm{d}z+C \int _{\mathbb{R}_{+}}\langle z \rangle ^{2l}\tilde{\varrho}^{2} \mathrm{d}z,
\end{align}
where we have used \eqref{vfi-I-0-x-bd-infty}, the bounds of $ \phi $ and $ v ^{B,\varepsilon} $ and the Cauchy-Schwarz inequality. Noting from \eqref{tild-varrho-esti-1} that $ \varepsilon ^{\alpha(2k+1)/2}\|\langle z \rangle ^{l} \partial _{t}^{k}  \tilde{\varrho}\|_{L _{T}^{2}H _{z}^{2-2k}} \leq C \,(k=0,1) $, we thus get after integrating \eqref{tild-te-esti-diff} over $ (0,t) $ that
\begin{align*}
 \displaystyle  \int _{\mathbb{R}_{+}} \langle z \rangle ^{2l}\tilde{\vartheta} _z^{2}(\cdot,t)\mathrm{d}z+\int _{0}^{t}\int _{\mathbb{R}_{+}} \langle z \rangle ^{2l}\tilde{\vartheta} _{\tau}^{2}\mathrm{d}z \mathrm{d}\tau \leq C v _{\ast}^{2} \varepsilon ^{- \alpha},
 \end{align*}
 where \eqref{esti-con-v-B-0-T} has been used, and the constant $ C>0 $ is independent of $ \varepsilon $. This along with $ \eqref{eq-for-esti-v-B-0}_{1} $, $\eqref{tild-vte-eq} _{1}  $ and \eqref{con-v-B-0-indep} implies that
 \begin{align}\label{v-B-0-zzz}
 \displaystyle  \int _{\mathbb{R}_{+}}\langle z \rangle ^{2l} \left(\vert \vartheta _{tz}\vert^{2} +\vert \vartheta _{zzz}\vert ^{2}\right)(\cdot,t)\mathrm{d}z + \int _{0}^{t} \int _{\mathbb{R}_{+}}\langle z \rangle ^{2l}\left(  \vert \vartheta  _{\tau \tau}\vert^{2} +\vert \partial _{z}^{4} \vartheta \vert ^{2} \right)  \mathrm{d}z \mathrm{d}\tau \leq  C\varepsilon ^{- \alpha}.
 \end{align}
 On the other hand, multiplying \eqref{tild-te-esti-diff} by $ t  $, we have
 \begin{align*}
 \displaystyle \frac{\mathrm{d}}{\mathrm{d}t} \left\{ t\int _{\mathbb{R}_{+}} \langle z \rangle ^{2l}\tilde{\vartheta} _z^{2}\mathrm{d}z \right\}+t\|\langle z \rangle ^{l}\tilde{\vartheta}_{t}\|_{L _{z}^{2}}^{2}  \leq  \|\langle z \rangle ^{l}\tilde{\vartheta}_{z}\|_{L _{z}^{2}}^{2}+ C t \left(\|\langle z \rangle ^{l}\tilde{\vartheta}\|_{L _{z}^{2}}^{2}+\|\langle z \rangle ^{l}\tilde{\vartheta}_{z}\|_{L _{z}^{2}}^{2}  \right) +t \|\langle z \rangle ^{l}\tilde{\rho}\|_{L _{z}^{2}}^{2},
 \end{align*}
 which along with \eqref{tild-varrho-esti-1} and \eqref{esti-con-v-B-0-T} immediately gives
 \begin{align*}
 \displaystyle  t \|\langle z \rangle ^{l}\tilde{\vartheta}_{z}(\cdot,t)\|_{L _{z}^{2}}^{2}+\int _{0}^{t}\tau \|\langle z \rangle ^{l}\tilde{\vartheta}_{\tau}\|_{L _{z}^{2}}^{2}\mathrm{d}\tau \leq C,
 \end{align*}
 where the constant $ C>0 $ is independent of $ \varepsilon $. By analogous arguments as proving \eqref{v-B-0-zzz}, we further get
 \begin{align*}
 \displaystyle  \displaystyle  t\int _{\mathbb{R}_{+}}\langle z \rangle ^{2l} \left(\vert \vartheta _{tz}\vert^{2} +\vert \vartheta _{zzz}\vert ^{2}\right)\mathrm{d}z + \int _{0}^{t} \tau\int _{\mathbb{R}_{+}}\left( \langle z \rangle ^{2l} \vert \vartheta _{\tau \tau}\vert^{2}+\langle z \rangle ^{2l} \vert \partial _{z}^{4} \vartheta \vert^{2} \right)   \mathrm{d}z \mathrm{d}\tau \leq C .
 \end{align*}
 Before proceeding, we would like to remark that for the problem \eqref{eq-for-esti-v-B-0}, we have estimates for $ \|\partial _{t}^{k}\vartheta\| _{ L _{T}^{2}H _{z}^{4-2k}}\,(k=0,1,2)$ provided $ \langle z \rangle ^{l} \partial _{t}^{k} \varrho \in L _{T}^{2}H _{z}^{2-2k}\,(k=0,1)$, and that the dependence of the norms of $ \partial _{t}^{k}\vartheta $ on $ \varepsilon $ is mainly determined by $ \varrho $. Notice that the initial data for the problem \eqref{tild-vte-eq} are compatible up to order one, and that $  \partial _{t}^{k} \tilde{\varrho} \,(k=0,1)$ possess the estimates in \eqref{tild-varrho-esti-1}. Therefore, by similar arguments as proving estimates for the problem \eqref{eq-for-esti-v-B-0}, we have for the problem \eqref{tild-vte-eq} that
 \begin{align}\label{esti-til-t-v-1}
  \displaystyle  \|(t ^{\frac{1}{2}}+\varepsilon ^{\frac{\alpha}{2}})\partial _{t}^{k}\tilde{\vartheta}\|_{L _{T}^{2}H _{z}^{2-2k}}^{2}+ \|(t ^{\frac{3}{2}}+\varepsilon ^{\frac{3 \alpha}{2}})\partial _{t}^{k}\partial _{z}^{3-2k}\tilde{\vartheta}\|_{L _{T}^{2}L _{z}^{2}}^{2} +\|(t ^{2}+\varepsilon ^{\frac{3 \alpha}{2}})\partial _{t}^{j}\partial _{z}^{4-2j}\tilde{\vartheta}\|_{L _{T}^{2}L _{z}^{2}}^{2}\leq C ,
  \end{align}
  where $ k=0,1 $, $ j=0,1,2 $. We should point out that the index of the power of the time weight function in the norm of $ \partial _{t}^{j} \partial _{z}^{4-2j}\tilde{\vartheta} $ is larger than that in the norm of $ \partial _{t}^{k}\partial _{z}^{3-2k}\tilde{\vartheta} $. This is because the former is derived based on the latter (need to control $ \|\langle z \rangle ^{l} t ^{2- \frac{1}{2}}\partial _{t} \partial _{z}\tilde{\vartheta}\|_{L _{T}^{2} L _{z}^{2}} $). Combining \eqref{esti-til-t-v-1} with \eqref{eq-for-esti-v-B-0}, \eqref{fip-vrho-esti}, \eqref{con-vte-indep} and \eqref{v-B-0-zzz} gives further that
  \begin{gather}\label{vete-highest}
  \displaystyle \|\langle z \rangle ^{l}(t ^{\frac{3+k}{2}}+\varepsilon ^{\frac{3 \alpha}{2}})\partial _{z}^{5+k}\vartheta\|_{L _{T}^{2}L ^{2}} \leq C, \ \ k=0,1.
  \end{gather}
  Collecting \eqref{con-v-B-0-indep}, \eqref{esti-til-t-v-1}--\eqref{vete-highest}, recalling \eqref{fip-vrho-esti} and the definition of $ \vartheta $, we then get \eqref{v-B-0-regularity-a}--\eqref{v-B-0-regularity-c}.

 To derive estimates for $ \varphi ^{B,\varepsilon} $, we need some more delicate estimates on $ v ^{B,\varepsilon} $. By \eqref{v-B-0-regularity} and Proposition \ref{prop-embeding-spacetime}, we get for any $ l \in \mathbb{N} $ and $ k=0,1 $ that
      \begin{align}
&\displaystyle \|\langle z \rangle ^{l}\partial _{t}^{k} v ^{B,\varepsilon}\|_{L _{T}^{\infty}H _{z}^{2-2k}} +\|\langle z \rangle ^{l}(t ^{\frac{1}{2}}+\varepsilon ^{\frac{\alpha}{2} })\partial _{t}^{k}\partial _{z}^{3-2k}v ^{B,\varepsilon}\|_{L _{T}^{\infty}L _{z}^{2}}+\|\langle z \rangle ^{l}(t ^{2}+\varepsilon ^{\frac{3 \alpha}{2}})\partial _{t}^{2}v ^{B,\varepsilon}\|_{L _{T}^{\infty}H _{z}^{1}}
 \nonumber \\
 &\displaystyle  \quad+\|\langle z \rangle ^{l}(t ^{\frac{3}{2}}+\varepsilon ^{\frac{3 \alpha}{2}})\partial _{t}^{k}\partial _{z}^{4-2k}v ^{B,\varepsilon}\|_{L _{T}^{\infty}L _{z}^{2}}+ \|\langle z \rangle ^{l}(t ^{2}+\varepsilon ^{\frac{3 \alpha}{2}})\partial _{t}^{k}\partial _{z}^{5-2k}v ^{B,\varepsilon}\|_{L _{T}^{\infty}L _{z}^{2}} \leq C ,\label{inft-B-0}
\end{align}
where we have used the following inequality for $ k=0,1 $,
\begin{align*}
\displaystyle t \| \partial _{t}^{k} \partial _{z}^{3-2k}v ^{B,\varepsilon}\|_{L _{z}^{2}}^{2}&= \int _{0}^{t} \left( \|\partial _{\tau}^{k}\partial _{z}^{3-2k}v ^{B,\varepsilon}\|_{L _{z}^{2}}^{2}+ 2\tau\int _{\mathbb{R}_{+}}\partial _{\tau}^{k}\partial _{z}^{3-2k}v ^{B,\varepsilon}\partial _{\tau}^{k+1} \partial _{z}^{3-2k}v ^{B,\varepsilon}\mathrm{d}z\right)\mathrm{d}\tau
 \nonumber \\
 & \leq C  \int _{0}^{t}\left(  \| \partial _{\tau}^{k}\partial _{z}^{3-2k}v ^{B,\varepsilon}\|_{L _{z}^{2}}^{2}+ \tau ^{2}\|\partial _{\tau} ^{k+1}\partial _{z}^{3-2k}v ^{B,\varepsilon}\|_{L _{z}^{2}}^{2} \right)\mathrm{d}\tau \leq C.
\end{align*}
Furthermore, in view of the Sobolev inequality \eqref{Sobolev-z-xi}, we have
    \begin{align}\label{inftY-V-B-0-st}
  &\displaystyle \|\langle z \rangle ^{l}\partial _{z}^{i} v ^{B,\varepsilon}\|_{L _{T}^{\infty}L _{z}^{\infty}}+ \|\langle z \rangle ^{l}(t ^{\frac{1}{4}}+\varepsilon ^{\frac{\alpha}{4}})\partial _{t}^{i}\partial _{z}^{2-2i}v ^{B,\varepsilon}\|_{L _{T}^{\infty}L _{z}^{\infty}} +\|t ^{2}\partial _{t}^{2} v ^{B,\varepsilon}\| _{L _{T}^{\infty}L _{z}^{\infty}}
   \nonumber \\
   & \quad+ \|\langle z \rangle ^{l}(t +\varepsilon ^{\alpha})\partial _{t}^{i}\partial _{z}^{3-2i}v ^{B,\varepsilon}\|_{L _{T}^{\infty}L _{z}^{\infty}}+\|\langle z \rangle ^{l}t ^{\frac{7}{4}}\partial _{t}^{i}\partial _{z}^{4-2i}v ^{B,\varepsilon}\|_{L _{T}^{\infty}L _{z}^{\infty}} \leq C
  \end{align}
  for $ i=0,1 $, and it holds for $ k=0,1,2 $ that
  \begin{align}
  \displaystyle  \displaystyle \varepsilon ^{\frac{3 \alpha}{2}}\|\langle z \rangle ^{l}\partial _{t}^{k}\partial _{z}^{4-2k}v ^{B,\varepsilon}\|_{L _{T}^{\infty}L _{z}^{\infty}} \leq C.
   \nonumber
  \end{align}
  Now we are ready to estimate $ \varphi ^{B,\varepsilon} $. Since
  \begin{gather*}
  \displaystyle  \partial _{z}^{\ell}      {\mathop{\mathrm{e}}}^{ v ^{B,\varepsilon}}= \sum _{  \ell _{1}+\cdots +\ell _{r}=\ell \atop
  1 \leq \ell _{1} \leq \cdots \leq\ell _{r} ,\,
  1 \leq r \leq \ell} C _{r}  {\mathop{\mathrm{e}}}^{v ^{B,\varepsilon}} \partial _{z}^{\ell _{1}}v ^{B,\varepsilon}\cdots \partial _{z}^{\ell _{r}}v ^{B,\varepsilon},\ \  \ell \geq 1
  \end{gather*}
  for some constant $ C _{r} $, then we get thanks to \eqref{v-B-0-regularity}, \eqref{inft-B-0} and \eqref{inftY-V-B-0-st} that for any $ l \in \mathbb{N} $,
\begin{gather}\label{L-INFT-H2-e}
\displaystyle \| \langle z \rangle ^{l} \partial _{t}^{i}(		{\mathop{\mathrm{e}}}^{v ^{B,\varepsilon}}-1  ) \|_{L _{T}^{\infty} H _{z} ^{2-2i}} \leq C, \ \ i=0,1,
\end{gather}
and that
  \begin{align}\label{pro-vfi-B-1-H3}
  \displaystyle  &\displaystyle \left\|   \langle z \rangle ^{l} ({\mathop{\mathrm{e}}}^{ v ^{B,\varepsilon}}-1)\right\|_{L _{T}^{2}H _{z}^{3}} ^{2}
   \nonumber \\
   &\displaystyle~ \leq  \sum _{  \ell _{1}+\cdots+ \ell _{r} \leq 3 \atop
  1 \leq \ell _{1} \leq \cdots \leq\ell _{r} ,\,
  1 \leq r \leq 3}C _{r}\|\langle z \rangle ^{l} \partial _{z}^{\ell _{1}}v ^{B,\varepsilon}\cdots \partial _{z}^{\ell _{r}}v ^{B,\varepsilon}\|_{L _{T}^{2}L _{z}^{2}}^{2}+\|\langle z \rangle ^{l} ({\mathop{\mathrm{e}}}^{ v ^{B,\varepsilon}}-1)\|_{L _{T}^{2}L _{z}^{2}}^{2}
   \nonumber \\
   & \displaystyle~ \leq  \sum _{  \ell _{1}+\cdots +\ell _{r} \leq 3 \atop
  1 \leq \ell _{1} \leq \cdots \leq\ell _{r} ,\,
  1 \leq r \leq 3} \makebox[-4pt]{~}C\int _{0}^{T}
  \|\partial _{z} ^{\ell _{1}}v ^{B,\varepsilon}\|_{L _{z}^{\infty}}^{2} \cdots \|\partial _{z} ^{\ell _{r-1}}v ^{B,\varepsilon}\|_{L _{z}^{\infty}}^{2}\|\langle z \rangle ^{l}\partial _{z} ^{\ell _{r}}v ^{B,\varepsilon}\|_{L _{z}^{2}}^{2}\mathrm{d}t +C\|\langle z \rangle ^{l} v ^{B,\varepsilon}\|_{L _{T}^{2}L _{z}^{2}}^{2}
   \nonumber \\
   &~\displaystyle \leq  \sum _{\ell=1}^{3}\| \langle z \rangle ^{l}\partial _{z}^{\ell}v ^{B,\varepsilon}\|_{L _{T}^{2}L _{z}^{2}   }^{2}+\|\langle z \rangle ^{l} v ^{B,\varepsilon}\|_{L _{T}^{2}L _{z}^{2}}^{2}  \leq C \|\langle z \rangle ^{l} v ^{B,\varepsilon}\|_{L _{T}^{2}H_{z}^{3}} ^{2} \leq C .
  \end{align}
  This along with \eqref{vfi-bd-1ord-lt-approxi} gives $\|\langle z \rangle ^{l}\varphi ^{B,\varepsilon}\|_{L _{T}^{2}H _{z}^{4}} \leq C   $. Noting $ \|\langle z \rangle ^{l}\partial _{z}^{i} v ^{B,\varepsilon}\|_{L _{T}^{\infty}L _{z}^{\infty}} \leq C ~(i=0,1)$ from \eqref{inftY-V-B-0-st} and
  \begin{align}
  \displaystyle  \partial _{z}^{4}   (   {\mathop{\mathrm{e}}}^{ v ^{B,\varepsilon}}-1)= \sum _{  \ell _{1}+\cdots +\ell _{r} = 4 \atop
  1 \leq \ell _{1} \leq \cdots \leq\ell _{r} \leq 3,\,
  1 \leq r \leq 4} C _{r}  {\mathop{\mathrm{e}}}^{v ^{B,\varepsilon}} \partial _{z}^{\ell _{1}}v ^{B,\varepsilon}\cdots \partial _{z}^{\ell _{r}}v ^{B,\varepsilon}+ 		{\mathop{\mathrm{e}}}^{v ^{B,\varepsilon}}\partial _{z}^{4}v ^{B,\varepsilon},
   \nonumber
  \end{align}
  by analogous arguments as in \eqref{pro-vfi-B-1-H3}, we have
  \begin{align*}%\label{pa-4-e-v-B-0}
  \displaystyle   \|\langle z \rangle ^{l}\partial _{z}^{4}     (   {\mathop{\mathrm{e}}}^{ v ^{B,\varepsilon}}-1)\|_{L _{T}^{2}L _{z}^{2}} ^{2}& \leq C \sum _{\ell=1}^{4}\| \langle z \rangle ^{l}\partial _{z}^{\ell}v ^{B,\varepsilon}\|_{L _{T}^{2}L _{z}^{2}   }^{2}+C \|\langle z \rangle ^{l}\partial _{z}^{2}v ^{B,\varepsilon}\| _{L _{T}^{2}L _{z}^{\infty}}^{2} \|\langle z \rangle ^{l}\partial _{z}^{2}v ^{B,\varepsilon}\| _{L _{T}^{\infty}L _{z}^{2}}^{2}
   \nonumber \\
   & \displaystyle  \leq C \varepsilon ^{- \alpha},
  \end{align*}
  where we have used \eqref{v-B-0-regularity} and \eqref{L-INFT-V-B-0}. Therefore we arrive at $ \|\langle z \rangle ^{l}\partial _{z}^{5}\varphi ^{  B,\varepsilon}\|_{L _{T}^{2}L _{z}^{2}} \leq C  \varepsilon ^{-\alpha/2} $. Similarly, we have
\begin{gather*}
\displaystyle  \|\langle z \rangle ^{l}\partial _{z}^{i}(		{\mathop{\mathrm{e}}}^{v ^{B,\varepsilon}}-1 )\| _{L _{T}^{2} L _{z}^{2}} \leq C\varepsilon ^{- \frac{3 \alpha}{2}} \mbox{ for }i=5,6,
\end{gather*}
and thus $ \|\langle z \rangle ^{l}\partial _{z}^{k}\varphi ^{B,\varepsilon}\|_{L _{T}^{2}L _{z}^{2}} \leq C\varepsilon ^{-\frac{3 \alpha}{2}}  $ for $ k=6,7 $. Next we shall estimate the time derivatives of $ \varphi ^{  B,\varepsilon} $. First, we get from \eqref{vfi-bd-1ord-lt-approxi}, \eqref{l-INFTY-VFI-I-0} and \eqref{L-INFT-H2-e} that
\begin{gather}\label{PA-T-VFI-B-1-INFTY}
\displaystyle  \|\varphi _{t} ^{ B,\varepsilon}\|_{L _{T}^{\infty}H _{z}^{1}}  \leq C .
\end{gather}
By \eqref{v-B-0-regularity-a} and \eqref{L-INFT-V-B-0}, it holds that
  \begin{align}\label{pa-t-e-v-B-0}
  \displaystyle \|\langle z \rangle ^{l} \partial _{t}\partial _{z}(		{\mathop{\mathrm{e}}}^{v ^{B,\varepsilon}}-1)\|_{L _{T}^{2}L _{z}^{2}} ^{2}& \leq C \|v _{z}^{ B,\varepsilon}v _{t}^{ B,\varepsilon}\|_{L _{T}^{2}L_{z}^{2}} ^{2}+ \|v _{tz}^{ B,\varepsilon}\|_{L _{T}^{2}L_{z}^{2}} ^{2} \leq C .
  \end{align}
  This along with \eqref{vfi-bd-1ord-lt-approxi}, \eqref{con-vfi-v-I-0-regula}, \eqref{L-INFT-H2-e} and \eqref{PA-T-VFI-B-1-INFTY} implies that
  \begin{align*}
  \displaystyle \| \langle z \rangle ^{l} \partial _{t}\varphi ^{B,\varepsilon}\|_{L _{T}^{2}H _{z}^{2}}^{2} \leq C .
  \end{align*}
  By analogous arguments as in \eqref{pa-t-e-v-B-0}, we further derive that
  \begin{align*}%\label{pa-t-z-2-e-v-B-0}
  \displaystyle \|\langle z \rangle ^{l}\partial _{t} \partial _{z}^{2}(		{\mathop{\mathrm{e}}}^{v ^{B,\varepsilon}}-1)\|_{L _{T}^{2} L _{z}^{2}} ^{2}& = \|\langle z \rangle ^{l}\partial _{z}^{2}(		{\mathop{\mathrm{e}}}^{ v ^{B,\varepsilon}}v _{t}^{ B,\varepsilon})\|_{L _{T}^{2} L _{z}^{2}} ^{2} \leq C\|\langle z \rangle ^{l}v _{t}^{ B,\varepsilon}\| _{L _{T}^{2}H _{z}^{2}}^{2} \leq C \varepsilon ^{- \alpha},
  \end{align*}
  \begin{align*}%\label{pa-t-e-v-B-0-z-high}
  &\displaystyle  \|\langle z \rangle ^{l}\partial _{t} \partial _{z}^{k}(		{\mathop{\mathrm{e}}}^{v ^{B, \varepsilon}}-1)\|_{L _{T}^{2} L _{z}^{2}} ^{2}= \|\langle z \rangle ^{l}\partial _{z}^{k}(		{\mathop{\mathrm{e}}}^{v ^{B,\varepsilon} }v _{t}^{ B,\varepsilon})\|_{L _{T}^{2} L _{z}^{2}} ^{2}
   \nonumber \\
   &~\displaystyle \leq  C\|\langle z \rangle ^{l}v _{t}^{B,\varepsilon}\| _{L _{T}^{2}H _{z}^{k}}^{2}+ \sum _{j=0}^{k} \|\langle z \rangle ^{l}({\mathop{\mathrm{e}}}^{v ^{B,\varepsilon}}-1)\|_{L _{T}^{2} H _{z}^{k+1-j}}^{2}  \|v _{t}^{ B,\varepsilon}\|_{L _{T}^{\infty}H _{z}^{j}}^{2} \leq C \varepsilon ^{-3 \alpha}
  \end{align*}
  for $ k =3,4 $. In a similar way, one can derive the estimates on higher-order derivatives of $ {\mathop{\mathrm{e}}}^{v ^{B,\varepsilon}}-1 $. Precisely, we have the following estimates
  \begin{align*}%\label{pa-t-i-e-v-B-0}
  \displaystyle \varepsilon ^{\frac{\alpha}{2}}\|\langle z \rangle ^{l}\partial _{t}^{2}(		{\mathop{\mathrm{e}}}^{v ^{B,\varepsilon}}-1)\| _{L _{T}^{2}L _{z}^{2}} + \varepsilon ^{\frac{3 \alpha}{2}}\|\langle z \rangle ^{l}\partial _{t}^{i}\partial _{z}^{k}(		{\mathop{\mathrm{e}}}^{v ^{B,\varepsilon}}-1)\| _{L _{T}^{2}L _{z}^{2}} \leq C
  \end{align*}
  for $5 \leq 2i+k \leq 6  $ with $ i \geq 2 $. This along with \eqref{vfi-bd-1ord-lt-approxi}, \eqref{con-vfi-v-I-0-regula} and \eqref{L-INFT-H2-e} yields that
  \begin{gather*}
 \displaystyle \varepsilon ^{\frac{\alpha}{2}}\|\langle z \rangle ^{l}\partial _{t}^{2}\varphi ^{B,\varepsilon}\|_{L _{T}^{2}L _{z} ^{2}} + \varepsilon ^{\frac{\alpha}{2}}\|\langle z \rangle ^{l}\partial _{t}^{2i}\partial _{z}^{k}\varphi ^{B,\varepsilon}\|_{L _{T}^{2} L _{z}^{2}} \leq C\ \ \mbox{for  } 2i+k =5, \\
 \displaystyle   \varepsilon ^{\frac{3 \alpha}{2}}\|\langle z \rangle ^{l}\partial _{t}^{2i}\partial _{z}^{k}\varphi ^{B,\varepsilon}\|_{L _{T}^{2} L _{z}^{2}} \leq C \ \ \mbox{for  } 6 \leq  2i+k \leq 7.
  \end{gather*}
  As for the time-weighted estimate for $ \varphi ^{B,\varepsilon} $, notice that the time-weighted estimates on $ v ^{B,\varepsilon} $ is independent of $ \varepsilon $, in other words, the time weight functions can remove the dependence of the norms on $ \varepsilon $ caused by the corner-correctors. Therefore the time-weighted estimate for $ \varphi ^{B,\varepsilon} $ can be derived by similar arguments as proving the $ \varepsilon $-dependent estimates. The details are omitted for brevity.
\end{proof}
The following lemma states the regularity of $ (\varphi ^{ b, \varepsilon}, v ^{b,\varepsilon}) $ which can be proved by a similar procedure as proving Lemma \ref{lem-v-B-0}.
\begin{lemma}\label{lem-first-bd-rt}
Assume the conditions in Lemma \ref{lem-regul-outer-layer-0} hold. Then for any $ T>0 $, there exists a constant $ \tilde{\varepsilon}_{T}>0 $ such that if $ \varepsilon \leq \tilde{\varepsilon}_{T} $, the problem \eqref{first-bd-pro-rt}, \eqref{firs-bd-1-rt} admits a unique solution $ (v ^{b,\varepsilon},\varphi ^{b,\varepsilon}) $ on $ [0,T] $ which satisfies
\begin{subequations}\label{con-b-0}
\begin{gather}
\displaystyle 0 \leq v ^{b,\varepsilon} \leq v _{\ast}, \ \ \|v _{t}^{ b,\varepsilon}\|_{L _{T}^{\infty}L _{\xi}^{2}}+\|\varphi _{t}^{ b,\varepsilon}\|_{L _{T}^{\infty}H _{\xi}^{1}} \leq C ,
 \\
 \displaystyle  \|\langle \xi \rangle ^{l}\partial _{t}^{k}v ^{b,\varepsilon}\|_{L _{T}^{2} H _{\xi}^{3-2k}}+ \|\langle \xi \rangle ^{l}\partial _{t}^{k}\varphi ^{b,\varepsilon}\|_{L _{T}^{2}H _{\xi}^{4-2k}} \leq C  \ \ \mbox{for  } k=0,1,\label{v-b-0-regularity-a}\\
\displaystyle  \|\langle \xi \rangle ^{l}(t ^{\frac{1}{2}}+\varepsilon ^{\frac{\alpha}{2}})\partial _{t}^{i}\partial _{\xi}^{4-2i}v ^{b,\varepsilon}\|_{L _{T}^{2}L _{\xi}^{2}} +\|\langle \xi \rangle ^{l}(t ^{\frac{3}{2}}+\varepsilon ^{\frac{3\alpha}{2}})\partial _{t}^{i}\partial _{\xi}^{5-2i}v ^{b,\varepsilon}\|_{L _{T}^{2}L _{\xi}^{2}}\leq C\ \ \mbox{for  } i=0,1,2,
\\
\displaystyle\|\langle \xi \rangle ^{l}(t ^{\frac{1}{2}}+\varepsilon ^{\frac{\alpha}{2}})\partial _{t}^{2}\varphi ^{b,\varepsilon}\|_{L _{T}^{2}L_{\xi}^{2}}+
 \|\langle \xi \rangle ^{l}(t ^{\frac{1}{2}}+\varepsilon ^{\frac{\alpha}{2}})\partial _{t}^{i}\partial _{\xi}^{k}\varphi ^{b,\varepsilon}\|_{L _{T}^{2}L _{\xi}^{2}} \leq C  \ \mbox{for }\, 2i+k =5,\\
\displaystyle  \|\langle z \rangle ^{l}(t ^{2}+\varepsilon ^{\frac{3 \alpha}{2}})\partial _{t}^{3}\varphi ^{b,\varepsilon}\|_{L _{T}^{2}L _{\xi}^{2}}+\|\langle \xi \rangle ^{l}(t ^{\frac{3}{2}}+\varepsilon ^{\frac{3 \alpha}{2}})\partial _{t}^{i}\partial _{\xi}^{6-2i}\varphi ^{b,\varepsilon}\|_{L _{T}^{2}L _{\xi}^{2}} \leq C\ \ \mbox{for }\, i=0,1,2,\\
\displaystyle \|\langle \xi \rangle ^{l}(t ^{2}+ \varepsilon ^{\frac{3 \alpha}{2}})\partial _{t}^{i}\partial _{\xi}^{6-2i}v ^{b,\varepsilon}\|_{L _{T}^{2} L _{\xi}^{2}}+\|\langle \xi \rangle ^{l}(t ^{2}+\varepsilon ^{\frac{3 \alpha}{2}})\partial _{t}^{i}\partial _{\xi}^{7-2i}\varphi ^{b,\varepsilon}\|_{L _{T}^{2}L _{\xi}^{2}} \leq C\ \ \mbox{for  } i=0,1,2,3,\label{v-b-0-regularity-c}
\end{gather}
\end{subequations}
where $ C >0$ is a constant depending on $ T $, but independent of $ \varepsilon $.
\end{lemma}
Similar to Corollary \ref{cor-V-B-0-INFTY}, in view of Proposition \ref{prop-embeding-spacetime} in Appendix A, we have the following delicate estimates for $ (\varphi ^{b,1},v ^{b,\varepsilon}) $.
\begin{corollary}\label{cor-V-b-0-INFTY}
Under the conditions of Lemma \ref{lem-first-bd-rt}, it holds for any $ l \in \mathbb{N} $ and $ k=0,1 $ that
\begin{align}\label{L-INFT-V-veb-0}
&\displaystyle \|\langle \xi \rangle ^{l}\partial _{t}^{k} v ^{b,\varepsilon}\|_{L _{T}^{\infty}H _{\xi}^{2-2k}} +\|\langle \xi \rangle ^{l}(t ^{\frac{1}{2}}+\varepsilon ^{\frac{\alpha}{2} })\partial _{t}^{k}\partial _{\xi}^{3-2k}v ^{b,\varepsilon}\|_{L _{T}^{\infty}L _{\xi}^{2}}+\|\langle \xi \rangle ^{l}(t ^{2}+\varepsilon ^{\frac{3\alpha}{2}})\partial _{t}^{2}v ^{b,\varepsilon}\|_{L _{T}^{\infty}H _{z}^{1}}
 \nonumber \\
 & \displaystyle \quad+  \|\langle \xi \rangle ^{l}(t ^{\frac{3}{2}}+\varepsilon ^{\frac{3 \alpha}{2}})\partial _{t}^{k}\partial _{\xi}^{4-2k}v ^{b,\varepsilon}\|_{L _{T}^{\infty}L _{\xi}^{2}}+ \|\langle \xi \rangle ^{l}(t ^{2}+\varepsilon ^{\frac{3 \alpha}{2}})\partial _{t}^{k}\partial _{\xi}^{5-2k}v ^{b,\varepsilon}\|_{L _{T}^{\infty}L _{\xi}^{2}}\leq C,
\end{align}
and that
\begin{subequations}\label{L-INFT-veVfi-b-0}
\begin{gather}
\displaystyle \|\langle \xi \rangle ^{l}\partial _{t}^{k}\varphi ^{b,\varepsilon}\|_{L _{T}^{\infty}H _{\xi}^{3-2k}}  +\|\langle \xi \rangle ^{l}(t ^{\frac{1}{2}}+\varepsilon ^{\frac{\alpha}{2} })\partial _{t}^{k}\partial _{\xi}^{4-2k}\varphi ^{b,\varepsilon}\|_{L _{T}^{\infty}L _{\xi}^{2}} \leq C,
  \label{L-INFT-veVfi-b-uniform-0}\\
  \|\langle \xi \rangle ^{l}(t ^{2}+\varepsilon ^{\frac{3 \alpha}{2}})\partial _{t}^{2}\varphi ^{b,\varepsilon}\|_{L _{T}^{\infty}H _{\xi}^{2}}+ \|\langle \xi \rangle ^{l}(t ^{\frac{3}{2}}+\varepsilon ^{\frac{3 \alpha}{2}})\partial _{t}^{k}\partial _{\xi}^{5-2k}\varphi ^{  b,\varepsilon}\|_{L _{T}^{\infty}L _{\xi}^{2}} \leq C,\\
  \displaystyle \|\langle \xi \rangle ^{l}(t ^{2}+\varepsilon ^{\frac{3 \alpha}{2}})\partial _{t}^{k}\partial _{\xi}^{6-2k}\varphi ^{  b,\varepsilon}\|_{L _{T}^{\infty}L _{\xi}^{2}} \leq C,
\end{gather}
\end{subequations}
where $ C >0$ is a constant depending on $ T $, but independent of $ \varepsilon $.
\end{corollary}
% subsection v_i_1 (end)

We next turn to the existence and regularity of the outer layer profile $ (\varphi ^{I,1}, v ^{I,1}) $.
\begin{lemma}\label{lemf-vfi-V-I-1}
Assume the conditions in Lemma \ref{lem-regul-outer-layer-0} hold, and let $ \varphi ^{ B,\varepsilon} $ and $ \varphi ^{ b,\varepsilon} $ be the solution obtained in Lemmas \ref{lem-v-B-0} and \ref{lem-first-bd-rt}, respectively. Then for any $ T>0 $, the problem \eqref{first-outer-problem} admits a unique classical solution $ (\varphi ^{I,1}, v ^{I,1}) $ on $ [0,T]$ which satisfies
\begin{gather*}
\displaystyle  \partial _{t}^{k}\varphi^{I,1} \in L _{T}^{2}H ^{6-2k} \ \mbox{for } k=0,1,2,3,\ \  v ^{I,1} \in L _{T}^{\infty}H ^{5} \ \ \mbox{and} \ \ \partial _{t}^{k}v ^{I,1} \in L _{T}^{2}H ^{7-2k}\  \mbox{for}\ k=1,2,3,
\end{gather*}
along with the bounds
\begin{subequations}\label{vfi-I-1-esti-first-con-lem}
 \begin{gather}
\displaystyle \|\partial _{t}^{i}\varphi ^{I,1}\|_{L _{T}^{2}H ^{2-2i}}+\|\partial _{t}^{k} v ^{I,1}\|_{L _{T}^{2}H ^{1}}
\leq C,\\
 \displaystyle \|t ^{\frac{1+i}{2}}\partial _{x}^{1+i}\varphi _{t}^{I,1}\|_{L _{T}^{2}L ^{2}}+\|\varepsilon ^{\frac{\alpha}{2}}\partial _{t}^{i}\partial _{x}^{3-2i}\varphi ^{I,1} \|_{L_{T}^{2} H ^{1}} + \|(t +\varepsilon ^{\frac{\alpha}{2}})\partial _{t}^{2}\varphi ^{I,1}\|_{L _{T}^{2}L ^{2}}\leq C,\\
\displaystyle  \|t ^{2+\frac{i}{2}}\partial _{x}^{1+i}\partial _{t}^{2}\varphi ^{I,1}\|_{L _{T}^{2}L ^{2}}+\|\varepsilon ^{\frac{3\alpha}{2}}\partial _{t}^{j}\partial _{x}^{5-2j}\varphi ^{I,1}\|_{L _{T}^{2}H ^{1}}+\|(t ^{\frac{5}{2}}+\varepsilon ^{\frac{3\alpha}{2}})\partial _{t}^{3}\varphi ^{I,1}\|_{L _{T}^{2} L ^{2}}\leq C,\\
\displaystyle \|(t ^{\frac{3i+2}{2}}+\varepsilon ^{\frac{\alpha( 2i+1)}{2}})\partial _{t}^{2+i}v ^{I,1}\|_{L _{T}^{2} H ^{1}}+\|\varepsilon ^{\frac{\alpha( 2i+1)}{2}}\partial _{t}^{k}\partial _{x}^{2+2i}v ^{I,1}\|_{L _{T}^{2}H ^{1}}+\|\varepsilon ^{\frac{3\alpha}{2}}\partial _{t}^{2}v _{xx}^{I,1}\|_{L _{T}^{2}H ^{1}}\leq C, \label{v-I1-txx}
\end{gather}
\end{subequations}
where $ i=0,1 $, $ k=0,1 $, $ j=0,1,2 $, $ C>0 $ is a constant independent of $ \varepsilon $.
\end{lemma}
\begin{proof}
  The local existence and uniqueness of solutions to the problem \eqref{first-outer-problem} on the outer layer profile pair $ (\varphi ^{I,1}, v ^{I,1}) $ can be proved by the classical PDE theory for linear parabolic equations (cf. \cite[Sec. 7.1]{evans-book}) along with the Banach fixed point theorem. In the following, we will devote ourselves to establishing some \emph{a priori} estimates which entail the global existence and the desired regularity of the solution.

  Denote $ b(x,t):=x \varphi ^{b,\varepsilon}(0,t)+(1-x)\varphi ^{B,\varepsilon}(0,t) $ and $ \tilde{\varphi}:=\varphi ^{I,1}+b(x,t) $ with
\begin{gather}\label{bdry-value-vfi-B-b-1}
\begin{cases}
  \displaystyle  \varphi ^{B,\varepsilon}(0,t)= -\int _{0}^{\infty}(\varphi _{x}^{I,0}(0,t)+M)\left(    {\mathop{\mathrm{e}}}^{v ^{B,\varepsilon}(y,t)}-1 \right) \mathrm{d}y,
   \nonumber \\[4mm]
   \displaystyle \varphi ^{b,\varepsilon}(0,t)=\int _{- \infty}^{0}(\varphi _{x}^{I,0}(1,t)+M)\left(    {\mathop{\mathrm{e}}}^{v ^{b,\varepsilon}(y,t)}-1 \right)\mathrm{d}y.
\end{cases}
\end{gather}
Then we deduce from \eqref{first-outer-problem} that
  \begin{gather}\label{eq-vfi-I-1-refor}
  \displaystyle  \begin{cases}
    \displaystyle \tilde{\varphi}_{t}=\tilde{\varphi}_{xx}-(\varphi _{x}^{I,0}+M)v _{x}^{I,1}-\tilde{\varphi}_{x}v _{x}^{I,0}+f _{1}(x,t),\\[2mm]
    \displaystyle   v_{t} ^{I,1}=-\left( \varphi _{x}^{I,0}+M \right)v ^{I,1}-\tilde{\varphi}_{x}v ^{I,0}+f _{2}(x,t),\\[2mm]
    \displaystyle \tilde{\varphi}(0,t)=\tilde{\varphi}(1,t)=0,\\[2mm]
    \displaystyle (\tilde{\varphi}, v ^{I,1})(x,0)=(0,0),
  \end{cases}
  \end{gather}
  where we have used the fact $ v ^{B,\varepsilon}(z,0)=v ^{b,\varepsilon}(\xi,0)=0 $, and $ f _{i}(x,t)~(i=1,2) $ are given by
  \begin{gather}\label{f-1-2-defi}
  \displaystyle f _{1}(x,t):=b _{t}+b _{x}v _{x}^{I,0},\ \  f _{2}(x,t):= b _{x}v ^{I,0},\ \ k=0,1.
  \end{gather}
  To ensure the desired regularity of the solution, it is necessary to derive some estimates for $ f _{1}(x,t) $ and $ f _{2}(x,t) $. By \eqref{con-vfi-B-1} and \eqref{Sobolev-infty}, we deduce for $ k=0,1 $ that
  \begin{subequations}\label{esti-vfi-B-0-BDY}
    \begin{gather}
  \displaystyle  \|\partial _{t}^{k}\varphi ^{B,\varepsilon}(0,t)\|_{L ^{2}((0,T))} \leq C \|\partial _{t}^{k}\varphi ^{B,\varepsilon}\|_{L _{T}^{2}L _{z}^{\infty}} \leq C \|\partial _{t}^{k}\varphi ^{B,\varepsilon}\|_{L _{T}^{2}H_{z}^{1}} \leq  C,\\
  \displaystyle \|(t ^{\frac{1+3k}{2}}+\varepsilon ^{\frac{(1+2k)\alpha}{2}})\partial _{t}^{2+k}\varphi ^{B,\varepsilon}(0,t)\|_{L ^{2}((0,T))}\leq \|(t ^{\frac{1+3k}{2}}+\varepsilon ^{\frac{(1+2k)\alpha}{2}})\partial _{t}^{2+k}\varphi ^{B,\varepsilon}\|_{L _{T}^{2}H _{z}^{1}} \leq C.
  \end{gather}
  \end{subequations}
      Similarly, by \eqref{con-b-0} and \eqref{Sobolev-infty}, we have for $ \varphi ^{b,\varepsilon}(0,t) $ that
       \begin{gather}\label{esti-vfi-b-0-bd}
  \displaystyle  \|\partial _{t}^{k}\varphi ^{b,\varepsilon}(0,t)\|_{L ^{2}((0,T))} + \|(t ^{\frac{1+3k}{2}}+\varepsilon ^{\frac{(1+2k)\alpha}{2}})\partial _{t}^{2+k}\varphi ^{b,\varepsilon}(0,t)\|_{L ^{2}((0,T))} \leq C  \ \ \mbox{for }k=0,1.
  \end{gather}
Recalling the definitions of $ f _{1} $ and $ f _{2} $ in \eqref{f-1-2-defi}, we deduce for $ k=0,1,2 $ that
\begin{subequations}\label{f-1-f-2-derive}
       \begin{align}
       \displaystyle \|\partial _{t}^{k}f _{1}\|_{L _{T}^{2}H ^{4-2k}}^{2}
     &\displaystyle \leq  C \left(  \|\partial _{t}^{k+1}\varphi ^{B,\varepsilon}(0,t)\|_{L ^{2}((0,T))}^{2}+ \|\partial _{t}^{k+1}\varphi ^{b,\varepsilon}(0,t)\|_{L ^{2}((0,T))}^{2} \right)
   \nonumber \\
   & \displaystyle \quad +\sum _{j=0}^{k}\left(  \|\partial _{t}^{j}\varphi ^{B,\varepsilon}(0,t)\|_{L ^{2}((0,T))}^{2}+ \|\partial _{t}^{j}\varphi ^{b,\varepsilon}(0,t)\|_{L ^{2}((0,T))}^{2}  \right)
    \nonumber \\
    & \displaystyle \quad\quad \quad~\quad \quad\times\|\partial _{t}^{k-j}v ^{I,0}\| _{L _{T}^{\infty}H ^{5-2k}}^{2} ,\\
    \displaystyle  \|\partial _{t}^{k}f _{2}\|_{L _{T}^{2}H ^{5-2k}}^{2} & \leq \sum _{j=0}^{k}\left(  \|\partial _{t}^{j}\varphi ^{B,\varepsilon}(0,t)\|_{L ^{2}((0,T))}^{2}+ \|\partial _{t}^{j}\varphi ^{b,\varepsilon}(0,t)\|_{L ^{2}((0,T))}^{2}  \right)
   \nonumber \\
   & \displaystyle \quad\quad \quad~ \quad\times\|\partial _{t}^{k-j}v ^{I,0}\| _{L _{T}^{\infty}H ^{5-2k}}^{2} ,
       \end{align}
 \end{subequations}
 which along with \eqref{con-vfi-v-I-0-regula}, \eqref{esti-vfi-B-0-BDY} and \eqref{esti-vfi-b-0-bd} implies that
   \begin{gather}
  \displaystyle  \varepsilon ^{\frac{(2k+1)\alpha}{2}}\|\partial _{t}^{1+k}f _{1}\|_{L _{T}^{2}H ^{2-2k}}+\varepsilon ^{\frac{\alpha}{2}}\|\partial _{t}^{2}f _{2}\|_{L _{T}^{2}H ^{1}}+\|f _{1}\| _{L _{T}^{2}H ^{4}} +\|\partial _{t}^{k}f _{2}\|_{L _{T}^{2}H ^{5-2k}} \leq C  , \ \ k=0,1.\label{uniform-f-1-f-2}
  \end{gather}
By analogous arguments as deriving \eqref{f-1-f-2-derive}, we get, thanks to \eqref{con-vfi-v-I-0-regula}, \eqref{esti-vfi-B-0-BDY} and \eqref{esti-vfi-b-0-bd},
  \begin{align}
  \displaystyle  \|t ^{\frac{3k+1}{2}}\partial _{t}^{1+k}f _{1}\|_{L _{T}^{2}H ^{2-2k}}+  \|t ^{\frac{1}{2}}\partial _{t}^{2}f _{2}\|_{L _{T}^{2}H ^{1}}\leq C ,\ \ k=0,1. \label{f-1-f-2-ve-depen}   \end{align}
With the temporal weight functions identified in \eqref{f-1-f-2-ve-depen},  we are ready to establish estimates for the solution. The remaining proof only involves tedious calculations by using the similar arguments as those in the proof of Lemma \ref{lem-v-B-0}, and hence we omit the details for brevity.
\end{proof}

With the estimates in \eqref{vfi-I-1-esti-first-con-lem}, one can utilize Proposition \ref{prop-embeding-spacetime} to get the following $L^\infty$ estimates.
\begin{corollary}\label{cor-vfi-I-1-infty}
Let $ (\varphi ^{I,1}, v ^{I,1}) $ be the solution obtained in Lemma \ref{lemf-vfi-V-I-1}. It holds that
\begin{subequations}\label{vfi-I-1-infty}
\begin{gather}
\displaystyle \|\varphi ^{I,1}\| _{L _{T}^{\infty}H ^{1}}+\|t \varphi _{xx}^{I,1}\|_{L _{T}^{\infty}L ^{2}} +\|\varepsilon ^{\frac{\alpha}{2}}\varphi _{xx}^{I,1}\|_{L _{T}^{\infty}H ^{1}}+\|\varepsilon ^{\frac{3\alpha}{2}}\partial _{x}^{4}\varphi ^{I,1}\|_{L _{T}^{\infty}H ^{1}}+\|\varepsilon ^{\frac{\alpha}{2}}\varphi_{t} ^{I,1}\|_{L _{T}^{\infty}H ^{1}}\leq C, \\
\displaystyle \|t \varphi_{t} ^{I,1}\|_{L _{T}^{\infty}L ^{2}}+\|t ^{2+\frac{i}{2}}\partial _{x}^{1+i}\varphi _{t}^{I,1}\|_{L _{T}^{\infty}L ^{2}} + \|\varepsilon ^{\frac{3\alpha}{2}}\partial _{x}^{2}\varphi_{t} ^{I,1}\|_{L _{T}^{\infty}H ^{1}}\leq C,\\
\displaystyle \|(t ^{\frac{5}{2}}+\varepsilon ^{\frac{3\alpha}{2}})\partial _{t}^{2}\varphi ^{I,1}\|_{L _{T}^{\infty}H ^{1}} \leq C,
\end{gather}
\end{subequations}
and that
\begin{subequations}\label{v-I-1-infty}
\begin{gather}
\displaystyle  \|v^{I,1}\| _{L _{T}^{\infty}H ^{1}}+\|\varepsilon ^{\frac{\alpha}{2}}v _{xx} ^{I,1}\|_{L _{T}^{\infty}H ^{1}}+\|\varepsilon ^{\frac{3\alpha}{2}}\partial _{x}^{4}v ^{I,1}\|_{L _{T}^{\infty}H ^{1}} \leq C ,\\
\displaystyle \|t ^{\frac{3i+2}{2}} \partial _{t}^{1+i} v ^{I,1}\|_{L _{T}^{\infty}H ^{1}}+\|\varepsilon ^{\frac{(1+2i)\alpha}{2}}\partial _{t}^{1+i} v ^{I,1}\|_{L _{T}^{\infty}H ^{2}}+\|\varepsilon ^{\frac{3\alpha}{2}}\partial _{t}\partial _{x}^{3} v ^{I,1}\|_{L _{T}^{\infty}H ^{1}} \leq C.
\end{gather}
Here $ i=0,1 $, the constant $ C>0 $ is independent of $ \varepsilon $.
\end{subequations}
\end{corollary}

With the estimates for lower-order profiles derived in Lemmas \ref{lem-regul-outer-layer-0}--\ref{lemf-vfi-V-I-1} at hand, we proceed to derived the estimates for higher-order profiles determined by the systems  \eqref{second-bd-eq} and \eqref{sec-bd-eq-rt}.
\begin{lemma}\label{lem-v-B-1}
Assume the conditions in Lemmas \ref{lem-regul-outer-layer-0}, \ref{lem-v-B-0} and \ref{lemf-vfi-V-I-1} hold. Then the problem \eqref{second-bd-eq} admits a unique solution $ (v ^{B,1}, \varphi ^{B,2}) $ on $ [0,T] $ for any $ T \in (0,\infty) $ which satisfies, for any $ l \in \mathbb{N} $,
\begin{gather*}
\displaystyle\langle z \rangle ^{l}\partial _{t}^{k}v ^{B,1} \in L _{T}^{2}H _{z}^{6-2k}\ \ \ \mbox{for }k=0,1,2,3,\  \mbox{and } \langle z \rangle ^{l}\partial _{t}^{k}\varphi ^{B,2}\in L _{T}^{2}H _{z}^{7-2k}\ \ \ \mbox{for }k=0,1,2.
 \end{gather*}
Precisely, we have
 \begin{subequations}\label{tild-v-B-1-lem}
\begin{gather}
\displaystyle \|\langle z \rangle ^{l}\partial _{t}^{k}v ^{B,1}\|_{L _{T}^{2}H _{z}^{2-2k}}
+\|\langle z \rangle ^{l} (t ^{\frac{3+3k}{2}}+\varepsilon ^{\frac{(1+2k)\alpha}{2}})\partial _{t}^{2+k}v ^{B,1}\|_{L _{T}^{2}L _{z}^{2}} \leq C ,\ k=0,1,
 \\
  \|(t+\varepsilon ^{\frac{\alpha}{2}})\langle z \rangle ^{l}\partial _{t}^{k}\partial _{z}^{3-2k}v ^{B,1}\| _{L _{T}^{2}L _{z}^{2}}+\|(t ^{\frac{3}{2}} +\varepsilon ^{\frac{\alpha}{2}})\langle z \rangle ^{l}\partial _{t}^{k}\partial _{z}^{4-2k}v ^{B,1}\| _{L _{T}^{2}L _{z}^{2}} \leq C,\ \ k=0,1,\\
  \|(t ^{\frac{5+k}{2}}+\varepsilon ^{\frac{3\alpha}{2}})\langle z \rangle ^{l}\partial _{t}^{i}\partial _{z}^{5+k-2i}v ^{B,1}\| _{L _{T}^{2}L _{z}^{2}}\leq C,\ i=0,1,2,\ k=0,1
\end{gather}
and
\begin{gather}
\displaystyle  \|\langle z \rangle ^{l}\varphi ^{B,2}\|_{L _{T}^{2}H _{z}^{3}}+\|(t ^{\frac{2+3j}{2} }+\varepsilon ^{\frac{(1+2j)\alpha}{2}})\langle z \rangle ^{l}\partial _{t}^{1+j}\varphi ^{B,2}\| _{L _{T}^{2}H _{z}^{2-2j}}\leq C ,\ \ j=0,1,  \\
 \displaystyle \|(t +\varepsilon ^{\frac{\alpha}{2}})\langle z \rangle ^{l}\partial _{z}^{4}\varphi ^{B,2}\|_{L _{T}^{2}L _{z}^{2}} +\|(t ^{\frac{3}{2}} +\varepsilon ^{\frac{3 \alpha}{4}})\langle z \rangle ^{l}\partial _{z}^{5}\varphi ^{B,2}\|_{L _{T}^{2}L _{z}^{2}} \leq C,
  \\
  \displaystyle \|(t ^{\frac{3}{2}} +\varepsilon ^{\frac{ \alpha}{2}})\langle z \rangle ^{l}\partial _{t}\partial _{z}^{3}\varphi ^{B,2}\|_{L _{T}^{2}L _{z}^{2}}  +\|(t ^{\frac{5}{2}}+\varepsilon ^{\frac{3\alpha}{2}})\langle z \rangle ^{l}\partial _{t}^{2}\varphi _{z}^{B,2}\| _{L _{T}^{2}L _{z}^{2}}  \leq C,
    \\
    \displaystyle \|(t ^{3}+\varepsilon ^{\frac{3\alpha}{2}})\langle z \rangle ^{l}\partial _{t}^{i}\partial _{z}^{6-2i}\varphi ^{B,2}\| _{L _{T}^{2}L _{z}^{2}} \leq C, \ \ i=0,1,2.
\end{gather}
\end{subequations}

\begin{proof}
The ideas and technicalities of the proof of Lemma \ref{lem-v-B-1} involves tedious and complicated calculations, the ideas and procedures are essentially similar to those in the proof of Lemma \ref{lemf-vfi-V-I-1}. For brevity, we conclude the proof without presenting the details here.
\end{proof}
\end{lemma}

With Lemma \ref{lem-v-B-1}, the following estimates can be readily derived based on  Proposition \ref{prop-embeding-spacetime} in the Appendix A.
\begin{corollary}\label{cor-INF-V-B-1-VFI-B-2}
Let $ (v ^{B,1}, \varphi ^{B,2})$ be the solution obtained in Lemma \ref{lem-v-B-1}.
It holds for $ i=0,1 $ that
\begin{subequations}\label{INF-V-B-1}
\begin{gather}
\displaystyle \|\langle z \rangle ^{l}v ^{B,1}\|_{L _{T}^{\infty}H _{z}^{1}}+ \|\varepsilon ^{\frac{(1+2i)\alpha}{2}}\langle z \rangle ^{l}\partial _{z}^{2+2i} v ^{B,1}\|_{L _{T}^{\infty}H _{z}^{1}} +\|t ^{\frac{2+i}{2}}\langle z \rangle ^{l}\partial _{z}^{2+i} v ^{B,1}\|_{L _{T}^{\infty}L _{z}^{2}} \leq C , \\
\displaystyle \|t ^{\frac{5+i}{2}}\langle z \rangle ^{l}\partial _{z}^{4+i} v ^{B,1}\|_{L _{T}^{\infty}L _{z}^{2}}+\|(t ^{\frac{3}{2}}+\varepsilon ^{\frac{\alpha}{2}})\langle z \rangle ^{l}\partial _{t} v ^{B,1}\|_{L _{T}^{\infty}H _{z}^{1}} +\|\varepsilon ^{\frac{3\alpha}{2}}\langle z \rangle ^{l}\partial _{t}\partial _{z}^{2} v ^{B,1}\|_{L _{T}^{\infty}H _{z}^{1}} \leq C ,\\
\displaystyle  \|t ^{\frac{5+i}{2}}\langle z \rangle ^{l}\partial _{t}\partial _{z}^{2+i} v ^{B,1}\|_{L _{T}^{\infty}L _{z}^{2}}+\|(t ^{3}+\varepsilon ^{\frac{3\alpha}{2}})\langle z \rangle ^{l}\partial _{t}^{2}v ^{B,1}\|_{L _{T}^{\infty}H _{z}^{1}} \leq C,\label{INF-V-B-1-c}
\end{gather}
\end{subequations}
and that
\begin{subequations}\label{INF-VFI-B-2}
\begin{gather}
\displaystyle
\|(t ^{\frac{1}{2}}+\varepsilon ^{\frac{\alpha}{4}})\|\langle z \rangle ^{l}\varphi ^{B,2}\|_{L _{T}^{\infty}H _{z}^{2}}+ \|(t  +\varepsilon ^{\frac{\alpha}{2}})\langle z \rangle ^{l}\partial _{z}^{3} \varphi ^{B,2}\|_{L _{T}^{\infty}L _{z}^{2}} \leq C,\\
\displaystyle \|(t ^{\frac{3}{2}}+\varepsilon ^{\frac{3\alpha}{4}}))\langle z \rangle ^{l}\partial _{z}^{4} \varphi ^{B,2}\|_{L _{T}^{\infty}L _{z}^{2}}+\|(t ^{3}+\varepsilon ^{\frac{3\alpha}{2}})\langle z \rangle ^{l}\partial _{z}^{5} \varphi ^{B,2}\|_{L _{T}^{\infty}L _{z}^{2}} \leq C,
  \\
 \displaystyle \|( t ^{2} +\varepsilon ^{\alpha})\langle z \rangle ^{l}\partial _{t}\varphi ^{B,2}\|_{L _{T}^{\infty}H _{z}^{2}}+\|(t ^{3}+\varepsilon ^{\frac{3\alpha}{2}})\langle z \rangle ^{l}\partial _{t}\partial _{z}^{3}\varphi ^{B,2}\|_{L _{T}^{\infty}L _{z}^{2}}\leq C,
\end{gather}
\end{subequations}
where $ C>0 $ is a constant independent of $ \varepsilon $.
\end{corollary}

Similar to Lemma \ref{lem-v-B-1}, we have the following existence and regularity  result for $ (\varphi ^{b,2},v ^{b,1}) $.
\begin{lemma}\label{lem-v-b-1}
Assume the conditions in Lemmas \ref{lem-regul-outer-layer-0}, \ref{lem-v-B-0} and \ref{lemf-vfi-V-I-1} hold. Then there exists a unique solution $ (\varphi ^{b,2}, v ^{b,1}) $ to the problem \eqref{sec-bd-eq-rt}  on $ [0,T] $ for any $ T \in (0,\infty) $ such that for any $ l \in \mathbb{N} $,
\begin{gather*}
\displaystyle \langle \xi \rangle ^{l}\partial _{t}^{k}v ^{b,1} \in L _{T}^{2}H _{\xi}^{6-2k}\ \ \ \mbox{for }k=0,1,2,3,\  \mbox{and } \langle \xi \rangle ^{l}\partial _{t}^{k}\varphi ^{b,2}\in L _{T}^{2}H _{\xi}^{7-2k}\ \ \ \mbox{for }k=0,1,2,
 \end{gather*}
 with
 \begin{subequations}\label{tild-v-b-1-in-lem}
\begin{gather}
\displaystyle \|\langle \xi \rangle ^{l}\partial _{t}^{k}v ^{b,1}\|_{L _{T}^{2}H _{\xi}^{2-2k}}
+\|\langle \xi \rangle ^{l} (t ^{\frac{3+3k}{2}}+\varepsilon ^{\frac{(1+2k)\alpha}{2}})\partial _{t}^{2+k}v ^{b,1}\|_{L _{T}^{2}L _{\xi}^{2}} \leq C ,\ \ k=0,1,
 \\
  \|(t+\varepsilon ^{\frac{\alpha}{2}})\langle \xi \rangle ^{l}\partial _{t}^{k}\partial _{\xi}^{3-2k}v ^{b,1}\| _{L _{T}^{2}L _{z}^{2}}+\|(t ^{\frac{3}{2}} +\varepsilon ^{\frac{\alpha}{2}})\langle \xi \rangle ^{l}\partial _{t}^{k}\partial _{\xi}^{4-2k}v ^{b,1}\| _{L _{T}^{2}L _{\xi}^{2}} \leq C,\ \ k=0,1,\\
  \|(t ^{\frac{5+k}{2}}+\varepsilon ^{\frac{3\alpha}{2}})\langle \xi \rangle ^{l}\partial _{t}^{i}\partial _{\xi}^{5+k-2i}v ^{b,1}\| _{L _{T}^{2}L _{\xi}^{2}}\leq C,\ i=0,1,2,\ \ k=0,1
\end{gather}
and
\begin{gather}
\displaystyle  \|\langle \xi \rangle ^{l}\varphi ^{b,2}\|_{L _{T}^{2}H _{\xi}^{3}}+\|(t ^{\frac{2+3j}{2} }+\varepsilon ^{\frac{(1+2j)\alpha}{2}})\langle \xi \rangle ^{l}\partial _{t}^{1+j}\varphi ^{b,2}\| _{L _{T}^{2}H _{\xi}^{2-2j}}\leq C ,\ \ j=0,1,  \\
 \displaystyle \|(t +\varepsilon ^{\frac{\alpha}{2}})\langle \xi \rangle ^{l}\partial _{\xi}^{4}\varphi ^{b,2}\|_{L _{T}^{2}L _{\xi}^{2}} +\|(t ^{\frac{3}{2}} +\varepsilon ^{\frac{3 \alpha}{4}})\langle z \rangle ^{l}\partial _{\xi}^{5}\varphi ^{b,2}\|_{L _{T}^{2}L _{\xi}^{2}} \leq C,
  \\
  \displaystyle \|(t ^{\frac{3}{2}} +\varepsilon ^{\frac{ \alpha}{2}})\langle \xi \rangle ^{l}\partial _{t}\partial _{\xi}^{3}\varphi ^{b,2}\|_{L _{T}^{2}L _{\xi}^{2}}  +\|(t ^{\frac{5}{2}}+\varepsilon ^{\frac{3\alpha}{2}})\langle \xi \rangle ^{l}\partial _{t}^{2}\varphi _{\xi}^{b,2}\| _{L _{T}^{2}L _{\xi}^{2}}  \leq C,
    \\
    \displaystyle \|(t ^{3}+\varepsilon ^{\frac{3\alpha}{2}})\langle \xi \rangle ^{l}\partial _{t}^{i}\partial _{\xi}^{6-2i}\varphi ^{b,2}\| _{L _{T}^{2}L _{\xi}^{2}} \leq C, \ \ i=0,1,2.
\end{gather}

\end{subequations}

\end{lemma}
Similar to Corollary \ref{cor-INF-V-B-1-VFI-B-2}, we have
\begin{corollary}\label{cor-INF-vb1-vfib2}
It holds that
\begin{subequations}\label{INF-vb1-vfib2}
\begin{gather}
\displaystyle \|\langle \xi \rangle ^{l}v ^{b,1}\|_{L _{T}^{\infty}H _{\xi}^{1}}+ \|\varepsilon ^{\frac{(1+2i)\alpha}{2}}\langle \xi \rangle ^{l}\partial _{\xi}^{2+2i} v ^{b,1}\|_{L _{T}^{\infty}H _{\xi}^{1}} +\|t ^{\frac{2+i}{2}}\langle \xi \rangle ^{l}\partial _{\xi}^{2+i} v ^{b,1}\|_{L _{T}^{\infty}L _{\xi}^{2}} \leq C , \\
\displaystyle \|t ^{\frac{5+i}{2}}\langle \xi \rangle ^{l}\partial _{\xi}^{4+i} v ^{b,1}\|_{L _{T}^{\infty}L _{\xi}^{2}}+\|(t ^{\frac{3}{2}}+\varepsilon ^{\frac{\alpha}{2}})\langle \xi \rangle ^{l}\partial _{t} v ^{b,1}\|_{L _{T}^{\infty}H _{\xi}^{1}} +\|\varepsilon ^{\frac{3\alpha}{2}}\langle \xi \rangle ^{l}\partial _{t}\partial _{\xi}^{2} v ^{b,1}\|_{L _{T}^{\infty}H _{\xi}^{1}} \leq C ,\\
\displaystyle  \|t ^{\frac{5+k}{2}}\langle \xi \rangle ^{l}\partial _{t}\partial _{\xi}^{2+k} v ^{b,1}\|_{L _{T}^{\infty}L _{\xi}^{2}}+\|(t ^{3}+\varepsilon ^{\frac{3\alpha}{2}})\langle \xi \rangle ^{l}\partial _{t}^{2}v ^{b,1}\|_{L _{T}^{\infty}H _{\xi}^{1}} \leq C ,
\end{gather}
\end{subequations}
and that
\begin{subequations}\label{INF-vfi-b2}
\begin{gather}
\displaystyle
\|(t ^{\frac{1}{2}}+\varepsilon ^{\frac{\alpha}{4}})\|\langle \xi \rangle ^{l}\varphi ^{b,2}\|_{L _{T}^{\infty}H _{\xi}^{2}}+ \|(t  +\varepsilon ^{\frac{\alpha}{2}})\langle \xi \rangle ^{l}\partial _{\xi}^{3} \varphi ^{b,2}\|_{L _{T}^{\infty}L _{\xi}^{2}} \leq C,\\
\displaystyle \|(t ^{\frac{3}{2}}+\varepsilon ^{\frac{3\alpha}{4}}))\langle \xi \rangle ^{l}\partial _{\xi}^{4} \varphi ^{b,2}\|_{L _{T}^{\infty}L _{\xi}^{2}}+\|(t ^{3}+\varepsilon ^{\frac{3\alpha}{2}})\langle \xi \rangle ^{l}\partial _{\xi}^{5} \varphi ^{b,2}\|_{L _{T}^{\infty}L _{\xi}^{2}} \leq C,
 \\
 \displaystyle \|(t ^{2} +\varepsilon ^{\alpha})\langle \xi \rangle ^{l}\partial _{t}\varphi ^{b,2}\|_{L _{T}^{\infty}H _{\xi}^{2}}+\|(t ^{3}+\varepsilon ^{\frac{3\alpha}{2}})\langle \xi \rangle ^{l}\partial _{t}\partial _{\xi}^{3}\varphi ^{b,2}\|_{L _{T}^{\infty}L _{\xi}^{2}}\leq C,
\end{gather}
\end{subequations}
where $ i=0,1 $, $ C>0 $ is a constant independent of $ \varepsilon $.
\end{corollary}

\section{Convergence of boundary layers} % (fold)
\label{sec:stability_of_boundary_layers}
This section is devoted to proving Theorem \ref{thm-stabi-refor}, while Theorem \ref{thm-original} is a consequence of Theorem \ref{thm-stabi-refor} by the change of variables as briefed in Section 1.
\subsection{Reformulation of the problem} % (fold)
\label{sub:reformulation_of_the_problem}
Denote by $ (\varphi ^{\varepsilon},v ^{\epsilon}) $ the solution to the problem \eqref{refor-eq}. To study the convergence of boundary-layer profiles stated in Theorem \ref{thm-stabi-refor}, by noticing that $\varphi ^{B,0}=\varphi ^{B,0}\equiv 0$, we write $ (\varphi ^{\varepsilon},v ^{\epsilon}) $ as
\begin{align*}
\displaystyle \begin{cases}
  \displaystyle\varphi ^{\varepsilon}=\varphi ^{I,0}(x,t)+\varepsilon ^{1/2}(\varphi ^{I,1}(x,t)+\varphi ^{B,1}(z,t)+\varphi ^{b,1}(\xi,t))+\mathcal{E}_{1}^{\varepsilon},\\[2mm]
  \displaystyle v ^{\varepsilon}=v ^{I,0}(x,t)+v ^{B,\varepsilon}(z,t)+v ^{b,\varepsilon}(\xi,t)+\mathcal{E}_{2}^{\varepsilon},
\end{cases}
\end{align*}
and derive equations for the remainder $ (\mathcal{E}_{1}^{\varepsilon},\mathcal{E}_{2}^{\varepsilon}) $ to show for any $ T>0 $ that
\begin{align}\label{cal-E}
\displaystyle \|\mathcal{E}_{1}^{\varepsilon}\|_{L _{T}^{\infty}L ^{\infty}}
  \leq C \varepsilon ^{\frac{3 \iota _{0}}{2}-\frac{3}{8}},\ \ \|t ^{\frac{5}{4}}\partial _{x}\mathcal{E}_{1}^{\varepsilon}\|_{L _{T}^{\infty}L ^{\infty}}  \leq C \varepsilon ^{2\iota _{0}-1},\ \ \ \|\mathcal{E}_{2}^{\varepsilon}\|_{L _{T}^{\infty}L ^{\infty}} \leq C  \varepsilon ^{\iota _{0}-\frac{1}{4}},
\end{align}
where $ 1/2<\iota _{0}<7/12 $ is as in \eqref{iota0-defi}. Note that the estimate \eqref{vfi-VI-0} naturally holds for small $\varepsilon>0$ under \eqref{cal-E} since $\frac{3 \iota _{0}}{2}-\frac{3}{8}<\frac{1}{2}$.  However, if we substitute the above perturbations into the equations in \eqref{refor-eq},  we shall find that the equations of $ \mathcal{E}_{i}^{\varepsilon} $ have some terms at the zero-th order of $ \varepsilon $, which are obstacles to derive the estimates in \eqref{cal-E}. To resolve this emerging issue, we  approximate the solution of \eqref{refor-eq}, \eqref{BD-POSITIVE-VE} with some higher-order profiles as follows:
%\begin{subequations}\label{approximate}
\begin{align}
\displaystyle \Phi ^{a}(x,t)&:=\varphi ^{I,0}+\varepsilon ^{\frac{1}{2}}\left( \varphi ^{I,1}(x,t)+\varphi ^{B,\varepsilon}(z,t)+\varphi ^{b,\varepsilon}(\xi,t) \right)
 \nonumber \\
 &\quad  +\varepsilon \left( \varphi ^{B,2}(z,t)+\varphi ^{b,2}(\xi,t) \right)
 +b _{\varphi}^{\varepsilon}(x,t),
  \nonumber \\
  \displaystyle V ^{a}(x,t)&:=  v ^{I,0}(x,t)+v ^{B,\varepsilon}(z,t)+v ^{b,\varepsilon}(\xi,t)+ \varepsilon^{\frac{1}{2}}\left( v ^{I,1}(x,t)+v ^{B,1}(z,t)+v ^{b,1}(\xi,t) \right)
 +b _{v}^{\varepsilon}(x,t),\label{approximate-v-formula}
\end{align}
%\end{subequations}
where the functions $ b _{\varphi}^{\varepsilon}(x,t) $ and $ b _{v}^{\varepsilon}(x,t) $ are constructed below to homogenize the boundary values of $ (\Phi^{a}, V^{a})  $:
\begin{subequations}
\begin{align}
\displaystyle \displaystyle\makebox[-0.5pt]{~} b _{\varphi}^{\varepsilon}(x,t)&=-(1-x)\left[ \varepsilon ^{\frac{1}{2}}\varphi ^{b,\varepsilon}(- \frac{1}{\varepsilon ^{1/2}} ,t)+\varepsilon \varphi ^{b,2}(-  \frac{1}{\varepsilon ^{1/2}},t) \right]-(1-x)		{\mathop{\mathrm{e}}}^{- \frac{x}{\varepsilon ^{\nu}}}\varepsilon \varphi^{B,2}(0,t)
 \nonumber \\
  & \displaystyle \quad-x \left[  \varepsilon ^{\frac{1}{2}}\varphi ^{B,\varepsilon}(\frac{1}{\varepsilon ^{1/2}},t)+\varepsilon \varphi ^{B,2}(\frac{1}{\varepsilon ^{1/2}},t) \right] - x 		{\mathop{\mathrm{e}}}^{- \frac{1-x}{\varepsilon ^{\nu}}}\varepsilon \varphi ^{b,2}(0,t),   \label{b-vfi-ve}
 \\
 \displaystyle \displaystyle\makebox[-0.5pt]{~} b _{v}^{\varepsilon}(x,t)&=(x-1)\left[v ^{b,\varepsilon}(- \frac{1}{\varepsilon ^{1/2}},t)+ \varepsilon ^{ \frac{1}{2}}v^{b,1}(-\frac{1}{\varepsilon ^{ 1/2}},t)+\cala_{1}^{\varepsilon}(t) \right]
  \nonumber \\
  & \displaystyle \quad -x \left[v ^{B,\varepsilon}( \frac{1}{\varepsilon ^{ 1/2}},t)+  \varepsilon ^{\frac{1}{2}}v ^{B,1}( \frac{1}{\varepsilon ^{1/2}},t)+\cala_{2}^{\varepsilon}(t)\right], \label{b-v-ve}
\end{align}
\end{subequations}
where $ \nu $ is as in \eqref{nu-constriant}. Then we perturb $ (\varphi ^{\varepsilon}, v ^{\varepsilon}) $ as
\begin{gather}\label{solution-recover-esti}
\displaystyle \varphi ^{\varepsilon}-\Phi^{a}= \varepsilon ^{\frac{1}{2}}\Phi ^{\varepsilon},\ \ v ^{\varepsilon}-V^{a}=\varepsilon ^{\frac{1}{2}}V ^{\varepsilon}.
\end{gather}
Clearly, we have
\begin{align}\label{cal-E-iden}
\begin{cases}
  \displaystyle \mathcal{E}_{1}^{\varepsilon}&=\varepsilon ^{1/2}\Phi ^{\varepsilon}+\varepsilon \left( \varphi ^{B,2}(z,t)+\varphi ^{b,2}(\xi,t) \right)  +b _{\varphi}^{\varepsilon}(x,t),\\[1mm]
\displaystyle \mathcal{E}_{2}^{\varepsilon}&= \varepsilon ^{1/2}V ^{\varepsilon}+\varepsilon ^{1/2}(v ^{I,1}(x,t)+v ^{B,1}(z,t)+v ^{b,1}(\xi,t))+b _{v}^{\varepsilon}(x,t).
\end{cases}
\end{align}
In view of the estimates in the previous section for the outer- and boundary-layer profiles, to get \eqref{cal-E}, it now remains to derive the equations for $ (\Phi ^{\varepsilon},V ^{\varepsilon}) $ and establish some desired estimates. Substituting \eqref{solution-recover-esti} into \eqref{refor-eq}, then the perturbation functions $ (\Phi ^{\varepsilon}, V ^{\varepsilon}) $ satisfy
\begin{gather}\label{eq-perturbation}
\displaystyle  \begin{cases}
\displaystyle \Phi _{t}^{\varepsilon}=\Phi _{xx}^{\varepsilon}-\varepsilon ^{\frac{1}{2}}\Phi _{x}^{\varepsilon}V _{x}^{\varepsilon}-\Phi _{x}^{\varepsilon}V_{x}^{a}- V _{x}^{\varepsilon}(\Phi_{x}^{a}+M)+\varepsilon ^{-\frac{1}{2}}\mathcal{R}_{1}^{\varepsilon},\\[1mm]
  \displaystyle V _{t}^{\varepsilon}=\varepsilon V _{xx}^{\varepsilon}-\varepsilon ^{\frac{1}{2}}\Phi _{x}^{\varepsilon}V ^{\varepsilon}-\Phi _{x}^{\varepsilon}V^{a}-(\Phi_{x}^{a}+M)V ^{\varepsilon}
  +\varepsilon ^{- \frac{1}{2}}\mathcal{R}_{2}^{\varepsilon},\\
  \displaystyle (\Phi ^{\varepsilon},V ^{\varepsilon})(x,0)=(0,0),\\
  \displaystyle (\Phi ^{\varepsilon},V ^{\varepsilon})(0,t)=(\Phi ^{\varepsilon},V ^{\varepsilon})(1,t)=(0,0),
  \end{cases}
\end{gather}
where the error terms $ \mathcal{R}_{i}^{\varepsilon}\,(i=1,2) $ are given by
\begin{align}\label{E-rror-fomula}
\displaystyle \mathcal{R}_{1}^{\varepsilon}=\Phi_{xx}^{a} -(\Phi_{x}^{a}+M)V_{x}^{a} -\Phi_{t}^{a}, \ \ \ \mathcal{R}_{2}^{\varepsilon}=\varepsilon V_{xx}^{a}-(\Phi_{x}^{a}+M)V ^{a}-V_{t}^{a}.
\end{align}
Notice that the coefficients and error terms in the problem \eqref{eq-perturbation} only involve the outer- and boundary-layer profiles which have been studied in the previous section. By standard arguments (cf. \cite{nishida-1978,zhaokun-2015-JDE}), one can prove the local existence of solutions to the problem \eqref{eq-perturbation} with $ \varepsilon>0 $ on a time interval $ [0,T _{\varepsilon}] $ for some $ T _{\varepsilon}>0 $ which may be small. Now the key of the matter is to establish some \emph{a priori} estimates for $ (\Phi ^{\varepsilon}, V ^{\varepsilon}) $ uniformly in $ \varepsilon $ from which one can extract the convergence of boundary layers. We also remark that following the procedure of establishing the uniform estimates, one can easily derive some higher-order estimates depending on $ \varepsilon $.  With these estimates, one can obtain the existence of solutions on the time interval $ [0,T] $ for any $ T>0 $ for fixed $ \varepsilon>0$.. Specifically, we will prove the following results for the problem \eqref{eq-perturbation}.

	\begin{proposition}\label{prop-pertur}
Assume the conditions in Theorem \ref{thm-stabi-refor} hold. Then for any $ T>0 $, there exists a constant $  \varepsilon _{T} >0$ depending on $ T $ such that for any $ \varepsilon \in (0,\varepsilon _{T}] $, the problem \eqref{eq-perturbation} admits a unique solution $ (\Phi ^{\varepsilon}, V ^{\varepsilon}) \in C([0,T];H ^{2}\times H ^{2}) $ which satisfies for any $ t \in [0,T] $,
\begin{gather*}
\displaystyle  \|\Phi ^{\varepsilon}(\cdot,t)\|_{L ^{2}}^{2} +\varepsilon ^{\frac{3}{2}-2 \iota _{0}}\|\Phi _{x} ^{\varepsilon}(\cdot,t)\|_{L ^{2}}^{2}+\|V ^{\varepsilon}(\cdot,t)\|_{L ^{2}}^{2}+ \varepsilon\|V _{x}^{\varepsilon}(\cdot,t)\|_{L ^{2}}^{2} \leq C \varepsilon ^{2 \iota _{0}-1},
 \nonumber \\
  \|t ^{\frac{5}{2}}\Phi _{xx}^{\varepsilon}(\cdot,t)\|_{ L ^{2}}^{2}+\varepsilon\|t ^{\frac{5}{2}}V _{xx}^{\varepsilon}(\cdot,t)\|_{L ^{2}}^{2} \leq C \varepsilon ^{4 \iota _{0}-\frac{7}{2}},
\end{gather*}
and
\begin{gather*}
\displaystyle \int _{0}^{t}\left(\|\Phi _{x}^{\varepsilon}\|_{L ^{2}}^{2}+\|V _{\tau}^{\varepsilon}\| _{ L ^{2}}^{2}+\varepsilon ^{\frac{3}{2}-2 \iota _{0}} \|\Phi _{\tau}^{\varepsilon}\| _{ L ^{2}}^{2} + \varepsilon ^{\frac{5}{2}-2 \iota _{0}} (\|\tau^{\frac{5}{2}}\Phi _{x \tau}^{\varepsilon}\|_{L ^{2}}^{2}+\varepsilon\|\tau^{\frac{5}{2}}V _{x \tau}^{\varepsilon}\|_{L ^{2}}^{2}) \right) \mathrm{d}\tau \leq C  \varepsilon ^{2\iota _{0}-1},
\end{gather*}
where $ C >0$ is a constant depending on $ T $ but independent of $ \varepsilon $.
\begin{proof}
 The proof of Proposition \ref{prop-pertur} consists of Lemmas \ref{lem-R-1-VE}--\ref{lem-FIDA-V-XX} in the sequel.
\end{proof}
\end{proposition}

\subsection{\emph{A priori} estimates} % (fold)
\label{subsec:em}

\subsubsection{Estimates on the error terms} % (fold)
\label{sub:estimates_on_the_error_terms}

% subsection estimates_on_the_error_terms (end)
Now let us turn to estimates on the error terms $ \mathcal{R}_{1}^{\varepsilon} $ and $ \mathcal{R}_{2}^{\varepsilon} $. Using \eqref{E-rror-fomula} and the first equations in \eqref{eq-outer-0} and in \eqref{first-outer-problem}, we get from a direct computation that
\begin{align}\label{R-1-SPLIT}
\displaystyle \mathcal{R}_{1}^{\varepsilon} =\mathcal{R}_{1,1}^{\varepsilon} +\mathcal{R}_{1,2}^{\varepsilon},
\end{align}
where
\begin{align}\label{R-1-1-defi}
\displaystyle \mathcal{R}_{1,1}^{\varepsilon}&=-   v _{x}^{ B,\varepsilon}\left[\varphi _{x}^{I,0}-\varphi _{x} ^{I,0}(0,t)-x\varphi _{xx} ^{I,0}(0,t)\right]- v _{x}^{ b,\varepsilon}\left[\varphi _{x} ^{I,0}-\varphi _{x}^{I,0}(1,t)-(x-1)\varphi _{xx}^{I,0}(1,t)) \right]
   \nonumber \\
   &\displaystyle \quad- \left[\varepsilon ^{\frac{1}{2}}v _{x}^{B,1}(\varphi _{x}^{I,0}-\varphi  _{x}^{I,0}(0,t))+  \varepsilon ^{\frac{1}{2}}v _{x}^{b,1}(\varphi _{x}^{I,0}-\varphi  _{x}^{I,0}(1,t))\right]
     \nonumber \\
     & \displaystyle \quad  - \varepsilon ^{\frac{1}{2}}\left[ \varphi _{x}^{ B,\varepsilon} (v _{x}^{I,0}-v _{x}^{I,0}(0,t))+\varphi _{x}^{ b,\varepsilon} (v _{x}^{I,0}-v _{x}^{I,0}(1,t)) \right] - \varepsilon ^{\frac{1}{2}}\left( v _{x}^{ B,\varepsilon}\varphi _{x}^{ b,\varepsilon}+v _{x}^{ b,\varepsilon}\varphi _{x}^{ B,\varepsilon} \right)
     \nonumber \\
                  &\displaystyle \quad  -\varepsilon(v _{x}^{B,1}\varphi _{x}^{ b,\varepsilon}+v _{x}^{b,1}\varphi _{x}^{ B,\varepsilon}+\varphi _{x} ^{B,2} v _{x}^{ b,\varepsilon}+\varphi _{x}^{b,2}v _{x}^{ B,\varepsilon})- \varepsilon \left[ \varphi _{x}^{I,1}(v _{x}^{B,1}+v _{x}^{b,1})  \right]-\varepsilon\varphi _{x}^{I,1}v _{x}^{I,1}
 \nonumber \\
               & \displaystyle \quad
            - \varepsilon \varphi _{x}^{B,2}\left[ v ^{I,0}
       +\varepsilon ^{\frac{1}{2}} (v ^{I,1}+v ^{B,1}+v ^{b,1}) \right] _{x}- \varepsilon \varphi _{x}^{b,2}\left[ v ^{I,0}
       +\varepsilon ^{\frac{1}{2}} (v ^{I,1}+v ^{B,1}+v ^{b,1}) \right] _{x}
      \nonumber \\
       & \displaystyle \quad
       - \left[ \varepsilon ^{\frac{1}{2}}( \varphi _{t}^{ B,\varepsilon}+\varphi _{t}^{  b,\varepsilon} )+ \varepsilon  ( \varphi _{t}^{B,2}+\varphi _{t} ^{b,2} ) \right]     +F ^{\varepsilon}=: \sum _{i=1} ^{11}\mathcal{P}_{i}+F ^{\varepsilon},
\end{align}
\begin{align}\label{R-1-2-defi}
\mathcal{R}_{1,2}^{\varepsilon}&=- \left[\varepsilon ^{\frac{1}{2}}v _{x}^{ B,\varepsilon}(\varphi _{x}^{I,1}-\varphi _{x}^{I,1}(0,t))+  \varepsilon ^{\frac{1}{2}}v _{x}^{ b,\varepsilon}(\varphi _{x}^{I,1}-\varphi _{x}^{I,1}(1,t))\right]  - \varepsilon v _{x}^{I,1}(\varphi _{x}^{ B,\varepsilon}+\varphi _{x}^{ b,\varepsilon})
\end{align}
with
\begin{equation}\label{F-defi}
\begin{aligned}
F^{\varepsilon}= & -\partial_x b_v^{\varepsilon}\left(\varphi_x^{I, 0}+M+\varepsilon^{\frac{1}{2}}\left(\varphi_x^{I, 1}+\varphi_x^{B, \varepsilon}+\varphi_x^{b, \varepsilon}\right)+\varepsilon\left(\varphi_x^{B, 2}+\varphi_x^{b, 2}\right)\right) \\
& -\partial_x b_{\varphi}^{\varepsilon}\left(v^{I, 0}+v^{B, \varepsilon}+v^{b, \varepsilon}+\varepsilon^{\frac{1}{2}}\big(v^{I, 1}+v^{B, 1}+v^{b, 1}\big)\right)_x-\partial_x b_{\varphi}^{\varepsilon} \partial_x b_v^{\varepsilon}-\partial_t b_{\varphi}^{\varepsilon} \\
& -\partial_x^2\left[(1-x) \mathrm{e}^{-\frac{x}{\varepsilon^\nu}} \varepsilon \varphi^{B, 2}(0, t)\right]-\partial_x^2\left[x \mathrm{e}^{-\frac{1-x}{\varepsilon^\nu}} \varepsilon \varphi^{b, 2}(0, t)\right].
\end{aligned}
\end{equation}

%\begin{align}\label{F-defi}
%\displaystyle F ^{\varepsilon}：&= - \partial _{x} b _{v}^{\varepsilon}\left( \varphi _{x}^{I,0}+M+\varepsilon ^{\frac{1}{2}}(\varphi_{x} ^{I,1}+ \varphi _{x}^{ B,\varepsilon}+\varphi _{x} ^{ b,\varepsilon}) +\varepsilon(\varphi  _{x}^{B,2}+\varphi _{x} ^{b,2})  \right)
% \nonumber \\
% & \displaystyle \quad  -\partial _{x} b _{\varphi}^{\varepsilon}\left( v ^{I,0}+v ^{B,\varepsilon}+v ^{b,\varepsilon}+\varepsilon ^{\frac{1}{2}}(v ^{I,1}+v ^{B,1}+v ^{b,1})\right)_{x}  -(\partial _{x} b _{\varphi}^{\varepsilon}\partial _{x} b _{v}^{\varepsilon})- \partial _{t} b _{\varphi}^{\varepsilon}  \nonumber \\
% %&~-\partial_x^2\left[(1-x) \mathrm{e}^{-\frac{x}{\varepsilon^\nu}} \varepsilon \varphi^{B, 2}(0, t)\right]-\partial_x^2\left[x \mathrm{e}^{-\frac{1-x}{\varepsilon^\nu}} \varepsilon \varphi^{b, 2}(0, t)\right].
%  %&~\displaystyle \quad - \partial _{x}^{2}\left[ (1-x)		{\mathop{\mathrm{e}}}^{- \frac{x}{\varepsilon ^{\nu}}}\varepsilon \varphi^{B,2}(0,t) \right]-\partial _{x}^{2}\left[ x		{\mathop{\mathrm{e}}}^{- \frac{1-x}{\varepsilon ^{\nu}}}\varepsilon \varphi^{b,2}(0,t) \right] .
%\end{align}
The estimates on $ \mathcal{R}_{1,1}^{\varepsilon} $ and $ \mathcal{R}_{1,2}^{\varepsilon} $ can be stated as follows.

\begin{lemma}\label{lem-R-1-VE}
Let $ 0< \varepsilon <1 $. It holds for any $ T>0 $ that
	\begin{gather}\label{con-R-1-in-lem}
 \displaystyle \makebox[-2pt]{~} \| \mathcal{R}_{1,1}^{\varepsilon}\|_{L _{T}^{2}L ^{2}} +\|t \mathcal{R}_{1,1}^{\varepsilon}\|_{L _{T}^{\infty}L _{z}^{2}}+\|t ^{\frac{5}{2}}\partial _{t}\mathcal{R}_{1,1}^{\varepsilon}\|_{L _{T}^{2}L ^{2}}\leq C  \varepsilon ^{\iota _{0}},\ \ \| \mathcal{R}_{1,2}^{\varepsilon}\|_{L _{T}^{2}L ^{2}}+ \| t\partial _{t}\mathcal{R}_{1,2}^{\varepsilon}\|_{L _{T}^{2}L ^{2}}\leq C \varepsilon ^{\frac{1}{2}},
\end{gather}
where $ \iota _{0}= \min \left\{ \frac{3}{4}-\nu, \frac{\alpha}{2}, \frac{2+\nu - \alpha}{2}, \frac{2-3 \nu}{2}  \right\} $, the constant $ C>0 $ may depend on $ T $ but independent of $ \varepsilon $.
\end{lemma}

\begin{proof}
By Taylor's formula, \eqref{con-vfi-v-I-0-regula}, \eqref{L-INFT-V-B-0}, \eqref{Sobolev-infty} and \eqref{z-transfer}, we have
\begin{align}\label{P-1-esti}
&\displaystyle \|\mathcal{P}_{1}\|_{L _{T}^{2}L ^{2}} =\left\|\frac{\varphi _{x}^{I,0}-\varphi _{x} ^{I,0}(0,t)-x\varphi _{xx}^{I,0}(0,t)}{x ^{2}}x ^{2}v _{x}^{ B,\varepsilon}\right\|_{L _{T}^{\infty}L ^{2}}
 \nonumber \\
 &~\displaystyle \leq  \frac{1}{2} \|\partial _{x}^{3}\varphi ^{I,0}\|_{L _{T}^{\infty}L ^{\infty}} \|x ^{2}v _{x}^{ B,\varepsilon}\|_{L _{T}^{2}L  ^{2}} \leq C \varepsilon \|\varphi ^{I,0}\|_{L _{T}^{\infty}H ^{4}}\|z ^{2}v _{x}^{ B,\varepsilon}\|_{L _{T}^{2}L  ^{2}}
  \nonumber \\
  & ~\displaystyle \leq C \varepsilon ^{\frac{3}{4}}\|\varphi ^{I,0}\|_{L _{T}^{\infty}H ^{4}}\|z ^{2}v _{z}^{ B,\varepsilon}\|_{L _{T}^{\infty}L  _{z}^{2}} \leq C \varepsilon ^{\frac{3}{4}}.
\end{align}
Similarly, by \eqref{l-INFTY-VFI-I-0}, \eqref{L-INFT-V-B-0}, \eqref{tild-v-B-1-lem}, \eqref{Sobolev-z-xi} and \eqref{xi-transfer}, we have
\begin{gather*}
\displaystyle \|\mathcal{P}_{2}\|_{L _{T}^{\infty}L ^{2}} \leq  \frac{1}{2} \|\partial _{x}^{3}\varphi ^{I,0}\|_{L_{T}^{\infty} L ^{\infty}} \|(x-1)^{2}v _{x}^{ b,\varepsilon}\|_{L _{T}^{2}L  ^{2}} \leq C  \varepsilon ^{\frac{3}{4}},\\
 \displaystyle \|\mathcal{P}_{3}\|_{L _{T}^{2}L  ^{2}} \leq C \varepsilon ^{\frac{1}{2}} \|\varphi _{xx}^{I,0}\|_{L_{T}^{\infty} L ^{\infty}}\left( \|x v _{x}^{B,1}\|_{L _{T}^{2}L ^{2}}+\|(x-1) v _{x}^{b,1}\|_{L _{T}^{2}L ^{2}} \right) \leq C  \varepsilon ^{\frac{3}{4}},\\
\displaystyle \|\mathcal{P}_{4}\|_{L _{T}^{2}L  ^{2}} \leq C \varepsilon ^{\frac{1}{2}} \|v _{x}^{I,0}\|_{L _{T}^{\infty}L ^{\infty}}\left( \|x \varphi _{x}^{ B,\varepsilon}\|_{L _{T}^{2}L ^{2}}+\|(x-1) \varphi _{x}^{ b,\varepsilon}\|_{L _{T}^{2}L ^{2}}  \right) \leq C   \varepsilon ^{\frac{3}{4}}.
\end{gather*}
Notice that $ \frac{1}{2 \varepsilon ^{1/2}} < z=\frac{x}{\varepsilon ^{1/2}}< \frac{1}{\varepsilon ^{1/2}} $ for $1/2 \leq x \leq 1 $, and that $ - \frac{1}{\varepsilon ^{1/2}} \leq \xi=\frac{x-1}{\varepsilon ^{1/2}} \leq -\frac{1}{2\varepsilon ^{1/2}} $ for $ 0 \leq x \leq 1/2$. This along with \eqref{L-INFT-V-B-0}, \eqref{L-INFT-V-veb-0}, \eqref{INF-V-B-1} and \eqref{INF-vb1-vfib2} implies for $ m \in \mathbb{N}$ and $  k=0,1,2 $ that
\begin{align}\label{half-l-infty}
&\|\partial _{t}^{k}\partial _{x}^{i} v ^{B,\varepsilon}\|_{L ^{\infty}((1/2, 1)\times(0,T))}+\|\partial _{t}^{k}\partial _{x}^{i} v ^{b,\varepsilon}\|_{L ^{\infty}((0,\frac{1}{2})\times(0,T))}
 \nonumber \\
 &~=\displaystyle \varepsilon ^{\frac{m}{2}}\left( \varepsilon ^{-\frac{m}{2}}\|\partial _{t}^{k}\partial _{x}^{i} v ^{B,\varepsilon}\|_{L ^{\infty}((1/2, 1)\times(0,T))}+\varepsilon ^{-\frac{m}{2}}\|\partial _{t}^{k}\partial _{x}^{i} v ^{b,\varepsilon}\|_{L ^{\infty}((0,\frac{1}{2})\times(0,T))}  \right)
 \nonumber \\
    &~\displaystyle\leq C \varepsilon ^{\frac{m}{2}}\|z ^{m+i} \partial _{t}^{k} \partial _{z}^{i}v ^{B,\varepsilon} \|_{L ^{\infty}(0,T;L _{z}  ^{\infty}(0, \varepsilon ^{-1/2}))}+C\varepsilon ^{\frac{m}{2}} \|\xi ^{m+i} \partial _{t}^{k} \partial _{\xi}^{i}v ^{b,\varepsilon} \|_{L ^{\infty}(0,T;L _{\xi}  ^{\infty}(-\varepsilon ^{-1/2},0))}
     \nonumber \\
     &~ \displaystyle
    \leq C \varepsilon ^{\frac{m}{2}}\|\langle z \rangle ^{m+i}\partial _{t}^{k}\partial _{z}^{i}v ^{B,\varepsilon}\|_{L _{T}^{2}L _{z}  ^{\infty}}
+ C \varepsilon ^{\frac{m}{2}}\|\langle \xi \rangle ^{m+i}\partial _{t}^{k}\partial _{\xi}^{i}v ^{b,\varepsilon}\|_{L _{T}^{2} L _{\xi}  ^{\infty}} \leq C  \varepsilon ^{\frac{m-3 \alpha}{2}} ,
\end{align}
where $ i=0,1,\cdots,4-2k $, and $ C $ is a positive constant independent of $ \varepsilon $. Similarly, we have for $ m \in \mathbb{N}$ that
\begin{gather}
\displaystyle |\partial _{t}^{k}\partial _{x}^{i} v ^{ B,1}\|_{L ^{\infty}((1/2, 1)\times(0,T))}+\|\partial _{t}^{k}\partial _{x}^{i} v ^{ b,1}\|_{L ^{\infty}((0,\frac{1}{2})\times(0,T))}\leq C \varepsilon ^{\frac{m-3 \alpha}{2}} ,\nonumber\\
\displaystyle \|\partial _{t}^{k} \partial _{x}^{i}\varphi ^{B,\varepsilon}\|_{L ^{\infty}((1/2, 1)\times(0,T))}+\|\partial _{t}^{k} \partial _{x}^{i}\varphi ^{b,\varepsilon}\|_{L ^{\infty}((0,1/2)\times(0,T))}  \leq C \varepsilon ^{\frac{m-3 \alpha}{2}} \label{vfi-B-1-HALF-INTY}
\end{gather}
for $ k=0,1,2$, $i=0,1,\cdots,4-2k $, and
\begin{gather}\label{vfi-B-2-half}
\displaystyle  \displaystyle  \left( \|\partial _{t}^{k} \partial _{x}^{i}\varphi ^{B,2}\|_{L ^{\infty}((1/2, 1)\times(0,T))}+ \|\partial _{t}^{k} \partial _{x}^{i}\varphi ^{b,2}\|_{L ^{\infty}((0,1/2)\times(0,T))} \right)  \leq C \varepsilon ^{\frac{m-3 \alpha}{2}},
\end{gather}
if $  k=0,1$, $i=0,1,\cdots,4-2k $. Therefore, we deduce for $ m \in \mathbb{N} \cap \left\{ m \geq 1+ 3\alpha \right\}$ that
\begin{align}\label{P-6-1}
\displaystyle  \varepsilon ^{\frac{1}{2}}\|v _{x}^{ B,\varepsilon}\varphi _{x}^{ b,\varepsilon}\|_{L _{T}^{2}L ^{2}}
 &\displaystyle \leq \varepsilon ^{\frac{1}{2}} \|v _{x}^{ B,\varepsilon}\varphi _{x}^{ b,\varepsilon}\|_{L ^{2}(0,T;L ^{2}((0,1/2)))}+\varepsilon ^{\frac{1}{2}} \|v _{x}^{ B,\varepsilon}\varphi _{x}^{ b,\varepsilon}\|_{L ^{2}(0,T;L ^{2}((0,1/2)))}
  \nonumber \\
  &\displaystyle \leq \varepsilon ^{\frac{1}{2}}\left( \|\varphi _{x}^{ b,\varepsilon}\|_{L ^{\infty}((0,1/2 )\times(0,T))}\|v _{x}^{ B,\varepsilon}\|_{L _{T}^{2}L ^{2}} \right.
   \nonumber \\
   &\displaystyle  \qquad \qquad \  \left.+\|v _{x}^{ B,\varepsilon}\|_{L ^{\infty}((1/2,1)\times (0,T))}\|\varphi _{x}^{ b,\varepsilon}\|_{L _{T}^{2}L ^{2}} \right)
    \nonumber \\
    &\displaystyle \leq C \varepsilon ^{\frac{2m+1-6 \alpha}{4}}   \left(  \|v _{z} ^{ B,\varepsilon}\|_{L _{T} ^{2}L _{z}^{2}}+\| \varphi_{\xi}^{ b,\varepsilon}\|_{L _{T}^{2}L _{\xi}^{2}}\right)
     \nonumber \\
     &\displaystyle \leq C   \varepsilon ^{\frac{2m+1-6 \alpha}{4}} \leq C \varepsilon ^{\frac{3}{4}} ,
\end{align}
\begin{align}%\label{p-6-2}
\displaystyle \varepsilon ^{\frac{1}{2}} \|v _{x}^{ b,\varepsilon}\varphi _{x}^{ B,\varepsilon}\|_{L _{T}^{2}L ^{2}}
 &\displaystyle \leq \varepsilon ^{\frac{1}{2}} \|v _{x}^{ b,\varepsilon}\varphi _{x}^{ B,\varepsilon}\|_{L ^{2}(0,T;L ^{2}((0,1/2)))}+\varepsilon ^{1/2} \|v _{x}^{ b,\varepsilon}\varphi _{x}^{ B,\varepsilon}\|_{L _{T}^{2}(0,T;L ^{2}((1/2,1)))}
  \nonumber \\
  &\displaystyle \leq \varepsilon ^{\frac{1}{2}}\left( \|v _{x}^{ b,\varepsilon}\|_{L ^{\infty}((0,1/2)\times (0,T))}\|\varphi _{x}^{ B,\varepsilon}\|_{L _{T}^{2}L ^{2}} \right.
   \nonumber \\
   &\displaystyle  \qquad \qquad \  \left.+\|\varphi _{x}^{ B,\varepsilon}\|_{L ^{\infty}((1/2,1)\times (0,T))}\|v_{x}^{ b,\varepsilon}\|_{L _{T}^{2}L ^{2}} \right)
    \nonumber \\
    &\displaystyle \leq C   \varepsilon ^{\frac{2m+1-6 \alpha}{4}}   \left(  \|\varphi _{z}^{ B,\varepsilon}\|_{L ^{2}L _{z}^{2}}+\| v_{\xi}^{ b,\varepsilon}\|_{L ^{2}L _{\xi}^{2}}\right)
     \nonumber \\
     &\displaystyle \leq C  \varepsilon ^{\frac{2m+1-6 \alpha}{4}}\leq C \varepsilon ^{\frac{3}{4}}  ,
     \end{align}
     where we have used \eqref{inte-transfer} and Corollaries \ref{cor-V-B-0-INFTY}--\ref{cor-INF-vb1-vfib2}. Thus we have
     \begin{align}\label{p-5-l2}
     \displaystyle  \|\mathcal{P }  _{5}\| _{L _{T}^{2}L ^ 2}\leq \varepsilon ^{\frac{1}{2}}\left(  \|v _{x}^{B,\varepsilon }\varphi _{x}^{ b,\varepsilon}\|_{L _{T}^{2}L ^{2}}+\|v _{x}^{ b,\varepsilon}\varphi _{x}^{ B,\varepsilon}\|_{L _{T}^{2}L ^{2}}\right) \leq C \varepsilon ^{\frac{3}{4}}.
     \end{align}
By similar arguments as proving estimates for $ \mathcal{P}_{5} $, one can infer that $ \|\mathcal{P}_{6}\| _{L _{T}^{2}L ^{2}} \leq C\varepsilon ^{\frac{3}{4}}  $. Thanks to \eqref{l-INFTY-VFI-I-0}, \eqref{vfi-I-1-esti-first-con-lem}, \eqref{inte-transfer} and Corollaries \ref{cor-V-B-0-INFTY}--\ref{cor-INF-vb1-vfib2}, we derive that
\begin{align}
\displaystyle  \|\mathcal{P}_{7}\|_{L _{T}^{2}L ^{2}}& \leq C \varepsilon ^{\frac{1}{2}} \|\varphi _{x}^{I,1}\|_{L _{T}^{2}L ^{\infty}} \left( \|v _{z}^{B,1}\|_{L _{T}^{\infty} L ^{2}}+\|v _{\xi}^{b,1}\|_{L _{T}^{\infty}L  ^{2}} \right)\leq C  \varepsilon ^{\frac{3}{4}},\nonumber
\end{align}
\begin{align}\label{p-8-l2l2}
\displaystyle \|\mathcal{P}_{8}\|_{L _{T}^{2}L ^{2}}  \leq C \varepsilon \|\varphi _{x}^{I,1}\|_{L _{T}^{2}L ^{\infty}} \|v _{x}^{I,1}\|_{L _{T}^{\infty}L ^{2}} \leq C \varepsilon ,
\end{align}
\begin{align*}
\displaystyle \|\mathcal{P}_{9}\|_{L _{T}^{2}L ^{2}}  &\leq \varepsilon ^{\frac{1}{2}}\| \varphi _{z}^{B,2}\|_{L _{T}^{2}L ^{2}}  \|v _{x}^{I,0}\|_{L _{T}^{\infty}L ^{\infty}}
+\varepsilon \|\varphi _{z}^{B,2}\|_{L _{T}^{2}L _{z}^{\infty}}\|v _{x}^{I,1}\|_{L _{T}^{\infty}L ^{2}}
 \nonumber \\
 &~\displaystyle \quad+\varepsilon  ^{\frac{1}{2}}\left\|\varphi _{z}^{B,2}\right\| _{L _{T}^{2}L _{z}^{\infty}} \left( \|v _{z}^{B,1}\|_{L _{T}^{\infty}L ^{2}} +\|v _{\xi}^{b,1}\|_{L _{T}^{\infty}L ^{2}}  \right) \leq C \varepsilon ^{\frac{3}{4}},
 \end{align*}
 \begin{align*}
 \displaystyle  \|\mathcal{P}_{10}\|_{L _{T}^{2}L ^{2}}& \leq\varepsilon ^{\frac{1}{2}}\| \varphi _{\xi}^{b,2}\|_{L _{T}^{2}L ^{2}}  \|v _{x}^{I,0}\|_{L _{T}^{\infty}L ^{\infty}}
+\varepsilon \|\varphi _{\xi}^{b,2}\|_{L _{T}^{2}L _{\xi}^{\infty}}\|v _{x}^{I,1}\|_{L _{T}^{\infty}L ^{2}}
 \nonumber \\
 &~\displaystyle \quad+\varepsilon  ^{\frac{1}{2}}\|\varphi _{\xi}^{b,2}\| _{L _{T}^{2}L ^{\infty}} \left( \|v _{z}^{B,1}\|_{L _{T}^{\infty}L ^{2}} +\|v _{\xi}^{b,1}\|_{L _{T}^{\infty}L ^{2}}  \right) \leq C \varepsilon ^{\frac{3}{4}},
 \end{align*}
 and
 \begin{align}\label{p-11-l2l2}
 \displaystyle  \|\mathcal{P}_{11}\|_{L _{T}^{2}L ^{2}}& \leq \varepsilon ^{\frac{1}{2}} \left( \|\varphi _{t}^{ B,\varepsilon}\|_{L _{T}^{2}L ^{2}}+\|\varphi _{t}^{ b,\varepsilon}\|_{L _{T}^{2}L ^{2}}+  \varepsilon ^{\frac{1}{2}} \|\varphi _{t}^{B,2}\|_{L _{T}^{2}L ^{2}}+ \varepsilon ^{\frac{1}{2}}\|\varphi _{t} ^{b,2}\|_{L _{T}^{2}L ^{2}}\right)
  \nonumber \\
  & \displaystyle \leq C\varepsilon ^{\frac{3}{4}}+ C \varepsilon ^{\frac{5}{4}- \frac{\alpha}{2}}.
 \end{align}
 %Similarly, we have also
%\begin{align}
%\displaystyle  \|\mathcal{P}_{i}\|_{L _{T}^{2}L ^{2}} \leq C \varepsilon ^{3/4},\ \ \ i =8,9,10,11.
%\end{align}
For the last term $ F ^{\varepsilon} $, by similar arguments as proving \eqref{half-l-infty}, we first deduce from \eqref{b-vfi-ve} and Corollaries \ref{cor-V-B-0-INFTY}--\ref{cor-INF-vb1-vfib2} that
\begin{align}
\displaystyle  \|\partial _{x} b _{\varphi}^{\varepsilon}\|_{L _{T} ^{2}L^{2}} & \leq C \varepsilon ^{\frac{m+1}{2}} \left( \left\|\varepsilon ^{-m/2} \varphi ^{B,\varepsilon}( \varepsilon ^{-1/2} ,t)\right
\|_{L ^{2}((0,T))} +\left\|\varepsilon ^{-m/2} \varphi ^{B,2}( \varepsilon ^{-1/2} ,t)\right
\|_{L ^{2}((0,T))} \right)
 \nonumber \\
 &~\displaystyle \quad +C \varepsilon ^{\frac{m+1}{2}} \left( \left\|\varepsilon ^{-m/2} \varphi ^{b,\varepsilon}( -\varepsilon ^{-1/2} ,t)\right
\|_{L ^{2}((0,T))} +\left\|\varepsilon ^{-m/2} \varphi ^{b,2}(- \varepsilon ^{-1/2} ,t)\right
\|_{L ^{2}((0,T))} \right)
 \nonumber \\
  &~\displaystyle \quad +  \varepsilon\| \varphi ^{B,2}(0,t) \|_{L ^{2}((0,T))}\|[(1-x)		{\mathop{\mathrm{e}}}^{- \frac{x}{\varepsilon ^{\nu}}}]'\|_{L ^{2}} +\varepsilon\| \varphi ^{b,2}(0,t) \|_{L ^{2}((0,T))}\|[x		{\mathop{\mathrm{e}}}^{- \frac{1-x}{\varepsilon ^{\nu}}}]'\|_{L ^{2}}
  \nonumber \\
 & \displaystyle ~  \leq C \varepsilon + C \varepsilon ^{\frac{2- \nu}{2}},\label{b-vfi-x-l2}
\end{align}
\begin{align}\label{b-vfi-L-2-l-infty}
\displaystyle  \|\partial _{x} b _{\varphi}^{\varepsilon}\|_{L _{T} ^{2}L^{\infty}} & \leq C \varepsilon ^{\frac{m+1}{2}} \left( \left\|\varepsilon ^{-m/2} \varphi ^{B,\varepsilon}( \varepsilon ^{-1/2} ,t)\right
\|_{L ^{2}((0,T))} +\left\|\varepsilon ^{-m/2} \varphi ^{B,2}( \varepsilon ^{-1/2} ,t)\right
\|_{L ^{2}((0,T))} \right)
 \nonumber \\
 &~\displaystyle \quad +C \varepsilon ^{\frac{m+1}{2}} \left( \left\|\varepsilon ^{-m/2} \varphi ^{b,\varepsilon}( -\varepsilon ^{-1/2} ,t)\right
\|_{L ^{2}((0,T))} +\left\|\varepsilon ^{-m/2} \varphi ^{b,2}(- \varepsilon ^{-1/2} ,t)\right
\|_{L ^{2}((0,T))} \right)
 \nonumber \\
 &~\displaystyle \quad +  \varepsilon\| \varphi ^{B,2}(0,t) \|_{L ^{2}((0,T))}\|[(1-x)		{\mathop{\mathrm{e}}}^{- \frac{x}{\varepsilon ^{\nu}}}]'\|_{L ^{\infty}} +\varepsilon\| \varphi ^{b,2}(0,t) \|_{L ^{2}((0,T))}\|[x		{\mathop{\mathrm{e}}}^{- \frac{1-x}{\varepsilon ^{\nu}}}]'\|_{L ^{\infty}}
  \nonumber \\
 & \displaystyle ~  \leq C \varepsilon + C \varepsilon ^{1- \nu },
\end{align}
and
\begin{align}\label{pa-t-b-vfi-ve-l2}
\displaystyle \|\partial _{t} b _{\varphi}^{\varepsilon}\|_{L _{T} ^{2}L ^{2}} & \leq C \varepsilon ^{\frac{m+1}{2}} \left( \left\|\varepsilon ^{-m/2} \varphi _{t}^{ B,\varepsilon}( \varepsilon ^{-1/2} ,t)\right
\|_{L ^{2}((0,T))} +\left\|\varepsilon ^{-m/2} \varphi_{t} ^{B,2}( \varepsilon ^{-1/2} ,t)\right
\|_{L ^{2}((0,T))} \right)
 \nonumber \\
 &~\displaystyle \quad +C \varepsilon ^{\frac{m+1}{2}} \left( \left\|\varepsilon ^{-m/2} \varphi _{t}^{ b,\varepsilon}( -\varepsilon ^{-1/2} ,t)\right
\|_{L ^{2}((0,T))} +\left\|\varepsilon ^{-m/2} \varphi _{t}^{b,2}(- \varepsilon ^{-1/2} ,t)\right
\|_{L ^{2}((0,T))} \right)
 \nonumber \\
  &~\displaystyle \quad +  \varepsilon\| \varphi _{t}^{B,2}(0,t) \|_{L ^{2}((0,T))}\|(1-x)		{\mathop{\mathrm{e}}}^{- \frac{x}{\varepsilon ^{\nu}}}\|_{L ^{2}} +\varepsilon\| \varphi _{t}^{b,2}(0,t) \|_{L ^{2}((0,T))}\|x		{\mathop{\mathrm{e}}}^{- \frac{1-x}{\varepsilon ^{\nu}}}\|_{L ^{2}}
  \nonumber \\
  &~\displaystyle\leq C \varepsilon + C \varepsilon ^{\frac{2+\nu- \alpha}{2}},
\end{align}
where $ m \in \mathbb{N} \cap \left\{ m \geq 1+ 3\alpha \right\}$ is fixed, and we have used the fact $ 1< \alpha<1+\nu<5/4 $. Similarly, given any integer $ m \geq 1 $, by \eqref{b-v-ve}, \eqref{Corre-property} and Corollaries \ref{cor-V-B-0-INFTY}--\ref{cor-INF-vb1-vfib2}, we have the following estimate for $ b _{v}^{\varepsilon} $:
  \begin{align}\label{b-v--h1-esti}
  \displaystyle &\displaystyle \|b _{v}^{\varepsilon}\|_{L _{T} ^{\infty}H ^{1}}
 \nonumber \\
 & ~\displaystyle \leq C \varepsilon ^{\frac{m}{2}} \left( \left\|\varepsilon ^{-m/2} v ^{B,\varepsilon}( \varepsilon ^{-1/2} ,t)\right
\|_{L ^{\infty}((0,T))} +\left\|\varepsilon ^{\frac{1-m}{2}}v ^{B,1}( \varepsilon ^{-1/2} ,t)\right
\|_{L ^{\infty}((0,T))}\right.
 \nonumber \\
 &~\displaystyle \quad \left. \quad \quad~\ \ \quad \quad + \left\| \varepsilon ^{-m/2}v ^{b,\varepsilon}(- \varepsilon ^{-1/2},t) \right\|_{L ^{\infty}((0,T))}
+\|\varepsilon ^{\frac{1-m}{2}}v^{b,1}(- \varepsilon ^{-1/2},t)\|_{L ^{\infty}((0,T))}\right)
  \nonumber \\
  &~\displaystyle \quad+ C \left( \|\cala_{1}^{\varepsilon}\|_{L ^{\infty}((0,T))}+\|\cala_{2}^{\varepsilon}\|_{L ^{\infty}((0,T))} \right)
   \nonumber \\
      & ~ \displaystyle \leq C \varepsilon ^{\frac{m}{2}}\left( \|\langle z\rangle ^{m} v ^{B,\varepsilon} \|_{L _{T}^{\infty}L _{z}^{\infty}}+\|\langle z \rangle ^{m-1} v ^{B,1} \|_{L _{T}^{\infty}L _{z}^{\infty}}+\| \langle \xi \rangle ^{m}v ^{b,\varepsilon} \|_{L _{T}^{\infty}L _{\xi}^{\infty}} \right.
    \nonumber \\
   &~\displaystyle \qquad~\qquad \qquad~\qquad\left. +\|\langle \xi\rangle ^{m-1}v ^{b,1} \| _{L _{T}^{\infty}L _{\xi}^{\infty}} \right) +C \varepsilon ^{\frac{\alpha}{2}}\leq C \varepsilon ^{\frac{m}{2}}+C \varepsilon ^{\frac{\alpha}{2}} \leq C  \varepsilon +C \varepsilon ^{\frac{\alpha}{2}}.
  \end{align}
 Therefore we get by \eqref{b-vfi-x-l2}--\eqref{b-v--h1-esti} and Corollaries \ref{cor-V-B-0-INFTY}--\ref{cor-INF-vb1-vfib2} that
  \begin{align}\label{F-es-t-1}
\displaystyle  \|F ^{\varepsilon}\|_{L _{T}^{2}L ^{2}}
  & \leq C \|\partial _{x}b _{v}^{\varepsilon}\|_{L ^{\infty}((0,T))}\left( 1+\left\|\varphi ^{I,0}\right\|_{ L _{T}^{2}H ^{1}}+\varepsilon ^{\frac{1}{2}}\|\varphi ^{I,1}\|_{L _{T}^{2}H ^{1}}+\varepsilon ^{\frac{1}{4}}\|\varphi ^{B,\varepsilon}\|_{L _{T}^{2}H _{z}^{1}}\right.
 \nonumber \\
 & \displaystyle \left. \qquad \qquad \qquad \quad~\ \ \quad+ \varepsilon ^{\frac{1}{4}}\|\varphi ^{b,\varepsilon}\| _{L _{T}^{2}H _{\xi}^{1}}+\varepsilon ^{\frac{3}{4}}\|\varphi ^{B,2}\|_{L _{T}^{2}H _{z}^{1}}+\varepsilon ^{\frac{3}{4}}\|\varphi ^{b,2}\|_{L _{T}^{2}H _{\xi}^{1}}\right)
  \nonumber \\
  &\displaystyle \quad +\|\partial _{x}b _{\varphi}^{\varepsilon}\|_{L _{T}^{2}L ^{\infty}}\left( \|v ^{I,0}\|_{L _{T}^{\infty}H ^{1}}+\varepsilon ^{\frac{1}{4}}\|v ^{B,1}\|_{L _{T}^{\infty}H _{z}^{1}}+ \varepsilon ^{\frac{1}{4}}\|v ^{b,1}\| _{L _{T}^{2}H _{\xi}^{1}} \right.
   \nonumber \\
   &\displaystyle \quad \qquad \qquad\ \ \  \qquad\left. +\varepsilon ^{- \frac{1}{4}} \|v _{z}^{ B,\varepsilon}\|_{L _{T}^{\infty}L _{z}^{2}}+ \varepsilon ^{- \frac{1}{4}}\|v _{\xi}^{ b,\varepsilon}\|_{L _{T}^{\infty}L _{\xi}^{2}}\right) + \|\partial _{x} b _{\varphi}^{\varepsilon}\|_{L _{T}^{2}L ^{2}}\|\partial _{x} b _{v}^{\varepsilon}\| _{L ^{\infty}((0,T))}
   \nonumber \\
   &\displaystyle \quad+ \|\partial _{t}b _{\varphi}^{\varepsilon}\|_{L _{T}^{2}L^{2}}     + \varepsilon (\|\varphi^{B,2}(0,t)\|_{L ^{2}((0,T))}+\|\varphi^{b,2}(0,t)\|_{L ^{2}((0,T))}) \|[ (1-x)		{\mathop{\mathrm{e}}}^{- \frac{x}{\varepsilon ^{\nu}}}  ]''\|_{L ^{2}}
    \nonumber \\
    &\displaystyle \leq C \varepsilon ^{\frac{\alpha}{2}}+C \varepsilon ^{\frac{3}{4}}+C  \varepsilon ^{\frac{3}{4}-\nu} +C \varepsilon + C \varepsilon ^{\frac{2+\nu- \alpha}{2}} +C \varepsilon ^{\frac{2-3 \nu}{2}}.
\end{align}
Summing up, we now have for $ 0< \varepsilon<1 $ that
\begin{gather}\label{R-1-first-esti}
\displaystyle \|\mathcal{R}_{1,1}^{\varepsilon}\|_{L _{T}^{2}L ^{2}} \leq \sum _{i=1}^{11}\|\mathcal{P}_{i}\|_{L _{T}^{2}L ^{2}}+\|F ^{\varepsilon}\|_{L _{T}^{2}L ^{2}} \leq C\varepsilon ^{\iota _{0}},
\end{gather}
where $ \iota _{0} $ is as in \eqref{iota0-defi}. That is, $ \iota _{0}= \min \left\{ \frac{3}{4}-\nu, \frac{\alpha}{2}, \frac{2+\nu - \alpha}{2}, \frac{2-3 \nu}{2},\frac{5}{4}- \frac{\alpha}{2}  \right\} $. For $ \| \mathcal{R}_{1,2}^{\varepsilon}\|_{L _{T}^{2}L ^{2}}  $, we get from \eqref{L-INFT-V-B-0}, \eqref{L-INFT-Vfi-B-0} and \eqref{L-INFT-V-veb-0}--\eqref{vfi-I-1-esti-first-con-lem} that
\begin{align}\label{R-1-2-L2l2}
 \displaystyle  \| \mathcal{R}_{1,2}^{\varepsilon}\|_{L _{T}^{2}L ^{2}} & \leq C \left( \int _{0}^{T}\int _{\mathcal{I}} \vert v _{z}^{ B,\varepsilon}\vert ^{2}\left\vert \int _{0}^{x}\varphi _{yy}^{I,1}\mathrm{d}y \right\vert ^{2}  \mathrm{d}x \mathrm{d}t\right)^{\frac{1}{2}}+\left( \int _{0}^{T}\int _{\mathcal{I}} \vert v _{\xi}^{ b,\varepsilon}\vert ^{2}\left\vert \int _{x}^{1}\varphi _{yy}^{I,1}\mathrm{d}y \right\vert ^{2}  \mathrm{d}x \mathrm{d}t\right)^{\frac{1}{2}}
  \nonumber \\
  &\displaystyle \quad+C \varepsilon ^{\frac{1}{2}}\| v _{x}^{I,1} \|_{L _{T}^{2}L ^{2}}\left( \|\varphi _{z}^{ B,\varepsilon}\|_{L _{T}^{\infty}L _{z}^{\infty}}+\|\varphi _{\xi}^{ b,\varepsilon}\|_{L _{T}^{\infty}L _{\xi}^{\infty}} \right)
   \nonumber \\
   & \displaystyle \leq C \varepsilon ^{\frac{1}{2}}\|\varphi _{xx}^{I,1}\|_{L _{T}^{2}L ^{2}}\left( \|\langle z \rangle v ^{B,\varepsilon}\|_{L _{T}^{\infty}L _{z}^{2}}+\|\langle \xi \rangle v ^{b,\varepsilon}\|_{L _{T}^{\infty}L _{\xi}^{2}} \right) +C \varepsilon ^{\frac{1}{2}}
   \leq C \varepsilon ^{\frac{1}{2}}.
 \end{align}
Next we shall estimate $ \|t \mathcal{R}_{1,1}^{\varepsilon}\|_{L _{T}^{\infty}L _{z}^{2}} $. Recalling the estimates on $ \|\mathcal{P} _{i}\|_{L _{T}^{2}L ^{2}}\,(1 \leq i \leq 6) $ in \eqref{P-1-esti}--\eqref{p-5-l2} and Corollaries \ref{cor-V-B-0-INFTY}--\ref{cor-INF-vb1-vfib2}, we have
\begin{align}\label{P1-6-infty-l2}
\displaystyle  \|\mathcal{P} _{i}\|_{L _{T}^{\infty}L ^{2}} \leq C \varepsilon ^{\frac{3}{4}},\ \ 1 \leq i \leq 6.
\end{align}
For $ \mathcal{P}_{7} $, we get
\begin{align}\label{cal-P-7-inf-l2}
\displaystyle  \|t ^{\frac{1}{2}}\mathcal{P}_{7}\|_{L _{T}^{\infty}L ^{2}}& \leq C \varepsilon ^{\frac{1}{2}} \|t ^{\frac{1}{2}}\varphi _{x}^{I,1}\|_{L _{T}^{\infty}L ^{\infty}} \left( \|v _{z}^{B,1}\|_{L _{T}^{\infty} L ^{2}}+\|v _{\xi}^{b,1}\|_{L _{T}^{\infty}L ^{2}} \right)
 \nonumber \\
   &\displaystyle \leq C \varepsilon ^{\frac{3}{4}} \left( \|\varphi _{x}^{I,1}\|_{L _{T}^{\infty}L ^{2}}^{1/2}\|t \varphi _{xx}^{I,1}\|_{L _{T}^{\infty}L ^{2}}^{1/2} + \|\varphi _{x}^{I,1}\|_{L _{T}^{\infty}L ^{2}}\right) \left( \|v _{z}^{B,1}\|_{L _{T}^{\infty} L _{z}^{2}}+\|v _{\xi}^{b,1}\|_{L _{T}^{\infty}L _{\xi}^{2}} \right)
   \nonumber \\
   &\displaystyle \leq C \varepsilon ^{\frac{3}{4}},
\end{align}
where we have used \eqref{L-INFT-Vfi-B-0}, \eqref{vfi-I-1-infty}, \eqref{v-I-1-infty}, \eqref{tild-v-B-1-lem}, \eqref{INF-vb1-vfib2} and \eqref{Sobolev-infty}. By analogous arguments as in deriving \eqref{cal-P-7-inf-l2} and Corollaries \ref{cor-V-B-0-INFTY}--\ref{cor-INF-vb1-vfib2}, one can modify the estimates \eqref{p-8-l2l2}--\eqref{p-11-l2l2} to get
\begin{align}\label{p7-11-infty-l2}
\displaystyle \|t ^{\frac{1}{2}}\mathcal{P}_{i}\|_{L _{T}^{\infty} L ^{2}} \leq C \varepsilon ^{\frac{3}{4}} \ \ \mbox{for } i=8,9,10,\ \ \ \|t \mathcal{P}_{11}\|_{L _{T}^{\infty}L ^{2}} \leq C \varepsilon ^{\frac{3}{4}}+C \varepsilon ^{\frac{5}{4}-\frac{\alpha}{2}},
\end{align}
where we have used the condition $ 0<\varepsilon < 1 $. Now let us turn to estimate for $ \|t ^{\frac{9}{8}} F\|_{L _{T}^{\infty}L ^{2}} $. Similar to the derivation of \eqref{b-vfi-x-l2}--\eqref{pa-t-b-vfi-ve-l2}, we get
\begin{subequations}\label{inty-vfi-t-vfix-l2}
\begin{align}\label{t-pa-x-b-vfi-infty}
\displaystyle  \|t ^{\frac{1}{2}}\partial _{x} b _{\varphi}^{\varepsilon}\|_{L _{T} ^{\infty}L^{\infty}} & \leq C \varepsilon +  C \varepsilon\|t ^{\frac{1}{2}} \varphi ^{B,2}(0,t) \|_{L ^{\infty}((0,T))}\|[(1-x)		{\mathop{\mathrm{e}}}^{- \frac{x}{\varepsilon ^{\nu}}}]'\|_{L ^{\infty}}
 \nonumber \\
 & \displaystyle \quad +\varepsilon\| t ^{\frac{1}{2}}\varphi ^{b,2}(0,t) \|_{L ^{\infty}((0,T))}\|[x		{\mathop{\mathrm{e}}}^{- \frac{1-x}{\varepsilon ^{\nu}}}]'\|_{L ^{\infty}}
 \leq C\varepsilon + C \varepsilon ^{1- \nu },
\end{align}
\begin{align}
&\displaystyle \|t\partial _{t} b _{\varphi}^{\varepsilon}\|_{L _{T} ^{\infty}L ^{2}}
 \nonumber \\
  &~\displaystyle \leq C \varepsilon  +  C\varepsilon\|t  \varphi _{t}^{B,2}(0,t) \|_{L ^{\infty}((0,T))}\|(1-x)		{\mathop{\mathrm{e}}}^{- \frac{x}{\varepsilon ^{\nu}}}\|_{L ^{2}} + C\varepsilon\| t \varphi _{t}^{b,2}(0,t) \|_{L ^{\infty}((0,T))}\|x		{\mathop{\mathrm{e}}}^{- \frac{1-x}{\varepsilon ^{\nu}}}\|_{L ^{2}}
  \nonumber \\
  &~ \displaystyle \leq C \varepsilon +C \varepsilon ^{1+\nu} \left[ \|t ^{2}\varphi _{t}^{B,2}\|_{L _{T}^{\infty}L _{z}^{\infty  }}^{\frac{1}{2}}\|\varphi _{t}^{B,2}\|_{L _{T}^{\infty}L _{z}^{\infty  }}^{\frac{1}{2}}+\|t ^{2}\varphi _{t}^{b,2}\|_{L _{T}^{\infty}L _{\xi}^{\infty  }}^{\frac{1}{2}}\|\varphi _{t}^{b,2}\|_{L _{T}^{\infty}L _{\xi}^{\infty  }}^{\frac{1}{2}} \right]
   \nonumber \\
   & ~\displaystyle
 \leq C \varepsilon + C \varepsilon ^{\frac{2+\nu-\alpha}{2}},
\end{align}
\end{subequations}
where we have used \eqref{INF-VFI-B-2} and \eqref{INF-vfi-b2}. In virtue of \eqref{INF-VFI-B-2}, \eqref{INF-vfi-b2}, \eqref{b-v--h1-esti}, \eqref{inty-vfi-t-vfix-l2} and Corollaries \ref{cor-V-B-0-INFTY}--\ref{cor-INF-vb1-vfib2}, we proceed to modify the estimate \eqref{F-es-t-1} as
\begin{align}\label{F-es-t-1-weighted}
&\displaystyle  \|t  F ^{\varepsilon}\|_{L _{T}^{\infty}L ^{2}}
 \nonumber \\
   &~ \leq C \|t ^{\frac{1}{2}}\partial _{x}b _{v}^{\varepsilon}\|_{L ^{\infty}((0,T))}\left( 1+\left\|\varphi ^{I,0}\right\|_{ L _{T}^{\infty}H ^{1}}+\varepsilon ^{\frac{1}{2}}\|\varphi ^{I,1}\|_{L _{T}^{\infty}H ^{1}}+\varepsilon ^{\frac{1}{4}}\|\varphi ^{B,\varepsilon}\|_{L _{T}^{\infty}H _{z}^{1}}\right.
 \nonumber \\
 &~ \displaystyle \left. \qquad \qquad \qquad \quad~\ \ \quad+ \varepsilon ^{\frac{1}{4}}\|\varphi ^{b,\varepsilon}\| _{L _{T}^{\infty}H _{\xi}^{1}}+\varepsilon ^{\frac{3}{4}}\|\varphi ^{B,2}\|_{L _{T}^{2}H _{z}^{1}}+\varepsilon ^{\frac{3}{4}}\|\varphi ^{b,2}\|_{L _{T}^{\infty}H _{\xi}^{1}}\right)
  \nonumber \\
  &~\displaystyle \quad +\|t ^{\frac{1}{2}}\partial _{x}b _{\varphi}^{\varepsilon}\|_{L _{T}^{\infty}L ^{\infty}}\left( \|v ^{I,0}\|_{L _{T}^{\infty}H ^{1}}+\varepsilon ^{\frac{1}{4}}\|v ^{B,1}\|_{L _{T}^{\infty}H _{z}^{1}}+ \varepsilon ^{\frac{1}{4}}\|v ^{b,1}\| _{L _{T}^{2}H _{\xi}^{1}} \right.
   \nonumber \\
   &~\displaystyle \quad \qquad \qquad\ \ \ \quad \qquad  \qquad\left. +\varepsilon ^{- \frac{1}{4}} \|v _{z}^{ B,\varepsilon}\|_{L _{T}^{\infty}L _{z}^{2}}+ \varepsilon ^{- \frac{1}{4}}\|v _{\xi}^{ b,\varepsilon}\|_{L _{T}^{\infty}L _{\xi}^{2}}\right)
   \nonumber \\
      &~\displaystyle \quad
    + \varepsilon \|t ^{\frac{1}{2}}\varphi^{B,2}(0,t)\|_{L ^{\infty}((0,T))} \|[ (1-x)		{\mathop{\mathrm{e}}}^{- \frac{x}{\varepsilon ^{\nu}}}  ]''\|_{L ^{2}}
    +C \|t ^{\frac{1}{2}} \varphi^{b,2}(0,t)\|_{L ^{\infty}((0,T))}\|[ x		{\mathop{\mathrm{e}}}^{- \frac{1-x}{\varepsilon ^{\nu}}}  ]''\|_{L ^{2}}
     \nonumber \\
      &~\displaystyle \quad+ \|t ^{\frac{1}{2}}\partial _{x} b _{\varphi}^{\varepsilon}\|_{L _{T}^{\infty}L ^{\infty}}\|\partial _{x} b _{v}^{\varepsilon}\| _{L ^{\infty}((0,T))}+ \|t\partial _{t}b _{\varphi}^{\varepsilon}\|_{L _{T}^{\infty}L^{2}}  \nonumber \\
     &~\displaystyle
    \leq C \varepsilon ^{\frac{\alpha}{2}}+C \varepsilon ^{\frac{3}{4}}+C  \varepsilon ^{\frac{3}{4}-\nu}  +C \varepsilon ^{\frac{2-3 \nu}{2}}+ C\varepsilon ^{\frac{2+\nu- \alpha}{2}}.
\end{align}
Gathering \eqref{P1-6-infty-l2}, \eqref{cal-P-7-inf-l2}, \eqref{p7-11-infty-l2} and \eqref{F-es-t-1-weighted}, we then arrive at
\begin{align*}%\label{R-11-t-l2}
\displaystyle \|t \mathcal{R}_{1,1}^{\varepsilon}\|_{L _{T}^{\infty}L^{2}} \leq C \varepsilon ^{\iota _{0}},
\end{align*}
where $ \iota _{0} $ is as in \eqref{iota0-defi}. For $ \mathcal{R}_{1,2}^{\varepsilon} $, by \eqref{Sobolev-infty}, Corollaries \ref{cor-V-B-0-INFTY}--\ref{cor-INF-vb1-vfib2} and analogous arguments as in \eqref{R-1-2-L2l2}, we get
\begin{align*}%\label{R-1-2-t-l2}
\displaystyle \|t ^{\frac{1}{2}}\mathcal{R}_{1,2}^{\varepsilon}\|_{L _{T}^{\infty}L _{z}^{2}}& \leq C \varepsilon ^{\frac{1}{2}} \|t ^{\frac{1}{2}}\varphi _{xx}^{I,1}\|_{L _{T}^{\infty}L ^{2}} \left( \|\langle z \rangle v _{z}^{ B,\varepsilon}\|_{L _{T}^{\infty}L _{z}^{2}} +\|\langle \xi \rangle v _{\xi}^{ b,\varepsilon}\|_{L _{T}^{\infty}L _{\xi}^{2}} \right)
 \nonumber \\
 &\displaystyle \quad+C \varepsilon ^{\frac{1}{2}} \|t ^{\frac{1}{2}}v _{x}^{I,1}\|_{L _{T}^{\infty}L ^{2}}\left( \|\varphi _{z}^{ B,\varepsilon}\|_{L _{T}^{\infty}L _{z}^{\infty}}+\|\varphi _{\xi}^{ b,\varepsilon}\|_{L _{T}^{\infty}L _{\xi}^{\infty}} \right)
    \leq C \varepsilon ^{\frac{1}{2}}.
\end{align*}
Now we proceed to estimate $ \|\partial _{t} \mathcal{R}_{1,1}^{\varepsilon}\|_{L _{T}^{2}L ^{2} } $ and $ \|\partial _{t} \mathcal{R}_{1,2}^{\varepsilon}\|_{L _{T}^{2}L ^{2} } $. Notice that if $ \|h \chi\|_{Z} \leq C \|h\|_{X}\| \chi\|_{Y} $ for $ h \in X $ and $ \chi \in Y $ with $ X $, $ Y $ and $ Z $ being Banach spaces, then
\begin{gather}\label{product-tim}
\displaystyle \|\partial _{t}(h \chi)\|_{Z} \leq C\|\partial _{t}h\|_{X}\|\chi\|_{Y}+C \|h\|_{X}\| \partial _{t}\chi\|_{Y},
\end{gather}
provided $ \partial _{t}f \in X $ and $ \partial _{t}\chi \in Y $, and that the time-weighted estimates (e.g., the weight function $ t ^{\beta} $ with $ \beta \leq 3 $) of the profiles are uniform with respect to $ \varepsilon $. Therefore, by \eqref{con-vfi-v-I-0-regula}, \eqref{z-transfer} and \eqref{L-INFT-V-B-0} and similar arguments as proving \eqref{P-1-esti}, we have
\begin{align}\label{P-1-t-esti}
&\displaystyle \|\partial _{t}\mathcal{P}_{1}\|_{L _{T}^{2}L ^{2}}
 \nonumber \\
 &~ \displaystyle \leq  \left\|\frac{\partial _{t}\partial _{x}\varphi ^{I,0}(x,t)-\partial _{t}\partial _{x}\varphi ^{I,0}(0,t)-x\partial _{t}\partial _{x}^{2}\varphi ^{I,0}(0,t)}{x ^{2}}x ^{2}v _{x}^{ B,\varepsilon}\right\|_{L _{T}^{2}L ^{2}}
 \nonumber \\
 & ~\displaystyle \quad+ \left\|\frac{\partial _{t}\partial _{x}\varphi ^{I,0}(x,t)-\partial _{t}\partial _{x}\varphi ^{I,0}(0,t)-x\partial _{t}\partial _{x}^{2}\varphi ^{I,0}(0,t)}{x ^{2}} x ^{2}\partial _{t} v _{x}^{ B,\varepsilon}\right\|_{L _{T}^{2}L ^{2}}
  \nonumber \\
  & ~\displaystyle \leq  \|\partial _{t}\partial _{x}^{3}\varphi ^{I,0}\|_{L _{T}^{\infty} L ^{\infty}} \|x ^{2}v _{x}^{ B,\varepsilon}\|_{L _{T}^{2}L  ^{2}}+ \|\partial _{x}^{3}\varphi ^{I,0}\|_{L _{T} ^{\infty}L ^{\infty}} \|x ^{2}\partial _{t}v _{x}^{ B,\varepsilon}\|_{L _{T}^{2}L  ^{2}}
   \nonumber \\
   & ~\displaystyle \leq C \varepsilon ^{\frac{3}{4}}\|\partial _{t}\varphi ^{I,0}\|_{L _{T}^{\infty}H ^{4}}\|z ^{2}v _{z}^{ B,\varepsilon}\|_{L _{T}^{2}L  _{z}^{2}}+C \varepsilon ^{\frac{3}{4}}\|\varphi ^{I,0}\|_{L _{T}^{\infty}H ^{4}}\|z ^{2}\partial _{t} v _{z}^{ B,\varepsilon}\|_{L _{T} ^{2}L  _{z}^{2}}
    \nonumber \\
    &~\displaystyle \leq C \varepsilon ^{\frac{3}{4}}.
\end{align}
Similar arguments further entails that
\begin{align}\label{P-i-t-esti}
\displaystyle  \|t ^{\frac{5}{2}}\partial _{t}\mathcal{P}_{i}\|_{L _{T}^{2}L ^{2}} \leq C \varepsilon ^{\iota _{0}} \ \ \mbox{for }i=2,3,\cdots,11,\ \ \ \|t \partial _{t}\mathcal{R}_{1,2}^{\varepsilon}\|_{L _{T}^{2}L ^{2}}\leq C \varepsilon ^{\frac{1}{2}},
\end{align}
where $ \iota _{0} $ is as in \eqref{R-1-first-esti}. Now it remains to prove $ \|t ^{\frac{5}{2}}\partial _{t}F\|_{L _{T}^{2}L ^{2}}\leq C \varepsilon ^{3/4- \nu} $. For this, we first derive some more estimates on $ b _{\varphi}^{\varepsilon} $ and $ b _{v}^{\varepsilon} $. Indeed, by analogous arguments as in \eqref{b-vfi-x-l2}--\eqref{pa-t-b-vfi-ve-l2}, we have
\begin{subequations}\label{some-more-b-vfi-b-v}
\begin{align}
\displaystyle \|t \partial _{t}\partial _{x}b _{\varphi}^{\varepsilon}\|_{L _{T}^{2}L ^{\infty}} & \leq C \varepsilon +  C\varepsilon\| \varphi _{t}^{B,2}(0,t) \|_{L ^{2}((0,T))}\|[(1-x)		{\mathop{\mathrm{e}}}^{- \frac{x}{\varepsilon ^{\nu}}}]'\|_{L ^{\infty}}
 \nonumber
 \\
 &~\displaystyle \quad +C\varepsilon\| t \varphi _{t}^{b,2}(0,t) \|_{L ^{2}((0,T))}\|[x		{\mathop{\mathrm{e}}}^{- \frac{1-x}{\varepsilon ^{\nu}}}]'\|_{L ^{\infty}}
 \leq C \varepsilon + C \varepsilon ^{1- \nu},\label{pa-t-x-bvfi-ve}
 \\
 \displaystyle  \|t ^{\frac{5}{2}}\partial _{t}^{2} b _{\varphi}^{\varepsilon}\|_{L _{T} ^{2}L ^{2}} & \leq C \varepsilon +  \varepsilon\| t ^{\frac{5}{2}}\partial _{t}^{2}\varphi ^{B,2}(0,t) \|_{L ^{2}((0,T))}\|(1-x)		{\mathop{\mathrm{e}}}^{- \frac{x}{\varepsilon ^{\nu}}}\|_{L ^{2}}
  \nonumber \\
  &~\displaystyle \quad+\varepsilon\| t ^{\frac{5}{2}}\partial _{t}^{2}\varphi ^{b,2}(0,t) \|_{L ^{2}((0,T))}\|x		{\mathop{\mathrm{e}}}^{- \frac{1-x}{\varepsilon ^{\nu}}}\|_{L ^{2}} \leq C \varepsilon + C \varepsilon ^{1+ \frac{\nu}{2}},\label{pa-t-2-b-vfi-ve-esti}\\
  \displaystyle \|\partial _{t}b _{v}^{\varepsilon}\|_{L _{T} ^{2}H ^{1}}
  & \displaystyle \leq C \varepsilon+ C \left( \|\partial _{t}\cala_{1}^{\varepsilon}\|_{L ^{2}((0,T))}+\|\partial _{t}\cala_{2}^{\varepsilon}\|_{L ^{2}((0,T))} \right)
   \leq C \varepsilon +C \varepsilon ^{\frac{\alpha}{2}}, \label{pa-t-b-v-l2-no-weig}
   \\
  \displaystyle \|t ^{\frac{3}{2}}\partial _{t}b _{v}^{\varepsilon}\|_{L _{T} ^{2}H ^{1}}
  & \displaystyle \leq C \varepsilon+ C \left( \|t ^{\frac{3}{2}}\partial _{t}\cala_{1}^{\varepsilon}\|_{L ^{2}((0,T))}+\|t ^{\frac{3}{2}}\partial _{t}\cala_{2}^{\varepsilon}\|_{L ^{2}((0,T))} \right)
   \leq C \varepsilon +C \varepsilon ^{\alpha}, \label{pa-t-bv-time-weight}
   \\
  \displaystyle \| t  \partial _{t}^{2} b _{v}^{\varepsilon}\|_{L _{T} ^{2}H ^{1}}
  & \displaystyle \leq C \varepsilon+ C \left( \|t \partial _{t}^{2}\cala_{1}^{\varepsilon}\|_{L ^{2}((0,T))}+\|t \partial _{t}^{2}\cala_{2}^{\varepsilon}\|_{L ^{2}((0,T))} \right) \leq C \varepsilon +C \varepsilon ^{\frac{\alpha}{2}},\label{pa-t-2-b-v-ve-wei}
  \end{align}
\end{subequations}
where we have used \eqref{tild-v-B-1-lem}, \eqref{Corre-property} and \eqref{cal-A-INFTY-T-SMALL}. Notice also that $ \partial _{x}b _{v}^{\varepsilon} $ is independent of $ x $. Thus it holds that
\begin{align}\label{b-time-esti}
\displaystyle  \|t\partial _{t}^{2}\partial _{x} b _{v}^{\varepsilon} \|_{L ^{2}((0,T))}+\|\partial _{t}\partial _{x}b _{v}^{\varepsilon}\|_{L ^{2}((0,T))} \leq C \varepsilon +C \varepsilon ^{\frac{\alpha}{2}}.
\end{align}
With \eqref{con-vfi-v-I-0-regula}, \eqref{inty-vfi-t-vfix-l2}, \eqref{some-more-b-vfi-b-v}, \eqref{b-time-esti} and Corollaries \ref{cor-V-B-0-INFTY}--\ref{cor-INF-vb1-vfib2}, recalling the definition of $ F ^{\varepsilon} $ in \eqref{F-defi}, we have
\begin{align*}
&\displaystyle \|t ^{\frac{5}{2}}\partial _{t}F ^{\varepsilon}\|_{L _{T}^{2}L  ^{2}}
 \nonumber \\
 &~\displaystyle \leq  C \|t  ^{\frac{5}{2}}\partial _{t}\partial _{x}b _{v}^{\varepsilon}\|_{L ^{2}((0,T))}\left( 1+\|\varphi ^{I,0}\|_{ L _{T}^{\infty}H ^{1}}+\varepsilon ^{\frac{1}{2}}\|\varphi ^{I,1}\|_{L _{T}^{2}H ^{1}}+\varepsilon ^{\frac{1}{4}}\|\varphi ^{B,\varepsilon}\|_{L _{T}^{\infty}H _{z}^{1}}\right.
 \nonumber \\
 & \displaystyle \qquad\ \left. \qquad \qquad \qquad \quad~\ \ \quad+ \varepsilon ^{\frac{1}{4}}\|\varphi ^{b,\varepsilon}\| _{L _{T}^{\infty}H _{\xi}^{1}}+\varepsilon ^{\frac{3}{4}}\|\varphi ^{B,2}\|_{L _{T}^{\infty}H _{z}^{1}}+\varepsilon ^{\frac{3}{4}}\|\varphi ^{b,2}\|_{L _{T}^{2}H _{\xi}^{1}}\right)
  \nonumber \\
  &~\displaystyle \quad +C \|\partial _{x}b _{v}^{\varepsilon}\|_{L ^{\infty}((0,T))}\left( 1+\|t  ^{\frac{5}{2}}\varphi _{t}^{I,0}\|_{ L _{T}^{2}H ^{1}}+\varepsilon ^{\frac{1}{2}}\|t ^{\frac{5}{2}}\varphi _{t}^{I,1}\|_{L _{T}^{2}H ^{1}}+\varepsilon ^{\frac{1}{4}}\|t  ^{\frac{5}{2}}\varphi _{t}^{ B,\varepsilon}\|_{L _{T}^{2}H _{z}^{1}}\right.
 \nonumber \\
 & \displaystyle \qquad \left. \qquad \qquad \qquad \quad~\ \ \quad+ \varepsilon ^{\frac{1}{4}}\|t  ^{\frac{5}{2}}\varphi _{t}^{ b,\varepsilon}\| _{L _{T}^{2}H _{\xi}^{1}}+\varepsilon ^{\frac{3}{4}}\|t  ^{\frac{5}{2}}\varphi _{t}^{B,2}\|_{L _{T}^{2}H _{z}^{1}}+\varepsilon ^{\frac{3}{4}}\|t  ^{\frac{5}{2}}\varphi _{t}^{b,2}\|_{L _{T}^{2}H _{\xi}^{1}}\right)
  \nonumber \\
  &~\displaystyle \quad + \|t  ^{\frac{5}{2}}\partial _{x}b _{\varphi}^{\varepsilon}\|_{L _{T}^{\infty}L ^{\infty}}\left( \|v _{t} ^{I,0}\|_{L _{T}^{2}H ^{1}}+\varepsilon ^{\frac{1}{4}}\|v _{t} ^{B,1}\|_{L _{T}^{2}H _{z}^{1}}+ \varepsilon ^{\frac{1}{4}}\|v _{t} ^{b,1}\| _{L _{T}^{2}H _{\xi}^{1}} \right.
   \nonumber \\
   &~\displaystyle \quad \qquad \qquad\ \ \ \quad \qquad  \qquad\left. +\varepsilon ^{- \frac{1}{4}} \|v _{z t}^{ B,\varepsilon}\|_{L _{T}^{2}L _{z}^{2}}+ \varepsilon ^{- \frac{1}{4}}\|v _{\xi t}^{ b,\varepsilon}\|_{L _{T}^{2}L _{\xi}^{2}}\right)
    \nonumber \\
    &~\displaystyle \quad+\|t ^{\frac{5}{2}}\partial _{t}\partial _{x}b _{\varphi}^{\varepsilon}\|_{L _{T}^{2}L ^{\infty}}\left( \|v ^{I,0}\|_{L _{T}^{\infty}H ^{1}}+\varepsilon ^{\frac{1}{4}}\|v ^{B,1}\|_{L _{T}^{\infty}H _{z}^{1}}+ \varepsilon ^{\frac{1}{4}}\|v ^{b,1}\| _{L _{T}^{2}H _{\xi}^{1}} \right.
   \nonumber \\
   &~\displaystyle \quad \qquad \qquad\ \ \ \quad \qquad  \qquad\left. +\varepsilon ^{- \frac{1}{4}} \|v _{z}^{ B,\varepsilon}\|_{L _{T}^{\infty}L _{z}^{2}}+ \varepsilon ^{- \frac{1}{4}}\|v _{\xi}^{ b,\varepsilon}\|_{L _{T}^{\infty}L _{\xi}^{2}}\right)
    \nonumber \\
    &~\displaystyle \quad+ \|t  ^{\frac{5}{2}}\partial _{t}\partial _{x} b _{\varphi}^{\varepsilon}\|_{L _{T}^{2}L ^{2}}\|\partial _{x} b _{v}^{\varepsilon}\| _{L ^{\infty}((0,T))}+\|\partial _{x} b _{\varphi}^{\varepsilon}\|_{L _{T}^{2}L ^{2}}\|t  ^{\frac{5}{2}}\partial _{t}\partial _{x} b _{v}^{\varepsilon}\| _{L ^{\infty}((0,T))}+ \|t  ^{\frac{5}{2}}\partial _{t}^{2} b _{\varphi}^{\varepsilon}\|_{L _{T}^{2}L^{2}}  \nonumber \\
    &~\displaystyle \leq  C \varepsilon ^{\frac{3}{4}- \nu},
\end{align*}
where we have used $ 0< \varepsilon<1 $ and $ \alpha>1 $. This along with \eqref{R-1-1-defi}, \eqref{P-1-t-esti} and \eqref{P-i-t-esti} gives  $ \|t ^{\frac{5}{2}}\partial _{t}\mathcal{R}_{1,1}^{\varepsilon}\|_{L _{T}^{2}L ^{2}} \leq C \varepsilon ^{\iota _{0}} $, and thus the proof is complete.
\end{proof}

\begin{lemma}\label{lem-R-2-ESTI}
Let $ 0<T<\infty $, $ 0< \varepsilon<1 $ and $ 1<\alpha<1+\nu<5/4 $. Then there exists a constant $ C >0$ independent of $ \varepsilon $ such that
\begin{align}\label{con-R-2-ESTI-in-lem}
\displaystyle   \|\mathcal{R}_{2}^{\varepsilon}\|_{L _{T}^{2}L ^{2}}+\|t\partial _{t}\mathcal{R}_{2}^{\varepsilon}\|_{L ^{2}L ^{2}} \leq C \varepsilon ^{\iota _{1}},\ \
 \|t ^{\frac{1}{2}}\mathcal{R}_{2}^{\varepsilon}\|_{L _{T}^{\infty}L ^{2}} \leq C \varepsilon ^{\iota _{2}},
\end{align}
where $ \iota _{1}=\min \left\{ 1, \frac{\alpha}{2}, \frac{3- \alpha}{2}, 1- \nu \right\} \geq \iota _{0} $, $ \iota _{3}=\min \left\{ 1, \frac{5}{4}- \frac{\alpha}{2}, \frac{3- \alpha}{2}, \frac{\alpha}{2}\right\} \geq \iota _{0} $, the constant $ C>0 $ is independent of $ \varepsilon $.
\end{lemma}
\begin{proof}
From \eqref{first-bd-layer-pro-appro}--\eqref{first-bd-pro-rt-approxi}, we know that
\begin{gather}\label{v-ZZ}
\displaystyle \begin{cases}
 \displaystyle v _{zz}^{ B,\varepsilon}=v _{t}^{ B,\varepsilon} +(\varphi _{x} ^{I,0}(0,t)+M)v ^{B,\varepsilon}+\varphi _{z}^{ B,\varepsilon}(v ^{B,\varepsilon}+v ^{I,0}(0,t)),\\[2mm]
 \displaystyle v _{\xi \xi}^{ b,\varepsilon} =v _{t}^{ b,\varepsilon} +(\varphi  _{x}^{I,0}(1,t)+M)v ^{b,\varepsilon}+\varphi _{z}^{ b,\varepsilon}(v ^{b,\varepsilon}+v ^{I,0}(1,t)).
\end{cases}
\end{gather}
Plugging \eqref{v-ZZ} into $ \mathcal{R}_{2}^{\varepsilon} $ in \eqref{E-rror-fomula}, recalling the definition of $ \Phi^{a} $ and $ V^{a} $, we have
\begin{align}\label{R-2-PRECISE}
\displaystyle \mathcal{R}_{2}^{\varepsilon}&=   -\left[ v ^{B,\varepsilon}( \varphi _{x}^{I,0}(x,t)-\varphi _{x} ^{I,0}(0,t)-x\varphi _{xx} ^{I,0}(0,t))+v ^{b,\varepsilon}( \varphi _{x} ^{I,0}(x,t)- \varphi _{x}^{I,0}(1,t)- (x-1)\varphi _{xx}^{I,0}(1,t))  \right]
   \nonumber \\
 & \displaystyle  \quad - \varepsilon ^{\frac{1}{2}}\left[v ^{B,\varepsilon}( \varphi _{x}^{I,1}(x,t)-\varphi _{x} ^{I,1}(0,t))  +v ^{b,\varepsilon}( \varphi _{x} ^{I,1}(x,t)- \varphi _{x}^{I,1}(1,t))\right]
    \nonumber \\
    & \displaystyle  \quad  - \varepsilon ^{\frac{1}{2}}\left[v ^{B,1}( \varphi _{x}^{I,0}(x,t)-\varphi _{x} ^{I,0}(0,t))  +v ^{b,1}( \varphi _{x} ^{I,0}(x,t)- \varphi _{x}^{I,0}(1,t))\right]
    \nonumber \\
&\displaystyle \quad -\varepsilon ^{\frac{1}{2}} \left[ \varphi _{x}^{ B,\varepsilon}(v ^{I,0}-v ^{I,0}(0,t)-xv _{x}^{I,0}(0,t))+\varphi _{x}^{ b,\varepsilon}\left( v ^{I,0}-v ^{I,0}(1,t)-(x-1)v _{x}^{I,0}(1,t) \right) \right]
               \nonumber \\
 & \displaystyle \quad -\varepsilon \left[ \varphi _{x}^{ B,\varepsilon}(v ^{I,1}(x,t)-v ^{I,1}(0,t))+\varphi _{x}^{ b,\varepsilon}(v ^{I,1}(x,t)-v ^{I,1}(1,t)) \right]
                \nonumber \\
&\displaystyle  \quad-\varepsilon \left[ \varphi _{x}^{B,2}(v ^{I,0}(x,t)-v ^{I,0}(0,t))+ \varphi _{x}^{b,2}(v ^{I,0}(x,t)-v ^{I,0}(1,t))\right]
                 \nonumber \\
& \quad -\varepsilon \varphi _{x} ^{I,1}(v ^{I,1}+ v ^{B,1}+v ^{b,1}) -\varepsilon ^{\frac{1}{2}}\left( \varphi _{x}^{ B,\varepsilon}v ^{b,\varepsilon}+\varphi _{x}^{ b,\varepsilon}v ^{B,\varepsilon} \right) - \varepsilon \left( \varphi _{x}^{ B,\varepsilon}v ^{b,1}+\varphi _{x}^{ b,\varepsilon}v ^{B,1} \right)
        \nonumber \\
& \displaystyle \quad - \varepsilon \left(  \varphi _{x}^{B,2} v ^{b,\varepsilon}+\varphi _{x}^{b,2} v ^{B,\varepsilon} \right)  -\varepsilon ^{\frac{3}{2}}  \varphi _{x}^{B,2}(v ^{I,1}+v ^{B,1}+v ^{b,1})  - \varepsilon ^{\frac{3}{2}} \varphi _{x}^{b,2}  (v ^{I,1}+v ^{B,1}+v ^{b,1})
         \nonumber \\
&\displaystyle \quad-b _{v}^{\varepsilon}[\varphi _{x}^{I,0}+M+\varepsilon ^{\frac{1}{2}}(\varphi_{x} ^{I,1}+\varphi_{x}^{ B,\varepsilon}+\varphi_{x} ^{ b,\varepsilon})+\partial _{x}b _{\varphi}^{\varepsilon}  ]   \nonumber \\
&\displaystyle \quad - \partial _{x}b _{\varphi}^{\varepsilon}(v ^{I,0}+v ^{B,\varepsilon}+v ^{b,\varepsilon}+\varepsilon ^{\frac{1}{2}}(v ^{I,1}+ v ^{B,1}+v ^{b,1}))
+\left[ \varepsilon v _{xx}^{I,0}+\varepsilon ^{\frac{3}{2}} v _{xx}^{I,1} \right]-\partial _{t} b _{v}^{\varepsilon}  =:\sum _{i=1}^{16}\mathcal{K}_{i}.
\end{align}
The proof of \eqref{con-R-2-ESTI-in-lem} is quite similar as the one for Lemma \ref{lem-R-1-VE}. We first prove $ \|\mathcal{R}_{2}^{\varepsilon}\|_{L _{T}^{\infty}L ^{2}} \leq C \varepsilon ^{3/4} $. By \eqref{L-INFT-V-B-0}, \eqref{L-INFT-V-veb-0}, \eqref{inte-transfer} and Taylor's formula, we get
\begin{align*}
 \displaystyle \|\mathcal{K}_{1}\|_{L _{T}^{2}L ^{2}}& \leq C \varepsilon ^{\frac{5}{4}} \|\partial _{x}^{3}\varphi ^{I,0}\|_{L _{T}^{\infty}L ^{\infty}}\left( \|\langle z \rangle ^{2} v ^{B,\varepsilon}\|_{L _{T}^{2}L _{z}^{2}}+\|\langle \xi \rangle ^{2} v ^{b,\varepsilon}\|_{L _{T}^{2}L _{\xi}^{2}} \right)  \leq C \varepsilon ^{\frac{5}{4}}.
\end{align*}
Similarly, by \eqref{l-INFTY-VFI-I-0}, \eqref{L-INFT-V-B-0}, \eqref{L-INFT-Vfi-B-0}, \eqref{tild-v-B-1-lem}, \eqref{Sobolev-z-xi} and \eqref{xi-transfer}, we have
\begin{gather*}
\displaystyle \begin{cases}
	 	\displaystyle \|\mathcal{K}_{3}\|_{L _{T}^{2}L ^{2}}  \leq C \varepsilon ^{\frac{5}{4}} \left( \| \langle z \rangle v ^{B,1}\|_{L _{T}^{2}L _{z} ^{2}}+\| \langle \xi \rangle v ^{b,1}\|_{L _{T}^{2}L _{\xi}^{2}} \right)  \|\varphi _{xx}^{I,0}\| _{L _{T}^{\infty}L ^{\infty}} \leq C  \varepsilon ^{\frac{5}{4}} , \\[2mm]
	\displaystyle \|\mathcal{K}_{4}\|_{L _{T}^{2}L ^{2}} \leq C \varepsilon ^{\frac{5}{4}}\left(  \| \varphi _{z} ^{ B,\varepsilon}\|_{L _{T}^{2}L _{z}^{2}}+ \| \varphi _{\xi} ^{ b,\varepsilon}\|_{L _{T}^{2}L _{\xi} ^{2}} \right)  \|v _{xx}^{I,0}\| _{L _{T}^{\infty}L ^{\infty}} \leq C \varepsilon ^{\frac{5}{4}} ,\\[2mm]
		\displaystyle  \|\mathcal{K}_{6}\|_{L _{T}^{2}L ^{2}} \leq
   C \varepsilon ^{\frac{5}{4}} \left( \|\varphi _{z}^{B,2}\|_{L _{T}^{2}L _{z}^{2}}+\|\varphi _{\xi} ^{b,2}\|_{L _{T}^{2}L _{\xi}^{2}} \right)\|v _{x}^{I,0}\|_{L _{T}^{\infty}L ^{\infty}} \leq C \varepsilon ^{\frac{5}{4}}.
\end{cases}
\end{gather*}
For $ \mathcal{K}_{2} $, we get
\begin{align*}
\displaystyle  \|\mathcal{K}_{2}\|_{L _{T}^{2}L ^{2}} &
\leq C  \varepsilon ^{\frac{1}{2}} \left[ \left( \int _{0}^{T}\int _{\mathcal{I}} \vert v ^{B,\varepsilon}\vert ^{2} \Big(\int _{0}^{x}\varphi _{yy}^{I,1}\mathrm{d}y\Big) ^{2} \mathrm{d}x \right) ^{\frac{1}{2}} +\left( \int _{0}^{T}\int _{\mathcal{I}} \vert v ^{b,\varepsilon}\vert ^{2} \Big(\int _{1}^{x}\varphi _{yy}^{I,1}\mathrm{d}y\Big) ^{2} \mathrm{d}x \right) ^{\frac{1}{2}}\right]
  \nonumber \\
  & \displaystyle \leq C \varepsilon \left(  \|\langle z \rangle v ^{B,\varepsilon}\| _{L _{T}^{\infty} L _{z} ^{2}}+ \|\langle \xi \rangle v ^{b,\varepsilon}\| _{L _{T}^{\infty} L _{\xi} ^{2}} \right)  \|\varphi _{xx}^{I,1}\| _{L _{T}^{2} L ^{2}} \leq \varepsilon,
\end{align*}
where we have used \eqref{L-INFT-V-B-0}, \eqref{L-INFT-V-veb-0} and \eqref{vfi-I-1-esti-first-con-lem}. Similarly, we get for $ \mathcal{K}_{5} $ that
\begin{align*}
\displaystyle  \displaystyle
	\|\mathcal{K}_{5}\|_{L _{T}^{2}L ^{2}} ^{2} \leq C \varepsilon \left( \| \langle z \rangle\varphi _{z} ^{  B,\varepsilon}\|_{L _{T}^{\infty}L _{z}^{2}}+\| \langle \xi \rangle\varphi _{\xi} ^{ b,\varepsilon}\|_{L _{T}^{\infty}L _{\xi} ^{2}}   \right) \|v _{x}^{I,1}\|_{L _{T}^{2}L ^{2}} \leq C \varepsilon.
\end{align*}
In view of \eqref{vfi-I-1-esti-first-con-lem}, \eqref{tild-v-B-1-lem}, \eqref{tild-v-b-1-in-lem} and \eqref{Sobolev-infty}, we have
\begin{align}\label{k-7-l2-esti}
\displaystyle  \|\mathcal{K}_{7}\|_{L _{T}^{2}L ^{2}}  \leq C \varepsilon  \|\varphi _{x}^{I,1}\|_{L _{T}^{2}L ^{2}} \left( \|v ^{I,1}\|_{L _{T}^{\infty}L ^{\infty}}+\|v ^{B,1}\|_{L _{T}^{\infty}L ^{\infty}}+\|v ^{b,1}\|_{L _{T}^{\infty}L ^{\infty}}\right) \leq C \varepsilon.
\end{align}
By analogous arguments as in \eqref{k-7-l2-esti} along with \eqref{vfi-I-1-esti-first-con-lem}, \eqref{tild-v-B-1-lem}, \eqref{b-vfi-x-l2}, \eqref{b-v--h1-esti} and Corollaries \ref{cor-V-B-0-INFTY}--\ref{cor-INF-vb1-vfib2}, we derive that
\begin{subnumcases}
   	   \displaystyle \displaystyle \|\mathcal{K}_{11}\|_{L _{T}^{2}L ^{2}} \leq C \varepsilon ^{\frac{5}{4}} \|\varphi _{z}^{B,2}\|_{L _{T}^{2}L _{z}^{2}} \left( \|v ^{I,1}\|_{L _{T}^{\infty}L ^{\infty}}+\|v ^{B,1}\|_{L _{T}^{\infty}L _{z}^{\infty}}+\|v ^{b,1}\|_{L _{T}^{\infty}L _{\xi}^{\infty}} \right) \leq C \varepsilon ^{\frac{5}{4}},\\[1mm]
   	   \|\mathcal{K}_{12}\|_{L _{T}^{2}L ^{2}} \leq C \varepsilon ^{\frac{5}{4}} \|\varphi _{\xi}^{b,2}\|_{L _{T}^{2}L _{z}^{2}} \left( \|v ^{I,1}\|_{L _{T}^{\infty}L ^{\infty}}+\|v ^{B,1}\|_{L _{T}^{\infty}L _{z}^{\infty}}+\|v ^{b,1}\|_{L _{T}^{\infty}L _{\xi}^{\infty}} \right) \leq C \varepsilon ^{\frac{5}{4}},\\[1mm]
   	   \|\mathcal{K}_{13}\|_{L _{T}^{2}L ^{2}} \leq C\left(1+\|\varphi _{x}^{I,0}\|_{L _{T}^{\infty} L ^{2}}+\varepsilon ^{\frac{1}{2}}\|\varphi _{x}^{I,1}\|_{L _{T}^{\infty}L ^{2}}+\varepsilon ^{\frac{1}{4}}\|\varphi _{z}^{B,\varepsilon }\|_{L _{T}^{\infty}L _{z}^{2}} +\varepsilon ^{\frac{1}{4}}\|\varphi _{\xi}^{ b,\varepsilon}\|_{L _{T}^{\infty}L _{\xi}^{2}}\right.
   	    \nonumber \\[1mm]
   	    \qquad \qquad ~\qquad \qquad\left. +\|\partial _{x}b _{\varphi}^{\varepsilon}\|_{L _{T}^{2}L ^{2}}\right)\|b _{v}^{\varepsilon}\|_{L _{T}^{\infty}L ^{\infty}}  \leq C  \varepsilon +C \varepsilon ^{\frac{\alpha}{2}}, \label{K-1-3-l2}\\[1mm]
        \displaystyle  \|\mathcal{K}_{14}\|_{L _{T}^{2}L ^{2}} \leq C\|\partial _{x}b _{\varphi}^{\varepsilon}\|_{L _{T}^{2}L ^{2}}\left( 1+ \varepsilon ^{\frac{1}{2}}( \|v ^{I,1}\|_{L _{T}^{\infty}L ^{\infty}}+\|v ^{B,1}\|_{L _{T}^{\infty}L ^{\infty}}+\|v ^{b,1}\|_{L _{T}^{\infty}L ^{\infty}}) \right)
         \nonumber \\[1mm]
         \qquad \qquad \quad\leq  C \varepsilon + C \varepsilon ^{\frac{2- \nu}{2}}.
\end{subnumcases}
Recalling the arguments in proving the estimate on $ \mathcal{P}_{5} $ in \eqref{p-5-l2}, we get, thanks to \eqref{half-l-infty}--\eqref{vfi-B-1-HALF-INTY} and Lemmas \ref{lem-v-B-0}--\ref{lem-v-B-1} that
\begin{align*}
\displaystyle  |\mathcal{K}_{8}\|_{L _{T}^{2}L ^{2}} +\|\mathcal{K}_{9}\|_{L _{T}^{2}L ^{2}} +\|\mathcal{K}_{10}\|_{L _{T}^{2}L ^{2}} \leq C \varepsilon ^{\frac{5}{4}}.
\end{align*}
For $ \mathcal{K}_{15} $ and $ \mathcal{K}_{16} $, it follows from \eqref{con-vfi-v-I-0-regula}, \eqref{vfi-I-1-esti-first-con-lem} and \eqref{pa-t-b-v-l2-no-weig} that
\begin{align*}
\displaystyle \|\mathcal{K}_{15}\| _{L _{T}^{2}L ^{2}} \leq \varepsilon \|v _{xx}^{I,0}\|_{L _{T}^{\infty}L ^{2}}+ \varepsilon ^{\frac{3}{2}}\|v _{xx}^{I,1}\|_{L _{T}^{2}L ^{2}} \leq C \varepsilon+C \varepsilon ^{\frac{3- \alpha}{2}}, \ \ \ \|\mathcal{K}_{16}\| _{L _{T}^{2}L ^{2}} \leq C \varepsilon +C \varepsilon ^{\frac{\alpha}{2}}.
\end{align*}
Substituting the estimates $ \|\mathcal{K}_{i}\|_{L _{T}^{2}L ^{2}} $ into \eqref{R-2-PRECISE}, we then get
\begin{align*}
\displaystyle \|\mathcal{R}_{2}^{\varepsilon}\|_{L _{T}^{2}L ^{2}} \leq C \varepsilon
 ^{\iota _{1}},
\end{align*}
where $ \iota _{1}=\min \left\{ 1, \frac{\alpha}{2}, \frac{3- \alpha}{2}, 1- \nu \right\}\geq \iota _{0} $ for $ 1< \alpha<5/4 $. Next we estimate $ \|t ^{1/2}\mathcal{R}_{2}^{\varepsilon}\|_{L _{T}^{\infty}L ^{2}} $. By repeating the above arguments on $ \mathcal{K} _{i}$ with  $ L _{T}^{2} $ replaced by $ L _{T}^{\infty} $, one has
\begin{gather*}
\displaystyle  \|\mathcal{K}_i\|_{L _{T}^{\infty}L ^{2}} \leq C \varepsilon ^{\frac{5}{4}},\ \ \mbox{for } i=1,3,4,7,8,9,10,\ \ \|\mathcal{K}_i\|_{L _{T}^{\infty}L ^{2}} \leq C \varepsilon ^{\frac{5}{4}-\frac{\alpha}{2}}\ \mbox{for }i=6,11,12,\\
\displaystyle \|\mathcal{K}_{5}\|_{L _{T}^{\infty}L ^{2}}  \leq C \varepsilon,\ \   \|\mathcal{K}_{15}\|_{L _{T}^{\infty}L ^{2}} \leq C\varepsilon+C\varepsilon ^{\frac{3- \alpha}{2}},
\end{gather*}
where \eqref{vfi-I-1-esti-first-con-lem}, \eqref{tild-v-B-1-lem}, \eqref{b-vfi-x-l2}, \eqref{b-v--h1-esti} and Corollaries \ref{cor-V-B-0-INFTY}--\ref{cor-INF-vb1-vfib2} have been used. For $ \mathcal{K}_{2} $, by virtue of \eqref{cor-vfi-I-1-infty}, \eqref{L-INFT-V-B-0}, \eqref{L-INFT-V-veb-0}, \eqref{vfi-I-1-infty} and \eqref{Sobolev-infty}, it holds that
\begin{align*}
\displaystyle \|\mathcal{K}_{2}\|_{L _{T}^{\infty}L ^{2}}  \leq C \varepsilon ^{5/4} \left(  \|\langle z \rangle v ^{B,\varepsilon}\| _{L _{T}^{\infty} L _{z} ^{2}}+ \|\langle \xi \rangle v ^{b,\varepsilon}\| _{L _{T}^{\infty} L _{\xi} ^{2}} \right)   \|\varphi _{xx}^{I,1}\| _{L _{T}^{\infty} L ^{\infty}} \leq C \varepsilon ^{\frac{5}{4}-\frac{\alpha}{2}}.
\end{align*}
Furthermore, thanks to \eqref{b-v--h1-esti} and \eqref{t-pa-x-b-vfi-infty}, we can modify \eqref{K-1-3-l2} to derive
\begin{align*}
\displaystyle \|t ^{\frac{1}{2}}\mathcal{K}_{13}\|_{L _{T}^{\infty}L ^{2}}& \leq \left(1+\|\varphi _{x}^{I,0}\|_{L _{T}^{\infty} L ^{2}}+\varepsilon ^{\frac{1}{2}}\|\varphi _{x}^{I,1}\|_{L _{T}^{\infty}L ^{2}}+\varepsilon ^{\frac{1}{4}}\|\varphi _{z}^{ B,\varepsilon}\|_{L _{T}^{\infty}L _{z}^{2}} +\varepsilon ^{\frac{1}{4}}\|\varphi _{\xi}^{ b,\varepsilon}\|_{L _{T}^{\infty}L _{\xi}^{2}}\right.
   	    \nonumber \\
   	    &\qquad \left. +\|t ^{\frac{1}{2}}\partial _{x}b _{\varphi}^{\varepsilon}\|_{L _{T}^{\infty}L ^{2}}\right)\|b _{v}^{\varepsilon}\|_{L _{T}^{\infty}L ^{\infty}}  \leq C  \varepsilon +C \varepsilon ^{\frac{\alpha}{2}},
\end{align*}
\begin{align*}
\displaystyle  \|t ^{\frac{1}{2}}\mathcal{K}_{14}\|_{L _{T}^{\infty}L ^{2}} &\leq \|t ^{\frac{1}{2}}\partial _{x}b _{\varphi}^{\varepsilon}\|_{L _{T}^{\infty}L ^{2}}\left( 1+ \varepsilon ^{\frac{1}{2}}( \|v ^{I,1}\|_{L _{T}^{\infty}L ^{\infty}}+\|v ^{B,1}\|_{L _{T}^{\infty}L ^{\infty}}+\|v ^{b,1}\|_{L _{T}^{\infty}L ^{\infty}}) \right)
         \nonumber \\
         &\displaystyle\leq  C  \varepsilon +C \varepsilon ^{\frac{\alpha}{2}},
\end{align*}
where we have used the condition $ \nu<1/4<1 $. To estimate $ \|t ^{1/2}\mathcal{K}_{16}\|_{L _{T}^{\infty}L ^{2}} $, we need to derive a more delicate estimate for $ \partial _{t}b _{v}^{\varepsilon} $. Indeed, by analogous arguments as in \eqref{b-v--h1-esti}, it holds that
\begin{align}\label{t-pa-t-b-v}
\displaystyle \|t ^{\frac{1}{2}}\partial _{t}b _{v}^{\varepsilon}\|_{L _{T}^{\infty}H ^{1}} & \leq C \varepsilon  + C \left( \|t ^{\frac{1}{2}}\partial _{t}\cala_{1}^{\varepsilon}\|_{L ^{\infty}((0,T))}+\|t ^{\frac{1}{2}}\partial _{t}\cala_{2}^{\varepsilon}\|_{L ^{\infty}((0,T))} \right) \leq C \varepsilon+ C \varepsilon ^{\frac{\alpha}{2}},
\end{align}
where we have used \eqref{cal-A-INFTY-T-SMALL}. Therefore we get
\begin{align*}
\displaystyle  \|t ^{\frac{1}{2}}\mathcal{K}_{16}\|_{L _{T}^{\infty}L ^{2}} \leq  \|t ^{\frac{1}{2}}\partial _{t} b _{v}^{\varepsilon}\| _{L _{T}^{\infty}L ^{2}}  \leq C \varepsilon +C \varepsilon ^{\frac{\alpha}{2}}.
\end{align*}
Gathering the estimates on the $ \|\cdot\|_{L _{T}^{\infty}L ^{2}} $ norms of $ \mathcal{K}_{i} $, we thus conclude that
\begin{align*}
\displaystyle \|t ^{\frac{1}{2}}\mathcal{R}_{2}^{\varepsilon}\|_{L _{T}^{\infty}L ^{2}} \leq C \varepsilon ^{\iota _{2}},
\end{align*}
where $ \iota _{2}=\min \left\{ 1, \frac{5}{4}- \frac{\alpha}{2}, \frac{3- \alpha}{2}, \frac{\alpha}{2}\right\}\geq \iota _{0} $ for $ 1< \alpha<5/4 $. In what follows, we will estimate $\|t \partial _{t}\mathcal{R}_{2}^{\varepsilon}\|_{L _{T}^{2}L ^{2}}  $. By repeating the arguments for estimates on $ \|\mathcal{R}_{2}^{\varepsilon}\|_{L _{T}^{2}L ^{2}} $, in view of \eqref{product-tim} and the fact that the time-weighted norms of the profiles involved in $ \|t \partial _{t}\mathcal{K}_{i}\|_{L _{T}^{2}L ^{2}}\,(1 \leq i \leq 12  ) $ are uniform with respect to $ \varepsilon $, says
\begin{align*}
\displaystyle \|t \varphi _{txx}^{I,1}\|_{L _{T}^{2}L ^{2}} + \|t\varphi _{zt}^{B,2}\|_{L _{T}^{2}L _{z}^{2}}+\|t \varphi _{\xi t}^{b,2}\|_{L _{T}^{2}L _{\xi}^{2}} +\|t v _{tx}^{I,1}\| _{L _{T}^{2}L ^{2}}\leq C
\end{align*}
for some constant $ C>0 $ independent of $ \varepsilon $, one has
\begin{align}\label{pa-t-2-bbv}
\displaystyle \|t \partial _{t}\mathcal{K}_{i}\| _{L _{T}^{2}L ^{2}} \leq C \varepsilon ^{\iota _{1}}\ \ \mbox{for } 1 \leq i \leq 12 .
\end{align}
It now remains to estimate $ \|t\partial _{t}\mathcal{K} _{i}\| _{L _{T}^{2}L ^{2}}$ for $ i=13,14,15,16 $. It follows from \eqref{vfi-I-1-esti-first-con-lem}, \eqref{tild-v-B-1-lem}, \eqref{b-vfi-x-l2},  \eqref{pa-t-b-vfi-ve-l2}, \eqref{b-v--h1-esti}, \eqref{some-more-b-vfi-b-v}, \eqref{t-pa-t-b-v} and Corollaries \ref{cor-V-B-0-INFTY}--\ref{cor-INF-vb1-vfib2} that
   \begin{align*}
   &\displaystyle  \|t \partial _{t}\mathcal{K}_{13}\|_{L _{T}^{2}L ^{2}}
    \nonumber \\
    &~\displaystyle\leq \|t \partial _{t} b _{v}^{\varepsilon}\|_{L _{T}^{2}L ^{\infty}}\left(\|\varphi _{x}^{I,0}\|_{L _{T}^{\infty} L ^{2}}+\varepsilon ^{\frac{1}{2}}\|\varphi _{x}^{I,1}\|_{L _{T}^{\infty}L ^{2}}+\varepsilon ^{\frac{1}{4}}\|\varphi _{z}^{ B,\varepsilon}\|_{L _{T}^{\infty}L _{z}^{2}}+\varepsilon ^{\frac{1}{4}}\|\varphi _{\xi}^{ b,\varepsilon}\|_{L _{T}^{\infty}L _{\xi}^{2}} \right)
    \nonumber \\
    &~\displaystyle  \quad+\| b _{v}^{\varepsilon}\|_{L _{T}^{\infty}L ^{\infty}}\left(\|\varphi _{xt}^{I,0}\|_{L _{T}^{2} L ^{2}}+\varepsilon ^{\frac{1}{2}}\|t \varphi _{xt}^{I,1}\|_{L _{T}^{2}L ^{2}}+\varepsilon ^{\frac{1}{4}}\|t \varphi _{zt}^{ B,\varepsilon}\|_{L _{T}^{2}L _{z}^{2}}+\varepsilon ^{\frac{1}{4}}\|t \varphi _{\xi t}^{ b,\varepsilon}\|_{L _{T}^{2}L _{\xi}^{2}} \right)
     \nonumber \\
         & ~\displaystyle \quad+\|\partial _{t} b _{v}^{\varepsilon}\|_{L _{T}^{2}L ^{2}}+\|t \partial _{t} b _{v}^{\varepsilon}\|_{L _{T}^{\infty}L ^{\infty}} \|\partial _{x}b _{\varphi}^{\varepsilon}\|_{L _{T}^{2}L ^{2}} +\|b _{v}^{\varepsilon}\|_{L _{T}^{\infty}L ^{\infty}} \|t \partial _{t}\partial _{x}b _{\varphi}^{\varepsilon}\|_{L _{T}^{2}L ^{2}}
          \nonumber \\
          &~\displaystyle\leq C \varepsilon ^{\frac{\alpha}{2}},
   \end{align*}
   \begin{align*}
\displaystyle  \|t \partial _{t}\mathcal{K}_{14}\|_{L _{T}^{2}L ^{2}} &\leq C\|t \partial _{t}\partial _{x}b _{\varphi}^{\varepsilon}\|_{L _{T}^{2}L ^{\infty}}\left( 1+ \varepsilon ^{\frac{1}{2}}( \|v ^{I,1}\|_{L _{T}^{\infty}L ^{2}}+\|v ^{B,1}\|_{L _{T}^{\infty}L ^{2}}+\|v ^{b,1}\|_{L _{T}^{\infty}L ^{2}}) \right)
         \nonumber \\
         & \displaystyle \quad + \|\partial _{x}b _{\varphi}^{\varepsilon}\|_{L _{T}^{2}L ^{\infty}}\left( 1+\|t\partial _{t}v ^{B,\varepsilon}\|_{L _{T}^{\infty}L ^{2}}+\|t\partial _{t}v ^{b,\varepsilon}\|_{L _{T}^{\infty}L ^{2}}\right.
          \nonumber \\
         &\displaystyle \quad \left.  \qquad+\varepsilon ^{\frac{1}{2}}(\|t \partial _{t}v ^{I,1}\|_{L _{T}^{\infty}L ^{2}}+\|t\partial _{t}v ^{B,1}\|_{L _{T}^{\infty}L ^{2}} +\|t\partial _{t}v ^{b,1}\|_{L _{T}^{\infty}L ^{2}}) \right) \leq  C  \varepsilon +C \varepsilon ^{\frac{\alpha}{2}}.
\end{align*}
   For $ \|t\partial _{t}\mathcal{K} _{15}\| _{L _{T}^{2}L ^{2}}$, due to \eqref{con-vfi-v-I-0-regula} and \eqref{v-I1-txx}, it holds that
   \begin{gather}
   \displaystyle \|t\partial _{t}\mathcal{K} _{15}\| _{L _{T}^{2}L ^{2}} \leq C \varepsilon\|v _{txx}^{I,0}\| _{L _{T}^{2}L ^{2}}+C \varepsilon ^{\frac{3}{2}}\|v _{txx}^{I,1}\| _{L _{T}^{2}L ^{2}} \leq C \varepsilon+ C \varepsilon ^{\frac{3- \alpha}{2}} \nonumber
   \end{gather}
   for $ 1<\alpha<5/4 $ and $ 0< \varepsilon \leq 1 $. Finally, we utilize \eqref{pa-t-2-b-v-ve-wei} to derive that
  \begin{align}\label{K-15-t2}
   \displaystyle \|t \partial _{t}\mathcal{K}_{16}\| _{L _{T}^{2}L ^{2}} = \|t\partial _{t}^{2} b _{v}^{\varepsilon}\| _{L _{T}^{2}L ^{2}} \leq C \varepsilon +C \varepsilon ^{\frac{\alpha}{2}}.
   \end{align}
   Therefore we summarize from \eqref{pa-t-2-bbv}--\eqref{K-15-t2} that $\|t \partial _{t}\mathcal{R}_{2}^{\varepsilon}\|_{L _{T}^{2}L ^{2}} \leq C \varepsilon ^{\iota _{1}}  $. The proof of Lemma \ref{lem-R-2-ESTI} is complete.
\end{proof}

\vspace{2mm}
To derive some uniform-in-$ \varepsilon $ estimates for $ (\Phi ^{\varepsilon}, V ^{\varepsilon}) $, we need to introduce some preparatory results. Denote $ \hat{\Phi}=\varphi _{x}^{I,0}+\varepsilon ^{1/2}(\varphi _{x}^{ B,\varepsilon}+\varphi _{x}^{ b,\varepsilon}) $. In view of \eqref{vfi-bd-1ord-lt} and \eqref{firs-bd-1-rt}, it holds for any $ T>0 $ that
\begin{gather}\label{postive-appr}
\displaystyle 0<K _{\ast}^{-1} \leq \hat{\Phi}+M \leq K _{\ast}
\end{gather}
for some constant $ K _{\ast}>0 $ independent of $ \varepsilon$, provided $ \varepsilon \leq \varepsilon _{1} $ for some small constant $ \varepsilon _{1}>0 $ which may depend on $ T $. Furthermore, from \eqref{vfi-bd-1ord-lt-approxi} and $\eqref{first-bd-pro-rt-approxi}_{2}$, we deduce that
\begin{align*}
    \varphi^{B,\varepsilon}_{zz}-\varphi^{B,\varepsilon}_{z}v_{z}^{B,\varepsilon}=(\varphi_{x}^{I,0}(0,t)+M)v_{z}^{B,\varepsilon},\ \ \varphi^{b,\varepsilon}_{\xi\xi}-\varphi^{b,\varepsilon}_{\xi}v_{\xi}^{b,\varepsilon}=(\varphi_{x}^{I,0}(1,t)+M)v_{\xi}^{b,\varepsilon}.
\end{align*}
This along with $ \eqref{eq-outer-0}_{1} $ and \eqref{approximate-v-formula} implies that
\begin{align}
\displaystyle \hat{\Phi}_{x}-(\hat{\Phi}+M)V _{x}^{a}&= - \left[ v _{x}^{ B,\varepsilon}(\varphi _{x}^{I,0}(x,t)-\varphi _{x}^{I,0}(0,t))+v _{x}^{ b,\varepsilon}(\varphi _{x}^{I,0}(x,t)-\varphi _{x}^{I,0}(1,t)) \right]
 \nonumber \\
 & \displaystyle \quad - \varepsilon ^{\frac{1}{2}}\varphi _{x}^{ B,\varepsilon}\left[ v _{x}^{I,0}+ v _{x}^{ b,\varepsilon}+ \varepsilon ^{\frac{1}{2}}(v _{x}^{I,1}+v _{x}^{B,1}+ v _{x}^{b,1})+\partial _{x} b _{v}^{\varepsilon}\right]
  \nonumber \\
  &\displaystyle  \quad -\varepsilon ^{\frac{1}{2}}\varphi _{x}^{ b,\varepsilon}\left[ v _{x}^{I,0}+ v _{x}^{ B,\varepsilon}+ \varepsilon ^{\frac{1}{2}}(v _{x}^{I,1}+v _{x}^{B,1}+ v _{x}^{b,1})+\partial _{x} b _{v}^{\varepsilon}\right]
   \nonumber \\
   & \displaystyle \quad- (\varphi _{x}^{I,0}+M)\left[ \varepsilon ^{\frac{1}{2}}(v _{x}^{I,1}+v _{x}^{B,1}+v _{x}^{b,1}) +\partial _{x}b _{v}^{\varepsilon}\right]+\varphi _{t}^{I,0}. \nonumber
\end{align}
We proceed to show that
\begin{gather}\label{hat-fida-esti-con}
\displaystyle \|\hat{\Phi}_{t}\|_{L _{T}^{2}L ^{\infty}}+ \|\hat{\Phi}_{x}-(\hat{\Phi}+M)V _{x}^{a}\|_{L _{T}^{2}L ^{\infty}} \leq C.
\end{gather}
Indeed, thanks to \eqref{con-vfi-v-I-0-regula}, \eqref{l-INFTY-VFI-I-0}, \eqref{L-INFT-Vfi-B-0}, \eqref{v-B-0-regularity-a}, \eqref{con-vfi-B-1-a} and \eqref{L-INFT-veVfi-b-uniform-0}, we get
\begin{align}\label{esti-approx-fida-2}
\displaystyle  \|\hat{\Phi}_{t}\|_{L _{T}^{2}L ^{\infty}} \leq \|\varphi _{xt}^{I,0}\|_{L _{T}^{2}L ^{\infty}}+\|\varphi _{zt}^{ B,\varepsilon}\|_{L _{T}^{2}L _{z}^{\infty}}+\|\varphi _{\xi t}^{ b,\varepsilon}\|_{L_{T}^{2}L _{\xi}^{\infty}} \leq C
\end{align}
and
\begin{align}
&\displaystyle  \|\hat{\Phi}_{x}-(\hat{\Phi}+M)V _{x}^{a}\|_{L _{T}^{2}L ^{\infty}}
 \nonumber \\
 &~\displaystyle \leq \|\varphi _{xx}^{I,0}\|_{L _{T}^{\infty}L ^{\infty}} \left( \|v _{z}^{ B,\varepsilon}\|_{L _{T}^{2}L _{z}^{\infty}}+\|v _{\xi}^{ b,\varepsilon}\|_{L _{T}^{2}L _{\xi}^{\infty}} \right)+ C \left(  \|\varepsilon ^{\frac{1}{2}}\varphi _{x}^{ B,\varepsilon} v _{x}^{ b,\varepsilon}\|_{L _{T}^{2}L _{z}^{\infty}}+\|\varepsilon ^{\frac{1}{2}}\varphi _{x}^{ b,\varepsilon} v _{x}^{ B,\varepsilon}\|_{L _{T}^{2}L _{\xi}^{\infty}} \right)
 \nonumber \\
 &~ \displaystyle \quad +\left( \|v _{x}^{I,0}\|_{L _{T}^{2}L ^{\infty}}+\varepsilon ^{\frac{1}{2}}\|v _{x}^{I,1}\|_{L _{T}^{2}L ^{\infty}}+\|v _{z}^{B,1}\|_{L _{T}^{2}L ^{\infty}}
  + \|v _{\xi}^{b,1}\|_{L _{T}^{2}L ^{\infty}}+\|\partial _{x}b _{v}^{\varepsilon}\|_{L _{T}^{2}L ^{\infty}}\right)
   \nonumber \\
   &~\displaystyle \quad \quad \quad\times(1+\|\varphi _{\xi}^{b,\varepsilon }\|_{L _{T}^{\infty}L ^{\infty}}+ \|\varphi _{z}^{B,\varepsilon }\|_{L _{T}^{\infty}L ^{\infty}}) + \|\varphi _{t}^{I,0}\|_{L _{T}^{2}L ^{\infty}} \leq C,\nonumber
\end{align}
where we have used
\begin{gather*}
\displaystyle  \varepsilon ^{\frac{1}{2}}\|v _{x}^{I,1}\|_{L _{T}^{2}L ^{\infty}} \leq C \varepsilon ^{\frac{1}{2}} \|v _{x}^{I,1}\|_{L _{T}^{2}L ^{2}}^{\frac{1}{2}}\left(   \|v _{xx}^{I,1}\|_{L _{T}^{2}L ^{2}}^{\frac{1}{2}}+ \|v _{x}^{I,1}\|_{L _{T}^{2}L ^{2}}^{\frac{1}{2}}\right) \leq C \varepsilon ^{\frac{1}{2}-\frac{\alpha}{4}} +C  \varepsilon ^{\frac{1}{2}}\leq C
\end{gather*}
for $ 1<\alpha<5/4 $, and the following inequality
\begin{align}
\displaystyle  C \left(  \|\varepsilon ^{\frac{1}{2}}\varphi _{x}^{ B,\varepsilon} v _{x}^{ b,\varepsilon}\|_{L _{T}^{2}L^{\infty}}+\|\varepsilon ^{\frac{1}{2}}\varphi _{x}^{ b,\varepsilon} v _{x}^{B,\varepsilon }\|_{L _{T}^{2}L ^{\infty}} \right) \leq C \varepsilon\nonumber
\end{align}
due to \eqref{half-l-infty}, \eqref{vfi-B-1-HALF-INTY} and similar arguments as in \eqref{P-6-1}. Furthermore, recall the definition of $ \Phi ^{a} $ and $ \hat{\Phi} $, we have
\begin{align}\label{fida-hat-err}
\displaystyle  \Phi_{x}^{a}+M=\hat{\Phi}+M+ \varepsilon ^{\frac{1}{2}}\varphi _{x}^{I,1}+\varepsilon \left( \varphi _{x}^{B,2}+\varphi _{x}^{b,2} \right)+ \partial _{x}b _{\varphi}^{\varepsilon}=:\hat{\Phi}+M+\hat{\Phi} ^{\text{err}},
\end{align}
with $\hat{\Phi} ^{\text{err}}= \varepsilon ^{\frac{1}{2}}\varphi _{x}^{I,1}+\varepsilon ( \varphi _{x}^{B,2}+\varphi _{x}^{b,2})+\partial _{x}b _{\varphi}^{\varepsilon}$ satisfying
\begin{align}\label{hat-fida-err-esti1}
\displaystyle  \|\hat{\Phi} ^{\text{err}}\|_{L _{T} ^{2}L ^{\infty}} &\leq C \varepsilon ^{\frac{1}{2}}\left( \|\varphi _{x}^{I,1}\|_{L _{T} ^{2}L ^{\infty}}+\|\varphi _{z}^{B,2}\|_{L _{T} ^{2}L _{z}^{\infty}}+\|\varphi _{\xi}^{b,2}\|_{L _{T} ^{2}L _{\xi}^{\infty}} +\varepsilon ^{-\frac{1}{2}}\|\partial _{x}b _{\varphi}^{\varepsilon}\|_{L _{T} ^{2}L _{\xi}^{\infty}}\right)
 \nonumber \\
& \leq C \varepsilon ^{\frac{1}{2}}(1+ \varepsilon ^{\frac{1}{2}}+\varepsilon ^{\frac{1}{2}-\nu}) \leq C \varepsilon ^{\frac{1}{2}},
\end{align}
\begin{align}
\displaystyle \displaystyle  \|\hat{\Phi} _{x}^{\text{err}}\|_{L _{T} ^{2}L ^{\infty}} &\leq C\left(  \varepsilon ^{\frac{1}{2}}\|\varphi _{xx}^{I,1}\|_{L _{T} ^{2}L ^{\infty}}+\|\varphi _{zz}^{B,2}\|_{L _{T} ^{2}L _{z}^{\infty}}+\|\varphi _{\xi \xi}^{b,2}\|_{L _{T} ^{2}L _{\xi}^{\infty}} +\|\partial _{x}^{2} b _{\varphi}^{\varepsilon}\|_{L _{T} ^{2}L _{\xi}^{\infty}}\right)
 \nonumber \\
& \leq C \varepsilon ^{\frac{1}{2}-\frac{\alpha}{4}}+C+C \varepsilon+C \varepsilon ^{1-2 \nu} \leq C  \nonumber
\end{align}
for $ 1< \alpha<1+\nu<5/4 $, where we have used \eqref{vfi-I-1-esti-first-con-lem}, \eqref{tild-v-B-1-lem}, \eqref{tild-v-b-1-in-lem}, \eqref{b-vfi-L-2-l-infty}, \eqref{Sobolev-infty} and analogous arguments as in \eqref{b-vfi-L-2-l-infty}. Finally, we introduce some time-weighted estimates on the approximate solutions $ (\Phi ^{a}, V ^{a}) $.
\begin{lemma}
Let $ 0< \varepsilon <1 $. It holds for any $ T>0 $ that
\begin{align}\label{time-weighted-for-high}
\displaystyle  \varepsilon ^{\frac{1}{2}}\|t ^{\frac{3}{2}} V _{x}^{a}\|_{L _{T}^{\infty} L ^{\infty}}+\|t ^{\frac{3}{2}} V _{t}^{a}\|_{L _{T}^{\infty} L ^{\infty}}+\varepsilon ^{\frac{1}{2}}\|t ^{\frac{3}{2}}V _{tx}^{a}\|_{L _{T}^{\infty} L^{\infty}}+\|t ^{\frac{9}{4}}\Phi _{xt}^{a}\|_{L _{T}^{\infty} L ^{\infty}} +\|t ^{\frac{1}{2}}\Phi _{x}^{a}\|_{L _{T}^{\infty} L ^{\infty}} \leq C,
\end{align}
where $ C >0$ is a constant independent of $ \varepsilon$.

\end{lemma}
\begin{proof}
From \eqref{l-INFTY-VFI-I-0}, \eqref{L-INFT-V-B-0}, \eqref{L-INFT-V-veb-0}, \eqref{v-I-1-infty}, \eqref{tild-v-B-1-lem}, \eqref{INF-vfi-b2}, \eqref{b-v--h1-esti} and \eqref{inte-transfer}, we get
\begin{align*}
   \|t  V _{x}^{a}\|_{L _{T}^{\infty} L ^{\infty}}& \leq \|t  v _{x}^{I,0}\|_{L _{T}^{\infty}L ^{\infty}}+\varepsilon ^{-\frac{1}{2}}\left( \|t  v _{z}^{ B,\varepsilon}\|_{L _{T}^{\infty}H _{z}^{1}}+\|t  v _{ \xi}^{ b,\varepsilon}\|_{L _{T}^{\infty} H _{\xi}^{1}} \right)
    \nonumber \\
    & \displaystyle \quad+\varepsilon ^{\frac{1}{2}}\left( \|v _{x}^{I,1}\|_{L _{T}^{\infty}L ^{2}}^{\frac{1}{2}} \|v _{xx}^{I,1}\|_{L _{T}^{\infty}L ^{2}}^{\frac{1}{2}}+\|v _{x}^{I,1}\|_{L _{T}^{\infty}L ^{2}}\right)
       \nonumber \\
   & \displaystyle \quad+ \left( \|t  v _{z}^{B,1}\|_{L _{T}^{\infty} L ^{\infty}}+\|t ^{\frac{3}{2}} v _{ \xi}^{b,1}\|_{L _{T}^{\infty} L ^{\infty}} \right)+\|t \partial _{x} b _{v}\|_{L _{T} ^{\infty}L ^{\infty}} \leq C \varepsilon ^{- \frac{1}{2}},
  \end{align*}
  where $ C>0 $ is a constant that may depend on $ T $ but independent of $ \varepsilon $. By \eqref{l-INFTY-VFI-I-0}, \eqref{L-INFT-V-B-0}, \eqref{L-INFT-V-veb-0}, \eqref{v-I-1-infty}, \eqref{tild-v-B-1-lem}, \eqref{INF-vfi-b2}, \eqref{approximate-v-formula}, \eqref{t-pa-t-b-v} and \eqref{inte-transfer}, we have
  \begin{align*}%\label{V-a-H1}
  \displaystyle  \|t ^{\frac{3}{2}} V _{t}^{a}\|_{L _{T}^{\infty} L ^{\infty}}& \leq \|t ^{\frac{3}{2}} v _{t}^{I,0}\|_{L _{T}^{\infty} H^{1}}+\left( \|t ^{\frac{3}{2}} v _{t}^{ B,\varepsilon}\|_{L _{T}^{\infty} H _{z}^{1}}+\|t ^{\frac{3}{2}} v _{t}^{ b,\varepsilon}\|_{L _{T}^{\infty} H _{\xi}^{1}} \right) +\varepsilon ^{\frac{1}{2}}\|t ^{\frac{3}{2}}v _{t}^{I,1}\|_{L _{T}^{\infty} H ^{1}}
   \nonumber \\
   & \displaystyle \quad+ \varepsilon ^{\frac{1}{2}}\left( \|t ^{\frac{3}{2}} v _{t}^{B,1}\|_{L _{T}^{\infty} H _{z}^{1}}+\|t ^{\frac{3}{2}} v _{t}^{b,1}\|_{L _{T}^{\infty} H _{\xi}^{1}} \right)+\|t ^{\frac{3}{2}}\partial _{t}b _{v}\|_{L _{T} ^{\infty}H ^{1}} \leq C .
  \end{align*}
  Notice that $ \partial _{x}b _{v} $ is independent of $ x $. Then we have from \eqref{t-pa-t-b-v} that $ \|t ^{\frac{1}{2}}\partial _{t}\partial _{x} b _{v}\|_{L _{T} ^{\infty}L ^{\infty}} \leq C$. Hence, it holds that
  \begin{align}\label{v-txa-time}
   \displaystyle  \|t ^{\frac{3}{2}} V _{tx}^{a}\|_{L _{T}^{\infty} L ^{\infty}}& \leq \|t ^{\frac{3}{2}} v _{tx}^{I,0}\|_{L _{T}^{\infty}L ^{\infty}}+\varepsilon ^{-\frac{1}{2}}\left( \|t ^{\frac{3}{2}} v _{tz}^{ B,\varepsilon}\|_{L _{T}^{\infty}H _{z}^{1}}+\|t ^{\frac{3}{2}} v _{t \xi}^{ b,\varepsilon}\|_{L _{T}^{\infty} H _{\xi}^{1}} \right)
    \nonumber \\
    & \displaystyle \quad+\varepsilon ^{\frac{1}{2}}\left( \|t ^{3} v _{tx}^{I,1}\|_{L _{T}^{\infty}L ^{2}}^{\frac{1}{2}} \|v _{txx}^{I,1}\|_{L _{T}^{\infty}L ^{2}}^{\frac{1}{2}}+\|t ^{\frac{3}{2}} v _{tx}^{I,1}\|_{L _{T}^{\infty}L ^{2}}\right)
       \nonumber \\
   & \displaystyle \quad+ \left( \|t ^{\frac{3}{2}} v _{tz}^{B,1}\|_{L _{T}^{\infty} L ^{\infty}}+\|t ^{\frac{3}{2}} v _{t \xi}^{b,1}\|_{L _{T}^{\infty} L ^{\infty}} \right)+\|t ^{\frac{3}{2}}\partial _{t}\partial _{x} b _{v}\|_{L _{T} ^{\infty}L ^{\infty}} \leq C \varepsilon ^{- \frac{1}{2}},
   \end{align}
   %Recalling \eqref{pa-t-bv-time-weight} and \eqref{pa-t-2-b-v-ve-wei}, we have
 %\begin{align}\label{bv-inty-h1-time}
  %\displaystyle \|t ^{\frac{3}{2}}\partial _{t}b _{v}\|_{L _{T} ^{\infty}H ^{1}} \leq C
  %\end{align}
  %for some constant $ C>0 $ independent of $ \varepsilon $.
   where we have used \eqref{l-INFTY-VFI-I-0}, \eqref{L-INFT-V-B-0}, \eqref{L-INFT-V-veb-0}, \eqref{v-I-1-infty}, \eqref{tild-v-B-1-lem}, \eqref{INF-vfi-b2}, \eqref{approximate-v-formula} and \eqref{inte-transfer}. By analogous arguments as proving \eqref{pa-t-2-b-vfi-ve-esti}, we know that
    \begin{align}
    \displaystyle  \|t ^{2}\partial _{t}\partial _{x}b _{\varphi}^{\varepsilon}\|_{L _{T}^{\infty}L ^{\infty}} & \leq C \varepsilon +  \varepsilon\|t ^{2} \varphi _{t}^{B,2}(0,t) \|_{L ^{\infty}((0,T))}\|[(1-x)		{\mathop{\mathrm{e}}}^{- \frac{x}{\varepsilon ^{\nu}}}]'\|_{L ^{\infty}}
 \nonumber
 \\
 &~\displaystyle \quad +\varepsilon\| t ^{2}\varphi _{t}^{b,2}(0,t) \|_{L ^{\infty}((0,T))}\|[x		{\mathop{\mathrm{e}}}^{- \frac{1-x}{\varepsilon ^{\nu}}}]'\|_{L ^{\infty}}
 \leq C \varepsilon + C \varepsilon ^{1- \nu} \leq C.\nonumber
    \end{align}
This along with analogous arguments as proving \eqref{v-txa-time} implies that
\begin{align}
\displaystyle  \|t ^{2}\Phi _{xt}^{a}\|_{L _{T}^{\infty} L ^{\infty}} & \leq \|t ^{2} \varphi _{tx}^{I,0}\|_{L _{T}^{\infty}L ^{\infty}}+\left( \|t ^{2} \varphi _{tz}^{ B,\varepsilon}\|_{L _{T}^{\infty}H _{z}^{1}}+\|t ^{2} \varphi _{t \xi}^{ b,\varepsilon}\|_{L _{T}^{\infty} H _{\xi}^{1}} \right) +\varepsilon ^{\frac{1}{2}}\|t ^{2}\varphi _{tx}^{I,1}\|_{L _{T}^{\infty} H ^{1}}
   \nonumber \\
   & \displaystyle \quad+ \varepsilon ^{\frac{1}{2}}\left( \|t ^{2} \varphi _{tz}^{B,2}\|_{L _{T}^{\infty} L ^{\infty}}+\|t ^{2} \varphi _{t \xi}^{b,2}\|_{L _{T}^{\infty} L ^{\infty}} \right)+\|t ^{2}\partial _{t}\partial _{x} b _{\varphi}^{\varepsilon}\|_{L _{T} ^{\infty}L ^{\infty}} \leq C\nonumber
\end{align}
for some constant $ C>0 $ independent of $ \varepsilon $. Similar arguments along with \eqref{t-pa-x-b-vfi-infty} further yields that  $\|t ^{\frac{1}{2}}\Phi _{x}^{a}\|_{L _{T}^{\infty} L ^{\infty}} \leq C  $. The proof is complete.
\end{proof}

\subsubsection{Lower-order estimates} % (fold)
\label{ssub:lower_order_estimates}
Now we shall establish some uniform-in-$\varepsilon$ estimates for $ (\Phi ^{\varepsilon}, V ^{\varepsilon})$ by beginning with the $ L ^{2} $ estimates. Throughout this section, we assume that $ (\Phi ^{\varepsilon}, V^{\varepsilon}) $ satisfies for any $ T>0 $,
\begin{gather}\label{apriori-assumption}
\displaystyle \sup _{t \in [0,T]} \left( \varepsilon\|V _{x}^{\varepsilon}(\cdot, t)\|_{L ^{2}}^{2} +\int _{0}^{t}\|\Phi ^{\varepsilon}\|_{L ^{\infty}}^{2}\mathrm{d}\tau \right) \leq \delta,
\end{gather}
for some small constant $ \delta>0 $ to be determined later. We remark that $ \delta $ could depend on $ T $ here since we are mainly concerned with the zero-diffusion limit of the problem with $ T>0 $ fixed. Besides, our estimates in the sequel are based on the estimates established in previous sections. Thus we always assume the conditions in Lemmas \ref{lem-regul-outer-layer-0}--\ref{lem-v-b-1}, \ref{lem-R-1-VE}, \ref{lem-R-2-ESTI} hold, and the results therein will be frequently used in the subsequent analysis without further clarification.

\begin{lemma}\label{lem-L2-pertur}
In addition to the conditions in Proposition \ref{prop-pertur}, we assume that the solution $ (\Phi ^{\varepsilon},V ^{\varepsilon}) $ to \eqref{eq-perturbation} on $ [0,T] $ satisfies \eqref{apriori-assumption}. Then there exists a constant $ \delta _{1}>0 $ such that for any $ t \in [0,T] $,
\begin{gather}\label{con-L2-PERT}
\|V ^{\varepsilon}(\cdot,t)\|_{L ^{2}}^{2}+ \|\Phi ^{\varepsilon}(\cdot,t)\|_{L ^{2}}^{2}+ \int _{0}^{t}\left( \|V ^{\varepsilon}\|_{L ^{2}}^{2}+\varepsilon \|V _{x}^{\varepsilon}\|_{L ^{2}}^{2}+\|\Phi _{x}^{\varepsilon}\|_{L ^{2}}^{2} \right) \mathrm{d}\tau \leq C \varepsilon ^{2\iota _{0}-1},
\end{gather}
provided $ \delta \leq \delta _{1} $, where $ \iota _{0} $ is as in \eqref{iota0-defi}, and $ C>0 $ is a constant depending on $ T $ but independent of $ \varepsilon $.
\end{lemma}
\begin{proof}
  Multiplying the first equation in \eqref{eq-perturbation} by $ \frac{\Phi ^{\varepsilon}}{\hat{\Phi}+M} $ followed by an integration over $ \mathcal{I} $ and integration by parts, one deduces that
\begin{align}\label{L-2-diff-id}
&\displaystyle \frac{1}{2}\frac{\mathrm{d}}{\mathrm{d}t}\int _{\mathcal{I} }\frac{\left\vert \Phi ^{\varepsilon}\right\vert ^{2}}{\hat{\Phi}+M} \mathrm{d}x +\int _{\mathcal{I} }\frac{\left\vert \Phi _{x}^{\varepsilon}\right\vert ^{2}}{\hat{\Phi}+M} \mathrm{d}x
 \nonumber \\
 &~\displaystyle = -\frac{1}{2}\int _{\mathcal{I} }\frac{\left\vert \Phi ^{\varepsilon}\right\vert ^{2}}{(\hat{\Phi}+M)^{2}}\hat{\Phi}_{t} \mathrm{d}x - \varepsilon ^{\frac{1}{2}}\int _{\mathcal{I} }  \frac{\Phi ^{\varepsilon}\Phi _{x}^{\varepsilon}V _{x}^{\varepsilon}}{\hat{\Phi}+M}\mathrm{d}x +\int _{\mathcal{I} }\varepsilon ^{- \frac{1}{2}}\mathcal{R}_{1}^{\varepsilon}\frac{\Phi^{\varepsilon}}{\hat{\Phi}+M} \mathrm{d}x
  \nonumber \\
  &~\displaystyle \quad - \int _{\mathcal{I} }V _{x}^{\varepsilon}\Phi ^{\varepsilon} \frac{\Phi _{x}^{a}+M}{\hat{\Phi}+M}\mathrm{d}x+\int _{\mathcal{I} }\frac{\Phi ^{\varepsilon}\Phi _{x}^{\varepsilon}}{(\hat{\Phi}+M)^{2}}\left( \hat{\Phi}_{x}-(\hat{\Phi}+M)V _{x}^{a} \right)  \mathrm{d}x.
\end{align}
The terms on the right hand side of \eqref{L-2-diff-id} can be treated as follows. By \eqref{postive-appr}, we derive that
\begin{gather}\label{fir-l2}
\displaystyle   -\frac{1}{2}\int _{\mathcal{I} }\frac{\left\vert \Phi ^{\varepsilon}\right\vert ^{2}}{(\hat{\Phi}+M)^{2}} \hat{\Phi}_{t} \mathrm{d}x \leq C \|\hat{\Phi}_{t}\|_{L ^{\infty}}\int _{\mathcal{I} }\frac{\left\vert \Phi ^{\varepsilon}\right\vert ^{2}}{\hat{\Phi}+M}\mathrm{d}x.
\end{gather}
Thanks to \eqref{postive-appr}, \eqref{apriori-assumption}, \eqref{Sobolev-modified} and the Cauchy-Schwarz inequality, it holds that
\begin{align}\label{firs-l2-nonlinear}
\displaystyle \varepsilon ^{\frac{1}{2}}\int _{\mathcal{I} }  \frac{\Phi ^{\varepsilon}\Phi _{x}^{\varepsilon}V _{x}^{\varepsilon}}{\hat{\Phi}+M}\mathrm{d}x &\leq \frac{1}{32}\int _{\mathcal{I} }\frac{\left\vert \Phi _{x}^{\varepsilon}\right\vert ^{2}}{\hat{\Phi}+M} \mathrm{d}x+ C (\varepsilon\|V _{x}^{\varepsilon}\|_{L ^{2}}^{2})\|\Phi ^{\varepsilon}\|_{L ^{\infty} }^{2}
 \nonumber \\
 &\displaystyle\leq \frac{1}{32}\int _{\mathcal{I} }\frac{\left\vert \Phi _{x}^{\varepsilon}\right\vert ^{2}}{\hat{\Phi}+M} \mathrm{d}x+ C \delta\int _{\mathcal{I} }\frac{\left\vert \Phi _{x}^{\varepsilon}\right\vert ^{2}}{\hat{\Phi}+M} \mathrm{d}x
  \leq \frac{1}{16}\int _{\mathcal{I} }\frac{\left\vert \Phi _{x}^{\varepsilon}\right\vert ^{2}}{\hat{\Phi}+M} \mathrm{d}x,
\end{align}
provided $ C \delta \leq \frac{1}{32} $. Recalling \eqref{R-1-SPLIT}, we have
\begin{align}\label{R-1-SPLIT-first}
\displaystyle \int _{\mathcal{I} }\varepsilon ^{-\frac{1}{2}}\mathcal{R}_{1}^{\varepsilon}\frac{\Phi^{\varepsilon}}{\hat{\Phi}+M} \mathrm{d}x = \int _{\mathcal{I}}\varepsilon ^{-\frac{1}{2}}\mathcal{R}_{1,1}^{\varepsilon}\frac{\Phi^{\varepsilon}}{\hat{\Phi}+M} \mathrm{d}x+ \int _{\mathcal{I}}\varepsilon ^{-\frac{1}{2}}\mathcal{R}_{1,2}^{\varepsilon}\frac{\Phi^{\varepsilon}}{\hat{\Phi}+M} \mathrm{d}x
,
\end{align}
where, due to \eqref{postive-appr} and the Cauchy-Schwarz inequality, it holds that
\begin{align}\label{R-1-1-integral-con}
\displaystyle \int _{\mathcal{I}}\varepsilon ^{-\frac{1}{2}}\mathcal{R}_{1,1}^{\varepsilon}\frac{\Phi^{\varepsilon}}{\hat{\Phi}+M} \mathrm{d}x
\leq  \varepsilon ^{-1}\|\mathcal{R}_{1,1}^{\varepsilon}\|_{L ^{2}}^{2} +C \|\Phi ^{\varepsilon}\|_{L ^{2}}^{2}.
\end{align}
For the second term on the right hand side of \eqref{R-1-SPLIT-first}, by \eqref{R-1-2-defi} and integration by parts, we first split it into three terms:
\begin{align}\label{R-1-2-esti}
&\displaystyle  \int _{\mathcal{I}}\varepsilon ^{-\frac{1}{2}}\mathcal{R}_{1,2}^{\varepsilon}\frac{\Phi^{\varepsilon}}{\hat{\Phi}+M} \mathrm{d}x
 \nonumber \\
 &~\displaystyle=-\int _{\mathcal{I}}v _{x}^{ B,\varepsilon}(\varphi _{x}^{I,1}(x,t)-\varphi _{x}^{I,1}(0,t)) \frac{\Phi^{\varepsilon}}{\hat{\Phi}+M} \mathrm{d}x+\int _{\mathcal{I}}   v _{x}^{ b,\varepsilon}(\varphi _{x}^{I,1}(x,t)-\varphi _{x}^{I,1}(1,t))\frac{\Phi^{\varepsilon}}{\hat{\Phi}+M} \mathrm{d}x
 \nonumber \\
 & \displaystyle \quad- \int _{\mathcal{I}}\varepsilon ^{\frac{1}{2}} v _{x}^{I,1}(\varphi _{x}^{ B,\varepsilon}+\varphi _{x}^{ b,\varepsilon}) \frac{\Phi^{\varepsilon}}{\hat{\Phi}+M} \mathrm{d}x=:\mathscr{D}_{1}+\mathscr{D}_{2}+\mathscr{D} _{3}.
\end{align}
Next we will estimate $ \mathscr{D}_{i}\,(i=1,2,3) $ term by term. For $ \mathscr{D}_{1} $, we get by virtue of \eqref{postive-appr}, \eqref{Sobolev-modified}, \eqref{z-transfer}, Corollaries \ref{cor-l-infty-vfi-I0}--\ref{cor-V-b-0-INFTY}, integration by parts and the Cauchy-Schwarz inequality that
\begin{align*}
\displaystyle \mathscr{D}_{1}&=\int _{\mathcal{I}} v ^{B,\varepsilon}(\varphi _{x}^{I,1}(x,t)-\varphi _{x}^{I,1}(0,t)) \frac{\Phi _{x}^{\varepsilon}}{\hat{\Phi}+M} \mathrm{d}x  + \int _{\mathcal{I}} v ^{B,\varepsilon}\varphi _{xx}^{I,1}\Phi ^{\varepsilon} \mathrm{d}x
 \nonumber \\
 & \displaystyle \quad-\int _{\mathcal{I}}v ^{B,\varepsilon}(\varphi _{x}^{I,1}(x,t)-\varphi _{x}^{I,1}(0,t)) \frac{\Phi^{\varepsilon}}{(\hat{\Phi}+M)^2}(\varphi _{xx}^{I,0}+\varepsilon ^{\frac{1}{2}}(\varphi _{xx}^{B,\varepsilon}+\varphi _{xx}^{b,\varepsilon}))  \mathrm{d}x
  \nonumber \\
  &~\displaystyle \leq \|\Phi _{x}\|_{L ^{2}} \left( \int _{\mathcal{I}} \vert v ^{B,\varepsilon}\vert ^{2}\left\vert \int _{0}^{x}\varphi _{yy}^{I,1}\mathrm{d}y \right\vert ^{2} \mathrm{d}x\right)^{\frac{1}{2}}+\|v ^{B,\varepsilon}\|_{L ^{2}} \|\Phi ^{\varepsilon}\| _{L ^{\infty}} \|\varphi _{xx}^{I,1}\|_{L ^{2}}
   \nonumber \\
   &~\displaystyle \quad +C \|\Phi ^{\varepsilon}\| _{L ^{\infty}}\left( \int _{\mathcal{I}} \vert v ^{B,\varepsilon}\vert ^{2}\left\vert \int _{0}^{x}\varphi _{yy}^{I,1}\mathrm{d}y \right\vert ^{2}  \mathrm{d}x\right)^{\frac{1}{2}} \left( \|\varphi _{xx}^{I,0}\|_{L ^{2}}+\varepsilon ^{\frac{1}{2}}(\|\varphi _{xx}^{ B,\varepsilon}\|_{L ^{2}} +\|\varphi _{xx}^{ b,\varepsilon}\|_{L ^{2}}) \right)     \nonumber \\
   &~\displaystyle \leq  \frac{1}{24}\int _{\mathcal{I} }\frac{\left\vert \Phi _{x}^{\varepsilon}\right\vert ^{2}}{\hat{\Phi}+M} \mathrm{d}x+ C\varepsilon ^{\frac{1}{2}} \|\langle z \rangle v ^{B,\varepsilon}\|_{L _{z}^{2}}^{2}\|\varphi _{xx}^{I,1}\|_{L ^{2}}^{2}(1+\|\varphi _{xx}^{I,0}\|_{L ^{2}}^{2}+\|\varphi _{zz}^{ B,\varepsilon}\|_{L _{z}^{2}}^{2}+\|\varphi _{\xi \xi}^{ b,\varepsilon}\|_{L _{\xi}^{2}}^{2} )
     \nonumber \\
     &~\displaystyle \leq  \frac{1}{24}\int _{\mathcal{I} }\frac{\left\vert \Phi _{x}^{\varepsilon}\right\vert ^{2}}{\hat{\Phi}+M} \mathrm{d}x+ C \varepsilon ^{\frac{1}{2}}\|\varphi _{xx}^{I,1}\|_{L ^{2}}^{2}.
\end{align*}
In the same method, we have
\begin{align*}
 \displaystyle \mathscr{D}_{2} \leq\frac{1}{24}\int _{\mathcal{I} }\frac{\left\vert \Phi _{x}^{\varepsilon}\right\vert ^{2}}{\hat{\Phi}+M} \mathrm{d}x+ C \varepsilon ^{\frac{1}{2}}\|\varphi _{xx}^{I,1}\|_{L ^{2}}^{2}.
 \end{align*}
 For $ \mathscr{D}_{3} $, it follows from \eqref{L-INFT-Vfi-B-0}, \eqref{v-I-1-infty}, \eqref{postive-appr}, \eqref{Sobolev-z-xi} and the Cauchy-Schwarz inequality that
 \begin{align*}
 \displaystyle \mathscr{D}_{3}  \leq  C \|v _{x}^{I,1}\|_{L ^{2}}\left( \|\varphi _{z}^{ B,\varepsilon}\|_{L ^{2}}+\|\varphi _{\xi}^{ b,\varepsilon}\|_{L ^{2}} \right)\|\Phi ^{\varepsilon}\|_{L ^{\infty}} \leq \frac{1}{24}\int _{\mathcal{I} }\frac{\left\vert \Phi _{x}^{\varepsilon}\right\vert ^{2}}{\hat{\Phi}+M} \mathrm{d}x +C \varepsilon ^{\frac{1}{2}}.
 \end{align*}
 Therefore we have from \eqref{R-1-2-esti} that
      \begin{align*}
      \displaystyle   \int _{\mathcal{I}}\varepsilon ^{- \frac{1}{2}}\mathcal{R}_{1,2}^{\varepsilon}\frac{\Phi^{\varepsilon}}{\hat{\Phi}+M} \mathrm{d}x \leq  \frac{1}{8}\int _{\mathcal{I} }\frac{\left\vert \Phi _{x}^{\varepsilon}\right\vert ^{2}}{\hat{\Phi}+M} \mathrm{d}x+
       C \varepsilon ^{\frac{1}{2}}\left( 1+\|\varphi _{xx}^{I,1}\|_{L ^{2}}^{2} \right).
      \end{align*}
      This along with \eqref{R-1-1-integral-con} gives
      \begin{align}\label{third-term}
      \displaystyle  \int _{\mathcal{I} }\varepsilon ^{- \frac{1}{2}}\mathcal{R}_{1}^{\varepsilon}\frac{\Phi^{\varepsilon}}{\hat{\Phi}+M} \mathrm{d}x &\leq  \frac{1}{8}\int _{\mathcal{I} }\frac{\left\vert \Phi _{x}^{\varepsilon}\right\vert ^{2}}{\hat{\Phi}+M} \mathrm{d}x+
       C \varepsilon ^{\frac{1}{2}}\left(1+ \|\varphi _{xx}^{I,1}\|_{L ^{2}}^{2} \right) +C\varepsilon ^{-1}\|\mathcal{R}_{1,1}^{\varepsilon}\|_{L ^{2}}^{2} +C \|\Phi ^{\varepsilon}\|_{L ^{2}}^{2}.
      \end{align}
Recalling \eqref{fida-hat-err}, we have
\begin{align}\label{L-1-L2-esti}
\displaystyle  - \int _{\mathcal{I} }V _{x}^{\varepsilon}\Phi ^{\varepsilon} \frac{\Phi _{x}^{a}+M}{\hat{\Phi} +M}\mathrm{d}x& =- \int _{\mathcal{I}} V _{x}^{\varepsilon}\Phi ^{\varepsilon}   \mathrm{d}x- \int _{\mathcal{I}}V _{x}^{\varepsilon}\Phi ^{\varepsilon} \frac{\hat{\Phi}^{\text{err}}}{\hat{\Phi} +M}  \mathrm{d}x =: \mathcal{L}_{1}+\mathcal{L}_{2},
\end{align}
where, thanks to \eqref{postive-appr}, integration by parts and the Cauchy-Schwarz inequality, it holds that
\begin{align}
\displaystyle \mathcal{L}_{1} = \int _{\mathcal{I}} V ^{\varepsilon}\Phi _{x}^{\varepsilon} \mathrm{d}x \leq \frac{1}{16}\int _{\mathcal{I}} \frac{\left\vert \Phi _{x}^{\varepsilon}\right\vert ^{2}}{\hat{\Phi}+M} \mathrm{d}x +C \|V ^{\varepsilon}\|_{L ^{2}}^{2}.\nonumber
\end{align}
For $ \mathcal{L}_{2} $, we get by \eqref{postive-appr}, \eqref{apriori-assumption}, \eqref{z-transfer} and the Cauchy-Schwarz inequality that
\begin{align}
 \displaystyle  \mathcal{L}_{2} &\leq C\|V _{x}^{\varepsilon}\|_{L ^{2}}\|\Phi ^{\varepsilon}\|_{L ^{2}}\|\hat{\Phi}^{\text{err}}\|_{L ^{\infty}}
  \leq \frac{\varepsilon}{32}\|V _{x}^{\varepsilon}\|_{L ^{2}}^{2} +C \varepsilon ^{-1}\|\hat{\Phi}^{\text{err}}\|_{L ^{\infty}}^{2}\|\Phi ^{\varepsilon}\|_{L ^{2}}^{2}.\nonumber
 \end{align}
 Therefore we have from \eqref{L-1-L2-esti} that
 \begin{align}\label{L2-Fourth-esti}
 \displaystyle   - \int _{\mathcal{I} }V _{x}^{\varepsilon}\Phi ^{\varepsilon} \frac{\Phi _{x}^{a}+M}{\hat{\Phi} +M}\mathrm{d}x & \leq \frac{1}{16}\int _{\mathcal{I}} \frac{\left\vert \Phi _{x}^{\varepsilon}\right\vert ^{2}}{\hat{\Phi}+M} \mathrm{d}x+ \frac{\varepsilon}{32}\|V _{x}^{\varepsilon}\|_{L ^{2}}^{2} + C \varepsilon ^{-1}\|\hat{\Phi}^{\text{err}}\|_{L ^{\infty}}^{2}\|\Phi ^{\varepsilon}\|_{L ^{2}}^{2}+C \|V ^{\varepsilon}\|_{L ^{2}}^{2}.
 \end{align}
For the last term in \eqref{L-2-diff-id}, from the Cauchy-Schwarz inequality, we get
\begin{align}\label{five-l2}
&\displaystyle \int _{\mathcal{I} }\frac{\Phi ^{\varepsilon}\Phi _{x}^{\varepsilon}}{(\hat{\Phi}+M)^{2}}\left( \hat{\Phi}_{x}-(\hat{\Phi}+M)V _{x}^{a} \right)  \mathrm{d}x
  \leq C \| \hat{\Phi}_{x}-(\hat{\Phi}+M)V _{x}^{a}\|_{L ^{\infty}} \int _{\mathcal{I} } \frac{\Phi ^{\varepsilon}\Phi _{x}^{\varepsilon}}{(\hat{\Phi}+M)^{2}}\mathrm{d}x
  \nonumber \\
   &\displaystyle ~ \leq \frac{1}{4}\int _{\mathcal{I}} \frac{\left\vert \Phi _{x}^{\varepsilon}\right\vert ^{2}}{\hat{\Phi}+M} \mathrm{d}x+ C \| \hat{\Phi}_{x}-(\hat{\Phi}+M)V _{x}^{a}\|_{L ^{\infty}} ^{2}\int _{\mathcal{I} }\frac{\left\vert \Phi ^{\varepsilon}\right\vert ^{2}}{\hat{\Phi}+M}\mathrm{d}x.
\end{align}
Thus, plugging \eqref{fir-l2}, \eqref{firs-l2-nonlinear}, \eqref{third-term}, \eqref{L2-Fourth-esti} and \eqref{five-l2} into \eqref{L-2-diff-id}, it follows that
\begin{align}\label{fida-l2-final-esti}
&\displaystyle  \frac{1}{2}\frac{\mathrm{d}}{\mathrm{d}t}\int _{\mathcal{I} }\frac{\left\vert \Phi ^{\varepsilon}\right\vert ^{2}}{\hat{\Phi}+M} \mathrm{d}x + \frac{1}{2}\int _{\mathcal{I} }\frac{\left\vert \Phi _{x}^{\varepsilon}\right\vert ^{2}}{\hat{\Phi}+M} \mathrm{d}x
 \nonumber \\
 &~\displaystyle\leq \frac{\varepsilon}{32}\|V _{x}^{\varepsilon}\|_{L ^{2}}^{2}+C\left( \|\hat{\Phi}_{t}\|_{L ^{\infty}}+\| \hat{\Phi}_{x}-(\hat{\Phi}+M)V _{x}^{a}\|_{L ^{\infty}} ^{2}+ \varepsilon ^{-1}\|\hat{\Phi}^{\text{err}}\|_{L ^{\infty}}^{2}+1 \right) \int _{\mathcal{I} }\frac{\left\vert \Phi ^{\varepsilon}\right\vert ^{2}}{\hat{\Phi}+M}\mathrm{d}x
  \nonumber \\
  &~\displaystyle \quad+ C \|V ^{\varepsilon}\|_{L ^{2}}^{2}+ C\varepsilon ^{-1}\|\mathcal{R}_{1,1}^{\varepsilon}\|_{L ^{2}}^{2} + C \varepsilon ^{\frac{1}{2}}\left(1+ \|\varphi _{xx}^{I,1}\|_{L ^{2}}^{2} \right),
\end{align}
provided $ C \delta \leq \frac{1}{32} $ in \eqref{firs-l2-nonlinear}, where we have used \eqref{postive-appr}.

To proceed, multiplying the second equation in \eqref{eq-perturbation} by $ V ^{\varepsilon} $ and integrating the resulting equation over $ \mathcal{I} $, we get
\begin{align}\label{V-L2-differ}
&\displaystyle \frac{1}{2}\frac{\mathrm{d}}{\mathrm{d}t}\int _{\mathcal{I} }\left\vert V ^{\varepsilon}\right\vert ^{2}\mathrm{d}x+ \varepsilon\int _{\mathcal{I} }\left\vert V _{x}^{\varepsilon}\right\vert ^{2} \mathrm{d}x
 \nonumber \\
  &~\displaystyle =- \int _{\mathcal{I} }\varepsilon ^{\frac{1}{2}}\Phi _{x}^{\varepsilon} \vert V ^{\varepsilon}\vert ^{2} \mathrm{d}x- \int _{\mathcal{I} }\Phi _{x}^{\varepsilon}V^{a} V ^{\varepsilon} \mathrm{d}x + \int _{\mathcal{I} }V ^{\varepsilon}\varepsilon ^{- \frac{1}{2}}\mathcal{R}_{2}^{\varepsilon} \mathrm{d}x- \int _{\mathcal{I} }(\Phi_{x}^{a}+M)\left\vert V ^{\varepsilon}\right\vert ^{2} \mathrm{d}x = \sum _{i=1}^{4}\mathcal{N}_{i},
\end{align}
where, due to integration by parts and the Cauchy-Schwarz inequality,
\begin{align}\label{N-1}
\displaystyle \mathcal{N}_{1} &= 2 \varepsilon ^{\frac{1}{2}}\int _{\mathcal{I} }\Phi ^{\varepsilon}V _{x}^{\varepsilon}V ^{\varepsilon} \mathrm{d}x \leq \frac{\varepsilon}{8}\int _{\mathcal{I} }\left\vert V _{x}^{\varepsilon}\right\vert ^{2}\mathrm{d}x+C \|\Phi ^{\varepsilon}\|_{L ^{\infty}}^{2}\int _{\mathcal{I} }\left\vert V ^{\varepsilon}\right\vert ^{2} \mathrm{d}x.
\end{align}
For $ \mathcal{N}_{2} $, we first notice from \eqref{b-v--h1-esti} and Corollaries \ref{cor-l-infty-vfi-I0}--\ref{cor-INF-vb1-vfib2} that
\begin{align}\label{v-a-bd}
\displaystyle \|V ^{a}\|_{L _{T}^{\infty}L ^{\infty}} \leq C
\end{align}
for some constant $ C>0 $ independent of $ \varepsilon $. This along with \eqref{postive-appr} and the Cauchy-Schwarz inequality yields that
\begin{align*}
\displaystyle \mathcal{N}_{2}& \leq C \|V ^{a}\|_{L ^{\infty}}\|\Phi _{x}^{\varepsilon}\|_{L ^{2}}\|V ^{\varepsilon}\|_{L ^{2}} \leq \frac{1}{8}\int _{\mathcal{I} }\frac{\left\vert \Phi _{x}^{\varepsilon}\right\vert ^{2}}{\hat{\Phi}+M} \mathrm{d}x +C\|V ^{\varepsilon}\|_{L ^{2}}  ^{2}.
\end{align*}
For $ \mathcal{N}_{3} $, by the Cauchy-Schwarz inequality, one has $ \mathcal{N}_{3} \leq \|V ^{\varepsilon}\|_{L ^{2}} ^{2}+C \varepsilon ^{-1}\|\mathcal{R}_{2}^{\varepsilon}\| _{L ^{2}} ^{2}$. With \eqref{fida-hat-err}, $ \mathcal{N}_{4} $ can be estimated as follows:
\begin{align}
\mathcal{N}_{4}&= - \int _{\mathcal{I} }(\hat{\Phi}+M)\vert V ^{\varepsilon}\vert ^{2} \mathrm{d}x- \int _{\mathcal{I}} \hat{\Phi}^{\text{err}}\vert V ^{\varepsilon}\vert ^{2}  \mathrm{d}x  \leq  - \int _{\mathcal{I} }(\hat{\Phi}+M)\vert V ^{\varepsilon}\vert ^{2} \mathrm{d}x+C \|\hat{\Phi}^{\text{err}}\|_{L ^{\infty}}\|V ^{\varepsilon}\|_{L ^{2}}^{2}.\nonumber
\end{align}
Inserting the estimates on $ \mathcal{N}_{i}~(1 \leq i \leq 4) $ into \eqref{V-L2-differ}, we get by virtue of \eqref{postive-appr} and Lemma \ref{lem-R-2-ESTI} that
\begin{align}\label{V-l2-ineq}
\displaystyle  &\displaystyle \frac{1}{2}\frac{\mathrm{d}}{\mathrm{d}t}\int _{\mathcal{I} }\left\vert V ^{\varepsilon}\right\vert ^{2}\mathrm{d}x+ \frac{7\varepsilon}{8}\int _{\mathcal{I} }\left\vert V _{x}^{\varepsilon}\right\vert ^{2} \mathrm{d}x+ \int _{\mathcal{I} }(\hat{\Phi}+M)\left\vert V ^{\varepsilon}\right\vert ^{2} \mathrm{d}x
 \nonumber \\
 &~\displaystyle\leq \frac{1}{8}\int _{\mathcal{I} }\frac{\left\vert \Phi _{x}^{\varepsilon}\right\vert ^{2}}{\hat{\Phi}+M} \mathrm{d}x+C(\|\Phi ^{\varepsilon}\|_{L ^{\infty}}^{2}+\|\hat{\Phi}^{\text{err}}\|_{L ^{\infty}})\|V ^{\varepsilon}\|_{L ^{2}}  ^{2}
 +C \varepsilon ^{-1}\|\mathcal{R}_{2}^{\varepsilon}\|_{L ^{2}}^{2},
\end{align}
Combining \eqref{V-l2-ineq} with \eqref{fida-l2-final-esti}, in view of \eqref{postive-appr}, we obtain that
\begin{align}
&\displaystyle \frac{\mathrm{d}}{\mathrm{d}t}\int _{\mathcal{I} } \left( \vert V ^{\varepsilon}\vert ^{2} +\frac{\left\vert \Phi ^{\varepsilon}\right\vert ^{2}}{\hat{\Phi}+M}\right) \mathrm{d}x+ \int _{\mathcal{I} }\left( \frac{\left\vert \Phi  _{x}^{\varepsilon}\right\vert ^{2}}{\hat{\Phi}+M}+ \varepsilon\vert V _{x}^{\varepsilon}\vert ^{2}+(\hat{\Phi}+M)\left\vert V ^{\varepsilon}\right\vert ^{2}\right)  \mathrm{d}x
 \nonumber \\
 &~\displaystyle\leq  C \varepsilon ^{\frac{1}{2}}\left(1+ \|\varphi _{xx}^{I,1}\|_{L ^{2}}^{2}\right)+C \left( 1+\|\Phi ^{\varepsilon}\|_{L ^{\infty}}^{2}+\|\hat{\Phi}^{\text{err}}\|_{L ^{\infty}}\right) \|V ^{\varepsilon}\|_{L ^{2}}  ^{2} + C\varepsilon ^{-1}(\|\mathcal{R}_{1,1}^{\varepsilon}\|_{L ^{2}}^{2} +|\mathcal{R}_{2}^{\varepsilon}\|_{L ^{2}}^{2})
  \nonumber \\
  &~\displaystyle \quad+C\left(  \|\hat{\Phi}_{t}\|_{L ^{\infty}}+\| \hat{\Phi}_{x}-(\hat{\Phi}+M)V _{x}^{a}\|_{L ^{\infty}} ^{2}+ \varepsilon ^{-1}\|\hat{\Phi}^{\text{err}}\|_{L ^{\infty}}^{2}+1  \right) \int _{\mathcal{I} }\frac{\left\vert \Phi ^{\varepsilon}\right\vert ^{2}}{\hat{\Phi}+M}\mathrm{d}x ,\nonumber
\end{align}
which along with \eqref{vfi-I-1-esti-first-con-lem}, \eqref{tild-v-B-1-lem}, \eqref{tild-v-b-1-in-lem}, \eqref{b-vfi-x-l2}, \eqref{b-vfi-L-2-l-infty}, \eqref{hat-fida-esti-con}, \eqref{hat-fida-err-esti1}, \eqref{Sobolev-infty}, \eqref{Sobolev-z-xi}, Lemmas \ref{lem-R-1-VE}, \ref{lem-R-2-ESTI} and the Gronwall inequality entails for any $ t \in [0,T] $ that
\begin{align}
&\displaystyle  \|V ^{\varepsilon}(\cdot,t)\|_{L ^{2}}^{2}+ \|\Phi ^{\varepsilon}(\cdot,t)\|_{L ^{2}}^{2}+ \int _{0}^{t}\left( \|V ^{\varepsilon}\|_{L ^{2}}^{2}+\varepsilon \|V _{x}^{\varepsilon}\| _{L ^{2}}^{2}+\|\Phi _{x}^{\varepsilon}\|_{L ^{2}}^{2} \right) \mathrm{d}\tau
 \nonumber \\
 &~\displaystyle \leq  C \varepsilon ^{\frac{1}{2}}\left( 1+\|\varphi _{xx}^{I,1}\|_{L _{T}^{2}L ^{2}}^{2}\right)+C\varepsilon ^{-1}(\|\mathcal{R}_{1,1}^{\varepsilon}\|_{L _{T}^{2} L ^{2}}^{2} +|\mathcal{R}_{2}^{\varepsilon}\|_{L _{T}^{2}L ^{2}}^{2}  ) \leq C \varepsilon ^{2 \iota _{0}-1},\nonumber
\end{align}
where we have used $ \iota _{1} \geq \iota _{0}$ and $ \frac{1}{2}< \iota _{0} <\frac{7}{12} $. The proof of Lemma \ref{lem-L2-pertur} is complete.
\end{proof}

% subsubsection_name (end)
We proceed to establish $ H ^{1} $ estimate for $ (\Phi ^{\varepsilon}, V ^{\varepsilon}) $.
% subsubsection lower_order_estimates (end)
\begin{lemma}\label{lem-V-x-per}
Under the conditions of Lemma \ref{lem-L2-pertur}, it holds for any $ t \in [0,T] $ that
\begin{gather}\label{con-fida-v-x-lem}
\displaystyle \varepsilon\|V _{x}^{\varepsilon}(\cdot,t)\|_{L ^{2}}^{2}+\int _{0}^{t}\|V _{\tau}^{\varepsilon}\| _{ L ^{2}}^{2} \mathrm{d}\tau \leq C \varepsilon ^{2\iota _{0}-1},\ \ \
\|\Phi _{x}^{\varepsilon}(\cdot,t)\|_{L ^{2}}^{2}+\int _{0}^{t} \|\Phi _{\tau}^{\varepsilon}\| _{ L ^{2}}^{2} \mathrm{d}\tau \leq C \varepsilon ^{4 \iota _{0}-\frac{5}{2}},
\end{gather}
provided $ \delta $ is suitably small, where the constant $ C>0 $ is independent of $ \varepsilon $.

\end{lemma}
\begin{proof}
  Multiplying the second equation in \eqref{eq-perturbation} by $ V _{t}^{\varepsilon} $ and integrating the resulting equation over $ \mathcal{I} $, we have
\begin{align}\label{V-x-diff-iden}
&\displaystyle \frac{1}{2} \frac{\mathrm{d}}{\mathrm{d}t}\int _{\mathcal{I}} \varepsilon\left\vert V _{x}^{\varepsilon}\right\vert ^{2} \mathrm{d}x+\int _{\mathcal{I} }\left\vert V _{t} ^{\varepsilon}\right\vert ^{2} \mathrm{d}x
 \nonumber \\
 &~\displaystyle = -\int _{\mathcal{I}}(\Phi _{x}^{a}+M) V ^{\varepsilon}V _{t}^{\varepsilon}\mathrm{d}x - \int _{\mathcal{I} }\Phi _{x}^{\varepsilon}V ^{a} V _{t}^{\varepsilon} \mathrm{d}x-\varepsilon ^{\frac{1}{2}} \int _{\mathcal{I}}\Phi _{x}^{\varepsilon}V ^{\varepsilon}V_{t}^{\varepsilon}\mathrm{d}x
 + \varepsilon ^{-\frac{1}{2}}\int _{\mathcal{I}}\mathcal{R}_{2}^{\varepsilon} V_{t}^{\varepsilon}\mathrm{d}x,
\end{align}
where, due to \eqref{postive-appr}, \eqref{fida-hat-err} and the Cauchy-Schwarz inequality, we get
\begin{align}\label{v-x-sti-one}
&\displaystyle  -\int _{\mathcal{I}}(\Phi _{x}^{a}+M) V ^{\varepsilon}V _{t}^{\varepsilon}\mathrm{d}x \leq \frac{1}{4}\int _{\mathcal{I}} \left\vert V _{t}^{\varepsilon}\right\vert ^{2} \mathrm{d}x
+C \Big( 1+\|\hat{\Phi}^{\text{err}}\|_{L ^{\infty}}^{2}\Big) \|V ^{\varepsilon}\|_{L ^{2}}  ^{2},
\end{align}
where $ \hat{\Phi}^{\text{err}} $ is as in \eqref{fida-hat-err}. By \eqref{v-a-bd} and the Cauchy-Schwarz inequality, we get
\begin{align}%\label{v-x-esti-first}
\displaystyle   \int _{\mathcal{I} }\Phi _{x}^{\varepsilon}V ^{a} V _{t}^{\varepsilon} \mathrm{d}x \leq \|V ^{a}\|_{L ^{\infty}}\|\Phi _{x}^{\varepsilon}\|_{L ^{2}}\|V _{t}^{\varepsilon}\|_{L ^{2}} \leq \frac{1}{4}\|V _{t} ^{\varepsilon}\|_{L ^{2}}^{2}+C \|\Phi _{x}^{\varepsilon}\|_{L ^{2}}^{2}.
\end{align}
Thanks to \eqref{con-L2-PERT}, \eqref{Sobolev-modified}, the Cauchy-Schwarz inequality and Lemma \ref{lem-R-2-ESTI}, we deduce that
\begin{align}
\displaystyle -\varepsilon ^{\frac{1}{2}} \int _{\mathcal{I}}\Phi _{x}^{\varepsilon}V ^{\varepsilon}V_{t}^{\varepsilon}\mathrm{d}x &\leq \frac{1}{8}\int _{\mathcal{I}} \left\vert V _{t}^{\varepsilon}\right\vert ^{2} \mathrm{d}x+C \varepsilon \|V ^{\varepsilon}\|_{L ^{\infty}}^{2}\int _{\mathcal{I}} \left\vert \Phi _{x}^{\varepsilon}\right\vert ^{2}\mathrm{d}x
 \nonumber \\
 & \displaystyle \leq \frac{1}{8}\int _{\mathcal{I}} \left\vert V _{t}^{\varepsilon}\right\vert ^{2} \mathrm{d}x+C \varepsilon \|V ^{\varepsilon}\|_{L ^{2}}\|V _{x}^{\varepsilon}\|_{L ^{2}}\int _{\mathcal{I}} \left\vert \Phi _{x}^{\varepsilon}\right\vert ^{2}\mathrm{d}x
  \nonumber \\
   &\displaystyle\leq \frac{1}{8}\int _{\mathcal{I}} \left\vert V _{t}^{\varepsilon}\right\vert ^{2} \mathrm{d}x+C \varepsilon ^{\iota _{0}+ \frac{1}{2}} \|V _{x}^{\varepsilon}\|_{L ^{2}}  \int _{\mathcal{I}} \left\vert \Phi _{x}^{\varepsilon}\right\vert ^{2}\mathrm{d}x
    \nonumber \\
    & \displaystyle \leq \frac{1}{8}\int _{\mathcal{I}} \left\vert V _{t}^{\varepsilon}\right\vert ^{2} \mathrm{d}x+C \varepsilon \|V _{x}^{\varepsilon}\| _{L ^{2}}^{2}\|\Phi _{x}^{\varepsilon}\|_{L ^{2}}^{2}+C \varepsilon ^{2 \iota _{0}}\|\Phi _{x}^{\varepsilon}\|_{L ^{2}}^{2},
\end{align}
and that
\begin{gather}\label{v-x-esti-last}
\displaystyle  \varepsilon ^{- \frac{1}{2}}\int _{\mathcal{I}}\mathcal{R}_{2}^{\varepsilon} V_{t}^{\varepsilon}\mathrm{d}x \leq \frac{1}{8}\int _{\mathcal{I}} \left\vert V _{t}^{\varepsilon}\right\vert ^{2} \mathrm{d}x+C \varepsilon ^{-1}\|\mathcal{R}_{2}^{\varepsilon}\|_{L ^{2}}^{2} .
\end{gather}
With \eqref{v-x-sti-one}--\eqref{v-x-esti-last} and the fact $ 0< \varepsilon<1 $, we thus update \eqref{V-x-diff-iden} as
\begin{align*}
\displaystyle  \frac{\mathrm{d}}{\mathrm{d}t}\int _{\mathcal{I}}\varepsilon\left\vert V _{x}^{\varepsilon}\right\vert ^{2} \mathrm{d}x+ \|V _{t} ^{\varepsilon}\|_{L ^{2}}^{2} & \leq C  \|\Phi _{x}^{\varepsilon}\| _{L ^{2}}^{2}(1+\varepsilon \|V _{x}^{\varepsilon}\| _{L ^{2}}^{2})+C \varepsilon ^{-1}\|\mathcal{R}_{2}^{\varepsilon}\|_{L ^{2}}^{2}+ C \Big( 1+\|\hat{\Phi}^{\text{err}}\|_{L ^{\infty}}^{2}\Big) \|V ^{\varepsilon}\|_{L ^{2}}  ^{2}  ,
\end{align*}
which followed by an integration over $ [0,t] $ for any $ t \in (0,T] $ gives
\begin{align}\label{esti-v-x-final}
\displaystyle  \varepsilon\|V _{x} ^{\varepsilon}(\cdot,t)\|_{L ^{2}}^{2}+ \int _{0}^{t}\|V _{\tau} ^{\varepsilon}\|_{L ^{2}}^{2}\mathrm{d}\tau \leq C\int _{0} ^{t}\|\Phi _{x}^{\varepsilon}\| _{L ^{2}}^{2}(\varepsilon \|V _{x}^{\varepsilon}\| _{L ^{2}}^{2}) \mathrm{d}\tau+C \varepsilon ^{ 2\iota _{0}-1}+C \varepsilon ^{2\iota _{1}-1},
\end{align}
where we have used \eqref{vfi-I-1-esti-first-con-lem}, \eqref{tild-v-B-1-lem}, \eqref{tild-v-b-1-in-lem}, \eqref{b-vfi-L-2-l-infty}, \eqref{con-R-2-ESTI-in-lem}, \eqref{postive-appr} and Lemma \ref{lem-L2-pertur}. Applying the Gronwall inequality to \eqref{esti-v-x-final}, we ultimately obtain that
\begin{align}\label{con-V-x-L2}
\displaystyle   \varepsilon\|V _{x} ^{\varepsilon}(\cdot,t)\|_{L ^{2}}^{2}+\|V ^{\varepsilon}(\cdot,t)\|_{L ^{2}}^{2}+ \int _{0}^{t}\|V _{\tau} ^{\varepsilon}\|_{L ^{2}}^{2}\mathrm{d}\tau \leq C \varepsilon ^{2\iota _{0}-1}+C \varepsilon ^{2\iota _{1}-1} \leq C \varepsilon ^{2 \iota _{0}-1},
\end{align}
where we have used the facts $ \iota _{1}\geq \iota _{0} $ and $ 0< \varepsilon<1 $. Now let us turn to estimates on $ \Phi _{x}^{\varepsilon} $. Taking the $ L ^{2} $ inner product of the first equation in \eqref{eq-perturbation} with $ \frac{\Phi _{t}^{\varepsilon}}{\hat{\Phi}+M} $ followed by an integration by parts, we have
\begin{align}\label{fida-x-diff-id}
&\displaystyle \frac{1}{2} \frac{\mathrm{d}}{\mathrm{d}t}\int _{\mathcal{I}}\frac{\left\vert \Phi _{x}^{\varepsilon}\right\vert ^{2}}{\hat{\Phi}+M}\mathrm{d}x+ \int _{\mathcal{I}}\frac{\left\vert \Phi _{t}^{\varepsilon}\right\vert ^{2}}{\hat{\Phi}+M} \mathrm{d}x
 \nonumber \\
&~\displaystyle = - \frac{1}{2}\int _{\mathcal{I}}\frac{\left\vert \Phi _{x}^{\varepsilon}\right\vert ^{2}}{(\hat{\Phi}+M)^{2}}\hat{\Phi}_{t}\mathrm{d}x +\int _{\mathcal{I}} \frac{\Phi _{t}^{\varepsilon}\Phi _{x}^{\varepsilon}}{\hat{\Phi}+M}\left( \hat{\Phi}_{x}-(\hat{\Phi}+M)V _{x}^{a} \right)  \mathrm{d}x +\int _{\mathcal{I}}   \frac{\Phi _{t}^{\varepsilon}}{\hat{\Phi}+M}\varepsilon ^{\frac{1}{2}}\Phi _{x}^{\varepsilon}V _{x}^{\varepsilon}\mathrm{d}x
 \nonumber \\
 &\displaystyle ~\quad- \int _{\mathcal{I}} V_{x}^{\varepsilon}\Phi _{t}^{\varepsilon}\frac{\Phi _{x}^{a}+M}{\hat{\Phi}+M} \mathrm{d}x+\int _{\mathcal{I}}  \frac{\Phi _{t}^{\varepsilon}}{\hat{\Phi}_{x}+M}\varepsilon ^{- \frac{1}{2}}\mathcal{R}_{1}^{\varepsilon} \mathrm{d}x=:\sum _{i=1}^{5}\mathcal{Q}_{i}.
\end{align}
We next estimate $ \mathcal{Q}_{i}~(1 \leq i \leq 5) $. Thanks to \eqref{postive-appr}, we get $ \mathcal{Q} _{1} \leq C \| \hat{\Phi}_{t}\|_{L ^{\infty}}  \|\Phi _{x}^{\varepsilon}\|_{ L ^{2}}^{2} $. From\eqref{postive-appr} and the Cauchy-Schwarz inequality, we deduce that
\begin{align*}
\displaystyle  \mathcal{Q} _{2} & \leq C \int _{\mathcal{I}} \frac{|\Phi _{t}^{\varepsilon}||\Phi _{x}^{\varepsilon}|}{\hat{\Phi}+M} \vert \hat{\Phi} _{x}- \hat{\Phi}V_{x}^{a}\vert \mathrm{d}x  \leq  C \|\hat{\Phi} _{x}- \hat{\Phi}V_{x}^{a}\| _{L ^{\infty}} \|\Phi _{t}^{\varepsilon}\|_{L ^{2}}\|\Phi _{x}^{\varepsilon}\|_{L ^{2}}
  \nonumber \\
  & \leq \frac{1}{8}\|\Phi _{t}^{\varepsilon}\|_{L ^{2}} ^{2}+ C\|\hat{\Phi} _{x}- \hat{\Phi}V_{x}^{a}\| _{L ^{\infty}} ^{2} \|\Phi _{x}^{\varepsilon}\|_{L ^{2}}^{2} .
\end{align*}
By \eqref{l-INFTY-VFI-I-0}, \eqref{L-INFT-V-B-0}, \eqref{L-INFT-V-veb-0}, \eqref{v-I-1-infty}, \eqref{tild-v-B-1-lem}, \eqref{INF-vfi-b2}, \eqref{b-v--h1-esti}, \eqref{postive-appr}, \eqref{fida-hat-err}, \eqref{con-V-x-L2}, \eqref{Sobolev-infty} and the first equation in \eqref{eq-perturbation}, we have
  \begin{align*}%\label{fida-xx-single-h1-esti}
  \displaystyle  \|\Phi _{xx}^{\varepsilon}\|_{L ^{2}}^{2}& \leq \|\Phi _{t}^{\varepsilon}\|_{L ^{2}}^{2}+\varepsilon \|\Phi _{x}^{\varepsilon}\|_{L ^\infty}^{2}\|V _{x}^{\varepsilon}\|_{L ^{2}}^{2}+\|v _{x}^{I,0}\|_{L ^{2}}^{2}\|\Phi _{x}^{\varepsilon}\|_{L ^{2}}^{2}+\|V _{x}^{a}-v _{x}^{I,0}\|_{L ^{2}}^{2}\|\Phi _{x}^{\varepsilon}\|_{L ^{\infty}}^{2}
   \nonumber \\
   & \displaystyle \quad+C \|V _{x}^{\varepsilon}\|_{L ^{2}}^{2} ( 1+\|\hat{\Phi}^{\text{err}}\|_{L ^{\infty}}^{2}) +\varepsilon ^{-1}\|\mathcal{R}_{1}^{\varepsilon}\|_{L ^{2}}^{2}
    \nonumber \\
       &\displaystyle \leq \|\Phi _{t}^{\varepsilon}\|_{L ^{2}}^{2}+C \|\Phi _{x}^{\varepsilon}\|_{L ^{2}} \|\Phi _{xx}^{\varepsilon}\|_{L ^{2}}(\varepsilon \|V _{x}^{\varepsilon}\|_{L ^{2}}^{2})+C \|\Phi _{x}^{\varepsilon}\|_{L ^{2}}^{2}+\varepsilon ^{-1}\|\mathcal{R}_{1}^{\varepsilon}\|_{L ^{2}}^{2} +\frac{1}{4}\|\Phi _{xx}^{\varepsilon}\|_{L ^{2}}^{2}
        \nonumber \\
        &~ \displaystyle \quad+C \|\Phi _{x}^{\varepsilon}\|_{L ^{2}}^{2}\left( \varepsilon ^{- \frac{1}{2}}\|v _{z}^{ B,\varepsilon}\|_{L _{z}^{2}}^{2}+\varepsilon ^{- \frac{1}{2}}\|v _{\xi}^{ b,\varepsilon}\|_{L _{\xi}^{2}}^{2}+\varepsilon ^{\frac{1}{2}}(\|v _{z}^{B,1}\|_{L _{z}^{2}}^{2}+\|v _{\xi}^{b,1}\|_{L _{\xi}^{2}}^{2}) +\|\partial _{x}b _{v}^{\varepsilon}\|_{L ^{2}}^{2}\right)^{2}
         \nonumber \\
          &~\displaystyle \quad+C \|V _{x}^{\varepsilon}\|_{L ^{2}}^{2} ( 1+\|\hat{\Phi}^{\text{err}}\|_{L ^{\infty}}^{2})+\varepsilon ^{-1}\|\mathcal{R}_{1}^{\varepsilon}\|_{L ^{2}}^{2}
         \nonumber \\
         &~\displaystyle \leq \frac{1}{2}\|\Phi _{xx}^{\varepsilon}\|_{L ^{2}}^{2}+C \|\Phi _{t}^{\varepsilon}\|_{L ^{2}}^{2} + C \varepsilon ^{-1}(\|\Phi _{x}^{\varepsilon}\|_{L ^{2}}^{2}+\|\mathcal{R}_{1}^{\varepsilon}\|_{L ^{2}}^{2})+C \|V _{x}^{\varepsilon}\|_{L ^{2}}^{2} ( 1+\|\hat{\Phi}^{\text{err}}\|_{L ^{\infty}}^{2})
  \end{align*}
  for $ 0< \varepsilon<1 $. Therefore we have
   \begin{align}\label{fida-xx-esti-ineq-fial}
   \displaystyle  \|\Phi _{xx}^{\varepsilon}\|_{L ^{2}} &\leq C \|\Phi _{t}^{\varepsilon}\|_{L ^{2}}+C\varepsilon ^{- \frac{1}{2}}(\|\mathcal{R}_{1}^{\varepsilon}\|_{L ^{2}}+\|\Phi _{x}^{\varepsilon}\|_{L ^{2}}) +C \|V _{x}^{\varepsilon}\|_{L ^{2}}( 1+\|\hat{\Phi}^{\text{err}}\|_{L ^{\infty}}).
   \end{align}
   With \eqref{postive-appr}, \eqref{apriori-assumption}, \eqref{con-V-x-L2}, \eqref{fida-xx-esti-ineq-fial}, \eqref{Sobolev-infty}, the condition $ 0<\varepsilon<1 $ and the Cauchy-Schwarz inequality, one can estimate $ \mathcal{Q}_{3} $ as follows:
   \begin{align}\label{Q-3-ESTI}
   \displaystyle  \mathcal{Q}_{3} &\leq C \|\Phi _{t}^{\varepsilon}\|_{L ^{2}} (\varepsilon ^{\frac{1}{2}}\|V _{x}^{\varepsilon}\|_{L ^{2}})\|\Phi _{x}^{\varepsilon}\|_{L ^{\infty}} \leq  C \|\Phi _{t}^{\varepsilon}\|_{L ^{2}} (\varepsilon ^{\frac{1}{2}} \|V _{x}^{\varepsilon}\|_{L ^{2}})\left( \|\Phi _{xx}^{\varepsilon}\|_{L ^{2}}^{\frac{1}{2}}\|\Phi _{x}^{\varepsilon}\|_{L ^{2}}^{\frac{1}{2}} +\|\Phi _{x}^{\varepsilon}\|_{L ^{2}}\right)
    \nonumber \\
    & \displaystyle\leq C \|\Phi _{t}^{\varepsilon}\|_{L ^{2}}(\varepsilon ^{\frac{1}{2}} \|V _{x}^{\varepsilon}\|_{L ^{2}})\Big( \|\Phi _{t}^{\varepsilon}\|_{L ^{2}}+\varepsilon ^{- \frac{1}{4}}\|\Phi _{x}^{\varepsilon}\|_{L ^{2}}+\varepsilon ^{- \frac{1}{4}}\|\mathcal{R}_{1}^{\varepsilon}\|_{L ^{2}} \Big)+\|\Phi _{t}^{\varepsilon}\|_{L ^{2}}(\varepsilon ^{\frac{1}{2}} \|V _{x}^{\varepsilon}\|_{L ^{2}}^{\frac{3}{2}})\|\Phi _{x}^{\varepsilon}\|_{L ^{2}}^{\frac{1}{2}}
    \nonumber \\
    &\displaystyle \quad+\|\Phi _{t}^{\varepsilon}\|_{L ^{2}}(\varepsilon ^{\frac{1}{2}} \|V _{x}^{\varepsilon}\|_{L ^{2}}^{\frac{3}{2}})\|\hat{\Phi}^{\text{err}}\|_{L ^{\infty}}^{\frac{1}{2}}\|\Phi _{x}^{\varepsilon}\|_{L ^{2}}^{\frac{1}{2}}
     \nonumber \\
     &\displaystyle \leq \Big( \frac{1}{16}+\varepsilon ^{\frac{1}{2}} \|V _{x}^{\varepsilon}\|_{L ^{2}}\Big) \|\Phi _{t}^{\varepsilon}\|_{L ^{2}}^{2}+ C\varepsilon \|V _{x}^{
         \varepsilon}\|_{L ^{2}}^{2}\left( \varepsilon ^{-\frac{1}{2}}\|\Phi _{x}^{\varepsilon}\|_{L ^{2}}^{2}+\varepsilon ^{- \frac{1}{2}}\|\mathcal{R}_{1}^{\varepsilon}\|_{L ^{2}}^{2} \right) +C  \varepsilon\|V _{x}^{\varepsilon}\|_{L ^{2}}^{3}\|\Phi _{x}^{\varepsilon}\|_{L ^{2}}
          \nonumber \\
          &\displaystyle \quad+C \varepsilon  \|V _{x}^{\varepsilon}\|_{L ^{2}}^{3}\|\hat{\Phi}^{\text{err}}\|_{L ^{\infty}}\|\Phi _{x}^{\varepsilon}\|_{L ^{2}}
           \nonumber \\
           & \displaystyle \leq\Big( \frac{1}{16}+ C  \delta ^{\frac{1}{2}}\Big) \|\Phi _{t}^{\varepsilon}\|_{L ^{2}}^{2}
           +C \varepsilon ^{2 \iota _{0}-\frac{3}{2}}(\|\mathcal{R}_{1}^{\varepsilon}\|_{L ^{2}}^{2}+\|\Phi _{x}^{\varepsilon}\|_{L ^{2}}^{2}) + C  \varepsilon ^{2 \iota _{0}- \frac{3}{2}}(\varepsilon ^{\frac{1}{2}}\|V _{x}^{\varepsilon}\|_{L ^{2}})\|\Phi _{x}^{\varepsilon}\|_{L ^{2}}
            \nonumber \\
            & \displaystyle \quad+C \varepsilon ^{3\iota _{0}-2}\|\hat{\Phi}^{\text{err}}\|_{L ^{\infty}}\|\Phi _{x}^{\varepsilon}\|_{L ^{2}}
             \nonumber \\
        & \displaystyle \leq \frac{1}{8} \|\Phi _{t}^{\varepsilon}\|_{L ^{2}}^{2}+ C \varepsilon ^{2\iota _{0}-\frac{3}{2}}( \varepsilon \|V _{x}^{\varepsilon}\|_{L ^{2}}^{2}+\|\Phi _{x}^{\varepsilon}\|_{L ^{2}}^{2})+C \varepsilon ^{2 \iota _{0}- \frac{3}{2}}\|\mathcal{R}_{1}^{\varepsilon}\|_{L ^{2}}^{2}+C \varepsilon ^{4 \iota _{0}-\frac{5}{2}}\|\hat{\Phi}^{\text{err}}\|_{L ^{\infty}}^{2},
   \end{align}
   provided $ C \delta ^{\frac{1}{2}} \leq \frac{1}{16} $. For $ \mathcal{Q} _{4} $, we utilize \eqref{fida-hat-err}, \eqref{con-V-x-L2}, integration by parts and the Cauchy-Schwarz inequality to derive that
\begin{align*}
\displaystyle \mathcal{Q}_{4}& =- \int _{\mathcal{I}} V_{x}^{\varepsilon}\Phi _{t}^{\varepsilon} \mathrm{d}x- \int _{\mathcal{I}} V_{x}^{\varepsilon}\Phi _{t}^{\varepsilon}\frac{\hat{\Phi}^{\text{err}}}{\hat{\Phi}+M}  \mathrm{d}x
\leq  \int _{\mathcal{I}} V^{\varepsilon}\Phi _{xt} ^{\varepsilon}\mathrm{d}x +C\|V _{x}^{\varepsilon}\|_{L ^{2}}\|\Phi _{t}^{\varepsilon}\|_{L ^{2}} \|\hat{\Phi}^{\text{err}}\|_{L ^{\infty}}
  \nonumber \\
  &\displaystyle \leq\frac{\mathrm{d}}{\mathrm{d}t}\int _{\mathcal{I}} V ^{\varepsilon}\Phi _{x}^{\varepsilon} \mathrm{d}x - \int _{\mathcal{I}} V _{t}^{\varepsilon}\Phi _{x}^{\varepsilon} \mathrm{d}x + \frac{1}{8}\|\Phi _{t}^{\varepsilon}\|_{L ^{2}}^{2}+C\|V _{x}^{\varepsilon}\|_{L ^{2}}^{2}\|\hat{\Phi}^{\text{err}}\|_{L ^{\infty}}^{2}
  \nonumber \\
 & \displaystyle\leq \frac{\mathrm{d}}{\mathrm{d}t}\int _{\mathcal{I}} V ^{\varepsilon}\Phi _{x}^{\varepsilon} \mathrm{d}x + \frac{1}{8}\|\Phi _{t}^{\varepsilon}\|_{L ^{2}}^{2}+ 2( \|V_{t}^{\varepsilon}\|_{L ^{2}}^{2} + \|\Phi _{x}^{\varepsilon}\|_{L ^{2}}^{2})+C \varepsilon ^{2 \iota _{0}-2}\|\hat{\Phi}^{\text{err}}\|_{L ^{\infty}}^{2}.
\end{align*}
Finally, by the Cauchy-Schwarz equality, we have for $ \mathcal{Q}_{5} $ that
   \begin{align*}
\displaystyle \displaystyle \mathcal{Q}_{5} \leq \frac{1}{8}\int _{\mathcal{I}} \left\vert \Phi _{t}^{\varepsilon}\right\vert ^{2} \mathrm{d}x+C \varepsilon ^{-1} \|\mathcal{R}_{1}^{\varepsilon} \| _{L ^{2}}^{2}.
\end{align*}
Therefore, inserting estimates on $ \mathcal{Q} _{i}~(0 \leq i  \leq 5) $ into  \eqref{fida-x-diff-id}, we get after taking $ \varepsilon $ suitably small that
\begin{align}
&\displaystyle   \frac{1}{2} \frac{\mathrm{d}}{\mathrm{d}t}\int _{\mathcal{I}}\frac{\left\vert \Phi _{x}^{\varepsilon}\right\vert ^{2}}{\hat{\Phi}_{x}+M}\mathrm{d}x+ \frac{1}{2}\int _{\mathcal{I}}\frac{\left\vert \Phi _{t}^{\varepsilon}\right\vert ^{2}}{\hat{\Phi}_{x}+M} \mathrm{d}x
 \nonumber \\
 &~\displaystyle\leq \frac{\mathrm{d}}{\mathrm{d}t}\int _{\mathcal{I}} V ^{\varepsilon}\Phi _{x}^{\varepsilon} \mathrm{d}x+C\left( \|\hat{\Phi}_{t}\|_{L ^{\infty}}+\|\hat{\Phi} _{x}- (\hat{\Phi}+M)V_{x}^{a}\| _{L ^{\infty}} ^{2} \right) \| \Phi_{x}^{\varepsilon}\| _{L ^{2}}^{2}+C\|V _{t}^{\varepsilon}\|_{L ^{2}}^{2}
  \nonumber \\
  &~\displaystyle \quad+ C  \varepsilon ^{\frac{4 \iota _{0} -3}{2}}( \varepsilon \|V _{x}\|_{L ^{2}}^{2}+\|\Phi _{x}^{\varepsilon}\|_{L ^{2}}^{2})+C( \varepsilon ^{2 \iota _{0}- \frac{3}{2}}+\varepsilon ^{-1})\|\mathcal{R}_{1}^{\varepsilon}\|_{L ^{2}}^{2}+C \varepsilon ^{4 \iota _{0}-\frac{5}{2}}\|\hat{\Phi}^{\text{err}}\|_{L ^{\infty}}^{2}.\nonumber
\end{align}
Integrating the above inequality over $ (0,t) $ for any $ t \in [0,T] $, we derive that
\begin{align}
&\displaystyle  \|\Phi _{x}^{\varepsilon}(\cdot,t)\|_{L ^{2}}+ \int _{0}^{t}\|\Phi _{\tau}^{\varepsilon}\|_{L ^{2}}^{2}\mathrm{d}\tau
 \nonumber \\
 &\leq   \int _{\mathcal{I}} V ^{\varepsilon}\Phi _{x}^{\varepsilon} \mathrm{d}x+\int _{0}^{t}\left( \|\hat{\Phi}_{\tau}\|_{L ^{\infty}}+\|\hat{\Phi} _{x}- (\hat{\Phi}+M)V_{x}^{a}\| _{L ^{\infty}} ^{2} \right) \|\hat{\Phi}_{x}^{\varepsilon}\| _{L ^{2}}^{2}\mathrm{d}\tau+C\|V _{t}^{\varepsilon}\|_{L _{T}^{2}L ^{2}}^{2}
 \nonumber \\
 & \displaystyle  \quad+ C  \varepsilon ^{\frac{4 \iota _{0}- 3}{2}}( \varepsilon \|V _{x}\|_{L _{T}^{2}L ^{2}}^{2}+\|\Phi _{x}^{\varepsilon}\|_{L _{T}^{2}L ^{2}}^{2})+C( \varepsilon ^{2 \iota _{0}- \frac{3}{2}}+\varepsilon ^{-1})\|\mathcal{R}_{1}^{\varepsilon}\|_{L _{T}^{2}L ^{2}}^{2}+C \varepsilon ^{4 \iota _{0}-\frac{5}{2}}\|\hat{\Phi}^{\text{err}}\|_{L _{T}^{2}L ^{\infty}}^{2}
  \nonumber \\
   & \displaystyle \leq \frac{1}{2}\|\Phi _{x}^{\varepsilon}(\cdot,t)\|_{L ^{2}}^{2}+C \|V ^{\varepsilon}(\cdot,t)\|_{L ^{2}}^{2}+\int _{0}^{t}\left( \|\hat{\Phi}_{\tau}\|_{L ^{\infty}}+\|\hat{\Phi} _{x}- (\hat{\Phi}+M)V_{x}^{a}\| _{L ^{\infty}} ^{2} \right) \|\Phi_{x}^{\varepsilon}\| _{L ^{2}}^{2}\mathrm{d}\tau +C \varepsilon ^{4 \iota _{0}-\frac{5}{2}}+C
     \nonumber \\
     & \displaystyle \leq \frac{1}{2}\|\Phi _{x}^{\varepsilon}(\cdot,t)\|_{L ^{2}}^{2}+\int _{0}^{t}\left( \|\hat{\Phi}_{t}\|_{L ^{\infty}}+\|\hat{\Phi} _{x}-( \hat{\Phi}+M)V_{x}^{a}\| _{L ^{\infty}} ^{2} \right) \|\Phi_{x}^{\varepsilon}\| _{L ^{2}}^{2}\mathrm{d}\tau +C \varepsilon ^{4 \iota _{0}-\frac{5}{2}},\nonumber
\end{align}
where we have used \eqref{con-R-1-in-lem} \eqref{b-vfi-L-2-l-infty}, \eqref{hat-fida-err-esti1}, \eqref{con-L2-PERT}, \eqref{con-V-x-L2}, $ \iota _{0}<\frac{7}{12} $ from \eqref{range-iota0}, $ 0<\varepsilon<1 $ and the Cauchy-Schwarz inequality. Therefore by \eqref{hat-fida-esti-con} and the Gronwall inequality, we get
\begin{align}
\displaystyle  \|\Phi _{x}^{\varepsilon}(\cdot,t)\|_{L ^{2}}+ 2\int _{0}^{t}\|\Phi _{\tau}^{\varepsilon}\|_{L ^{2}}^{2}\mathrm{d}\tau \leq C \varepsilon ^{4 \iota _{0}-\frac{5}{2}}.\nonumber
\end{align}
This along with \eqref{con-V-x-L2} gives \eqref{con-fida-v-x-lem}. Thus we finish the proof of Lemma \ref{lem-V-x-per}.
\end{proof}
As a direct consequence of the Lemmas \ref{lem-L2-pertur} and \ref{lem-V-x-per}, we have the following corollary.
\begin{corollary}\label{cor-in-lower-order}
Assume the conditions of Lemmas \ref{lem-L2-pertur} and \ref{lem-V-x-per} hold. Then for any solution $ (\Phi ^{\varepsilon},V ^{\varepsilon}) $ to \eqref{eq-perturbation} on $ [0,T] $ satisfying \eqref{apriori-assumption}, we have
\begin{gather}\label{vfid-l-infty-verify}
\displaystyle \sup _{t \in [0,T]} \left( \varepsilon ^{\frac{1}{2}}\|V ^{\varepsilon}(\cdot,t)\|_{L ^{\infty}}^{2}+\varepsilon ^{\frac{3}{4}- \iota _{0}} \|\Phi ^{\varepsilon}(\cdot, t)\|_{L ^{\infty}} ^{2} \right) +\int _{0}^{T}\|\Phi ^{\varepsilon}\|_{L ^{\infty}}^{2}\mathrm{d}t \leq C \varepsilon ^{2\iota _{0}-1}.
\end{gather}
where $ C>0 $ is a constant depending on $ T $ but independent of $ \varepsilon $.
\end{corollary}
\begin{proof}
  From \eqref{con-L2-PERT}, \eqref{con-fida-v-x-lem}, \eqref{Sobolev-infty}, \eqref{Sobolev-modified}, $ 0< \varepsilon<1 $ and the fact $ \iota _{0}<\frac{7}{12} $, it holds that
  \begin{align}
  \displaystyle \sup _{t \in [0,T]}\|V ^{\varepsilon}\|_{L ^{\infty}}^{2}& \leq \sup _{t \in [0,T]} \left( \|V _{x}^{\varepsilon}\|_{L ^{2}}\|V ^{\varepsilon}\|_{L ^{2}} +\|V ^{\varepsilon}\|_{L ^{2}}^{2}\right) \leq C \varepsilon ^{\frac{2 \iota _{0}-2}{2}+\frac{2 \iota _{0}-1}{2}}+ C \varepsilon ^{2\iota _{0}-1} \leq C \varepsilon ^{2\iota _{0}- \frac{3}{2}},
   \nonumber \\
    \sup _{t \in [0,T]}\|\Phi ^{\varepsilon}\|_{L ^{\infty}}^{2}& \leq \sup _{t \in [0,T]} \left( \|\Phi _{x}^{\varepsilon}\|_{L ^{2}}\|\Phi ^{\varepsilon}\|_{L ^{2}} +\|\Phi ^{\varepsilon}\|_{L ^{2}}^{2}\right) \leq C \varepsilon ^{\frac{1}{2}(4 \iota _{0}-\frac{5}{2})+\frac{2 \iota _{0}-1}{2}}+ C \varepsilon ^{2 \iota _{0}-1} \leq C \varepsilon ^{3 \iota _{0}- \frac{7}{4}}\nonumber
    \end{align}
    and
    \begin{align}
    \displaystyle \int _{0}^{t} \|\Phi ^{\varepsilon}\|_{L ^{\infty}}^{2} \mathrm{d}\tau \leq C\int _{0}^{t} \|\Phi _{x}^{\varepsilon}\|_{L ^{2}}^{2}\mathrm{d}\tau \leq C \varepsilon ^{2\iota _{0}-1}\nonumber
    \end{align}
    for any $ t \in [0,T] $. The proof is complete.
\end{proof}

\begin{remark}
In view of \eqref{con-fida-v-x-lem} and \eqref{vfid-l-infty-verify}, the \emph{a priori} assumption \eqref{apriori-assumption} can be verified. Indeed, the estimates \eqref{con-fida-v-x-lem} and \eqref{vfid-l-infty-verify} hold provided $ \delta $ is suitably small, says $ \delta \leq \delta _{0} $ for some constant $ \delta _{0}<1 $ (see \eqref{firs-l2-nonlinear} and \eqref{Q-3-ESTI}). Therefore if we fix $ \delta= \frac{\delta _{0}}{2} $, then we have
\begin{align}
\displaystyle  \sup _{t \in [0,T]} \left( \varepsilon\|V _{x}^{\varepsilon}(\cdot, t)\|_{L ^{2}}^{2} +\int _{0}^{t}\|\Phi ^{\varepsilon}\|_{L ^{\infty}}^{2}\mathrm{d}\tau \right)  \leq C \varepsilon ^{\iota _{4}} < \delta,\nonumber
\end{align}
provided $ \varepsilon>0 $ is small enough. Hence, for suitably small $ \varepsilon $, all the estimates in Lemmas \ref{lem-L2-pertur}--\ref{lem-V-x-per} and Corollary \ref{cor-in-lower-order} hold true.
\end{remark}
\subsubsection{Higher-order estimates} % (fold)
\label{sub:higher_order_estimates}
To finish the proof of Proposition \ref{prop-pertur}, we need to derive some higher-order estimates for $ (\Phi ^{\varepsilon},V ^{\varepsilon}) $ in this subsection.
\begin{lemma}\label{lem-FIDA-V-XX}
Under the conditions of Lemmas \ref{lem-L2-pertur} and \ref{lem-V-x-per}, the solution $ (\Phi ^{\varepsilon},V ^{\varepsilon}) $ to \eqref{eq-perturbation} on $ [0,T] $ possesses the following estimate:
\begin{align*}
&\displaystyle   \|t ^{\frac{5}{2}}\Phi _{t} ^{\varepsilon}(\cdot,t)\|_{L ^{2}}^{2}+\|t ^{\frac{5}{2}}V _{t}^{\varepsilon}(\cdot,t)\|_{L ^{2}}^{2}+\|t ^{\frac{5}{2}}\Phi _{xx}^{\varepsilon}(\cdot,t)\|_{ L ^{2}}^{2}+\varepsilon\|t ^{\frac{5}{2}}V _{xx}^{\varepsilon}(\cdot,t)\|_{L ^{2}}^{2}
 \nonumber \\
 & \displaystyle \quad + \int _{0}^{t} \left( \|\tau^{\frac{5}{2}}\Phi _{x \tau}^{\varepsilon}\|_{L ^{2}}^{2}+\varepsilon\|\tau^{\frac{5}{2}}V _{x \tau}^{\varepsilon}\|_{L ^{2}}^{2}\right) \mathrm{d}\tau  \leq C\varepsilon ^{4\iota _{0}- \frac{7}{2}},\ \ \ t \in (0,T],
\end{align*}
where $ C>0 $ is a constant independent of $ \varepsilon $.
\end{lemma}
\begin{proof}
Differentiating the equations in \eqref{eq-perturbation} with respect to $ t $, then it follows that
\begin{gather}\label{high-order-eq}
\displaystyle \displaystyle \makebox[-4.7pt]{~} \begin{cases}
\displaystyle \Phi _{tt}^{\varepsilon}=\Phi _{xxt}^{\varepsilon}-\varepsilon ^{\frac{1}{2}}\Phi _{xt}^{\varepsilon}V _{x}^{\varepsilon} -\varepsilon ^{\frac{1}{2}}\Phi _{x}^{\varepsilon}V _{xt}^{\varepsilon} -\Phi _{xt}^{\varepsilon} V_{x}^{a}-\Phi _{x}^{\varepsilon}V_{xt}^{a}- V _{xt}^{\varepsilon}(\Phi_{x}^{a}+M)- V _{x}^{\varepsilon}\Phi_{xt}^{a}+\varepsilon ^{-\frac{1}{2}}\partial _{t}\mathcal{R}_{1}^{\varepsilon},\\[1mm]
  \displaystyle V _{tt}^{\varepsilon}=\varepsilon V _{xxt}^{\varepsilon}-\varepsilon ^{\frac{1}{2}}\Phi _{xt}^{\varepsilon}V ^{\varepsilon}-\varepsilon ^{\frac{1}{2}}\Phi _{x}^{\varepsilon}V _{t}^{\varepsilon}-\Phi _{xt}^{\varepsilon}V^{a}-\Phi _{x}^{\varepsilon}V_{t}^{a}-\Phi_{xt}^{a}V ^{\varepsilon}-(\Phi_{x}^{a}+M)V _{t}^{\varepsilon}
  +\varepsilon ^{-\frac{1}{2}}\partial _{t}\mathcal{R}_{2}^{\varepsilon}.
    \end{cases}\makebox[-5pt]{~}
\end{gather}
Multiplying the first equation in \eqref{high-order-eq} by $ t ^{5}\Phi _{t}^{\varepsilon} $ and then integrating the resulting equation over $ \mathcal{I} $, we get after using integration by parts that
\begin{align}\label{fida-t-diff-iden}
&\displaystyle \frac{1}{2}\frac{\mathrm{d}}{\mathrm{d}t}\int _{\mathcal{I}}  t ^{5}\left\vert\Phi _{t}^{\varepsilon}\right\vert ^{2} \mathrm{d}x
+\int _{\mathcal{I}} t ^{5}\left\vert \Phi _{xt}^{\varepsilon}\right\vert ^{2} \mathrm{d}x
 \nonumber \\
&~\displaystyle = \frac{5}{2}\int _{\mathcal{I}} t ^{4}\left\vert\Phi _{t}^{\varepsilon}\right\vert ^{2}  \mathrm{d}x-\varepsilon ^{\frac{1}{2}}\int _{\mathcal{I}} t ^{5}\Phi _{xt}^{\varepsilon}V_{x}^{\varepsilon}\Phi _{t}^{\varepsilon} \mathrm{d}x - \varepsilon ^{\frac{1}{2}}\int _{\mathcal{I}}t ^{5} \Phi _{x}^{\varepsilon}V _{xt}^{\varepsilon}\Phi _{t}^{\varepsilon} \mathrm{d}x- \int _{\mathcal{I}} t ^{5}\Phi _{xt}^{\varepsilon}V_{x}^{a}\Phi _{t}^{\varepsilon} \mathrm{d}x
 \nonumber \\
 & \displaystyle~ \quad - \int _{\mathcal{I}} t ^{5}\Phi _{x}^{\varepsilon}V_{xt}^{a}\Phi _{t}^{\varepsilon} \mathrm{d}x - \int _{\mathcal{I}} t ^{5} V _{xt}^{\varepsilon}(\Phi_{x}^{a}+M) \Phi _{t}^{\varepsilon} \mathrm{d}x-  \int _{\mathcal{I}}t ^{5}  V _{x}^{\varepsilon}\Phi_{xt}^{a}\Phi _{t}^{\varepsilon} \mathrm{d}x
 \nonumber \\
 &~\displaystyle \quad +  \varepsilon ^{- \frac{1}{2}}\int _{\mathcal{I}} t ^{5}\partial _{t}\mathcal{R}_{1}^{\varepsilon}\Phi _{t}^{\varepsilon} \mathrm{d}x=:\sum _{i=1}^{8}\mathcal{H}_{i}.
\end{align}
Next we estimate $ \mathcal{H}_{i}\,(1 \leq i \leq 7) $ term by term. Clearly, we have $ \mathcal{H} _{1} \leq C  \|\Phi _{t}^{\varepsilon}\|_{L ^{2}}^{2}$. In view of \eqref{con-fida-v-x-lem}, \eqref{Sobolev-modified} and the Cauchy-Schwarz inequality, we have
\begin{align}\label{cal-H-2}
\displaystyle  \mathcal{H}_{2} &\leq \frac{1}{8}\|t ^{\frac{5}{2}}\Phi _{xt}^{\varepsilon}\|_{L ^{2}}^{2} +C\varepsilon \|V _{x}^{\varepsilon}\|_{L ^{2}}^{2}\|t ^{\frac{5}{2}}\Phi _{t}^{\varepsilon}\|_{L ^{\infty}}^{2} \leq (\frac{1}{8}+ C \varepsilon  ^{2\iota _{0}-1})\|t ^{\frac{3}{2}}\Phi _{xt}^{\varepsilon}\|_{L ^{2}}^{2}.
 \end{align}
 Similarly, by \eqref{con-fida-v-x-lem}, \eqref{Sobolev-infty}, $ 0< \varepsilon<1 $, the fact $ \iota _{0} <\frac{7}{12}<\frac{5}{8} $ and the Cauchy-Schwarz inequality, we have
 \begin{align}
 \displaystyle  \mathcal{H}_{3}& \leq \frac{\varepsilon}{16}\|t ^{\frac{5}{2}}V_{xt}^{\varepsilon}\|_{L ^{2}}^{2}+C\|\Phi _{x}^{\varepsilon}\|_{L ^{2}}^{2}\|t ^{\frac{5}{2}}\Phi _{t}^{\varepsilon}\|_{L ^{\infty}}^{2}
  \nonumber \\
  & \displaystyle\leq \frac{\varepsilon}{16}\|t ^{\frac{5}{2}}V_{xt}^{\varepsilon}\|_{L ^{2}}^{2}+C t ^{\frac{5}{2}}\|\Phi _{x}^{\varepsilon}\|_{L ^{2}}^{2}\left( \|\Phi _{t}^{\varepsilon}\|_{L ^{2}}\|\Phi _{tx}^{\varepsilon}\|_{L ^{2}}+\|\Phi _{t}^{\varepsilon}\|_{L ^{2}}^{2} \right)
  \nonumber \\
  &\displaystyle\leq \frac{1}{16}(\varepsilon\|t ^{\frac{5}{2}}V_{xt}^{\varepsilon}\|_{L ^{2}}^{2}+\|t ^{\frac{5}{2}}\Phi _{xt}^{\varepsilon}\|_{L ^{2}}^{2}) +\|\Phi _{t}^{\varepsilon}\|_{L ^{2}}^{2}\|\Phi _{x}^{\varepsilon}\|_{L ^{2}}^{2}(1+\|\Phi _{x}^{\varepsilon}\|_{L ^{2}}^{2})
   \nonumber \\
   & \displaystyle \leq \frac{1}{16}(\varepsilon\|t ^{\frac{5}{2}}V_{xt}^{\varepsilon}\|_{L ^{2}}^{2}+\|t ^{\frac{5}{2}}\Phi _{xt}^{\varepsilon}\|_{L ^{2}}^{2})+C \varepsilon ^{8 \iota _{0}-5}\|\Phi _{t}^{\varepsilon}\|_{L ^{2}}^{2}.\nonumber
 \end{align}
 For $ \mathcal{ H}_{i}~(4 \leq i \leq 7) $, thanks to \eqref{time-weighted-for-high} and the Cauchy-Schwarz inequality, we have
\begin{gather}\label{h-esti}
 \displaystyle \begin{cases}
 	\displaystyle  \mathcal{H}_{4} \leq \frac{1}{8}\|t ^{\frac{5}{2}}\Phi _{xt}^{\varepsilon}\|_{L ^{2}}^{2}+C \|t ^{\frac{5}{2}}V _{x}^{a}\|_{L ^{\infty}}^{2}\|\Phi _{t}^{\varepsilon}\|_{L ^{2}} ^{2} \leq \frac{1}{8}\|t ^{\frac{5}{2}}\Phi _{xt}^{\varepsilon}\|_{L ^{2}}^{2} + C \varepsilon ^{- 1}\|\Phi _{t}^{\varepsilon}\|_{L ^{2}} ^{2},
 	 \\[2mm]
 	  \displaystyle\mathcal{H}_{5} \leq C \|t ^{\frac{5}{2}} V _{xt}^{a}\|_{L ^{\infty}}\left( \|\Phi _{x}^{\varepsilon}\|_{L ^{2}} ^{2}+\|\Phi _{t}^{\varepsilon}\|_{L ^{2}}^{2}\right) \leq C \varepsilon ^{-
 	  \frac{1}{2}}\left( \|\Phi _{x}^{\varepsilon}\|_{L ^{2}} ^{2}+\|\Phi _{t}^{\varepsilon}\|_{L ^{2}}^{2}\right),
 	\\[2mm]
 	\displaystyle \mathcal{H}_{6} \leq  \frac{\varepsilon}{16}\|t ^{\frac{5}{2}}V_{xt}^{\varepsilon}\|_{L ^{2}}^{2}+C \varepsilon ^{-1} \|t ^{\frac{5}{2}}(\Phi _{x}^{a}+M)\|_{L ^{\infty}}^{2}\|\Phi _{t}^{\varepsilon}\|_{L ^{2}}^{2} \leq \frac{\varepsilon}{16}\|t ^{\frac{5}{2}}V_{xt}^{\varepsilon}\|_{L ^{2}}^{2}+ \varepsilon ^{-1}\|\Phi _{t}^{\varepsilon}\|_{L ^{2}}^{2},
 	\\[2mm]
 	 	\displaystyle
 	\mathcal{H}_{7} \leq C \|t ^{\frac{5}{2}}\Phi _{xt}^{a}\|_{L ^{\infty}}\left( \varepsilon ^{\iota _{0}-\frac{1}{4}} \|V _{x}^{\varepsilon}\|_{L ^{2}}^{2}+\varepsilon ^{-\iota _{0}+\frac{1}{4}}\|\Phi _{t}^{\varepsilon}\|_{L ^{2}}^{2} \right) \leq C  \varepsilon ^{\iota _{0}-\frac{1}{4}} \|V _{x}^{\varepsilon}\|_{L ^{2}}^{2}+\varepsilon ^{-\iota _{0}+\frac{1}{4}}\|\Phi _{t}^{\varepsilon}\|_{L ^{2}}^{2},\\[2mm]
 	\displaystyle\mathcal{H}_{8} \leq  \|t ^{\frac{5}{2}}\Phi _{t}^{\varepsilon}\|_{L ^{2}}^{2}+\varepsilon ^{-1}\|t ^{\frac{5}{2}}\partial _{t}\mathcal{R}_{1}^{\varepsilon}\|_{L  ^{2}}^{2}.
 	\end{cases}
 \end{gather}
 Plugging estimates on $ \mathcal{H }_{i}\,(1 \leq i \leq 8) $ into \eqref{fida-t-diff-iden} followed by an integration over $ (0,t) $ for any $ t \in (0,T] $, it follows that
\begin{align}\label{fida-t-final-diff-per}
&\displaystyle \|t ^{\frac{5}{2}}\Phi _{t} ^{\varepsilon}(\cdot,t)\|_{L ^{2}}^{2}+ \int _{0}^{t} \tau ^{\frac{5}{2}}\|\Phi _{x \tau}^{\varepsilon}\|_{L ^{2}}^{2}\mathrm{d}\tau
 \nonumber \\
 &~\displaystyle\leq \frac{\varepsilon}{4}\int _{0}^{t}\tau ^{5}\|V _{x \tau}^{\varepsilon}\| _{L ^{2}}^{2}\mathrm{d}\tau+ C (1+\varepsilon ^{8 \iota _{0}-5}+\varepsilon ^{-1}+\varepsilon ^{-\frac{1}{2}}+\varepsilon ^{- \iota _{0}+\frac{1}{4}})\int _{0}^{t} \|\Phi _{\tau} ^{\varepsilon}(\cdot,\tau)\|_{L ^{2}}^{2}\mathrm{d}\tau
  \nonumber \\
  &~\displaystyle \quad+C \varepsilon ^{- \frac{1}{2}}\int _{0}^{t}\|\Phi _{x}^{\varepsilon}\|_{L ^{2}}^{2} \mathrm{d}\tau
   +\int _{0}^{t}\varepsilon ^{\iota _{0}-\frac{1}{4}} \|V _{x}^{\varepsilon}\|_{L ^{2}}^{2}\mathrm{d}\tau + C \varepsilon ^{-1}\|t ^{\frac{5}{2}}\partial _{t}\mathcal{R}_{1}^{\varepsilon}\|_{L _{T}^{2} L  ^{2}}^{2}
   \nonumber \\
   &~\displaystyle \leq  \frac{\varepsilon}{4}\int _{0}^{t}\tau ^{5}\|V _{x \tau}^{\varepsilon}\| _{L ^{2}}^{2}\mathrm{d}\tau+C \varepsilon ^{4 \iota _{0}- \frac{7}{2}}+C\varepsilon ^{3 \iota _{0}- \frac{9}{4}} +C \varepsilon ^{2 \iota _{0}-\frac{3}{2}}
    \leq  \frac{\varepsilon}{4}\int _{0}^{t}\tau ^{5}\|V _{x \tau}^{\varepsilon}\| _{L ^{2}}^{2}\mathrm{d}\tau+C \varepsilon ^{4 \iota _{0}- \frac{7}{2}},
\end{align}
provided $ C \varepsilon  ^{2\iota _{0}-1} \leq \frac{1}{8} $ in \eqref{cal-H-2},
where we have used $ \frac{1}{2}<\iota _{0}<\frac{7}{12} $, $ 0< \varepsilon <1 $ and Lemmas \ref{lem-R-1-VE}, \ref{lem-L2-pertur} and \ref{lem-V-x-per}. To proceed, multiplying the second equation in \eqref{high-order-eq} by $ t ^{5} V _{t}^{\varepsilon} $, we get after integrating the resulting equation over $ \mathcal{I} $ that
\begin{align}\label{V-t-est-diff-iden}
&\displaystyle \frac{1}{2}\frac{\mathrm{d}}{\mathrm{d}t} \int _{\mathcal{I}} t ^{5}\left\vert V _{t}^{\varepsilon}\right\vert ^{2} \mathrm{d}x+ \varepsilon \int _{\mathcal{I}}t ^{5}\left\vert V _{xt}^{\varepsilon}\right\vert ^{2} \mathrm{d}x
 \nonumber \\
 &~\displaystyle \leq C\int _{\mathcal{I}} \left\vert V _{t}^{\varepsilon}\right\vert ^{2} \mathrm{d}x- \varepsilon ^{ \frac{1}{2}}\int _{\mathcal{I}} t ^{5}\Phi _{xt}^{\varepsilon}V ^{\varepsilon}V _{t}^{\varepsilon} \mathrm{d}x -\int _{\mathcal{I}} \varepsilon ^{ \frac{1}{2}}t ^{5}\Phi _{x}^{\varepsilon}V _{t}^{\varepsilon} V _{t}^{\varepsilon}\mathrm{d}x - \int _{\mathcal{I}} t ^{5}\Phi _{xt}^{\varepsilon}V^{a} V _{t}^{\varepsilon} \mathrm{d}x - \int _{\mathcal{I}}t ^{5} \Phi _{x}^{\varepsilon}V_{t}^{a}V _{t}^{\varepsilon} \mathrm{d}x
  \nonumber \\
  & ~\displaystyle \quad - \int _{\mathcal{I}} t ^{5}\Phi_{xt}^{a} V ^{\varepsilon}V _{t}^{\varepsilon} \mathrm{d}x + \varepsilon ^{- \frac{1}{2}}\int _{\mathcal{I}}t ^{5} \partial _{t}\mathcal{R}_{2}^{\varepsilon}V _{t}^{\varepsilon} \mathrm{d}x-\int _{\mathcal{I}} t ^{5}\left( \Phi_{x}^{a}+M \right)\left\vert V _{t}^{\varepsilon}\right\vert ^{2}  \mathrm{d}x
   \nonumber \\
   &~\displaystyle=:C\int _{\mathcal{I}} \left\vert V _{t}^{\varepsilon}\right\vert ^{2} \mathrm{d}x+\sum _{i=1}^{7}\mathcal{L}_{i},
\end{align}
where, similar to \eqref{h-esti}, $ \mathcal{L}_{i}\,(1 \leq i \leq 6) $ enjoy the following estimates:
\begin{align}\label{L-ESTI-fINAL}
\begin{cases}	
\displaystyle \mathcal{L}_{1} \leq \frac{1}{4}\|t ^{\frac{5}{2}}\Phi _{xt}^{\varepsilon}\|_{ L ^{2}}^{2}+ C \varepsilon \|V ^{\varepsilon}\|_{L ^{\infty}}^{2}\|V _{t}^{\varepsilon}\|_{L ^{2}}^{2},\ \ \ \mathcal{L}_{2}  \leq \frac{\varepsilon}{8}\|t ^{\frac{5}{2}} V _{tx}^{\varepsilon}\|_{L ^{2}}^{2}+\|\Phi ^{\varepsilon}\|_{L ^{\infty}}^{2}\|V _{t}^{\varepsilon}\|_{L ^{2}}^{2},\\[2mm]
\displaystyle \mathcal{L}_{3} \leq \frac{1}{4}\|t ^{\frac{5}{2}}\Phi _{xt}^{\varepsilon}\|_{ L ^{2}}^{2}+C \|V ^{a}\|_{L ^{\infty}}^{2}\|V _{t}^{\varepsilon}\|_{L ^{2}}^{2},\ \ \ \mathcal{L}_{4} \leq \|t ^{\frac{5}{2}}V _{t} ^{a}\|_{L ^{\infty}} \left( \|\Phi _{x}^{\varepsilon}\| _{L ^{2}}^{2}+\|V _{t}^{\varepsilon}\|_{L ^{2}}^{2}\right),\\[2mm]
\displaystyle \mathcal{L}_{5} \leq C \|t ^{\frac{5}{2}}\Phi _{xt}^{a}\|_{L ^{\infty}}\left( \|V ^{\varepsilon}\|_{L ^{2}}^{2}+\|V _{t}^{\varepsilon}\|_{L ^{2}}^{2} \right),\ \ \ \mathcal{L }_{6} \leq \|V _{t}^{\varepsilon}\|_{L ^{2}}^{2}+C \varepsilon ^{-1}\|t ^{\frac{5}{2}}\partial _{t}\mathcal{R}_{2}^{\varepsilon}\|_{L ^{2}}^{2},\\[2mm]
\displaystyle \mathcal{L}_{7} \leq C  \|t ^{\frac{5}{2}}(\Phi _{x}^{a}+M)\|_{L ^{\infty}}\|V _{t}^{\varepsilon}\|_{L ^{2}}^{2} \leq C \|V _{t}^{\varepsilon}\|_{L ^{2}}^{2}.
\end{cases}
\end{align}
Therefore, we integrate \eqref{V-t-est-diff-iden} over $ (0,t)\subset (0,T] $ to get
\begin{align}\label{V-t-final-diff-pert}
&\displaystyle  \|t ^{\frac{5}{2}} V _{t}^{\varepsilon}(\cdot,t)\|_{L ^{2}}^{2}+ \varepsilon\int _{0}^{t}\|\tau ^{\frac{5}{2}}V _{x \tau}^{\varepsilon}\|_{L ^{2}}^{2}\mathrm{d}\tau
 \nonumber \\
 &~\displaystyle \leq \frac{1}{2}\int _{0}^{t}\|\tau ^{\frac{5}{2}}\Phi _{x \tau}^{\varepsilon}\|_{L ^{2}}^{2}\mathrm{d}\tau +C \int _{0}^{t} \left( 1+\varepsilon\|V ^{\varepsilon}\|_{L ^{\infty}}^{2}+\|\Phi^{\varepsilon}\|_{L ^{\infty}}^{2} +\|\tau ^{\frac{5}{2}}(\Phi _{x}^{a}+M)\|_{L ^{\infty}}\right)\|V _{\tau}^{\varepsilon}(\cdot,\tau)\| _{L ^{2}}^{2}\mathrm{d}\tau
  \nonumber \\
  &~\displaystyle \quad +C\|\Phi _{x}^{\varepsilon}\|_{L _{T}^{2}L ^{2}}^{2}+C \|V ^{\varepsilon}\|_{L _{T}^{2}L ^{2}} ^{2}+C \varepsilon ^{-1}\|t ^{\frac{3}{2}}\partial _{t}\mathcal{R}_{2}^{\varepsilon}\|_{L _{T}^{2} L ^{2}}^{2}
  \nonumber \\
  &~\displaystyle \leq  \frac{1}{2}\int _{0}^{t}\|\tau ^{\frac{5}{2}}\Phi _{x \tau}^{\varepsilon}\|_{L ^{2}}^{2}\mathrm{d}\tau+C (1+\varepsilon ^{2\iota _{0}+\frac{1}{2}}+ \varepsilon ^{3 \iota _{0}- \frac{7}{4}})\varepsilon ^{2\iota _{0}-1} + C \varepsilon ^{4\iota _{0}- \frac{7}{2}}
   \nonumber \\
   &~\displaystyle\leq  \frac{1}{2}\int _{0}^{t}\|\tau ^{\frac{5}{2}}\Phi _{x \tau}^{\varepsilon}\|_{L ^{2}}^{2}\mathrm{d}\tau+C \varepsilon ^{4\iota _{0}- \frac{7}{2}}.
\end{align}
where we have used \eqref{time-weighted-for-high}, \eqref{con-L2-PERT}, \eqref{con-fida-v-x-lem}, \eqref{vfid-l-infty-verify}, $ 0< \varepsilon <1 $, $ \iota _{2}\geq \iota _{0} $, $ \frac{1}{2}< \iota _{0}<\frac{7}{12} $ and Lemma \ref{lem-R-2-ESTI}. Combining \eqref{V-t-final-diff-pert} with \eqref{fida-t-final-diff-per}, we arrive at
\begin{align*}%\label{Final-tim-deriv-esti}
&\displaystyle  \|t ^{\frac{5}{2}}\Phi _{t} ^{\varepsilon}(\cdot,t)\|_{L ^{2}}^{2}+\|t ^{\frac{5}{2}}V _{t}^{\varepsilon}(\cdot,t)\|_{L ^{2}}^{2}+ \int _{0}^{t} \left( \|\tau ^{\frac{5}{2}}\Phi _{x \tau}^{\varepsilon}\|_{L ^{2}}^{2}+\varepsilon\|\tau ^{\frac{5}{2}}V _{x \tau}^{\varepsilon}\|_{L ^{2}}^{2} \right) \mathrm{d}\tau \leq C \varepsilon ^{4\iota _{0}- \frac{7}{2}}.
\end{align*}
This along with \eqref{eq-perturbation},\eqref{time-weighted-for-high}, \eqref{con-L2-PERT}, \eqref{con-fida-v-x-lem} and Lemmas \ref{lem-R-1-VE}, \ref{lem-R-2-ESTI} further entails that
\begin{align}
\displaystyle  \varepsilon ^{\frac{7}{2}-4 \iota _{0}}\|t ^{\frac{5}{2}}\Phi _{xx}^{\varepsilon}\|_{ L _{T}^{\infty}L ^{2}}^{2}+\varepsilon ^{\frac{9}{2}-4 \iota _{0}}\|t ^{\frac{5}{2}}V _{xx}^{\varepsilon}\|_{L _{T}^{\infty}L ^{2}}^{2} \leq C.\nonumber
\end{align}
Consequently, we finish the proof of Lemma \ref{lem-FIDA-V-XX}.
\end{proof}

% subsection higher_order_estimates (end)

\subsection{Proof of Theorem \ref{thm-stabi-refor}} % (fold)
\label{sub:proof_of_theorem_ref}
Based on Proposition \ref{prop-pertur}, we prove the estimates of $ \mathcal{E}_{1}^{\varepsilon} $ and $ \mathcal{E}_{2}^{\varepsilon} $ asserted in \eqref{cal-E} and further show the convergence results stated in Theorem \ref{thm-stabi-refor}. Recalling from \eqref{b-v--h1-esti} and Corollaries \ref{cor-V-B-0-INFTY}--\ref{cor-INF-vb1-vfib2}, we get for $ 1<\alpha<5/4 $ that
    \begin{subequations}\label{infty-profiles-for-conver}
   \begin{gather}
  \displaystyle \varepsilon ^{\frac{\alpha}{4}}\left( \|\varphi ^{B,2}\|_{L _{T}^{\infty}L _{z}^{\infty}}+\|\varphi ^{b,2}\|_{L _{T}^{\infty}L _{\xi}^{\infty}} \right)  +\|\varphi^{I,1}\|_{L _{T}^{\infty}L ^{\infty}} \leq C,\\
  \|t ^{\frac{1}{2}}\varphi ^{B,2}\|_{L _{T}^{\infty}H_{z} ^{2}} +\|t ^{\frac{1}{2}}\varphi^{b,2}\|_{L _{T}^{\infty}H _{\xi}^{2}}+\|t ^{\frac{1}{2}}\varphi _{x}^{I,1}\|_{L _{T}^{\infty}L ^{\infty}} \leq C,\label{infty-profiles-for-conver-time}\\
  \displaystyle \| v ^{B,1}\|_{ L ^{\infty}(0,T;L _{z}^{\infty})}+\|v ^{b,1}\|_{ L ^{\infty}(0,T;L _{\xi}^{\infty})}+
  \|v^{I,1}\|_{L _{T}^{\infty} L ^{\infty}}+  \varepsilon ^{-\frac{\alpha}{2}}\|\partial _{x}^{l} b _{v} ^{\varepsilon}\|_{L _{T} ^{\infty}L ^{\infty}} \leq C,
   \end{gather}
   \end{subequations}
   where $ l=0,1 $. By analogous arguments as proving \eqref{b-vfi-L-2-l-infty}, we get by virtue of \eqref{vfi-B-1-HALF-INTY}, \eqref{vfi-B-2-half}, \eqref{INF-VFI-B-2} and \eqref{INF-vfi-b2} that
\begin{align}
\displaystyle
\| b _{\varphi}^{\varepsilon}\|_{L _{T} ^{\infty}L^{\infty}} & \leq C\varepsilon +  \varepsilon\| \varphi ^{B,2}(0,t) \|_{L ^{\infty}((0,T))}\|[(1-x)   {\mathop{\mathrm{e}}}^{- \frac{x}{\varepsilon ^{\nu}}}]\|_{L ^{\infty}}
 \nonumber \\
  & \displaystyle \quad+\varepsilon\| \varphi ^{b,2}(0,t) \|_{L ^{\infty}((0,T))}\|[x   {\mathop{\mathrm{e}}}^{- \frac{1-x}{\varepsilon ^{\nu}}}]\|_{L ^{\infty}}\leq C \varepsilon + C \varepsilon ^{1- \frac{\alpha}{4} } \leq C \varepsilon ^{\frac{5}{8}},
  \label{bvfi-infty-last}\\
 % \|t ^{\frac{1}{2}} b _{\varphi}^{\varepsilon}\|_{L _{T} ^{\infty}L^{\infty}} & \leq C\varepsilon +  \varepsilon\|t ^{\frac{1}{2}} \varphi ^{B,2}(0,t) \|_{L ^{\infty}((0,T))}\|[(1-x)   {\mathop{\mathrm{e}}}^{- \frac{x}{\varepsilon ^{\nu}}}]\|_{L ^{\infty}}
 %\nonumber \\
 % & \displaystyle \quad+\varepsilon\| t ^{\frac{1}{2}}\varphi ^{b,2}(0,t) \|_{L ^{\infty}((0,T))}\|[x   {\mathop{\mathrm{e}}}^{- \frac{1-x}{\varepsilon ^{\nu}}}]\|_{L ^{\infty}}\leq C \varepsilon ,
  %\label{time-weigh-bvfi-infty-last}\\
\displaystyle  \|t ^{\frac{1}{2}}\partial _{x} b _{\varphi}^{\varepsilon}\|_{L _{T} ^{\infty}L^{\infty}} & \leq C\varepsilon +  \varepsilon\| t ^{\frac{1}{2}}\varphi ^{B,2}(0,t) \|_{L ^{\infty}((0,T))}\|[(1-x)   {\mathop{\mathrm{e}}}^{- \frac{x}{\varepsilon ^{\nu}}}]'\|_{L ^{\infty}}
 \nonumber \\
 & \displaystyle \quad+\varepsilon\| t ^{\frac{1}{2}}\varphi ^{b,2}(0,t) \|_{L ^{\infty}((0,T))}\|[x    {\mathop{\mathrm{e}}}^{- \frac{1-x}{\varepsilon ^{\nu}}}]'\|_{L ^{\infty}} \leq C \varepsilon + C \varepsilon ^{1- \nu} \leq C \varepsilon ^{1- \nu},\label{time-pa-x-bvfi-infty-last}
\end{align}
where we have used $ 1<\alpha<5/4 $. Therefore, we get from \eqref{cal-E-iden}, \eqref{vfid-l-infty-verify}, \eqref{infty-profiles-for-conver}, \eqref{bvfi-infty-last}, $ 1< \alpha<1+\nu<5/4 $ and $ \frac{1}{2}< \iota _{0}<\frac{7}{12} $ that
\begin{align}
&\displaystyle  \|\mathcal{E}_{1}^{\varepsilon}\|_{L _{T}^{\infty}L ^{\infty}}
 \nonumber \\
 &~\displaystyle \leq  \varepsilon \left(\|\varphi ^{\varepsilon B,2}\|_{L _{T}^{\infty}L _{z}^{\infty}}+\|\varphi ^{\varepsilon b,2}\|_{L _{T}^{\infty}L _{\xi}^{\infty}}\right) + \|b _{\varphi}^{\varepsilon}\|_{L _{T}^{\infty}L ^{\infty}}+\varepsilon ^{\frac{1}{2}}\|\Phi ^{\varepsilon}\|_{L _{T}^{\infty}L ^{\infty}}
  \nonumber \\
  &~\displaystyle\leq C \varepsilon ^{1- \frac{\alpha}{4}}+C \varepsilon ^{\frac{5}{8}}+C \varepsilon ^{\frac{3 \iota _{0}}{2}-\frac{3}{8}} \leq C \varepsilon ^{\frac{3 \iota _{0}}{2}-\frac{3}{8}}\nonumber
\end{align}
and
\begin{align}
\displaystyle  \|\varphi ^{\varepsilon}- \varphi ^{I,0}\|_{L _{T}^{\infty}L ^{\infty}} &\leq
C \varepsilon ^{\frac{1}{2}}\left( \|\varphi ^{ I,1}\|_{L _{T}^{\infty}L ^{\infty}} +\|\varphi ^{B,\varepsilon}\|_{L _{T}^{\infty}L _{z}^{\infty}}+\|\varphi ^{b,\varepsilon}\|_{L _{T}^{\infty}L _{\xi}^{\infty}}\right)+C  \|\mathcal{E}_{1}^{\varepsilon}\|_{L _{T}^{\infty}L ^{\infty}}
 \nonumber \\
  &\displaystyle\leq C \varepsilon ^{\frac{1}{2}}+C \varepsilon ^{\frac{3 \iota _{0}}{2}-\frac{3}{8}}\leq C \varepsilon ^{\frac{3 \iota _{0}}{2}-\frac{3}{8}},\nonumber
\end{align}
which gives \eqref{vfi-VI-0}. Furthermore,
\begin{align}
\displaystyle  \|\mathcal{E}_{2}^{\varepsilon}\|_{L _{T}^{\infty}L ^{\infty}} & \leq C \varepsilon ^{\frac{1}{2}}\left( \|v ^{ I,1}\|_{L _{T}^{\infty}L ^{\infty}} +\|v ^{ B,1}\|_{L _{T}^{\infty}L _{z}^{\infty}}+\|v ^{ b,1}\|_{L _{T}^{\infty}L _{\xi}^{\infty}}\right)
 \nonumber \\
 & \displaystyle  \quad+C \|b _{v}^{\varepsilon}\|_{L _{T}^{\infty}L _{z}^{\infty}} +\varepsilon ^{\frac{1}{2}}\|V ^{\varepsilon}\|_{L _{T}^{\infty}L ^{\infty}} \leq C \varepsilon ^{\frac{1}{2}}+C \varepsilon ^{\frac{\alpha}{2}}+C  \varepsilon ^{\iota _{0}-\frac{1}{4}}
  \nonumber \\
  & \displaystyle  \leq C  \varepsilon ^{\iota _{0}-\frac{1}{4}}.
\end{align}
This yields \eqref{vfi-VI-01}. To prove \eqref{vfi-VI-02},
which is associated with the estimate on $ \partial _{x}\mathcal{E}_{1}^{\varepsilon}$, we first utilize \eqref{Sobolev-modified}, Lemmas \ref{lem-V-x-per} and \ref{lem-FIDA-V-XX} to derive that
\begin{align}\label{fida-x-infty-final}
\displaystyle \displaystyle \|t ^{\frac{5}{4}}\Phi _{x}^{\varepsilon}\| _{L _{T}^{\infty}L ^{\infty}} \leq C\left(\|\Phi _{x}^{\varepsilon}\|_{L _{T}^{\infty}L ^{2}} +\|\Phi _{x}^{\varepsilon}\|_{L _{T}^{\infty}L ^{2}}^{\frac{1}{2}}\|t ^{\frac{5}{2}}\Phi _{xx}^{\varepsilon}\|_{L _{T} ^{\infty}L ^{2}}^{\frac{1}{2}} \right)  \leq C \varepsilon ^{2 \iota _{0}- \frac{3}{2}}.
\end{align}
Then by \eqref{cal-E-iden}, \eqref{infty-profiles-for-conver-time}, \eqref{time-pa-x-bvfi-infty-last}, \eqref{fida-x-infty-final}, we have
\begin{align*}
&\displaystyle  \|t ^{\frac{5}{4}}\partial _{x}\mathcal{E}_{1}^{\varepsilon}\|_{L _{T}^{\infty}L ^{\infty}}
 \nonumber \\
 &~\displaystyle \leq C\varepsilon ^{\frac{1}{2}} \left(\|t ^{\frac{1}{2}}\varphi _{z}^{ B,2}\|_{L _{T}^{\infty}L _{z}^{\infty}}+\|t ^{\frac{1}{2}}\varphi _{\xi}^{ b,2}\|_{L _{T}^{\infty}L _{\xi}^{\infty}}\right) + C\|t ^{\frac{1}{2}}\partial _{x} b _{\varphi}^{\varepsilon}\|_{L _{T}^{\infty}L ^{\infty}}+\varepsilon ^{\frac{1}{2}}\|t ^{\frac{5}{4}}\Phi _{x}^{\varepsilon}\|_{L _{T}^{\infty}L ^{\infty}}
  \nonumber \\
  &~\displaystyle \leq C \varepsilon ^{2\iota _{0}-1}
\end{align*}
for $ \frac{1}{2}< \iota _{0}<\frac{7}{12} $. This gives \eqref{vfi-VI-02} and completes the proof of Theorem \ref{thm-stabi-refor}.\hfill $ \square $

%\vspace{5mm}
% subsection proof_of_theorem_ref (end)
\appendix

\section{Some basic facts}\label{Appen-Analysis}
To make our work self-contained, below we will list some frequently used Sobolev-type inequalities and an embedding theorem on space-time Sobolev spaces (see also the Appendix of \cite{Corrillo-Hong-Wang-vanishing}).

\begin{lemma}[{\cite[Page 236]{brezis-ham-book}}]
Let $ p>1$. Then for any $ \epsilon>0 $, there exists a positive constant $ C=C(\epsilon,p) $ such that
\begin{gather}\label{l-infty-Sobolev}
\displaystyle \|h\|_{L ^{\infty}} \leq \epsilon \|h _{x}\|_{L ^{p}(\mathcal{I})}+ C\|h\|_{L ^{1}(\mathcal{I})}
\end{gather}
for any $ h \in W ^{1,p}(\mathcal{I}) $.
\end{lemma}
\begin{lemma}\label{lem-}
It holds that
\begin{gather}\label{Sobolev-infty}
\displaystyle \|h\|_{ L ^{\infty}}\leq C \left( \|h\|_{L ^{2}}+\|h\|_{L ^{2}}^{1/2}\|h _{x}\|_{L ^{2}}^{1/2} \right)
\end{gather}
for any $ h \in H ^{1}(\mathcal{I}) $, where $ C>0 $ is a positive constant independent of $ h $.
\end{lemma}
We remark that if $ h \in H _{0}^{1}(\mathcal{I})$, then
\begin{align}\label{Sobolev-modified}
\displaystyle \|h\|_{L ^{\infty}} \leq \sqrt{2}\|h\|_{L ^{2}}^{1/2}\|h _{x}\|_{L ^{2}}^{1/2}\ \  \mbox{and}\ \  \|h\|_{L ^{\infty}} \leq C \|h _{x}(\cdot,t)\|_{L ^{2}},
\end{align}
and that if $ h \in H _{z}^{1}\,(\text{resp}. H _{\xi}^{1} ) $, then
\begin{align}\label{Sobolev-z-xi}
\displaystyle \|h\|_{L _{z}^{\infty}} \leq C \|h \|_{L _{z}^{2}}^{1/2}\|h _{z}\|_{L _{z}^{2}}^{1/2} \leq C \|h\|_{H _{z}^{1}} \,(\text{resp}.~ \|h\|_{L _{\xi}^{\infty}} \leq C \|h \|_{L _{z}^{2}}^{1/2}\|h _{\xi}\|_{L _{\xi}^{2}}^{1/2} \leq C \|h\|_{H _{\xi}^{1}}),
\end{align}
where $ C >0$ is a constant.
\vspace{0.2cm}

The following embedding theorem is also frequently used in our analysis.
\begin{proposition}[\mbox{\cite[Lemma 1.2]{temam-book-2001}}]\label{prop-embeding-spacetime}
Let $ V $, $ H $ and $ V' $ be three Hilbert spaces satisfying $ V \subset H \subset V ' $ with $ V' $ being the dual of $ V$. If a function $ u $ belongs to $ L ^{2} (0,T;V)$ and its time derivatives $ u _{t} $ belongs to $ L ^{2}(0,T;V') $, then
\begin{gather}
\displaystyle u \in C([0,T];H) \ \mbox{ and }\ \|u\|_{L ^{\infty}(0,T;H)} \leq C \left( \|u\|_{L ^{2}(0,T;V)}+\|u _{t}\|_{L ^{2}(0,T;V')} \right),
\end{gather}
where the constant $ C>0 $ depends on $ T $ but independent of $ u $. In particular, we have $ u \in C([0,T];H) $ if $ u \in L ^{2}(0,T;H) $ and $ u _{t} \in L ^{2}(0,T;H) $.
\end{proposition}
\begin{remark}
Proposition \ref{prop-embeding-spacetime} entails the fact that for any $ m \in \mathbb{N} $,
\begin{gather*}
\displaystyle   \left\{ u \vert\, u \in L ^{2}(0,T;X ^{m+2}), u _{t}\in L ^{2}(0,T;X ^{m}) \right\}\hookrightarrow C([0,T];X ^{m+1})
\end{gather*}
continuously, where $ X ^{m}:=H ^{m}\,(resp.\,H _{z}^{m} \, or\, H _{\xi}^{m} )  $.
\end{remark}
Finally, in view of the change of variables in \eqref{bd-layer-variable}, we have for any $ G _{1}(z,t) \in H _{z}^{m} $ and $ G _{2}(\xi,t) \in H _{\xi}^{m} $ with $ m \in \mathbb{N} $ that
\begin{subequations}\label{inte-transfer}
\begin{gather}
\displaystyle \left\|\partial _{x}^{m}G _{1}\left(\frac{x}{\varepsilon ^{1/2}},t \right)\right\|_{L ^{2}}= \varepsilon ^{\frac{1-2m}{4}}\|\partial _{z}^{m}G _{1}(z,t)\|_{L _{z}^{2}},\ \ \  \left\|\partial _{x}^{m}G _{1}\left(z,t \right)\right\|_{L ^{\infty}}= \varepsilon ^{-\frac{m}{2}}\|\partial _{z}^{m}G _{1}(z,t)\|_{L _{z}^{\infty}}, \label{z-transfer}\\
\displaystyle
\displaystyle \left\|\partial _{x}^{m}G _{2}\left(\frac{x-1}{\varepsilon ^{1/2}},t \right)\right\|_{L ^{2}}= \varepsilon ^{\frac{1-2m}{4}}\|\partial _{\xi}^{m}G _{2}(\xi,t)\|_{L _{\xi}^{2}},\ \ \ \left\|\partial _{x}^{m}G _{2}\left(\xi,t \right)\right\|_{L ^{\infty}}= \varepsilon ^{-\frac{m}{2}}\|\partial _{\xi}^{m}G _{2}(\xi,t)\|_{L _{\xi}^{\infty}}. \label{xi-transfer}
\end{gather}
\end{subequations}

\section{Approximation of the leading-order boundary-layer profiles}\label{appen-approxi}
In this section, we shall show that $ (v ^{B,\varepsilon},\varphi ^{B,\varepsilon})\,(\mbox{resp. } (v ^{b,\varepsilon}, \varphi ^{b,\varepsilon})) $ is a good approximation of $ (v ^{ B,0},\varphi ^{ B,1})$ (resp. $ (v ^{ b,0}, \varphi ^{ b,1}) $). More precisely, we will prove \eqref{aprr-esti-v-B}, \eqref{apr-esti-vfi-B} and \eqref{appro-v-b-property}, that is,
\begin{align}\label{con-V-B-0-appendix}
\displaystyle   \Big\|(v ^{B,\varepsilon}-v ^{B,0})\Big(\frac{x}{\varepsilon ^{1/2}},t \Big)\Big\|_{L _{T}^{\infty}L ^{\infty}}+\Big\|\varepsilon ^{k/2}\partial _{x}^{k} (\varphi ^{B,\varepsilon}-\varphi  ^{B,1})\Big(\frac{x}{\varepsilon ^{1/2}},t\Big)\Big\|_{L _{T}^{\infty}L ^{\infty}} \leq C\varepsilon ^{\frac{\alpha}{2}} ,
\end{align}
and
\begin{align}\label{con-V-b-0-appendix}
\displaystyle   \Big\|(v ^{b,\varepsilon}-v ^{b,0})\Big(\frac{x}{\varepsilon ^{1/2}},t \Big)\Big\|_{L _{T}^{\infty}L ^{\infty}}+\Big\|\varepsilon ^{k/2}\partial _{x}^{k} (\varphi ^{b,\varepsilon}-\varphi  ^{b,1})\Big(\frac{x}{\varepsilon ^{1/2}},t\Big)\Big\|_{L _{T}^{\infty}L ^{\infty}} \leq C\varepsilon ^{\frac{\alpha}{2}} .
\end{align}
In what follows, we will focus only on \eqref{con-V-B-0-appendix}. One can prove \eqref{con-V-b-0-appendix} similarly. Firstly, we notice from Lemma \ref{lem-v-B-0} that the problem \eqref{first-bd-layer-pro-appro} admits a unique global solution $ v ^{B,\varepsilon}   $ such that for any $ l \in \mathbb{N} $ and $ T>0 $, $ \langle z \rangle ^{l}\partial _{t}^{k} v ^{B,\varepsilon} \in L _{T}^{2}H _{z}^{6-2k} $ for $ k=0,1,2,3 $ and
\begin{subequations}\label{v-B-0-regularity-appendix}
\begin{gather}
\displaystyle 0 \leq v ^{B,\varepsilon} \leq v _{\ast},\ \ \ \|\langle z \rangle ^{l}v _{t}^{ B,\varepsilon}\|_{L _{T}^{\infty}L _{z}^{2}}+ \|\langle z \rangle ^{l}\partial _{t}^{k}v ^{B,\varepsilon}\|_{L _{T}^{2} H _{z}^{3-2k}} \leq C \ \ \mbox{for  } k=0,1,\label{v-B-0-regularity-appendix-a}\\
\displaystyle  \displaystyle  \|\langle z \rangle ^{l}t ^{\frac{1}{2}}\partial _{t}^{i}\partial _{z}^{4-2i}v ^{B,\varepsilon}\|_{L _{T}^{2}L _{z}^{2}}+\|\langle z \rangle ^{l}t ^{\frac{3}{2}}\partial _{t}^{i}\partial _{z}^{5-2i}v ^{B,\varepsilon}\|_{L _{T}^{2} L _{z}^{2}} \leq C \ \ \mbox{for  }  i=0,1,2,\\
\displaystyle  \|\langle z \rangle ^{l}t ^{2}\partial _{t}^{k}\partial _{z}^{6-2k}v ^{B,\varepsilon}\|_{L _{T}^{2} L _{z}^{2}}\leq C \ \ \mbox{for  } k=0,1,2,3.\label{v-B-0-regularity-c-appendix}
%\displaystyle \textcolor{red}{ \|\langle z \rangle ^{l} t^{5/2}\partial _{t}^{4}v ^{B,\varepsilon}\|_{L _{T}^{2}L _{z}^{2}} \leq Cv _{\ast}.}\label{v-B-0-regularity-d}
\end{gather}
\end{subequations}
This along with the weak convergence argument and the Aubin-Lions lemma (cf. Corollary 4 of \cite{Simon-1987}) yields that the problem \eqref{first-bd-layer-pro} admits a unique global solution $ v ^{B,0} $ such that
\begin{subequations}\label{v-B-0-regularity-ori-appdix}
\begin{gather}
\displaystyle 0 \leq v ^{ B,0} \leq v _{\ast},\ \ \ \|\langle z \rangle ^{l}v _{t}^{ B,0}\|_{L _{T}^{\infty}L _{z}^{2}}+ \|\langle z \rangle ^{l}\partial _{t}^{k}v ^{ B,0}\|_{L _{T}^{2} H _{z}^{3-2k}} \leq C \ \ \mbox{for  } k=0,1,\label{v-B-0-regularity-ori-appdix-a}\\
\displaystyle  \|\langle z \rangle ^{l}t ^{\frac{1}{2}}\partial _{t}^{i}\partial _{z}^{k}v ^{ B,0}\|_{L _{T}^{2}L _{z}^{2}} \leq C\ \ \mbox{for  } 2i+k =4,\ \ \|\langle z \rangle ^{l}t ^{\frac{3}{2}}\partial _{t}^{i}\partial _{z}^{k}v ^{ B,0}\|_{L _{T}^{2} L _{z}^{2}} \leq C \ \ \mbox{for  }  2i+k =5,\label{v-B-0-regularity-c-ori-appdix} \\
\displaystyle \|\langle z \rangle ^{l}t ^{2}\partial _{t}^{i}\partial _{z}^{k}v ^{ B,0}\|_{L _{T}^{2} L _{z}^{2}} \leq C \ \ \mbox{for  }  2i+k =6.
%\displaystyle \textcolor{red}{ \|\langle z \rangle ^{l} t^{5/2}\partial _{t}^{4}v ^{B,\varepsilon}\|_{L _{T}^{2}L _{z}^{2}} \leq Cv _{\ast}.}\label{v-B-0-regularity-d}
\end{gather}
\end{subequations}
With \eqref{v-B-0-regularity-appendix} and \eqref{v-B-0-regularity-ori-appdix}, we are now ready to prove \eqref{con-V-B-0-appendix}. Before the start, we underline that the second estimate in \eqref{con-V-B-0-appendix} follows directly from the first one along with \eqref{vfi-bd-1ord-lt}, \eqref{vfi-bd-1ord-lt-approxi}. Therefore it remains to prove the first estimate in \eqref{con-V-B-0-appendix}. To this end, we denote
\begin{align*}
\displaystyle  V= v ^{B,\varepsilon}-v ^{B,0}- \eta (z)\cala_{1}^{\varepsilon}(t),
\end{align*}
with $ \eta (z) $ as in \eqref{eta-defi}. Then, in view of \eqref{Corre-property}, to prove the first estimate in \eqref{con-V-B-0-appendix}, it suffices to show that
\begin{gather}\label{V-infty-con-apdix}
\displaystyle   \|V\|_{L _{T}^{\infty}L ^{\infty}} \leq C \varepsilon ^{\frac{\alpha}{2}}.
\end{gather}
From $ \eqref{first-bd-layer-pro} $ and $ \eqref{first-bd-layer-pro-appro} $, we see that $ V $ solves
\begin{align}\label{eq-per-in-appdix}
\displaystyle \begin{cases}
	\displaystyle V _{t}=V _{zz}-(\varphi _{x} ^{I,0}(1,t)+M)v ^{I,0}(1,t)(		{\mathop{\mathrm{e}}}^{  v ^{B,\varepsilon}}-		{\mathop{\mathrm{e}}}^{v ^{B,0}})-(\varphi _{x} ^{I,0}(0,t)+M) \mathrm{e} ^{v ^{B,\varepsilon}}(V+\eta (z)\cala_{1}^{\varepsilon}(t))
	  \\
	 \displaystyle \quad \quad-(\varphi _{x} ^{I,0}(0,t)+M) v ^{B,0}\left(		{\mathop{\mathrm{e}}}^{v ^{B,\varepsilon}}- {\mathop{\mathrm{e}}}^{v ^{ B,0}}\right)+\eta''(z)\cala_{1}^{\varepsilon}(t)- \eta(z)\partial _{t}\cala_{1}^{\varepsilon}(t),\\
	 \displaystyle V(0,t)=0,\ \ V(+\infty,t)=0,\\
	 \displaystyle V(z,0)=0.
\end{cases}
\end{align}
Testing the equation \eqref{eq-per-in-appdix} against $ V $, by the integration by parts, we get
\begin{align}\label{appr-CON-diff}
&\displaystyle \frac{1}{2}\frac{\mathrm{d}}{\mathrm{d}t}\int _{\mathbb{R}_{+}}V ^{2} \mathrm{d}z+ \int _{\mathbb{R}_{+}} (\varphi _{x} ^{I,0}(0,t)+M) \mathrm{e} ^{v ^{B,\varepsilon}} V ^{2}\mathrm{d}z+\int _{\mathbb{R}_{+}}V _{z}^{2} \mathrm{d}z
 \nonumber \\
 &~\displaystyle = -\int _{\mathbb{R}_{+}}(\varphi _{x} ^{I,0}(0,t)+M)(v ^{I,0}(0,t)+v ^{B,0})(		{\mathop{\mathrm{e}}}^{  v ^{B,\varepsilon}}-		{\mathop{\mathrm{e}}}^{v ^{B,0}}) V \mathrm{d}z
  \nonumber \\
  &~\displaystyle \quad + \int _{\mathbb{R}_{+}} \eta''(z)\cala_{1}^{\varepsilon}(t) V\mathrm{d}z - \int _{\mathbb{R}_{+}}\eta(z)\partial _{t}\cala_{1}^{\varepsilon}(t) V \mathrm{d}z ,
   \nonumber \\
   & ~\displaystyle \leq C \int _{\mathbb{R}_{+}}\left( \vert V\vert +\vert  \eta (z)\cala_{1}^{\varepsilon}(t)\vert\right)\vert V\vert  \mathrm{d}z+\int _{\mathbb{R}_{+}} \eta''(z)\cala_{1}^{\varepsilon}(t) V\mathrm{d}z + \int _{\mathbb{R}_{+}}\eta(z)\partial _{t}\cala_{1}^{\varepsilon}(t) V \mathrm{d}z
    \nonumber \\
    & ~ \displaystyle \leq C \int _{\mathbb{R}_{+}}V ^{2} \mathrm{d}z+ C \int _{\mathbb{R}_{+}}\left(\vert  \eta (z)\cala_{1}^{\varepsilon}(t)\vert ^{2}+\vert \eta''(z)\cala_{1}^{\varepsilon}(t)\vert ^{2}+  \vert \eta(z)\partial _{t}\cala_{1}^{\varepsilon}(t)\vert ^{2}\right)  \mathrm{d}z,
\end{align}
where we have used \eqref{l-INFTY-VFI-I-0}, \eqref{v-B-0-regularity-ori-appdix}, \eqref{Sobolev-infty} and the Cauchy-Schwarz inequality. Integrating \eqref{appr-CON-diff} over $ (0,t) $ for any $ t \in (0,T] $, by \eqref{con-vfi-I-0-only}, \eqref{Corre-property}, \eqref{v-B-0-regularity-appendix} and the fact that $ \eta $ is smooth and compactly supported, we have
\begin{align*}
\displaystyle  \int _{\mathbb{R}_{+}}V ^{2}(\cdot,t) \mathrm{d}z+\int _{0}^{t}\int _{\mathbb{R}_{+}}\left( V _{z}^{2} +V ^{2} \right) \mathrm{d}z \mathrm{d}\tau \leq C \int _{0}^{t}\int _{\mathbb{R}_{+}}V ^{2} \mathrm{d}z \mathrm{d}\tau+C \varepsilon ^{\alpha}.
\end{align*}
This along with the Gronwall inequality gives
\begin{align}\label{V-L2-CON-apdix}
\displaystyle  \int _{\mathbb{R}_{+}}V ^{2}(\cdot,t) \mathrm{d}z+\int _{0}^{t}\int _{\mathbb{R}_{+}}\left( V _{z}^{2} +V ^{2} \right) \mathrm{d}z \mathrm{d}\tau \leq C \varepsilon ^{\alpha}.
\end{align}
Furthermore, testing the equation \eqref{eq-per-in-appdix} against $ V _{t} $, by similar arguments as deriving \eqref{appr-CON-diff}, we get
\begin{align*}
\displaystyle  \displaystyle \frac{\mathrm{d}}{\mathrm{d}t}\int _{\mathbb{R}_{+}}V _{z} ^{2} \mathrm{d}z+\int _{\mathbb{R}_{+}}V _{t}^{2} \mathrm{d}z \leq  C \int _{\mathbb{R}_{+}}\left(\vert  \eta (z)\cala_{1}^{\varepsilon}(t)\vert ^{2}+\vert \eta''(z)\cala_{1}^{\varepsilon}(t)\vert ^{2}+  \vert \eta(z)\partial _{t}\cala_{1}^{\varepsilon}(t)\vert ^{2}\right)  \mathrm{d}z+ C \int _{\mathbb{R}_{+}}V ^{2} \mathrm{d}z.
\end{align*}
This along with \eqref{con-vfi-I-0-only}, \eqref{Corre-property}, \eqref{v-B-0-regularity-appendix}, \eqref{V-L2-CON-apdix} and the fact $ \eta $ is smooth and compactly supported yields that
\begin{align}\label{V-z-esti-apdix}
 \displaystyle  \int _{\mathbb{R}_{+}}V _{z} ^{2}(\cdot,t) \mathrm{d}z+\int _{0}^{t}\int _{\mathbb{R}_{+}}V _{\tau}^{2} \mathrm{d}z  \mathrm{d}\tau \leq C \varepsilon ^{\alpha}
 \end{align}
 for any $ t \in [0,T] $. Combining \eqref{V-L2-CON-apdix}, \eqref{V-z-esti-apdix} and \eqref{Sobolev-z-xi}, we get
 \begin{align*}
 \displaystyle \|V \|_{L _{T}^{\infty}L ^{\infty}} \leq C \varepsilon ^{\frac{\alpha}{2}}.
 \end{align*}
 This gives \eqref{V-infty-con-apdix}. The proof is complete.

 \vspace{4mm}

 \section*{Acknowledgement} % (fold)
The research of G.-Y. Hong was partially supported by the National Natural Science Foundation (grant Nos. 12201221 and 12371203), the Guangdong Basic and Applied Basic Research Foundation (grant Nos. 2021A1515111038 and 2024A1515012306), the Guangzhou Municipal Science and Technology Project (grant No. 2024A04J3788), and the CAS AMSS-POLYU Joint Laboratory of Applied Mathematics postdoctoral fellowship scheme. The research of Z.A. Wang was supported in part by the Hong Kong RGC GRF grant No. PolyU 15306121 and internal grant No. 4-ZZPY from HKPU.

%\bibliography{chemotaxisref}

\begin{thebibliography}{10}

\bibitem{Adler66}
J.~Adler.
\newblock Chemotaxis in bacteria.
\newblock {\em Science}, 153:708--716, 1966.

\bibitem{Marcel-Lankeit}
M.~Braukhoff and J.~Lankeit.
\newblock Stationary solutions to a chemotaxis-consumption model with realistic
  boundary conditions for the oxygen.
\newblock {\em Math. Models Methods Appl. Sci.}, 29(11):2033--2062, 2019.

\bibitem{brezis-ham-book}
H.~Brezis.
\newblock {\em Functional analysis, {S}obolev spaces and partial differential
  equations}.
\newblock Universitext. Springer, New York, 2011.

\bibitem{Kyudong-vasseur-2020-M3as}
Kyudong C., M.-J. Kang, Y.-S. Kwon, and A.F. Vasseur.
\newblock Contraction for large perturbations of traveling waves in a
  hyperbolic-parabolic system arising from a chemotaxis model.
\newblock {\em Math. Models Methods Appl. Sci.}, 30(2):387--437, 2020.

\bibitem{Corrillo-Hong-Wang-vanishing}
J.A. Carrillo, G.-Y. Hong, and Z.-A. Wang.
\newblock Convergence of boundary layers of chemotaxis models with physical
  boundary conditions~{I}: degenerate initial data.
\newblock {\em SIAM J. Math. Anal.}, 56(6):7576--7643, 2024.

\bibitem{Carrillo-li-wang-PLms}
J.A. Carrillo, J.~Li, and Z.-A. Wang.
\newblock Boundary spike-layer solutions of the singular {K}eller-{S}egel
  system: existence and stability.
\newblock {\em Proc. Lond. Math. Soc. (3)}, 122(1):42--68, 2021.

\bibitem{carrillo2024boundary}
J.A. Carrillo, J.~Li, Z.-.A Wang, and W.~Yang.
\newblock Boundary spike-layer solutions of the singular {K}eller-{S}egel
  system: existence, profiles and stability.
\newblock {\em arXiv preprint arXiv:2410.09572}, 2024.

\bibitem{ChoiK2}
M.~Chae and K.~Choi.
\newblock Nonlinear stability of planar traveling waves in a chemotaxis model
  of tumor angiogenesis with chemical diffusion.
\newblock {\em J. Differential Equations}, 268(7):3449--3496, 2020.

\bibitem{chertock2012sinking}
A.~Chertock, K.~Fellner, A.~Kurganov, A.~Lorz, and P.-A. Markowich.
\newblock Sinking, merging and stationary plumes in a coupled chemotaxis-fluid
  model: a high-resolution numerical approach.
\newblock {\em J. Fluid Mech.}, 694:155--190, 2012.

\bibitem{ChoiK1}
K.~Choi, M.-J. Kang, and A.F. Vasseur.
\newblock Global well-posedness of large perturbations of traveling waves in a
  hyperbolic-parabolic system arising from a chemotaxis model.
\newblock {\em J. Math. Pures Appl. (9)}, 142:266--297, 2020.

\bibitem{evans-book}
L.C. Evans.
\newblock {\em Partial differential equations}, volume~19 of {\em Graduate
  Studies in Mathematics}.
\newblock American Mathematical Society, Providence, RI, second edition, 2010.

\bibitem{fan2017global}
L.L. Fan and H.-Y. Jin.
\newblock Global existence and asymptotic behavior to a chemotaxis system with
  consumption of chemoattractant in higher dimensions.
\newblock {\em J. Math. Phys.}, 58(1):011503, 22, 2017.

\bibitem{Grenier-1998-JDE}
E.~Grenier and O.~Gu\`es.
\newblock Boundary layers for viscous perturbations of noncharacteristic
  quasilinear hyperbolic problems.
\newblock {\em J. Differential Equations}, 143(1):110--146, 1998.

\bibitem{Holems-1995-book}
M.H. Holmes.
\newblock {\em Introduction to perturbation methods}, volume~20.
\newblock Springer-Verlag, New York, 1995.

\bibitem{Hong-Wang-large-time}
G.-Y. Hong and Z.-A. Wang.
\newblock Asymptotic stability of exogenous chemotaxis systems with physical
  boundary conditions.
\newblock {\em Quart. Appl. Math.}, 79:717--743, 2021.

\bibitem{hou-liu-wang-wang-2018-SIAM}
Q.~Hou, C.-J. Liu, Y.-G. Wang, and Z.-A. Wang.
\newblock Stability of boundary layers for a viscous hyperbolic system arising
  from chemotaxis: one-dimensional case.
\newblock {\em SIAM J. Math. Anal.}, 50(3):3058--3091, 2018.

\bibitem{hou-wang-2019-JMPA}
Q.~Hou and Z.-A. Wang.
\newblock Convergence of boundary layers for the {K}eller-{S}egel system with
  singular sensitivity in the half-plane.
\newblock {\em J. Math. Pures Appl. (9)}, 130:251--287, 2019.

\bibitem{Hou-wang-2016-JDE}
Q.~Hou, Z.-A. Wang, and K.~Zhao.
\newblock Boundary layer problem on a hyperbolic system arising from
  chemotaxis.
\newblock {\em J. Differential Equations}, 261(9):5035--5070, 2016.

\bibitem{KellerOdell75ms}
E.F. Keller and G.M. Odell.
\newblock Necessary and sufficient conditions for chemotactic bands.
\newblock {\em Math. Biosci.}, 27:309--317, 1975.

\bibitem{keller1971traveling}
E.F. Keller and L.A. Segel.
\newblock Traveling bands of chemotactic bacteria: a theoretical analysis.
\newblock {\em J. Theor. Biol.}, 30(2):235--248, 1971.

\bibitem{lankeit2021radial}
J.~Lankeit and M.~Winkler.
\newblock Radial solutions to a chemotaxis-consumption model involving
  prescribed signal concentrations on the boundary.
\newblock {\em Nonlinearity}, 35(1):719, 2021.

\bibitem{lee2019boundary}
C.-C. Lee, Z.-A. Wang, and W.~Yang.
\newblock Boundary-layer profile of a singularly perturbed non-local
  semi-linear problem arising in chemotaxis.
\newblock {\em Nonlinearity}, 33(10):5111--5141, 2020.

\bibitem{lee2015numerical}
H.G. Lee and J.~Kim.
\newblock Numerical investigation of falling bacterial plumes caused by
  bioconvection in a three-dimensional chamber.
\newblock {\em Eur. J. Mech. B Fluids}, 52:120--130, 2015.

\bibitem{zhaokun-2015-JDE}
H.C. Li and K.~Zhao.
\newblock Initial-boundary value problems for a system of hyperbolic balance
  laws arising from chemotaxis.
\newblock {\em J. Differential Equations}, 258(2):302--338, 2015.

\bibitem{LLW}
J.~Li, T.~Li, and Z.-A. Wang.
\newblock Stability of traveling waves of the {K}eller-{S}egel system with
  logarithmic sensitivity.
\newblock {\em Math. Models Methods Appl. Sci.}, 24(14):2819--2849, 2014.

\bibitem{li2023stability}
J.~Li and X.~Li.
\newblock Stability of stationary solutions to a multidimensional
  parabolic--parabolic chemotaxis-consumption model.
\newblock {\em Math. Models Methods Appl. Sci.}, 33(14):2879--2904, 2023.

\bibitem{Li-pan-zhao-2012-Siam-applied}
T.~Li, R.H. Pan, and K.~Zhao.
\newblock Global dynamics of a hyperbolic-parabolic model arising from
  chemotaxis.
\newblock {\em SIAM J. Appl. Math.}, 72(1):417--443, 2012.

\bibitem{LW091}
T.~Li and Z.-A. Wang.
\newblock Nonlinear stability of traveling waves to a hyperbolic-parabolic
  system modeling chemotaxis.
\newblock {\em SIAM J. Appl. Math.}, 70:1522--1541, 2009.

\bibitem{martinez2018asymptotic}
V.R. Martinez, Z.-A. Wang, and K.~Zhao.
\newblock Asymptotic and viscous stability of large-amplitude solutions of a
  hyperbolic system arising from biology.
\newblock {\em Indiana Univ. Math. J.}, 67:1383--1424, 2018.

\bibitem{nishida-1978}
T.~Nishida.
\newblock {\em Nonlinear hyperbolic equations and related topics in fluid
  dynamics}.
\newblock D\'{e}partement de Math\'{e}matique, Universit\'{e} de Paris-Sud,
  Orsay, 1978.
\newblock Publications Math\'{e}matiques d'Orsay, No. 78-02.

\bibitem{peng2018nonlinear}
H.Y. Peng and Z.-A. Wang.
\newblock Nonlinear stability of strong traveling waves for the singular
  keller--segel system with large perturbations.
\newblock {\em J. Differential Equations}, 265(6):2577--2613, 2018.

\bibitem{Rousset-2005-JDE}
F.~Rousset.
\newblock Characteristic boundary layers in real vanishing viscosity limits.
\newblock {\em J. Differential Equations}, 210(1):25--64, 2005.

\bibitem{Schw1}
H.R. Schwetlick.
\newblock Travelling fronts for multidimensional nonlinear transport equations.
\newblock {\em Ann. Inst. H. Poincar\'{e} Anal. Non Lin\'{e}aire},
  17(4):523--550, 2000.

\bibitem{Simon-1987}
J.~Simon.
\newblock Compact sets in the space {$L^p(0,T;B)$}.
\newblock {\em Ann. Mat. Pura Appl. (4)}, 146:65--96, 1987.

\bibitem{song2023convergence}
X.~Song and J.~Li.
\newblock Convergence rate of solutions towards spiky steady state for the
  keller--segel system with logarithmic sensitivity.
\newblock {\em Nonlinear Analysis}, 232:113284, 2023.

\bibitem{tao2011boundedness}
Y.S. Tao.
\newblock Boundedness in a chemotaxis model with oxygen consumption by
  bacteria.
\newblock {\em J. Math. Anal. Appl.}, 381(2):521--529, 2011.

\bibitem{tao2012eventual}
Y.S. Tao and M.~Winkler.
\newblock Eventual smoothness and stabilization of large-data solutions in a
  three-dimensional chemotaxis system with consumption of chemoattractant.
\newblock {\em J. Differential Equations}, 252(3):2520--2543, 2012.

\bibitem{temam-book-2001}
R.~Temam.
\newblock {\em Navier-{S}tokes equations}.
\newblock AMS Chelsea Publishing, Providence, RI, 2001.
\newblock Theory and numerical analysis, Reprint of the 1984 edition.

\bibitem{tuval2005bacterial}
I.~Tuval, L.~Cisneros, C.~Dombrowski, C.-W. Wolgemuth, J.-O. Kessler, and R.-E.
  Goldstein.
\newblock Bacterial swimming and oxygen transport near contact lines.
\newblock {\em Proc. Natl. Acad. Sci. USA}, 102(7):2277--2282, 2005.

\bibitem{wang2021cauchy}
D.~Wang, Z.-A. Wang, and K.~Zhao.
\newblock Cauchy problem of a system of parabolic conservation laws arising
  from the singular {K}eller-{S}egel model in multi-dimensions.
\newblock {\em Indiana Univ. Math. J.}, 70(1):1--47, 2021.

\bibitem{wang2024smooth}
Y.~Wang, M.~Winkler, and Z.~Xiang.
\newblock Smooth solutions in a three-dimensional chemotaxis-stokes system
  involving dirichlet boundary conditions for the signal.
\newblock {\em NoDEA Nonlinear Differential Equations Appl.}, 31(5):87, 20
  pages, 2024.

\bibitem{wang2013mathematics}
Z.-A. Wang.
\newblock Mathematics of traveling waves in chemotaxis--review paper.
\newblock {\em Discrete Contin. Dyn. Syst. Ser. B}, 18(3):601--641, 2013.

\end{thebibliography}
%\bibliographystyle{plain}
%%\bibliographystyle{mysiam}
\end{document}